\documentclass[11pt,a4paper,oneside]{report}
\usepackage[utf8]{inputenc}
\usepackage[T1]{fontenc}
\usepackage{amsmath,amssymb,amsthm}
\usepackage{mathrsfs}
\usepackage{enumerate}
\usepackage{titlesec}
\usepackage{wasysym}
\usepackage{cite}
\usepackage{graphicx}
\usepackage{braket}
\usepackage[dvipsnames]{xcolor}
\usepackage{tikz}
\usepackage{tcolorbox}
\usepackage{hyperref}
\usepackage{geometry}
\usepackage{esint}

\tcbuselibrary{breakable}

\newtheorem{theorem}{{Theorem}}[section]
\newtheorem{lemma}[theorem]{{Lemma}}
\newtheorem{proposition}[theorem]{{Proposition }}
\newtheorem{definition}[theorem]{{Definition}}
\newtheorem{bemerkung}[theorem]{{Remark}}

\newtheorem{korollar}[theorem]{{Corollary}}

\numberwithin{equation}{section}

\newcommand{\Sp}{\mathbb S}

\newcommand{\N}{\mathbb{N}}
\newcommand{\R}{\mathbb{R}}

\definecolor{gray75}{gray}{0.75}
\newcommand{\hsp}{\hspace{20pt}}
\titleformat{\chapter}[hang]{\Huge\bfseries}{\thechapter\hsp\textcolor{gray75}{|}\hsp}{0pt}{\Huge\bfseries}

\begin{document}
\pagenumbering{roman}

\Huge
 \begin{center}
Parabolic and Elliptic Schauder Theory on Manifolds for a Fourth-Order Problem\\ with a First- and a Third-Order Boundary Condition\\ 

\hspace{10pt}

\Large
Jan-Henrik Metsch$^*$ \\

\hspace{10pt}

\normalsize  
$^*$ Corresponding author: J.-H. Metsch, Department of Mathematics,\\ University of Freiburg, Germany ({\tt\small jan.metsch@math.uni-freiburg.de})

\end{center}

\hspace{10pt}

\normalsize

\setcounter{tocdepth}{1}
\begin{center}
\textbf{Abstract}
\end{center}
   These notes provide a self-contained introduction to Schauder theory on manifolds. First, we derive Schauder estimates for a fourth-order parabolic linear problem with a first- and third-order boundary condition on a smooth compact manifold $M$. Then, applying these, we prove the existence of solutions and study the associated elliptic problem. \\  
   


\newpage
   \begingroup
\let\clearpage\relax
\tableofcontents
\endgroup

\section*{The Purpose of these Notes}
In these notes, I present a self-contained derivation of the Schauder estimates for a fourth-order parabolic problem, which I have encountered during my Ph.D. studies. Even though this problem has essentially been studied in the literature, I could not track down a completely self-contained proof in the literature that does not require potential theory. Among other things, it was, therefore, that my supervisor Prof. Dr. Ernst Kuwert, encouraged me to write up a proof myself and include it as an appendix in my dissertation. In the hope that it will be helpful to others, I then wrote this appendix in a form that resembles more lecture notes rather than an article.\\

Even though some of the proofs I present here are novel, there is probably nothing in these notes that is new to any expert in the field.\\

Finally, I wish to mention that the thesis `Biharmonischer Wärmefluss' of Tobias Lamm \cite{LammBiharmonisch} was immensely helpful throughout the writing of these notes. In the first part of his thesis, Lamm follows a similar ambition and establishes Schauder estimates for higher-order problems with Dirichlet boundary conditions. Still, the problem I study here is not covered by Lamm's analysis and therefore needed `its own proof'. 

\section*{Introduction}
Schauder estimates are a priori estimates for solutions of linear partial differential equations (PDEs) that are immensely useful for studying linear and nonlinear problems. The pioneering work, which includes the first formulation and proof of Schauder estimates, is the 1934 article \cite{schauder01} by Juliusz Schauder himself. Given a sufficiently regular domain $\Omega\subset\R^ n$, sufficiently regular functions $a^ {ij}$, $b^ i$ and $c$ as well as  $f\in C^ {0,\gamma}(\Omega)$ and $\varphi\in C^ {2,\gamma}(\partial\Omega)$, he considers the problem
\begin{equation}\nonumber
\begin{array}{ll}
\left\{ 
\begin{aligned}
a^{ij}\partial_{ij}u+b^i\partial_i u+cu=f&\hspace{.5cm}\textrm{in }\Omega,\\
u=\varphi&\hspace{.5cm}\textrm{along  }\partial\Omega.\\
\end{aligned}
\right.
\end{array}
\end{equation}
Schauder proves that any solution $u\in C^ {2,\gamma}(\Omega)$ satisfies the estimate
\begin{equation}\nonumber
\|u\|_{C^ {2,\gamma}(\Omega)}\leq C\left(\|f\|_{C^ {0,\gamma}(\Omega)}+\|\varphi\|_{C^ {2,\gamma}(\partial\Omega)}+\|u\|_{L^ \infty(\Omega)}\right).
\end{equation}
Without claiming any completeness or historical accuracy, we collect some well-known and classical sources that have generalized the work of Schauder to include general boundary conditions, higher-order problems, systems of PDEs and parabolic problems. First, we mention the articles \cite{adn1} and \cite{adn2}, in which Agmon Douglis and Nirenberg generalize to systems of elliptic equations with general boundary conditions. Parabolic equations have, for example, been studied in the book of Ladyženskaja, Solonnikov and Ural'ceva \cite{ladyvzenskaja1988linear}, the book of Eidelman\cite{eidelman}, the book of Eidelman and Zhitarashu \cite{eidelman1999parabolic} and the book of Friedmann \cite{friedman2008partial}.\\

The classical approach to deriving Schauder estimates is based on potential theoretic estimates. A modern approach that replaced these intricate estimates with a blow-up argument has been introduced by Simon in the pioneering article \cite{simon}.\\

Simon first considers constant coefficient operators that he assumes to be hypoelliptic (i.e. every solution $u\in L^2_{\operatorname{loc}}$ is smooth). He then generalizes to more general situations by perturbative arguments.\\

 To execute Simon's method, one generally only has to verify a hypoellipticity assumption, which he uses to give an elegant proof of certain \emph{Liouville type theorems}. To verify hypoellipticity, the go-to strategy is to develop suitable elliptic or parabolic regularity theory for weak solutions. While these theories are very important, these notes aim to develop Schauder theory and not $L^2$-theory. Therefore the proof we give here follows another approach: First, Liouville-type theorems are proven for smooth solutions and then extended to weaker solutions by mollification. Likewise, when discussing existence theory, we avoid introducing the concept of parabolic Sobolev spaces and weak solutions in favor of a Galerkin approximation. Still, even this approach requires the weak theory of the Neumann Laplacian, which we, therefore, also establish.\\

 \paragraph{Acknowledgements}\ \\
 I want to thank my supervisor Prof. Dr. Ernst Kuwert, for his encouragement to work out these notes, as doing so massively improved my understanding of the subject. Additionally, I want to thank him for his guidance and the many helpful discussions.

\newpage
\section*{Structure of these Notes}
These notes are split into four chapters. In the following, we summarize their contents. Note that some of the notions will be defined later in the text.\\

\hyperref[SchauderCh_Ch01FunctionalSpaces]{\textbf{Chapter 1}} recalls the definition of Hölder and Sobolev spaces and extends them to functions defined on manifolds. Additionally, the parabolic Hölder spaces $C^{k,[k/4],\gamma}(\Omega\times[0,T])$ are introduced and generalised to the manifold setting. Lastly, we prove various \emph{interpolation estimates} for functions in  $C^{k,[k/4],\gamma}(\Omega\times[0,T])$.\\

\hyperref[SchauderCh_Ch02NeumannLaplacian]{\textbf{Chapter 2}} discusses the Neumann Laplacian on a Riemannian manifold $(M,g)$. We establish the existence of an $L^2_g(M)$-orthonormal basis of $L^2_g(M)$ consisting of eigenfunctions $\phi_k$ -- that is solutions to 
 $$\begin{array}{rll}
    \displaystyle-\Delta_g \phi_k\hspace{-.2cm}&\displaystyle=\lambda_k \phi_k&\displaystyle\textrm{in $M$},\vspace{.2cm}\\
    \displaystyle\frac{\partial \phi_k}{\partial\nu}&\displaystyle=0\hspace{-.2cm}&\displaystyle\textrm{along $\partial M$}.
\end{array}$$
By first working on $\R^n$ and then lifting the results to $M$, we then develop elliptic regularity theory for the more general problem 
 $$
\left\{
\begin{aligned}
    -\Delta_g u&=f\hspace{.5cm}\textrm{in $M$},\\
    \frac{\partial u}{\partial\nu}&=h\hspace{.5cm}\textrm{along $\partial M$}.
\end{aligned}
\right.
$$
Thereby we establish smoothness and suitable estimates for the eigenfunctions $\phi_k$. Our main result is Theorem \ref{SchauderCh_neumannlaplacianResultsTheorem}. \\

\hyperref[SchauderCh_Ch03SchauderEstiamtesRn]{\textbf{Chapter 3}} establishes parabolic Schauder estimates on domains $V\times[0,T]\subset\R^n\times\R$ by following the strategy from Simon \cite{simon}. First, suitable Liouville-type theorems are derived that allow us to execute Simon's blow-up argument. Still following \cite{simon}, the Schauder estimates are localized to parabolic balls $U_R(p)$ and extended to non-constant operators. Finally, a covering argument is used to deduce Schauder estimates on general domains. The main results here are Corollaries \ref{SchauderCh_InteriorSchauderEstimateCorollary} and \ref{SchauderCh_BoundarySchauderEstimateCorollary}\\

\hyperref[SchauderCh_Ch04SchauderTheorMfd]{\textbf{Chapter 4}} is split into four parts.  In Section \ref{SchauderCh_Ch4Section1}, the results from Chapter 3 are lifted onto smooth, compact manifolds. The main result is Theorem \ref{SchauderCh_GoodAPrioriEstimateOnManifold}. In Section \ref{SchauderCh_Ch4Sec2}, we specialize to \emph{canonical boundary conditions}. These allow for a variational formulation and hence the derivation of an $L^2$ a priori estimate. The main result is Corollary \ref{SchauderCh_SchauderEstimateonmanifolds}. In Section \ref{SchauderCh_Ch04Sec3}, we study existence of solutions -- the main result is Theorem \ref{SchauderCh_FinalExistenceThmParabolic}. In Section \ref{SchauderCh_Ch04Sec4}, we refine Corollary \ref{SchauderCh_SchauderEstimateonmanifolds} by establishing a decay estimate for the initial boundary value problem 
$$
\left\{
\begin{array}{ll}
\displaystyle \dot u+\frac12\Delta(\Delta+2)u=0&\displaystyle \textrm{in $\Sp^2_+\times[0,T]$},\vspace{.15cm}\\
\displaystyle \frac{\partial u}{\partial\nu}=\frac{\partial(\Delta+2)u}{\partial\nu}=0&\displaystyle \textrm{along $\partial\Sp^2_+\times[0,T]$},\vspace{.15cm}\\
 \displaystyle u(0)=u_0&\displaystyle \textrm{on $\Sp^2_+$}.
\end{array}
\right.
$$
The author used this result in \cite{Metsch}. Finally, in Section \ref{SchauderCh_Ch04Sec05}, we use the parabolic estimates to derive elliptic Schauder estimates. Given a compact Riemannian manifold, we then use the elliptic estimates to analyze the existence of solutions of the problem 
$$
\left\{
\begin{aligned}
&\Delta_g^2u=f\hspace{2.73cm}\textrm{in $M$,}\\
&\frac{\partial u}{\partial\nu}=h_1,\ \frac{\partial\Delta_g u}{\partial\nu}=h_2\hspace{.5cm}\textrm{along $\partial M$,}\\
&\int_M ud\mu_g=0.
\end{aligned}
\right.$$
The main results are Lemma \ref{SchauderCh_EllipticSchauderEstimate} and Corollary \ref{SchauderCh_EllipticProblemExistenceKorollar}.

\section*{Prerequisites}
These notes assume that the reader is familiar with the following concepts:
\begin{itemize}
    \item Definition and most elementary properties of the Hölder and Sobolev spaces $C^{k,\gamma}(\Omega)$ and $W^{k,p}(\Omega)$. A good reference is the book of Gilbarg and Trudinger \cite{gilbarg1977elliptic} (Chapter 4 for Hölder spaces and Chapter 7 for Sobolev spaces). Another excellent reference is Chapter 5 in Evan's book \cite{evans}.
    \item Some basic knowledge of functional analysis. We do not require anything not covered in a typical introductory course. 
    \item Some basic notions of differential and Riemannian geometry. These are required only in the fourth Chapter, where we lift Schauder estimates from $\R^n$ to a manifold $M$.
\end{itemize}
Whenever we use an advanced result from one of these fields, we will provide a literature reference in the text.

\section*{Notation and Conventions}
\begin{itemize}
\item We use the multi-index notation. A multi-index is a vector $\alpha\in\N_0^n$. For such $\alpha$, $x\in\R^n$ and sufficiently regular $f:\R^n\rightarrow\R$, we put 
$$
x^\alpha:=x_1^{\alpha_1}\cdot...\cdot x_n^{\alpha_n}
\hspace{.5cm}\textrm{and}\hspace{.5cm}
\nabla_\alpha f:=\frac{\partial}{\partial x_1^{\alpha_1}}...\frac{\partial}{\partial x_n^{\alpha_n}}f.
$$  
Additionally for $\alpha=(\alpha_1,...,\alpha_n)$, we put $\alpha!:=\alpha_1!\cdot...\cdot\alpha_n!$ and define the \emph{length} of $\alpha$ as $|\alpha|:=\alpha_1+...+\alpha_n$.  
\item As we deal with functions $u(x,t)$ depending on a space variable $x$ and a time variable $t$, we use the symbol $\nabla$ instead of $D$ to denote spatial derivatives, i.e. $\nabla u$, $\nabla^2 u$ etc.
\item Given a Riemannian manifold $(M,g)$, the metric in local coordinates is denoted by $g_{ij}$. As usual, $g^{ij}$ denotes the inverse of $g_{ij}$.
\item Given a Riemannian manifold $(M,g)$ we denote the Laplace Beltrami operator by $\Delta_g$. If $\partial M\neq\emptyset$, we denote the exterior unit normal along $\partial M$ by $\nu$. 
\item The Hölder exponents $\gamma$ are always some number in $(0,1)$.
\item Given a manifold $M$ with boundary $\partial M$, we say that a chart $\phi:U\rightarrow V$ is an interior chart when $U\cap\partial M=\emptyset$. A \emph{boundary chart} is a chart $\phi:U\rightarrow V$ where $U\cap\partial M\neq\emptyset$ and $V\subset(\R^{n-1}\times[0,\infty))$. That is, the model for a manifold with boundary is the upper half-space. Therefore, the exterior unit normal along $\partial M$ in local coordinates is $\nu=-\left(g^{nn}\right)^{-\frac12}g^{ni}\partial_i$.
\item If $X$ and $Y$ are Banach spaces and $f:X\rightarrow Y$ is a compact embedding, we write $f:X\hookrightarrow\hookrightarrow Y$ or simply $X\hookrightarrow\hookrightarrow Y$, when $f$ is clear from context.
\end{itemize}

\chapter{Functional Spaces}\label{SchauderCh_Ch01FunctionalSpaces}
\pagenumbering{arabic}
\section{Hölder Spaces on \texorpdfstring{$\R^n$}{Rn}}
\subsection{Elliptic Hölder spaces}
Let $\Omega\subset\R^n$ be an open set. We denote the set of $m$-times continuously differentiable functions on $\Omega$ by $C^ m(\Omega)$. \begin{definition}[Differentiability on $\bar\Omega$]\label{SchauderCh_EllipticHolderSpaceClosedDomain}
Let $\Omega\subset\R^ n$ be open and bounded and $u\in C^ m(\Omega)$ We say that $u\in C^ m(\bar\Omega)$, if for all multiindices $\alpha\in\N_0^ n$ with length $|\alpha|\leq m$ the function $\nabla_ \alpha u$ has a continuous extension onto $\bar\Omega$. 
We equip $C^ m(\bar\Omega)$ with the norm 
$$\|u\|_{C^m(\bar\Omega)}:=\sum_{k=0}^m\sup_{x\in \Omega}|\nabla^k u(x)|.$$ 
\end{definition}

For $u\in C^0(\Omega)$ and $\gamma\in(0,1)$ we define the \emph{Hölder seminorm} as
$$
[u]_{\gamma,\Omega}:=\sup_{\substack{x,y\in\Omega\\ x\neq y}}\frac{| u(x)- u(y)|}{|x-y|^\gamma}.
$$
More generally, for $u\in C^m(\Omega)$, we define 
$$[\nabla^ m u]_{\gamma,\Omega}:=\sum_{\substack{\alpha\in\N_0^ n\\ |\alpha|=m}}[\nabla_\alpha u]_{\gamma,\Omega}.$$
\begin{definition}[Hölder Space]
Let $\Omega\subset\R^ n$ be open and bounded. We say that a function $u\in C^m(\bar\Omega)$ belongs to the space $C^ {m,\gamma}(\Omega)$, if 
$$\|u\|_{C^{m,\gamma}(\Omega)}:=\|u\|_{C^m(\Omega)}+[\nabla^m u]_{\gamma,\Omega}<\infty.$$
\end{definition}
The spaces $C^{m,\gamma}(\Omega)$ equipped with the $\|\cdot\|_{C^ {m,\gamma}(\Omega)}$-norm are Banach spaces.\\

Given a function $u\in C^{m,\gamma}(\Omega)$, it is easy to see that $\nabla^ku$ has a unique continuous extension to $\bar\Omega$ for all $0\leq k\leq m$. To make explicit in the notation, that $u$ is in fact defined on $\bar\Omega$, we sometimes write $u\in C^{m,\gamma}(\bar\Omega)$.\\

The following lemma is an important tool that allows estimating the Hölder seminorm by covering arguments.
\begin{lemma}\label{SchauderCh_SemiNormSubbadditiveonConvexSpaces}
Let $\Omega\subset\R^n$ be open and convex and $U_1,...,U_N\subset\Omega$ be open and convex sets such that $\Omega\subset \bigcup_{i=1}^N U_i$. Then for all functions $u\in C^{0,\gamma}(\Omega)$ we have 
$$[u]_{\gamma,\Omega}\leq\sum_{i=1}^N[u]_{\gamma, U_i}.$$
\end{lemma}
\begin{proof}
For any $x\neq y\in\Omega$ we define $\alpha_{x,y}:[0,1]\rightarrow\Omega,\ \alpha_{x,y}(t):=(1-t)x+ty$ and put
$$S(x,y):=\set{1\leq i\leq N\ |\  U_i\cap\alpha_{x,y}([0,1])\neq\emptyset}
\hspace{.5cm}\textrm{and}\hspace{.5cm}
P(x,y):=|S(x,y)|\leq N.$$
To prove the lemma, we prove the following estimate:
\begin{equation}\label{SchauderCh_SubadditivityToprove}
\frac{|u(x)-u(y)|}{|x-y|^\gamma}\leq\sum_{i\in S(x,y)}[u]_{\gamma,U_i}
\end{equation}
We prove Estimate \ref{SchauderCh_SubadditivityToprove} by an inductive argument. First, assume that for some points $x\neq y\in\Omega$, we have $|S(x,y)|=1$. Then we can assume without loss of generality that $\alpha_{x,y}(t)\in U_1$ for all $t\in[0,1]$ and, in particular $x,y\in U_1$. In this case, Estimate \eqref{SchauderCh_SubadditivityToprove} is trivial.\\

Now suppose that we already know Estimate \eqref{SchauderCh_SubadditivityToprove} to be true for all points $x\neq y\in\Omega$ with $P(x,y)\leq m\leq N$. If $m=N$, the lemma is already proven. If $m\leq N-1$ we argue as follows. Let $x\neq y\in\Omega$ with $P(x,y)=m+1$. Without loss of generality we assume that $S(x,y)=\set{1,...,m+1}$ and that $x\in U_1$. We define 
$$t_1:=\sup\set{t\in[0,1]\ |\ \alpha_{x,y}(s)\in U_1\textrm{ for all $s\in[0,t]$}}.$$
If $t_1=1$, then $\alpha_{x,y}(t)\in U_1$ for all $t\in[0,1)$ and Estimate \eqref{SchauderCh_SubadditivityToprove} follows as 
$$
\frac{|u(y)-u(x)|}{|y-x|^\gamma}
=
\lim_{\epsilon\rightarrow 0^+}\frac{|u(\alpha_{x,y}(1-\epsilon))-u(x)|}{|\alpha_{x,y}(1-\epsilon)-x|^\gamma}\leq [u]_{\gamma, U_1}.
$$
If $t_1<1$ we define $z:=\alpha_{x,y}(t_1)$. Since $U_1$ is convex and $\alpha_{x,y}(t_1)\not\in U_1$, we deduce $\alpha_{z,y}(t)\not\in U_1$ for all $t\in[0,1]$ and hence $S(z,y)\subset S(x,y)\backslash\set1$. In particular $P(z,y)\leq P(x,y)-1=m$. By the inductive hypothesis, we have 
\begin{equation}\label{SchauderCh_Subadditivityproof_InductiveHelp}
\frac{|u(z)-u(y)|}{|z-y|^\gamma}\leq\sum_{i\in S(z,y)}[u]_{\gamma,\gamma}.
\end{equation}
For $\epsilon>0$ small we estimate 
$$\frac{|u(x)-u(y)|}{|x-y|^\gamma}\leq \frac{|u(x)-u(\alpha_{x,y}(t_1-\epsilon))|}{|x-y|^\gamma}+\frac{|u(\alpha_{x,y}(t_1-\epsilon))-u(y)|}{|x-y|^\gamma}.$$
Note that $|x-y|\geq |x-\alpha_{x,y}(t_1-\epsilon)|$ and $|x-y|\geq |y-\alpha_{x,y}(t_1-\epsilon)|$. Using that for small $\epsilon>0$ we have $\alpha_{x,y}(t_1-\epsilon)\in U_1$, we get 
$$\frac{|u(x)-u(y)|}{|x-y|^\gamma}\leq [u]_{\gamma,U_1,}+\frac{|u(\alpha(t_1-\epsilon))-u(y)|}{|\alpha_{x,y}(t_1-\epsilon)-y|^\gamma}.$$
As $t_1<1$ we can take the limit $\epsilon\rightarrow 0^+$ and, using Estimate \eqref{SchauderCh_Subadditivityproof_InductiveHelp} as well as $1\in S(x,y)$ and $S(z,y)\subset S(x,y)\backslash\set 1$, we obtain 
$$\frac{|u(x)-u(y)|}{|x-y|^\gamma}\leq
[u]_{\gamma, U_1}+\sum_{i\in S(z,y)}[u]_{\gamma,U_i}
\leq\sum_{i\in S(x,y)}[u]_{\gamma,U_i\gamma}.$$
\end{proof}

\subsection{Parabolic Hölder Spaces}
Let $\Omega\subset\R^n$ be open and $T>0$. We consider $\Omega\times(0,T)\subset\R^n\times\R$. Given a function $u:\Omega\times(0,T)\rightarrow\R$, we denote derivatives with respect to the time variable $t\in(0,T)$ by a dot $\cdot$ or by $\partial_t$ and derivatives with respect to the space variable $x$ by  $\partial_i$, $\partial_{x_i}$, etc. or simply by $\nabla$. 

\begin{definition}
    Let $\Omega\subset\R^n$ be open, $T>0$ and $m\in\N_0$. We define 
    $$C^{m,[m/4]}(\Omega\times(0,T)):=\set{u:\Omega\times(0,T)\rightarrow\R\ |\ \textrm{$\partial_t^j\nabla^k u$ exists for all $4j+k\leq m$ and is continuous}}.$$
\end{definition}

It can quickly be checked that for $u\in C^{m,[m/4]}(\Omega\times(0,T))$ we have $\nabla u\in C^{m-1,[(m-1)/4]}(\Omega\times(0,T))$ and $\partial_t u\in C^{m-4,[m/4]-1}(\Omega\times(0,T))$.\\

Next, we formulate the analog of Definition \ref{SchauderCh_EllipticHolderSpaceClosedDomain}.
\begin{definition}[Differentiability on $\bar\Omega\times[0,T\operatorname{]}$]
Let $\Omega\subset\R^n$ be open, $T>0$ and $u\in C^{m,[m/4]}(\Omega\times(0,T))$. We say that $u\in C^{m,[m/4]}(\bar\Omega\times[0,T])$ if for all multiindices $\alpha\in\N_0^n$ and $j\in\N_0$ such that $|\alpha|+4j\leq m$ the function $\partial_t^j\nabla_\alpha u$ has a continuous extension onto $\bar\Omega\times[0,T]$. We equip $C^{m,[m/4]}(\bar\Omega\times[0,T])$ with the norm 
$$\|u\|_{C^{m,[m/4]}(\Omega\times(0,T))}:=\sum_{4j+|\alpha|\leq m}\sup_{(x,t)\in\Omega\times(0,T)}|\partial_t^j\nabla_\alpha u(x,t)|.$$
\end{definition}

\begin{definition}[Parabolic Hölder Condition]
     Let $\Omega\subset\R^n$ be open, $T>0$, $\alpha,\beta\in(0,1)$ and $u\in C^{0,0}(\Omega\times(0,T))$. We define the \emph{spatial} and \emph{temporal} Hölder seminorm by
     \begin{align*}
    [u]^{\operatorname{space}}_{\alpha,\Omega\times(0,T)}&:=\sup_{\substack{x,y\in\Omega,\ t\in(0,T)\\ x\neq y}}\frac{|u(x,t)-u(y,t)|}{|x-y|^{\alpha}},\\
    [u]^{\operatorname{time}}_{\beta,\Omega\times(0,T)}&:=\sup_{\substack{x\in\Omega,\ t,s\in(0,T)\\ t\neq s}}\frac{|u(x,t)-u(x,s)|}{|t-s|^{\beta}}.
\end{align*}
\end{definition}

\begin{definition}[Parabolic Hölder Spaces]\label{SchauderCh_DefinitionofParabolicHölderSpaces}
Let $\Omega\subset\R^n$ be open and bounded, $T>0$ and $m\in\N_0$. For a function $u\in C^{m,[m/4]}(\Omega\times(0,T))$, we define 
\begin{equation}\label{SchauderCh_parabolicHölderseminormdefinition}
    [u]_{\gamma,\Omega\times(0,T)}^{(m)}:=\sum_{4j+k= m}[\partial_t^j\nabla^k u]_{\gamma,\Omega\times(0,T)}^{\operatorname{space}}
+\sum_{0<\frac{m+\gamma-k}4-j<1}[\partial_t^j \nabla^k u]^{\operatorname{time}}_{\frac{m+\gamma-k}4-j,\Omega\times(0,T)}.
\end{equation}
We say that $u\in C^{m,[m/4],\gamma}(\Omega\times(0,T))$, if 
$$\|u\|_{C^{m,[m/4],\gamma}(\Omega\times(0,T))}:=\|u\|_{C^{m,[m/4]}(\Omega\times(0,T))}+ [u]_{\gamma,\Omega\times(0,T)}^{(m)}<\infty.$$
\end{definition}
It can quickly be checked, that $C^{m,[m/4],\gamma}(\Omega\times(0,T))$ equipped with $\|\cdot\|_{C^{m,[m/4],\gamma}(\Omega\times(0,T))}$ is Banach.\\

Given a function $u\in C^{m,[m/4],\gamma}(\Omega\times[0,T])$, it is easy to see that $\partial_t^l\nabla^ku$ has a unique continuous extension to $\bar\Omega\times[0,T]$ whenever $k+4l\leq m$. To make explicit in the notation, that $u$ is in fact defined on $\bar\Omega\times[0,T]$, we sometimes write $u\in C^{m,\gamma}(\bar\Omega\times[0,T])$.

\paragraph{Motivation of Definition \ref{SchauderCh_DefinitionofParabolicHölderSpaces}}\ \\
We give a brief justification for the spatial and temporal Hölder norms appearing in Equation \eqref{SchauderCh_parabolicHölderseminormdefinition}. Given a function $u\in C^{m,[m/4]}(\Omega\times(0,T))$ and $\lambda>0$, we consider $u_\lambda(x,t):=u(\lambda x, \lambda^4 t)$. A quick computation shows
\begin{align*}
[\partial_t^j\nabla^k u_\lambda]_{\alpha,\Omega\times(0,T)}^{\operatorname{space}}&=\lambda^{4j+k}\lambda^\alpha[\partial_t^j\nabla^k u]_{\alpha,\Omega\times(0,T)}^{\operatorname{space}},\\
[\partial_t^j\nabla^k u_\lambda]_{\beta,\Omega\times(0,T)}^{\operatorname{time}}&=\lambda^{4j+k}\lambda^{4\beta}[\partial_t^j\nabla^k u]_{\beta,\Omega\times(0,T)}^{\operatorname{time}}.
\end{align*}
The expected scaling behaviour of `a $(m+\gamma)$-th derivative' is $\lambda^{m+\gamma}$. As $\alpha,\gamma\in(0,1)$ and $j,k,m\in\N_0$, this already implies that for the spatial seminorm we require $4j+k=m$ and $\alpha=\gamma$. In the temporal case, we require $\beta=\frac{m+\gamma-k}4-j$.\\

The following proposition can be proven by following the same arguments that lead to Lemma \ref{SchauderCh_SemiNormSubbadditiveonConvexSpaces}:

\begin{proposition}\label{SchauderCh_SemiNormSubbadditiveonConvexSpaces_Parabolic}
Let $\Omega\subset\R^n$ be open and convex, $T>0$, $U_1,...,U_N\subset\Omega\times(0,T)$ be open and convex sets and $I_1,...,I_N\subset(0,T)$ be intervals such that $\Omega\times(0,T)\subset \bigcup_{i=1}^N U_i\times I_i$. Then for all functions $u\in C^{0,0}(\Omega\times(0,T))$ and $\alpha,\beta\in(0,1)$
$$[u]_{\alpha,\Omega\times(0,T)}^{\operatorname{space}}\leq\sum_{i=1}^N[u]^{\operatorname{space}}_{\alpha, U_i\times I_i}
\hspace{.5cm}\textrm{and}\hspace{.5cm}
[u]_{\beta,\Omega\times(0,T)}^{\operatorname{time}}\leq\sum_{i=1}^N[u]^{\operatorname{time}}_{\beta, U_i\times I_i}.$$
\end{proposition}

Finally, we give the following proposition.
\begin{proposition}[An equivalent Norm]\label{SchauderCh_EquivalentNOrmC41Gamma}\ \\
Let $\Omega$ be a bounded domain. For $u\in C^{4,1,\gamma}(\Omega\times(0,T))$, we define 
$$[D^{4,1}u]^{(0)}_{\gamma,\Omega\times(0,T)}:=
[\nabla^4 u]_{\gamma,\Omega\times(0,T)}^{(0)}
+
[\dot u]_{\gamma,\Omega\times(0,T)}^{(0)}.$$
    Then, the following norm is equivalent to $\|\cdot\|_{C^{4,1,\gamma}(\Omega\times(0,T))}$:
    $$\|u\|'_{C^{4,1,\gamma}(\Omega\times(0,T))}:=\|u\|_{C^{4,1}(\Omega\times(0,T))}+[D^{4,1}u]_{\gamma,\Omega\times(0,T)}$$
\end{proposition}
\begin{proof}
    In view of Proposition \ref{SchauderCh_SemiNormSubbadditiveonConvexSpaces_Parabolic}, this is a consequence of Theorem \ref{SchauderCh_InterpolationEstimates} and Lemma \ref{SchauderCh_D41Estimateslemma}
\end{proof}

\begin{tcolorbox}[colback=white!20!white,colframe=black!100!white,sharp corners, breakable]
In the following we will always use the norm $\|u\|'_{C^{4,1,\gamma}(\Omega\times(0,T)}$ from Proposition \ref{SchauderCh_EquivalentNOrmC41Gamma} without including the `$'$' in the notation. 
\end{tcolorbox}

\section{Technical Lemmas}
\begin{definition}[Parabolic Ball]\label{SchauderCh_ParabolicBallDefnition}
Let $p_0=(x_0, t_0)\in\R^n\times\R$ and $\rho>0$. We define the parabolic ball of radius $\rho$ with center $p_0$ as
$$U_\rho(p_0):=\set{p=(x,t)\in\R^n\times\R\ |\ \max\{|x-x_0|, |t-t_0|^{\frac14}\}<\rho}=B_\rho(x_0)\times(t_0-\rho^4,t_0+\rho^4).$$
Additionally, we define the following truncated parabolic balls:
$$\begin{array}{l l }
    U_\rho^+(p)&\hspace{-.3cm}:=U_\rho(p)\cap(\R^n\times[0,\infty))\\
    U_\rho^-(p)&\hspace{-.3cm}:=U_\rho(p)\cap(\R^n\times(-\infty,0])\\
    U_{\rho+}(p)&\hspace{-.3cm}:=U_\rho(p)\cap(\R^{n-1}\times[0,\infty)\times\R)\\
    U_{\rho+}^+(p)&\hspace{-.3cm}:=U_\rho(p)\cap(\R^{n-1}\times[0,\infty)\times[0,\infty))\\
    U_{\rho+}^-(p)&\hspace{-.3cm}:=U_\rho(p)\cap(\R^{n-1}\times[0,\infty)\times(-\infty,0])
\end{array}$$
\end{definition}

\begin{definition}[Floor and Side Boundary]\label{SchauderCh_FloorSideBoundaryDef}
For $U\subset\R^n\times\R$ we define the \emph{floor boundary} and \emph{side boundary} respectively as 
$$\partial_FU:=U\cap(\R^n\times 0)
\hspace{.5cm}\textrm{and}\hspace{.5cm}
\partial_SU:=U\cap(\R^{n-1}\times 0\times\R).$$
\end{definition}

\begin{theorem}[Interpolation Estimates]\label{SchauderCh_InterpolationEstimates}\ \\
Let $p_0\in\R^n\times\R$, $\rho>0$, $U_\rho:=U_\rho(p_0)$ and $u\in C^{4,1,\gamma}(U_\rho)$. There exists $\epsilon_0(n)>0$ such that for all $\epsilon\in (0,\epsilon_0)$ and $0\leq k\leq 4$
\begin{align*}
    \rho^k \|\nabla^k u\|_{L^\infty(U_\rho)}&\leq \epsilon\rho^{4+\gamma}[D^{4,1} u]^{(0)}_{\gamma,U_\rho}+C(n)\epsilon^{-\frac{k}{4-k+\gamma}}\|u\|_{L^\infty(U_\rho)},\\
     \rho^4 \|\dot u\|_{L^\infty(U_\rho)}&\leq \epsilon\rho^{4+\gamma}[D^{4,1} u]^{(0)}_{\gamma,U_\rho}+C(n)\epsilon^{-\frac{k}{\gamma}}\|u\|_{L^\infty(U_\rho)}.
\end{align*}
Additionally, for $0\leq l\leq 3$ and $\epsilon\in (0,\epsilon_0)$
\begin{align*}
\rho^{l+\gamma} [\nabla^l u]_{\gamma,U_\rho}^{(0)}&\leq \epsilon\rho^{4+\gamma}[D^{4,1} u]^{(0)}_{\gamma,U_\rho}+C(n)\epsilon^{-\frac{l+\gamma}{4-l}}\|u\|_{L^\infty(U_\rho)}.
\end{align*}
\end{theorem}

\begin{proof}     
By translation and scaling, we can assume that $p_0=0$ and $\rho=1$.
Let $ f\in C^{1,0,\gamma}(U_{1})$, $q=(x_1,t_1)\in U_1:=U_1(0)$ and $\sigma\in(0,1)$ such that $U_\sigma(q)\subset U_1$. Denoting the exterior normal along $\partial B_\sigma(x_1)$ by $\nu$ we use the divergence theorem to estimate for $(x,t)\in U_\sigma(q)$
\begin{align}
    |\partial_ if(x,t)|\leq &\frac1{|B_\sigma(x_1)|} 
\left|\int_{B_\sigma(x_1)} \partial_i f(y,t)dy\right|+\frac1{|B_\sigma(x_1)|}\int_{B_\sigma(x_1)}|\partial_i f(x,t)-\partial_i f(y,t)|dy\label{SchauderCh_InterpolationProof02}\\
\leq & \frac1{|B_\sigma(x_1)|} 
\left|\int_{\partial B_\sigma(x_1)} f(y,t)\nu_i dy\right|+\frac1{|B_\sigma(x_1)|}\int_{B_\sigma(x_1)}[\nabla f]_{\gamma,B_\sigma(x_1)}^{\operatorname{space}}|x-y|^\gamma dy\nonumber\\
\leq & \frac{|\partial B_\sigma(x_1)|}{|B_\sigma(x_1)|} 
\|f\|_{L^\infty(B_\sigma(x_1))}+[\nabla f]_{\gamma,B_\sigma(x_1)}^{\operatorname{space}}\sigma^\gamma.\nonumber
\end{align}
We have $|B_\sigma(x_1)|=c(n)\sigma^n$ and $|\partial B_\sigma(x_1)|=c'(n)\sigma^{n-1}$. So, taking the supremum over all $(x,t)$ in $U_\sigma(p)$ and $1\leq i\leq n$, we obtain
\begin{equation}\label{SchauderCh_Interpolationproof01}
    \|\nabla f\|_{L^\infty(U_\sigma(q))}\leq \frac{C(n)}\sigma \|f\|_{L^\infty(U_\sigma(q))}+C(n)\sigma^\gamma[\nabla f]_{\gamma,U_\sigma(q)}^{(0)}.
\end{equation}
Since for all $q'\in U_1$ there exists $q\in U_1$ such that $q'\in U_\sigma(q)\subset U_1$, we deduce that for all $\sigma\in(0,1)$
\begin{equation}\label{SchauderCh_InterpolationProof03}\|\nabla f\|_{L^\infty(U_1)}\leq C(n)\left(\sigma^\gamma[\nabla f]_{\gamma,U_1}^{(0)}+\sigma^{-1}\|f\|_{L^\infty(U_1)}\right).
\end{equation}
Instead of estimating $|\partial_i f(x,t)-\nabla_i f(y,t)|$ with the Hölder seminorm, we can write $|\partial_i f(x,t)-\partial_i f(y,t)|\leq C(n)\|\nabla^2 f\||x-y|$ if $f$ is $C^{2,0,\gamma}(U_1)$. Following the same reasoning that led to \eqref{SchauderCh_InterpolationProof03}, we then obtain
\begin{equation}\label{SchauderCh_InterpolationProof04}
\|\nabla f\|_{L^\infty(U_1)}\leq C(n)\left(\sigma\|\nabla^2 f\|_{L^\infty(U_1)}+\sigma^{-1}\|f\|_{L^\infty(U_1)}\right).
\end{equation}
Taking $f=u$ in Estimate \eqref{SchauderCh_InterpolationProof04}, we get 
\begin{equation}\label{SchauderCh_InterpolationProof04_02}
    \|\nabla u\|_{L^\infty(U_1)}
     \leq C(n)\left(\sigma \|\nabla^2 u\|_{L^\infty(U_1)}
     +
     \sigma^{-1}\| u\|_{L^\infty(U_1)}\right).
\end{equation}
Next, taking $f=\nabla u$ in Estimate \eqref{SchauderCh_InterpolationProof04} and subsequently inserting Estimate \eqref{SchauderCh_InterpolationProof04_02} with $\sigma\rightarrow\mu$, implies 
\begin{align}
    \|\nabla^2 u\|_{L^\infty(U_1)}
     &\leq C(n)\left(\sigma \|\nabla^3 u\|_{L^\infty(U_1)}
     +
     \sigma^{-1}\| \nabla u\|_{L^\infty(U_1)}\right)\nonumber\\
     \ &\leq C(n)\left(\sigma \|\nabla^3 u\|_{L^\infty(U_1)}
     +
     \sigma^{-1}\left(
            \mu \|\nabla^2 u\|_{L^\infty(U_1)}
     +
     \mu^{-1}\| u\|_{L^\infty(U_1)}\label{SchauderCh_InterpolationProof05}
     \right)\right).
\end{align}
Choosing $\mu=\epsilon(n)\sigma$ with a small enough number $\epsilon(n)\in(0,1)$, we obtain
\begin{equation}\label{SchauderCh_InterpolationProof06}
\|\nabla^2 u\|_{L^\infty(U_1)}\leq C(n)\left(
    \sigma\|\nabla^{3} u\|_{L^\infty(U_1)}
    +
    \sigma^{-2}\| u\|_{L^\infty(U_1)}
\right).
\end{equation}
Inserting Estimate \eqref{SchauderCh_InterpolationProof06} into Estimate \eqref{SchauderCh_InterpolationProof04_02}, we get 
\begin{equation}\label{SchauderCh_InterpolationProof07}
\|\nabla u\|_{L^\infty(U_1)}\leq C(n)\left(
    \sigma^2\|\nabla^{3} u\|_{L^\infty(U_1)}
    +
    \sigma^{-1}\| u\|_{L^\infty(U_1)}
\right).
\end{equation}
Next, we take $f=\nabla^2 u$ in Estimate \eqref{SchauderCh_InterpolationProof04} and use Estimate \eqref{SchauderCh_InterpolationProof06} with $\sigma\rightarrow\mu$ to estimate 
\begin{align*}
    \|\nabla^3 u\|_{L^\infty(U_1)}
     &\leq C(n)\left(\sigma \|\nabla^4 u\|_{L^\infty(U_1)}
     +
     \sigma^{-1}\|  \nabla^2u\|_{L^\infty(U_1)}\right)\\
      &\leq C(n)\left(\sigma \|\nabla^4 u\|_{L^\infty(U_1)}
     +
     \sigma^{-1}\left(
            \mu \|\nabla^3 u\|_{L^\infty(U_1)}
     +
     \mu^{-2}\| u\|_{L^\infty(U_1)}
     \right)\right).
\end{align*}
Choosing again $\mu=\epsilon(n)\sigma$ with $\epsilon(n)\in(0,1)$ small enough yields 
\begin{equation}\label{SchauderCh_InterpolationProof08}
    \|\nabla^3 u\|_{L^\infty(U_1)}
     \leq C(n)\left(\sigma \|\nabla^4 u\|_{L^\infty(U_1)}
     +
     \sigma^{-3}\| u\|_{L^\infty(U_1)}\right).
\end{equation}
Inserting Estimate \eqref{SchauderCh_InterpolationProof08} back into Estimates \eqref{SchauderCh_InterpolationProof06} and \eqref{SchauderCh_InterpolationProof07}, we have shown that for $k=1,2,3$
\begin{equation}\label{SchauderCh_InterpolationProof09}
    \|\nabla^k u\|_{L^\infty(U_1)}
     \leq C(n)\left(\sigma^{4-k} \|\nabla^4 u\|_{L^\infty(U_1)}
     +
     \sigma^{-k}\|   u\|_{L^\infty(U_1)}\right).
\end{equation}
Finally, we use Estimate \eqref{SchauderCh_InterpolationProof03} with $f\rightarrow\nabla^3 u$ and subsequently Estimate \eqref{SchauderCh_InterpolationProof09} with $k=3$ to deduce that 
\begin{align*}
     \|\nabla^4 u\|_{L^\infty(U_1)} \leq &C(n)\left(\sigma^{\gamma} [\nabla^4 u]_{\gamma,U_1}^{(0)}
     +
     \sigma^{-1}\|   \nabla^3u\|_{L^\infty(U_1)}\right)\\
     \leq &C(n)\left(\sigma^{\gamma} [\nabla^4 u]_{\gamma,U_1}^{(0)}
     +
     \sigma^{-1}  \left(
        \mu  \|\nabla^4 u\|_{L^\infty(U_1)}+\mu^{-3} \|u\|_{L^\infty(U_1)}
     \right)\right).
\end{align*}
Taking $\mu=\epsilon(n)\sigma$ with small enough $\epsilon(n)\in(0,1)$ gives 
\begin{equation}\label{SchauderCh_InterpolationProof10}
     \|\nabla^4 u\|_{L^\infty(U_1)}\leq C(n)\left(\sigma^{\gamma} [\nabla^4 u]_{\gamma,U_1}^{(0)}
     +
     \sigma^{-4}\| u\|_{L^\infty(U_1)}\right).
\end{equation}
Estimates \eqref{SchauderCh_InterpolationProof09} and \eqref{SchauderCh_InterpolationProof10} together imply that for $0\leq k\leq 4$
\begin{equation}\label{SchauderCh_InterpolationProof10NEW}
     \|\nabla^k u\|_{L^\infty(U_1)}\leq C(n)\left(\sigma^{4-k+\gamma} [\nabla^4 u]_{\gamma,U_1}^{(0)}
     +
     \sigma^{-k}\| u\|_{L^\infty(U_1)}\right),
\end{equation}
which is the first estimate claimed in the theorem. 
To prove the second estimate, let again $q=(x_1,t_1)\in U_1$ and $\sigma>0$ such that $U_\sigma(q)\subset U_1$. Then, for any $(x,t)\in U_\sigma(q)$ we have $t\in(t_1-\sigma^4, t_1+\sigma^4)$. Consequently 
\begin{align}
    |\dot u(x,t)| &\leq \frac1{2\sigma^4}\left|\int_{t_1-\sigma^4}^{t_1+\sigma^4}\dot u(x,s)ds\right|+\frac1{2\sigma^4}\int_{t_1-\sigma^4}^{t_1+\sigma^4}|\dot u(x,t)-\dot u(x,s)|ds\nonumber\\
    &\leq C(n)\sigma^{-4}\| u\|_{L^\infty(U_1)}
    +\frac1{2\sigma^4}\int_{t_1-\sigma^4}^{t_1+\sigma^4}[\dot u]^{\operatorname{time}}_{\frac\gamma4, U_1}|t-s|^{\frac\gamma4}ds\nonumber\\
    &\leq C(n)\left(\sigma^{-4}\|\dot u\|_{L^\infty(U_1)}
    +\sigma^{\gamma}[u]_{\frac\gamma4, U_1}^{\operatorname{time}}\right).\label{SchauderCh_InterpolationProof11}
\end{align}
For any $q'\in U_1$ there exists $q\in U_1$ such that $q'\in U_\sigma(q)\subset U_1$. Thus, Estimate \eqref{SchauderCh_InterpolationProof11} implies the second estimate claimed in the theorem.

\paragraph{Seminorms}\ \\
We are left with estimating the spatial and temporal Hölder seminorms of $\nabla^l u$ for $l=0,1,2,3$. To do so, we consider two points $q=(x,t)\neq (y,s)=q'\in U_1$. We first consider the case where $|t-s|^{\frac14}\geq\sigma$ and $|x-y|\geq\sigma$. Using Estimate \eqref{SchauderCh_InterpolationProof10NEW}, we estimate 
\begin{align}
    &\frac{|\nabla^lu(x,t)-\nabla^lu(x,s)|}{|t-s|^{\frac\gamma4}}+
    \frac{|\nabla^lu(x,t)-\nabla^lu(y,t)|}{|x-y|^\gamma}\nonumber\\
    \leq& 4\sigma^{-\gamma}\|\nabla^l u\|_{L^\infty(U_1)}\nonumber\\
    \leq &C(n)\left(
        \sigma^{4-l}[D^{4,1} u]_{\gamma, U_1}^{(0)}+\sigma^{-l-\gamma}\|u\|_{L^\infty(U_1)}
    \right)\label{SchauderCh_InterpolationProof12}.
\end{align}
Now we consider the case where $|t-s|^{\frac14}\leq\sigma$ or $|x-y|\leq\sigma$ respectively. If $|x-y|\leq\sigma$, we use Estimates \eqref{SchauderCh_InterpolationProof10NEW} to estimate
\begin{align}
    \frac{|\nabla^lu(x,t)-\nabla^lu(y,t)|}{|x-y|^\gamma}&\leq\|\nabla^{l+1} u\|_{L^\infty(U_1)} |x-y|^{1-\gamma}\nonumber\\
    &\leq C(n)\left(\sigma^{4-l}[D^{4,1} u]_{\gamma, U_1}^{(0)}+\sigma^{-l-\gamma}\|u\|_{L^\infty(U_1)}
    \right).\label{SchauderCh_InterpolationProof122}
\end{align}
Estimates \eqref{SchauderCh_InterpolationProof12} and \eqref{SchauderCh_InterpolationProof122} already imply the claimed estimates for the spatial seminorms. It remains to consider the case $|t-s|\leq \sigma^4$. We use Estiamte \eqref{SchauderCh_InterpolationProof11} to estiamte 
\begin{align}
   \frac{ |u(x,t)-u(x,s)|}{|t-s|^{\frac\gamma4}}\leq& \|\dot u\|_{L^\infty(U_1)}|t-s|^{1-\frac{\gamma}4}\nonumber\\
   \leq &C(n)(\sigma^\gamma[D^{4,1} u]_{\gamma, U_1}^{(0)}+\sigma^{-4}\|u\|_{L^\infty(U_1)})\sigma^{4-\gamma}\nonumber\\
   \leq & C(n)\left(\sigma^4 [D^{4,1} u]_{\gamma, U_1}^{(0)}+\sigma^{-\gamma}\|u\|_{L^\infty(U_1)}\right).\label{SchauderCh_InterpolationProof1222}
\end{align}
Taking $l=0$ in Estimates  \eqref{SchauderCh_InterpolationProof12} and \eqref{SchauderCh_InterpolationProof122} and combining with Estimate \eqref{SchauderCh_InterpolationProof1222}, we obtain
\begin{equation}\label{SchauderCh_InterpolationProofuTimeSemi}
[u]_{\gamma,U_1}^{(0)}\leq C(n)\left(\sigma^4 [D^{4,1} u]_{\gamma, U_1}^{(0)}+\sigma^{-\gamma}\|u\|_{L^\infty(U_1)}\right).
\end{equation}
Next, we consider $l=1,2,3$. Let $x\in B_1(0)$, $s,t\in (-1,1)$ such that $|t-s|\leq\sigma^4$ and let $x_1\in B_1(0)$ such that $x\in B_\sigma:= B_\sigma(x_1)\subset B_1(0)$. Denoting the exterior normal along $\partial B_\sigma$ by $\nu$, we estimate  
\begin{align}
    |\nabla^lu (x,t)-\nabla^l u(x,s)|\leq & \frac1{|B_\sigma|}\left|\int_{B_\sigma} \nabla^l u(x,t)-\nabla^l u(y,t) dy
    -\int_{B_\sigma} \nabla^l u(x,s)-\nabla^l u(y,s) dy\right|\nonumber\\
    &
    +\frac1{|B_\sigma|}\left|\int_{B_\sigma}\nabla^l u(y,s)-\nabla^l (y,t) dy\right|.\label{SchauderCh_LongINTERPOLATIONEstimate} 
\end{align}
To estimate the integral in the first line of Estimate \eqref{SchauderCh_LongINTERPOLATIONEstimate}, we write 
\begin{align}
    &\left|\int_{B_\sigma} \nabla^l u(x,t)-\nabla^l u(y,t) dy
    -\int_{B_\sigma} \nabla^l u(x,s)-\nabla^l u(y,s) dy\right|\nonumber\\
    \leq &\left|\int_0^1 \int_{B_\sigma}\nabla ^{l+1} u(y+\lambda(x-y),t)(x-y)- \nabla ^{l+1} u(y+\lambda(x-y),s)(x-y)dy d\lambda\right|\nonumber\\
    \leq &\left|\int_0^1 \int_{B_\sigma}[\nabla^{l+1} u]_{\frac\gamma4, U_1}^{\operatorname{time}}||t-s|^{\frac\gamma4}|x-y|dy d\lambda\right|\nonumber\\ 
    \leq & |B_\sigma|[\nabla^{l+1} u]_{\frac\gamma4, U_1}^{\operatorname{time}}||t-s|^{\frac\gamma4}\sigma.\label{SchauderCh_LongINTERPOLATIONEstimate1}
\end{align}
To estimate the integral in the second line of Estimate \eqref{SchauderCh_LongINTERPOLATIONEstimate}, we use the divergence theorem to write 
\begin{align}
       & \left|\int_{B_\sigma}\nabla^l u(y,s)-\nabla^l (y,t) dy\right|\nonumber\\
  \leq & \left|\int_{\partial B_\sigma}\langle \nabla^{l-1} u(y,s)-\nabla^{l-1} u(y,t),\nu\rangle dS(y)\right|\nonumber\\
    \leq &\int_{\partial B_\sigma}| \nabla^{l-1} u(y,s)-\nabla^{l-1} u(y,t)| dS(y)\nonumber\\
    \leq &|\partial B_\sigma|[\nabla^{l-1} u]_{\frac\gamma4, U_1}^{\operatorname{time}}|t-s|^{\frac\gamma4}\label{SchauderCh_LongINTERPOLATIONEstimate2}
\end{align}
Inserting Estimates \eqref{SchauderCh_LongINTERPOLATIONEstimate1} and \eqref{SchauderCh_LongINTERPOLATIONEstimate2} into Estimate \eqref{SchauderCh_LongINTERPOLATIONEstimate}, we get
\begin{equation}\label{SchauderCh_LongINTERPOLATIONEstimate3}
     |\nabla^lu (x,t)-\nabla^l u(x,s)| \leq C(n)\left(\sigma[\nabla^{l+1} u]_{\frac\gamma4, U_1}^{\operatorname{time}}+\sigma^{-1}[\nabla^{l-1} u]_{\frac\gamma4, U_1}^{\operatorname{time}}\right)|t-s|^{\frac\gamma4}.
\end{equation}
Combining Estimates \eqref{SchauderCh_InterpolationProof12} and \eqref{SchauderCh_LongINTERPOLATIONEstimate3} shows that for $l=1,2,3$ and $\sigma>0$
\begin{equation}\label{SchauderCh_integralestimateSemiNorm}
    \hspace{-.2cm}[\nabla ^l u]_{\frac\gamma4, U_1}^{\operatorname{time}}\hspace{-.1cm}\leq\hspace{-.05cm} C(n)
    \left(
    \sigma[\nabla^{l+1} u]_{\frac\gamma4, U_1}^{\operatorname{time}}
    \hspace{-.1cm}+\sigma^{-1}[\nabla^{l-1} u]_{\frac\gamma4, U_1}^{\operatorname{time}}
    \hspace{-.1cm}+\sigma^{4-l}[D^{4,1} u]_{\gamma, U_1}^{(0)}
    \hspace{-.1cm}+\sigma^{-l-\gamma}\|u\|_{L^\infty(U_1)}
    \right).
\end{equation}
Using Estimate \eqref{SchauderCh_InterpolationProofuTimeSemi}, we obtain
\begin{align}
    [\nabla u]_{\frac\gamma4, U_1}^{\operatorname{time}}&\leq C(n)\left(\sigma[\nabla^{2} u]_{\frac\gamma4, U_1}^{\operatorname{time}}+\sigma^{3}[D^{4,1} u]_{\gamma, U_1}^{(0)}+\sigma^{-1-\gamma}\|u\|_{L^\infty(U_1)}\right),\label{SchauderCh_nabla1uEstimate}\\
    [\nabla ^2 u]_{\frac\gamma4, U_1}^{\operatorname{time}}&\leq C(n)\left(\sigma[\nabla^{3} u]_{\frac\gamma4, U_1}^{\operatorname{time}}+\sigma^{-1}[\nabla u]_{\frac\gamma4, U_1}^{\operatorname{time}}+\sigma^{2}[D^{4,1} u]_{\gamma, U_1}^{(0)}+\sigma^{-2-\gamma}\|u\|_{L^\infty(U_1)}\right).\label{SchauderCh_nabla2uEstimate}
\end{align}
Let $\epsilon(n)\in(0,1)$ be a small number that is allowed to depend on $n$. We insert Estimate \eqref{SchauderCh_nabla1uEstimate} $\sigma\rightarrow \epsilon(n)\sigma$ into Estimate \eqref{SchauderCh_nabla2uEstimate}. Choosing $\epsilon(n)$ small enough then shows 
\begin{equation}\label{SchauderCh_Improvednabla2uEstimate}
    [\nabla ^2 u]_{\frac\gamma4, U_1}^{\operatorname{time}}\leq C(n)\left(\sigma[\nabla^{3} u]_{\frac\gamma4, U_1}^{\operatorname{time}}+\sigma^{2}[D^{4,1} u]_{\gamma, U_1}^{(0)}+\sigma^{-2-\gamma}\|u\|_{L^\infty(U_1)}\right).
\end{equation}
Taking $l=3$ in Estimate \eqref{SchauderCh_integralestimateSemiNorm} gives
\begin{equation}\label{SchauderCh_Nabla3uEstimate}
    [\nabla ^3 u]_{\frac\gamma4, U_1}^{\operatorname{time}}\leq C(n)\left(\sigma^{-1}[\nabla^{2} u]_{\frac\gamma4, U_1}^{\operatorname{time}}+\sigma[D^{4,1} u]_{\gamma, U_1}^{(0)}+\sigma^{-3-\gamma}\|u\|_{L^\infty(U_1)}\right).
\end{equation}
Again, let $\epsilon(n)\in(0,1)$ be a small number that is allowed to depend on $n$. We insert Estimate \eqref{SchauderCh_Improvednabla2uEstimate} with  $\sigma\rightarrow \epsilon(n)\sigma$ into Estimate \eqref{SchauderCh_Nabla3uEstimate}. Choosing $\epsilon(n)$ small enough gives 
\begin{equation}\label{SchauderCh_nabla3uFinalEst}
    [\nabla ^3 u]_{\frac\gamma4, U_1}^{\operatorname{time}}\leq C(n)\left(\sigma[D^{4,1} u]_{\gamma, U_1}^{(0)}+\sigma^{-3-\gamma}\|u\|_{L^\infty(U_1)}\right).
\end{equation}
Inserting  Estimate \eqref{SchauderCh_nabla3uFinalEst} back into Estimate \eqref{SchauderCh_Improvednabla2uEstimate}, we deduce
\begin{equation}\label{SchauderCh_Finalnabla2uEst}
    [\nabla ^2 u]_{\frac\gamma4, U_1}^{\operatorname{time}}\leq C(n)\left(\sigma^{2}[D^{4,1} u]_{\gamma, U_1}^{(0)}+\sigma^{-2-\gamma}\|u\|_{L^\infty(U_1)}\right).
\end{equation}
Finally, inserting Estimate \eqref{SchauderCh_Finalnabla2uEst} into Estimate \eqref{SchauderCh_nabla1uEstimate} gives 
\begin{equation}\label{SchauderCh_Finalnabla1uEst}
    [\nabla  u]_{\frac\gamma4, U_1}^{\operatorname{time}}\leq C(n)\left(\sigma^{3}[D^{4,1} u]_{\gamma, U_1}^{(0)}+\sigma^{-1-\gamma}\|u\|_{L^\infty(U_1)}\right).
\end{equation}
Estimates \eqref{SchauderCh_InterpolationProofuTimeSemi}, \eqref{SchauderCh_nabla3uFinalEst}, \eqref{SchauderCh_Finalnabla2uEst} and \eqref{SchauderCh_Finalnabla1uEst} are the claimed estimates for the temporal seminorms. 
\end{proof}

\begin{lemma}\label{SchauderCh_D41Estimateslemma}
Let $\rho>0$, $p_0\in\R^n\times\R$, $U_\rho:=U_\rho(p_0)$ and $u\in C^ {4,1,\gamma}(U_\rho)$. Then, for $l\in\set{0,1,2,3}$ 
$$[\nabla^ l u]^ {\operatorname{time}}_{\frac{4-l+\gamma}4, U_\rho}\leq C(n)\left([D^ {4,1} u]_{\gamma,U_\rho}^{(0)} +\rho^{-4-\gamma}\|u\|_{L^\infty(U_\rho)}\right).$$
\end{lemma}
\begin{proof}
By scaling and translation, we may assume $\rho=1$ and $p_0=0$. Let $(x,t)\neq(x,s)\in U_1$ and $\epsilon:=|t-s|^{\frac14}$. We assume, without loss of generality, that $t<s$. If $\epsilon\geq \frac16$ we use Theorem \ref{SchauderCh_InterpolationEstimates} to estimate
\begin{equation}\label{SchauderCh_HoldereSemiNorm4thorderTrivialEstimate}
\frac{|\nabla^l u(x,t)-\nabla u(x,s)|}{|t-s|^{\frac{4-l+\gamma}4}}
\leq 2\|\nabla ^l u\|_{L^\infty(U_1)}6^{\frac{4-l+\gamma}4}\leq C(n)\left([D^{4,1} u]_{\gamma, U_1}^{(0)}+\|u\|_{L^\infty(U_1)}\right).
\end{equation}
So we may also assume that $\epsilon\leq \frac16$. Given $i\in\set{1,...,n}$ and a function $f: B_{1-\epsilon}(0)\rightarrow\R$, we put 
 $$D_i^\epsilon f(x,t):= \frac{f(x+\epsilon e_i, t)- f(x,t)}\epsilon=\frac1\epsilon\int_0^\epsilon \partial_i f(x+\rho e_i,t)d\rho.$$
For higher difference quotients, we write e.g. $D_{ij}^\epsilon f:=D_i^\epsilon D_j^\epsilon f$. For $x\in B_{1-3\epsilon}(0)$, the difference quotients $D_i^\epsilon u(x,t)$, $D_{ij}^\epsilon u(x,t)$ and $D_{ijk}^\epsilon u(x,t)$ are all well-defined and we have the following formulas:
\begin{align}
    D_i^\epsilon u(x,t)&=\frac1\epsilon\int_0^\epsilon \partial_i u(x+\rho_1 e_i,t)d\rho_1\label{SchauderCh_BetterHoelderRegularityFormula1order}\\
    D_{ij}^\epsilon u(x,t)&=\frac1{\epsilon^2}\int_0^\epsilon \int_0^\epsilon \partial_{ij} u(x+\rho_1 e_i+\rho_2 e_j,t)d\rho_1d\rho_2\label{SchauderCh_BetterHoelderRegularityFormula1orde2}\\
    D_{ijk}^\epsilon u(x,t)&=\frac1{\epsilon^3}\int_0^\epsilon \int_0^\epsilon\int_0^\epsilon \partial_{ijk} u(x+\rho_1 e_i+\rho_2e_j+\rho_3e_k,t)d\rho_1d\rho_2d\rho_3\label{SchauderCh_BetterHoelderRegularityFormula1orde3}
\end{align}
\ \\\noindent
\textbf{Improved Hölder Continuity of $\nabla^3 u$}\ \\
Using Equation \eqref{SchauderCh_BetterHoelderRegularityFormula1orde3}, we write 
\begin{align*}
    &D_{ijk}^\epsilon u(x,t)-\partial_{ijk} u(x,t)\nonumber\\
    =&\frac1{\epsilon^3}\int_0^\epsilon \int_0^\epsilon\int_0^\epsilon u(x+\rho_1 e_i+\rho_2e_j+\rho_3e_k,t)d\rho_1d\rho_2d\rho_3- \partial_{ijk} u(x,t)\nonumber\\
    =&\frac1{\epsilon^3}\int_0^\epsilon \int_0^\epsilon\int_0^\epsilon\int_0^1  \langle\nabla \partial_{ijk} u(x+\mu\left(\rho_1 e_i+\rho_2e_j+\rho_3e_k\right),t),\rho_1 e_i+\rho_2e_j+\rho_3e_k\rangle d\mu d\rho_1d\rho_2d\rho_3.
\end{align*}
So, using the temporal Hölder continuity of the fourth derivative $\nabla^4 u$ and recalling that $\epsilon^4=|t-s|$, we may estimate 
\begin{align}
    &\left|\partial_{ijk} u(x,t)-D_{ijk}^\epsilon u(x,t)-\partial_{ijk} u(x,s)+D_{ijk}^\epsilon u(x,s)\right|\nonumber\\
 \leq & \frac1{\epsilon^3}\int_0^\epsilon\int_0^\epsilon\int_0^\epsilon\int_0^1 [\nabla^4 u]_{\frac\gamma4,U_1}^{\operatorname{time}}|t-s|^{\frac\gamma4}\mu(\rho_1+\rho_2+\rho_3)d\mu d\rho_1d\rho_2d\rho_3\nonumber\\
 \leq & \frac1{\epsilon^3} [\nabla^4 u]_{\frac\gamma4,U_1}^{\operatorname{time}}|t-s|^{\frac\gamma4} \int_0^\epsilon\int_0^\epsilon\int_0^\epsilon\int_0^13\epsilon d\mu d\rho_1d\rho_2d\rho_3\nonumber\\
    = &3 [\nabla^4 u]_{\frac\gamma4, U_1}^{\operatorname{time}}|t-s|^{\frac{1+\gamma}4}.\label{SchauderCh_BetterHoelderRegularityFormula1orde3_2}
\end{align}
Recalling that $t<s$, we estimate
\begin{align}
    |D_{ijk}^\epsilon u(x,t)- D_{ijk}^\epsilon u(x,s)|\leq &|t-s| \sup_{t\leq \tau\leq s}\sup_{x\in B_{3-\epsilon}(0)} |D_{ijk} \dot u(x,\tau)|\nonumber\\
                \leq & |t-s| \sup_{t\leq \tau\leq s} \sup_{x\in B_{1-3\epsilon}(0)} \frac1\epsilon \left(|D_{jk} \dot u(x+\epsilon e_i,\tau)|+|D_{jk} \dot u(x,\tau)|\right)\nonumber\\
                \leq & |t-s| \frac2\epsilon \sup_{t\leq \tau\leq s} \sup_{x\in B_{1-2\epsilon}(0)}|D_{jk} \dot u(x,\tau)|\nonumber\\
                \leq & |t-s| \frac4{\epsilon^2} \sup_{t\leq \tau\leq s} \sup_{x\in B_{1-\epsilon}(0)}|D_{k} \dot u(x,\tau)|\nonumber\\
                \leq & |t-s| \frac4{\epsilon^2} \sup_{t\leq \tau\leq s} \sup_{x\in B_{1-\epsilon}(0)}| \frac{|\dot u(x+\epsilon e_k,\tau)-\dot u(x,\tau)|}\epsilon\nonumber\\
                \leq & |t-s| \frac4{\epsilon^3}  [\dot u]^{\operatorname{space}}_{\frac\gamma4, U_1}\epsilon^\gamma\nonumber\\
                \leq &  4 [\dot u]^{\operatorname{space}}_{\frac\gamma4, U_1}|t-s|^{\frac{1+\gamma}4}.\label{SchauderCh_BetterHoelderRegularityFormula1orde3_3}
\end{align}
Combining Estimates \eqref{SchauderCh_BetterHoelderRegularityFormula1orde3_2} and \eqref{SchauderCh_BetterHoelderRegularityFormula1orde3_3} implies that for $x\in B_{1-3\epsilon}(0)$ and $|t-s|\leq\frac16$ we have 
\begin{equation}\label{SchauderCh_ClaimedTimeHölderForNabla3}
    \frac{|\nabla^3 u(x,t)-\nabla^3 u(x,s)|}{|t-s|^{\frac{1+\gamma}4}}\leq C(n)[D^{4,1} u]^{(0)}_{\gamma, U_1}.
\end{equation}
If now $x\not\in B_{1-3\epsilon}(0)$ but still $|t-s|\leq\frac16$, we can find $z\in B_{1-3\epsilon}(0)$ such that $|x-z|<3\epsilon$. We write
$$\partial_{ijk} u(x,t)=\partial_{ijk} u(z,t)+\int_0^1\langle\nabla \partial_{ijk} u(z+\rho(x-z),t), x-z\rangle d\rho.$$
Making use of Estimate \eqref{SchauderCh_ClaimedTimeHölderForNabla3} and $|z-x|\leq 3\epsilon=3|t-s|^{\frac14}$, we get 
\begin{align}
    \frac{|\nabla^3 u(x,t)-\nabla^3 u(x,s)|}{|t-s|^{\frac{1+\gamma}4}}
    \leq &C(n)[D^{4,1} u]^{(0)}_{\gamma, U_1}+\frac1{|t-s|^{\frac{1+\gamma}4}}\int_0^1[\nabla^4 u]_{\frac\gamma4, U_1}^{\operatorname{time}}|t-s|^{\frac\gamma4}|x-z|d\rho\nonumber\\
    \leq & C(n)[D^{4,1} u]^{(0)}_{\gamma, U_1}.\label{SchauderCh_ClaimedTimeHölderForNabla3FullDisk}
\end{align} 
Combining Estimates \eqref{SchauderCh_HoldereSemiNorm4thorderTrivialEstimate} and \eqref{SchauderCh_ClaimedTimeHölderForNabla3FullDisk} we deduce the lemma for $l=3$.\\

\noindent
\textbf{Improved Hölder Continuity of $\nabla^2 u$ and $\nabla u$}\ \\
Estimating the temporal seminorms for $\nabla^2 u$ and $\nabla u$ is very much analogous to the argument we have shown for $\nabla^3 u$. We begin with $\nabla^2 u$. Using Equation \eqref{SchauderCh_BetterHoelderRegularityFormula1orde2} instead of \eqref{SchauderCh_BetterHoelderRegularityFormula1orde3} as well as the already established Hölder regularity of $\nabla^3 u$, we can derive that for $x\in B_{1-3\epsilon}(0)$ and $|t-s|\leq\frac16$
\begin{equation}\label{SchauderCh_BetterHoelderRegularityFormula1orde2_2}
    \left|\partial_{ij} u(x,t)-D_{ij}^\epsilon u(x,t)-\partial_{ij} u(x,s)+D_{ij}^\epsilon u(x,s)\right|
    \leq 2[\nabla^3 u]_{\frac{1+\gamma}4, U_1}^{\operatorname{time}}|t-s|^{\frac{1+\gamma}4}
\end{equation}
instead of Estimate \eqref{SchauderCh_BetterHoelderRegularityFormula1orde3_2}. Next, following the analysis that led to Estimate \eqref{SchauderCh_BetterHoelderRegularityFormula1orde3_3}, we derive
\begin{equation}\label{SchauderCh_BetterHoelderRegularityFormula1orde2_3}
    |D_{ij}^\epsilon u(x,t)- D_{ij}^\epsilon u(x,s)|\leq   4 [\dot u]_{\frac\gamma4, U_1}^{\operatorname{space}}|t-s|^{\frac{2+\gamma}4}.
\end{equation}
Finally, we generalise to $x\in B_1(0)$; this time by using that 
$$\partial_{ij} u(x,t)=\partial_{ij} u(z,t)+\int_0^1 \langle\nabla\partial_{ij}u(z+\rho(x-z),t), x-z\rangle d\rho.$$
Combining with Estimate \eqref{SchauderCh_HoldereSemiNorm4thorderTrivialEstimate}, the lemma follows for $l=2$. Finally, for $l=1$, we first use Equation \eqref{SchauderCh_BetterHoelderRegularityFormula1order} and the now established Hölder continuity of $\nabla^2 u$ to establish the following analog of Estimate \eqref{SchauderCh_BetterHoelderRegularityFormula1orde2_2}. For $x\in B_{1-3\epsilon}(0)$ 
\begin{equation}\label{SchauderCh_BetterHoelderRegularityFormula1orde1_2}
    \left|\partial_{i} u(x,t)-D_{i}^\epsilon u(x,t)-\partial_{i} u(x,s)+D_{i}^\epsilon u(x,s)\right|\nonumber\\
\leq[\nabla^2 u]_{\frac{2+\gamma}4, U_1}^{\operatorname{time}}|t-s|^{\frac{3+\gamma}4}.
\end{equation}
Next, following the analysis that lead to Estimate \eqref{SchauderCh_BetterHoelderRegularityFormula1orde3_3}, we derive
\begin{equation}\label{SchauderCh_BetterHoelderRegularityFormula1orde1_3}
    |D_{i}^\epsilon u(x,t)- D_{i}^\epsilon u(x,s)|\leq   [\dot u]_{\frac\gamma4, U_1}^{\operatorname{space}}|t-s|^{\frac{3+\gamma}4}.
\end{equation}
We deduce the lemma for $l=1$ by following the same arguments as for $l=2,3$.
\end{proof}

\begin{lemma}[Interpolation Estimates for Temporal Seminorms]\label{SchauderCh_TemporalInterpolationLemma}\ \\
There exists $\epsilon_0(n)>0$ such that for $\epsilon\in(0,\epsilon_0)$, $p_0\in\R^n$, $\rho>0$, $U_\rho:=U_\rho(p_0)$, $u\in C^{4,1,\gamma}(U_\rho)$, $l\in\set{0,1,2,3}$ and $k\in\set{0,...,3-l}$
$$\rho^{l+k+\gamma}[\nabla^l u]^{\operatorname{time}}_{\frac{k+\gamma}4,U_\rho} \leq\epsilon\rho^{4+\gamma}[D^{4,1}u]^{(0)}_{\gamma, U_\rho}+C(n,\gamma)\epsilon^{-\frac{l+k+\gamma}{4-l-k}}\|u\|_{L^\infty(U_\rho)}.$$
\end{lemma}
\begin{proof}
Without loss of generality $\rho=1$ and $p_0=0$. Let $\sigma\in(0,1)$ and $(x,t), (x,s)\in U_1:=U_1(0)$. We first assume that $|t-s|\geq \sigma^4$. Using Theorem \ref{SchauderCh_InterpolationEstimates} we deduce that for $\sigma<\sigma_0(n)$ 
\begin{align}
    \frac{|\nabla^lu(x,t)-\nabla^lu(x,s)|}{|t-s|^ {\frac{k+\gamma}4}}&\leq 2\sigma^{-k-\gamma}\|\nabla^l u\|_{L^\infty(U_1)}\nonumber\\
    &\leq 2\sigma^{-k-\gamma}\left(
        \sigma^{4+\gamma-l}[D^{4,1} u]_{\gamma, U_1}^{(0)}+\sigma^{-l}\|u\|_{L^\infty(U_1)}
    \right)\nonumber\\
    &\leq C(n)\left(
        \sigma^{4-l-k}[D^{4,1} u]_{\gamma, U_1}^{(0)}+\sigma^{-l-k-\gamma}\|u\|_{L^\infty(U_1)}
    \right)\label{SchauderCh_Interpolation2Proof12}.
\end{align}
So, it suffices to prove the lemma for $|t-s|\leq \sigma^4$. Let $t_0\in(0,1)$ such that $t,s\in(t_0-\sigma^4, t_0+\sigma^4)\subset(0,1)$. If $B_\sigma(x)\subset B_1(0)$, we can use Lemma \ref{SchauderCh_D41Estimateslemma} and estimate
\begin{align}
    \frac{|\nabla^lu(x,t)-\nabla^lu(x,s)|}{|t-s|^ {\frac{k+\gamma}4}}\leq& \frac{|\nabla^lu(x,t)-\nabla^lu(x,s)|}{|t-s|^ {\frac{4-l+\gamma}4}}\frac{|t-s|^ {\frac{4-l+\gamma}4}}{|t-s|^ {\frac{k+\gamma}4}}\nonumber\\
    &\leq [\nabla^ l u]_{\frac{4-l+\gamma}4,U_\sigma((x,t_0))}^ {\operatorname{time}}|t-s|^ {\frac{4-l-k}4}\nonumber\\
     &\leq C(n)\left(\sigma^ {4-l-k}[D^ {4,1} u]_{\gamma, U_1}^{(0)}+\sigma^ {-l-k-\gamma}\|u\|_{L^ \infty(U_1)}\right).\label{SchauderCh_Interpolation2_CaseReduction}
\end{align}
If $B_\sigma(x)\not\subset B_1(0)$ there is a point $y\in B_1(0)$ such that $x\in B_\sigma(y)\subset B_1(0)$. Using Theorem \ref{SchauderCh_InterpolationEstimates}, we estimate for $\tau\in\set{t,s}$
$$\frac{|\nabla ^l u(x,\tau)-\nabla ^l u(y,\tau)|}{|t-s|^{\frac{k+\gamma}4}}
\leq 
\sigma^{-k-\gamma}[\nabla^l u]_{\gamma, U_1}^{(0)}|x-y|^\gamma
\leq
\sigma^{-k}\left(\sigma^{4-l}[D^{4,1} u]_{\gamma, U_1}^{(0)}+\sigma^{-l-k-\gamma}\|u\|_{L^\infty(U_1)}\right).
$$
Hence Estimate \eqref{SchauderCh_Interpolation2_CaseReduction} is valid for all $x\in B_1(0)$. Recalling Estimate \eqref{SchauderCh_Interpolation2Proof12}, the lemma follows.
\end{proof}

The following proposition can be established by following the same arguments that lead to Theorem \ref{SchauderCh_InterpolationEstimates} and Lemmas \ref{SchauderCh_D41Estimateslemma} and \ref{SchauderCh_TemporalInterpolationLemma}.
\begin{proposition}
    Theorem \ref{SchauderCh_InterpolationEstimates} and Lemmas \ref{SchauderCh_D41Estimateslemma} and \ref{SchauderCh_TemporalInterpolationLemma} are also true, when the parabolic balls $U_\rho$ are replaced by the truncated parabolic balls $U_\rho^\pm(p)$, $U_{\rho+}(p)$ and $U_{\rho+}^{\pm}(p)$. 
\end{proposition}

\section{Hölder Spaces on Manifolds}\label{SchauderCh_SectiononHolderSpacesonMfds}
\subsection{Elliptic Hölder Spaces on Manifolds}
Throughout this subsection, we fix a smooth, compact, $n$-dimensional manifold $M$, possibly with boundary $\partial M\neq\emptyset$.
\begin{definition}[Good Atlas]
A finite atlas $\mathcal A=\set{(\varphi_i, U_i)\ |\ 1\leq i\leq m}$ of $M$ consisting of charts $\varphi_i:U_i\subset M\rightarrow V_i\subset\R^n$ or $\R^{n-1}\times[0,\infty)$ is called \emph{a good atlas of $M$}, if:
\begin{enumerate}[(1)]
    \item For all $1\leq i\leq m$ the sets $V_i$ are convex.
    \item For all $1\leq i\leq m$ there exist an open set $\tilde U_i\supset\supset U_i$, an open and convex set $\tilde V_i\supset\supset V_i$ and a smooth extension $\tilde\varphi_i:\tilde U_i\rightarrow\tilde V_i$ of $\varphi_i$.
\end{enumerate} 
\end{definition}

If $\partial M\neq\emptyset$, a good atlas on $M$ provides us with a good atlas on $\partial M$ by considering only those charts $U_i$ that contain boundary points i.e. $U_i\cap\partial M\neq\emptyset$ and restricting to the boundary. That is 
$$\varphi_i\big|_{U_i\cap \partial M}:U_i\cap\partial M\rightarrow V_i\cap (\R^{n-1}\times0).$$

In the following definition, we define Hölder spaces on $M$. Afterwards, we will have to justify this definition.
\begin{definition}[Hölder Spaces on Manifolds]\label{SchauderCh_Manifold_Elliptic_Def}
    Let $\mathcal A=\set{(\varphi_i, U_i)\ |\ 1\leq i\leq m}$ be a good atlas of $M$, $V_i:=\varphi_i(U_i)$ and $\gamma\in(0,1)$. We say that $u\in C^k(M)$ is in $C^{k,\gamma}_{\mathcal A}(M)$, if 
    $$\|u\|_{C^{k,\gamma}_\mathcal A(M)}:=\max_{1\leq i\leq m}\|u\circ\varphi_i^{-1}\|_{C^{k,\gamma}(V_i)}<\infty.$$ 
\end{definition}
 It is easy to see that all the usual properties of Hölder spaces remain true. For example $(C^{k,\gamma}_\mathcal A(M),\|\cdot\|_{C^{k,\gamma}_\mathcal A(M)})$ is a Banach space. In the next lemma, we show that the space $C^{k,\gamma}_\mathcal A(M)$ is independent of the choice of good atlas.

\begin{lemma}\label{SchauderCh_EquivalenceofHolderSpaceDefinition}
Let $\mathcal A$ and $\mathcal B$ be two good atlases on a smooth manifold $M$. Then $C^{k,\gamma}_{\mathcal A}(M)=C^{k,\gamma}_{\mathcal B}(M)$ for all $k\in \N_0$ and $\gamma\in (0,1)$. Additionally there exists a constant $C=C(\mathcal A, \mathcal B, k,\gamma)$ such that
$$
\|u\|_{C^ {k,\gamma}_{\mathcal A}(M)} \leq C(\mathcal A,\mathcal B,k,\gamma)\|u\|_{C^ {k,\gamma}_{\mathcal B}(M)}.
$$
\end{lemma}
\begin{proof}
It suffices to prove the estimate. Let $\varphi:U\rightarrow V$ be a good chart from $\mathcal A$ and $\hat\varphi:\hat U\rightarrow\hat V$  from $\mathcal B$ such that $U\cap\hat U\neq\emptyset$. We put $V':=\varphi(U\cap\hat U)\subset V$ and put $\psi:=\hat\varphi\circ\varphi^{-1}:V'\rightarrow \hat V$. Then $\psi\in C^\infty(V')$ with bounded derivatives. For $x\in V'$ and a multi-index $\alpha$ of length $l$ we compute 
\begin{equation}\label{SchauderCh_HolderSpaceEstimate_GoodAtlas0}
\nabla_\alpha  (u\circ \varphi^{-1})(x)
=\nabla_\alpha (u\circ\hat\varphi^{-1}\circ\psi)(x)
=\sum_{|\beta|\leq l} c_\beta (\psi)(x)\nabla_\beta (u\circ \hat \varphi^{-1})(\psi(x)).
\end{equation}
Here $c_\beta(\psi)(x)$ are smooth functions that depend on $\psi$ and derivatives of $\psi$. Using this formula it is easily seen that 
\begin{equation}\label{SchauderCh_HolderSpaceEstimate_GoodAtlas1}
\|u\circ\varphi^{-1}\|_{C^{k}(V')}
\leq\hspace{-.1cm} C(\psi,k)\|u\circ\tilde \varphi^{-1}\|_{C^{k}(\psi(V'))}
\leq\hspace{-.1cm}  C(\psi,k)\|u\circ\hat \varphi^{-1}\|_{C^{k}(\hat V)}
\leq\hspace{-.1cm}  C(\psi,k)\|u\|_{C^{k}_{\mathcal B}(M)}.
\end{equation}
As $\mathcal B$ is a good atlas, we can cover $U$ by good chart domains $\hat U_1,...,\hat U_m$ in $\mathcal B$. Using Estimate \eqref{SchauderCh_HolderSpaceEstimate_GoodAtlas1} we then get
\begin{equation}\label{SchauderCh_NormEquivalence_NormPartResult}
\|u\circ\varphi^{-1}\|_{C^{k}(V)}\leq \sum_{i=1}^m\|u\circ\varphi^{-1}\|_{C^{k}(\varphi(U\cap\tilde U_i))}
\leq C(\mathcal A,\mathcal B,k)\|u\|_{C^{k}_{\mathcal B}(M)}.
\end{equation}
Since this is true for all good charts in $\mathcal A$, we have proven $\|u\|_{C^{k}_{\mathcal A}(M)}\leq C(\mathcal A,\mathcal B,k)\|u\|_{C^{k}_{\mathcal B}(M)}$. It remains to estimate the seminorms.\\

Let again $\varphi:U\rightarrow V$ be a good chart from $\mathcal A$, $f\in C^{0,\gamma}(V)$,  $\hat\varphi:\hat U\rightarrow\hat V$ be a good chart from $\mathcal B$ such that $\hat U\cap U\neq\emptyset$ and $S\subset\varphi( U\cap \hat U)$. We estimate
\begin{align}
[f]_{\gamma, S}=&\sup_{x\neq y\in S}
\frac{|f(x)-f(y)|}{|x-y|^\gamma}\nonumber\\
\leq &\sup_{x\neq y\in S}
\frac{|(f\circ\varphi\circ\hat \varphi^{-1})(\hat\varphi(\varphi^{-1}(x))-(f\circ\varphi\circ\hat \varphi^{-1})(\hat\varphi(\varphi^{-1}(y))|}{|x-y|^\gamma}\nonumber\\
\leq &[f\circ\varphi\circ\hat\varphi^{-1}]_{\gamma, \hat\varphi(\varphi^{-1}(S))}\sup_{x\neq y\in S}\left|\frac{\hat\varphi(\varphi^{-1}(x))-\hat\varphi(\varphi^{-1}(y))}{|x-y|}\right|^\gamma\nonumber\\
\leq & [f\circ\varphi\circ\hat\varphi^{-1}]_{\gamma, \hat\varphi(\varphi^{-1}(S))}\left(\|\nabla(\hat\varphi\circ\varphi^{-1})\|_{C^0(\varphi(U\cap\hat U))}\right)^\gamma\nonumber\\
\leq& C(\mathcal A,\mathcal B,\gamma) [f\circ\varphi\circ\hat\varphi^{-1}]_{\gamma, \hat\varphi(\varphi^{-1}(S))}.\label{SchauderCh_SchauderNromEquivBasicEstimate}
\end{align}
For general $x,y\in V$ there doesn't exist a good chart in $\mathcal B$ such that $\varphi^{-1}(x),\varphi^{-1}(y)\in \hat U$. Since $\mathcal A$ is a good atlas, there is a smooth extension $\varphi:\tilde U\rightarrow \tilde V$ with $U\subset\subset \tilde U$ and $V\subset\subset\tilde V$. Let $\mathcal B=\set{(\hat\varphi_i, \hat U_i)\ |\ 1\leq i\leq m}$ and put $\Omega_i:=\varphi( \tilde U\cap \hat U_i)$. Then $\Omega_i\subset \tilde V$ are open and cover $\bar V$. Additionally, we can write 
$$\Omega_i=\bigcup_{x\in\Omega_i} B_{r_{x,i}}(x)\hspace{.5cm}\textrm{with $r_{x,i}>0$ such that $B_{r_{x,i}}(x)\subset \Omega_i$}.$$ 
This implies that $\bar V\subset\bigcup_{i=1}^m\Omega_i\subset \bigcup_{i=1}^m\bigcup_{x\in\Omega_i} B_{r_{x,i}}(x)$. By compactness, we can select points $x_1,...., x_J$ and radii $r_j$ such that $V\subset\bar V\subset \bigcup_{j=1}^J B_{r_j}(x_j)$. Note that for each $1\leq j\leq J$ there exists $1\leq i(j)\leq m$ such that $B_{r_j}(x_j)\subset\Omega_{i(j)}$. Using Estimate \eqref{SchauderCh_SchauderNromEquivBasicEstimate}, we obtain
$$[f]_{\gamma, B_{r_j}(x_j)}
\leq \max_{1\leq i\leq m}[f]_{\Omega_i,\gamma}
\leq C(\mathcal A,\mathcal B,\gamma)\max_{1\leq i\leq m}[f\circ\varphi\circ\hat\varphi_i^{-1}]_{\gamma, \hat V_i}.$$
Since $V$ is convex, the sets $B_{r_j}(x_j)\cap V$ are convex and cover $V$. So, using Lemma \ref{SchauderCh_SemiNormSubbadditiveonConvexSpaces} and noting that $J$ is a number only depending on $\mathcal A$ and $\mathcal B$, we obtain 
\begin{equation}\label{SchauderCh_EquivalenceHolderNormsBASICestimate}
[f]_{\gamma, V}\leq \sum_{j=1}^J[f]_{\gamma, B_{r_j}(x_j)}
\leq C(\mathcal A,\mathcal B,\gamma)\max_{1\leq i\leq m}[f\circ\varphi\circ\hat\varphi_i^{-1}]_{\gamma, \hat V_i}.
\end{equation}
Let $\alpha\in \N_0^n$ be a multi-index of length $k$. We apply Estimate \eqref{SchauderCh_EquivalenceHolderNormsBASICestimate} with $f=\nabla_\alpha(u\circ\varphi^{-1})$. In view of Equation \eqref{SchauderCh_HolderSpaceEstimate_GoodAtlas0}, we deduce that for all good charts $\varphi:U\rightarrow V$ from $\mathcal A$
\begin{equation}\label{SchauderCh_SemiNormsestimaedResult_EquivalenceNorms}
[\nabla_\alpha (u\circ\varphi)^{-1}]_{\gamma, U}\leq C(\mathcal A,\mathcal B,k,\gamma)\max_{1\leq i\leq m}[\nabla_\alpha(u\circ\hat\varphi_i^{-1})]_{\gamma, \hat V_i}
.\end{equation}
The lemma follows by combining Estimates \eqref{SchauderCh_NormEquivalence_NormPartResult} and \eqref{SchauderCh_SemiNormsestimaedResult_EquivalenceNorms}.
\end{proof}
 \begin{tcolorbox}[colback=white!20!white,colframe=black!100!white,sharp corners, breakable]
We sometimes write $\|u\|_{C^{k,\gamma}(M)}$ without making the choice of good atlas explicit in the notation. This is now justified by
 Lemma \ref{SchauderCh_EquivalenceofHolderSpaceDefinition}.
 \end{tcolorbox}

\subsection{Parabolic Hölder Spaces on Manifolds}
Throughout this subsection we fix $T\in(0,\infty)$ and  a smooth, compact, $n$-dimensional manifold $M$, possibly with boundary $\partial M\neq\emptyset$. For $m\in\N_0$, we say that a function $u$ is $C^{m,[m/4]}(M\times(0,T))$, if for $k+4l$ the derivatives $\partial_t^l\nabla^k u$ exist and are continuous.

\begin{definition}[Parabolic Hölder Spaces on Manifolds]\label{SchauderCh_MANIFOLDParabolicHolderSpaceDef}
    Let $k\in\N_0$ and $\mathcal A=\set{(\varphi_i, U_i)\ |\ 1\leq i\leq m}$  be a good atlas on $M$. We say that a function $u\in C^{k,[k/4]}(M\times(0,T))$ is in $C^{k,[k/4],\gamma}(M\times[0,T])$,~if 
    $$\|u\|_{C^{k,[k/4],\gamma}_{\mathcal A}(M\times[0,T])}:=\max_{1\leq i\leq m}\|u\circ\varphi_i^{-1}\|_{C^{k,[k/4],\gamma}(V_i\times[0,T])}<\infty.$$
    Here, in a slight abuse of notation, $(u\circ\varphi_i^{-1})(x,t):=u(\varphi_i^{-1}(x),t).$
\end{definition}

As in the elliptic case, we need to justify that this definition does not depend on the choice of good atlas. 

\begin{lemma}\label{SchauderCh_NormEquivalenceParabolicHolderSpace}
Let $M$ be a smooth, $n$-dimensional manifold and let $\mathcal A$ and $\mathcal B$ be good atlases on $M$. There exists a constant $C(\mathcal A,\mathcal B,k,\gamma)$ such that for any $T>0$
$$
\|u\|_{C^ {k,[k/4],\gamma}_\mathcal A(M\times[0,T])}
\leq 
C(\mathcal A,\mathcal B,k,\gamma)
\|u\|_{C^ {k,[k/4],\gamma}_\mathcal B(M\times[0,T])}
.
$$
\end{lemma}
 \begin{tcolorbox}[colback=white!20!white,colframe=black!100!white,sharp corners, breakable]
The important observation in Lemma \ref{SchauderCh_NormEquivalenceParabolicHolderSpace} lies in the fact that the constant $C(\mathcal A,\mathcal B,k,\gamma)$ does not depend on $T$.
\end{tcolorbox}
\begin{proof}
Let $\varphi:U\rightarrow V$ be a good chart from $\mathcal A$ and $\hat \varphi:\hat U\rightarrow\hat V$ be a good chart from $\mathcal B$ with $U\cap\hat U\neq\emptyset$. We define $\psi:=\hat\varphi\circ\varphi^{-1}:\varphi(U\cap\hat U)\rightarrow \hat\varphi(U\cap\hat U)$. For $4j+l\leq k$ and a multi index $\alpha$ of length $l$, we note the analogue of Equation \eqref{SchauderCh_HolderSpaceEstimate_GoodAtlas0}
\begin{equation}\label{SchauderCh_HolderSpaceEstimate_Parabolic_GoodAtlas0}
\nabla_\alpha\partial_t^j  (u\circ \varphi^{-1})(x,t)
=\nabla_\alpha\partial_t^j (u\circ\hat\varphi^{-1}\circ\psi)(x,t)
=\sum_{|\beta|\leq l} c_\beta (\psi)(x)\nabla_\beta\partial_t^j (u\circ \hat \varphi^{-1})(\psi(x),t).
\end{equation}
For a multi-index $\alpha$, we can use Equation \eqref{SchauderCh_HolderSpaceEstimate_Parabolic_GoodAtlas0} to conclude
\begin{alignat}{2}
&\|\nabla_\alpha\partial_t^j  (u\circ \varphi^{-1})(x,t)\|_{C^{0}(V\times[0,T])}
\leq 
C(\mathcal A,\mathcal B,k,\gamma)\|u\|_{C^{k,[k/4]}_{\mathcal B}(M\times[0,T])}&\hspace{.3cm}\textrm{ if }|\alpha|+4j\leq k,\label{SchauderCh_HolderSpaceEstimate_Parabolic_GoodAtlas1}\\
&[\nabla_\alpha\partial_t^j  (u\circ \varphi^{-1})]_{\gamma,V\times[0,T]}^{\operatorname{space}}
\leq 
C(\mathcal A,\mathcal B,k,\gamma)\|u\|_{C^{k,[k/4],\gamma}_{\mathcal B}(M\times[0,T])}&\hspace{-2cm}\textrm{ if }|\alpha|+4j=k ,\label{SchauderCh_HolderSpaceEstimate_Parabolic_GoodAtlas2}
\end{alignat}
by following the same arguments that we presented in the proof of Lemma \ref{SchauderCh_EquivalenceofHolderSpaceDefinition}.
It remains to estimate the temporal Hölder seminorms. Let $t\neq s\in [0,T]$ and $j,l\in\N_0$ such that $\eta:=\frac{k+\gamma-l}4-j\in(0,1)$. Let $\varphi:U\rightarrow V$ be a good chart in $\mathcal A$ and $x\in V$. We aim to prove 
\begin{equation}\label{SchauderCh_HolderSpaceEstimate_Parabolic_GoodAtlas1GOAL}
     \left|\frac{\partial_t^j\nabla^l (u\circ\varphi^{-1})(x,t)-\partial_t^j\nabla^l (u\circ\varphi^{-1})(x,s)}{|t-s|^\eta}\right|
     \leq 
     C(\mathcal A,\mathcal B, k,\gamma)\|u\|_{C^{k,[k/4],\gamma}_{\mathcal B}(M\times[0,T])}.
\end{equation}
Once this is established, the lemma follows by combining Estimates \eqref{SchauderCh_HolderSpaceEstimate_Parabolic_GoodAtlas1}, \eqref{SchauderCh_HolderSpaceEstimate_Parabolic_GoodAtlas2} and \eqref{SchauderCh_HolderSpaceEstimate_Parabolic_GoodAtlas1GOAL}. If $|t-s|\geq 1$, we use Estimate \eqref{SchauderCh_HolderSpaceEstimate_Parabolic_GoodAtlas1} to estimate 
\begin{align}
    &\left|\frac{\partial_t^j\nabla^l (u\circ\varphi^{-1})(x,t)-\partial_t^j\nabla^l (u\circ\varphi^{-1})(x,s)}{|t-s|^\eta}\right|\nonumber\\
\leq &2\sup_{\tau\in[0,T]}\left|\partial_t^j\nabla^l (u\circ\varphi^{-1})(x,\tau)\right|\nonumber\\
\leq &2\|u\circ\varphi^{-1}\|_{C^{k,[k/4]}(V\times[0,T])}\nonumber\\
\leq &C(\mathcal A,\mathcal B,k)\|u\|_{C^{k,[k/4],\gamma}_{\mathcal B}(M\times[0,T])}.\label{SchauderCh_HolderSpaceEstimate_Parabolic_GoodAtlas4}
\end{align}
Next, we consider the case where $|t-s|\leq 1$. Let $\hat\varphi:\hat U\rightarrow\hat V$ be a good chart from $\mathcal B$ such that $\varphi^{-1}(x)\in\hat U$ and put $\psi:=\hat\varphi\circ\varphi$. Using Equation \eqref{SchauderCh_HolderSpaceEstimate_Parabolic_GoodAtlas0}, we estimate 
\begin{align}
    &\left|\frac{\partial_t^j\nabla^l (u\circ\varphi^{-1})(x,t)-\partial_t^j\nabla^l (u\circ\varphi^{-1})(x,s)}{|t-s|^\eta}\right|\nonumber\\
  \leq &\sum_{|\beta|\leq l} |c_\beta (\psi)(x)| \left|\frac{\partial_t^j\nabla_\beta (u\circ\hat \varphi^{-1})(\psi(x),t)-\partial_t^j\nabla_\beta (u\circ\hat \varphi^{-1})(\psi(x),s)}{|t-s|^\eta}\right|.\label{SchauderCh_HolderSpaceEstimate_Parabolic_GoodAtlas5}
\end{align}
We estimate the terms individually. If $|\beta|+4j\leq k-4$, then $\partial_t^j \nabla_\beta u$ is $C^1$ in time and hence, using $|t-s|\leq 1$ and Estimate \eqref{SchauderCh_HolderSpaceEstimate_Parabolic_GoodAtlas1}, we can estimate 
\begin{align}
&\left|\frac{\partial_t^j\nabla_\beta (u\circ\hat\varphi^{-1})(\psi(x),t)-\partial_t^j\nabla_\beta (u\circ\hat\varphi^{-1})(\psi(x),s)}{|t-s|^\eta}\right|\nonumber\\
\leq 
&\sup_{\tau\in[0,T]}\left|\partial_t^{j+1}\nabla_\beta (u\circ\hat\varphi^{-1})(\psi(x),\tau)\right|\nonumber\\
\leq 
&\|\partial_t^{j+1}\nabla_\beta (u\circ\hat\varphi^{-1})\|_{C^0(\hat V\times[0,T])}\nonumber\\
\leq &
\|u\|_{C^{k,[k/4],\gamma}_{\mathcal B}(M\times[0,T])}.\label{SchauderCh_HolderSpaceEstimate_Parabolic_GoodAtlas6}
\end{align}
Finally, we need to consider the case where $k-4<|\beta|+4j\leq k$ or, equivalently 
$1>\frac{k-|\beta|}4-j\geq 0$. Then $1>\frac{k+\gamma-|\beta|}4-j> 0$ and since $|\beta|\leq l$ we get 
$0<\eta\leq \frac{k+\gamma-|\beta|}4-j<1$. 
Using that $|t-s|\leq 1$ and the definition of the $C^{k,[k/4],\gamma}$-norm we then get 
\begin{align}
&\left|\frac{\partial_t^j\nabla_\beta (u\circ\hat\varphi^{-1})(\psi(x),t)-\partial_t^j\nabla_\beta (u\circ\hat\varphi^{-1})(\psi(x),s)}{|t-s|^\eta}\right|\nonumber\\
\leq &
\left|\frac{\partial_t^j\nabla_\beta (u\circ\hat\varphi^{-1})(\psi(x),t)-\partial_t^j\nabla_\beta (u\circ\hat\varphi^{-1})(\psi(x),s)}{|t-s|^{\frac{k+\gamma-|\beta|}4-j}}\right|\nonumber\\
\leq&
\|u\|_{C^{k,[k/4],\gamma}_{\mathcal B}(M\times[0,T])}.\label{SchauderCh_HolderSpaceEstimate_Parabolic_GoodAtlas7}
\end{align}
Estimate \eqref{SchauderCh_HolderSpaceEstimate_Parabolic_GoodAtlas1GOAL} and hence the lemma follows by combining Estimates
\eqref{SchauderCh_HolderSpaceEstimate_Parabolic_GoodAtlas4}
\eqref{SchauderCh_HolderSpaceEstimate_Parabolic_GoodAtlas5}, 
\eqref{SchauderCh_HolderSpaceEstimate_Parabolic_GoodAtlas6}
and 
\eqref{SchauderCh_HolderSpaceEstimate_Parabolic_GoodAtlas7}. 
\end{proof}

 \begin{tcolorbox}[colback=white!20!white,colframe=black!100!white,sharp corners, breakable]
We sometimes write $\|u\|_{C^{k,[k/4],\gamma}(M\times[0,T])}$ without making the choice of good atlas explicit in the notation. This is now justified by
 Lemma \ref{SchauderCh_NormEquivalenceParabolicHolderSpace}.
 \end{tcolorbox}

\section{Sobolev Spaces on Manifolds}
Throughout this section, let $M$ be a smooth, $n$-dimensional, compact manifold possibly with boundary $\partial M\neq\emptyset$. In this section, we define the notion of Sobolev spaces on $M$. 

\begin{definition}
We say that a set $A\subset M$ is measurable if $\phi(A)$ is Lebesgue measurable for all charts $\phi:U\subset M\rightarrow V\subset\R^n$. We say that $A\subset M$ is a null set if $\phi(A)$ is a Lebesgue null set for all charts $\phi:U\subset M\rightarrow V\subset\R^n$.
\end{definition}
We denote by $\sim$ the equivalence relation that identifies functions $u$ and $v$ if they agree apart from a null set. As usual, we will often not distinguish between a function $u$ and its equivalence class $[u]_\sim$.

\begin{definition}
A function $u:M\rightarrow\R$ is in $W^{k,p}_{\operatorname{loc}}(M)$ if $u\circ\phi^{-1}\in W^{k,p}_{\operatorname{loc}}(V)$ for all charts $\phi:U\subset M\rightarrow V\subset\R^n$.
\end{definition}
\noindent
\textbf{Remarks:}
\begin{enumerate}[(1)]
    \item In case of a boundary chart where $V\subset\R^ {n-1}\times[0,\infty)$ and $V\cap(\R^ {n-1}\times 0)\neq\emptyset$, we say that $v\in W^ {k,p}(V)$ if $v\in W^ {k,p}(V^+)$, where $V^+:=\set{x\in V\ |\ x_n>0}$. 
    \item We stress that for $V\subset\R^ {n-1}\times[0,\infty)$ and $V\cap(\R^ {n-1}\times 0)\neq\emptyset$ the `$\operatorname{loc}$' refers to sets $W\subset V$ open in the relative topology on $\R^ {n-1}\times[0,\infty)$ with closure $\bar W\subset V$. Such $W$ are allowed to contain points of the form $(x,0)$. 
\end{enumerate}

To define a sensible $W^{k,p}(M)$-norm, we can not use an arbitrary chart. For example we consider $M=(0,1)$ and $u:M\rightarrow\R,\ u(x)=x$. A sensible definition of $W^{k,p}(M)$ must surely include $u$. However, taking the chart $\varphi:(0,1)\rightarrow (0,\infty), \varphi(x)=\frac1x$ produces the chart representation $u(\varphi^{-1}(t))=\frac1t$, which is not even in $L^1((0,\infty))$. The issue is, of course, that the chosen chart is \emph{not good}. To eliminate this problem, we make the following definition:

\begin{definition}\label{SchauderCh_SobolevNormDefGoodTalas}
Let $\mathcal A=\set{(\varphi_i, U_i)\ |\ 1\leq i\leq m}$ be a good atlas on $M$. It is easy to see that for all $u\in W^{k,p}_{\operatorname{loc}}(M)$ the expression 
\begin{equation}\label{SchauderCh_SobolevSpaceDefinition}
\|u\|_{W^{k,p}_\mathcal A(M)}:=\sum_{i=1}^m\|u\circ\varphi_i^{-1}\|_{W^{k,p}(V_i)}
\end{equation}
is finite\footnotemark. We define $W^{k,p}(M):=W^{k,p}_{\operatorname{loc}}(M)$,  equipped with $\|\cdot \|_{W^{k,p}_\mathcal A(M)}$.
\end{definition}
\footnotetext{Indeed, any good chart $\varphi_i:U_i\rightarrow V_i$ can be extended to $\varphi_i:\tilde U_i\rightarrow\tilde V_i$ with $V_i\subset\subset \tilde V_i$ and by definition of $W^{k,p}_{\operatorname{loc}}(M)$ we get $u\circ\varphi_i^{-1}\in W^{k,p}(V_i)$.}

Similar to the proof of Lemma \ref{SchauderCh_EquivalenceofHolderSpaceDefinition}, we can establish the following proposition, which shows, that Definition \ref{SchauderCh_SobolevSpaceDefinition} does not depend on the choice of the good atlas. 

\begin{proposition}\label{SchauderCh_GoodatlasSobolevNorm}
Let $M$ be a compact $n$-dimensional manifold and $\mathcal A$ and $\mathcal B$ be two good atlases on $M$. Then, for all $u\in W^{k,p}_{\operatorname{loc}}(M)$ we have 
$$\|u\|_{W^{k,p}_\mathcal A(M)}\leq C(\mathcal A,\mathcal B,k,p)\|u\|_{W^{k,p}_{\mathcal B}(M)}.$$
\end{proposition}

It is easy to see that the usual properties of Sobolev spaces, e.g. their completeness, carry over to the manifold setting. Also, we often use the following embedding theorems. For a proof in $\R^n$, we refer to Evan's book \cite{evans} (see Chapter 5, Section 7, Theorem 1 with the subsequent remark for Rellich's theorem and Chapter 5, Section 6 Theorem 4 for Morrey's inequality).
\begin{proposition}[Sobolev Embedding Theorem]\ \\
Let $M$ be a  compact, $n$-dimensional manifold, possibly with boundary. The following hold:
\begin{enumerate}[(1)]
    \item \emph{(Rellich's Theorem)}  The inclusion $W^{1,p}(M)\hookrightarrow L^p(M)$ is compact for all $1\leq p <\infty$.
    \item \emph{(Morrey's Inequality)} If $r\in\N_0$, $\alpha\in(0,1)$ and $k-\frac np=r+\alpha$, then the inclusion $W^{k,p}(M)\hookrightarrow C^{r,\alpha}(M)$ is continuous.
\end{enumerate}
\end{proposition}

Another result that carries over from $\R^n$ is the trace theorem. For $\R^n$, this result can be found in Evan's book \cite{evans} (see Chapter 5, Section 5, Theorem 1). 
\begin{proposition}[Trace Operator]\label{SchauderCh_TraceTheoremStandard}
Let $M$ be a compact manifold with boundary $\partial M\neq\emptyset$ and $p\in[1,\infty)$. There exists a linear and continuous map $T:W^{1,p}(M)\rightarrow L^p(\partial M)$ such that $Tu=u|_{\partial M}$ for $u\in C^0(M)$.
\end{proposition}
As usual, we write $u$ instead of $Tu$.

\subsection{The Weak Gradient}
Let $(M,g)$ be a compact, Riemannian manifold, $\mathcal A=\set{(\varphi_i, U_i)\ |\ 1\leq i\leq m}$ be a good atlas and $u\in W^ {1,p}(M)$. For each $1\leq i\leq m$, we can define the vector field
$$X_i:U_i\rightarrow TU_i,\ X_i(p):=g^ {ab}(\varphi_i(p))\partial_a (u\circ\varphi_i^ {-1})(\varphi_i(p))\frac{\partial}{\partial x^b}\bigg|_p.$$
Given two overlapping chart domains $U_i\cap U_j\neq\emptyset$, it is easy to see that $X_i$ and $X_j$ coincide on $U_i\cap U_j$ apart from perhaps a null set. So, we may define the weak gradient $\nabla u$ globally on $M$ by the local coordinates formula 
$$\nabla_g u(p)=g^ {ab}(\varphi_i(p))\partial_a (u\circ\varphi_i^ {-1})(\varphi_i(p))\frac{\partial}{\partial x^b}\bigg|_p.$$

\begin{proposition}\label{SchauderCh_MetricindeucedNormEquivalent}
Let $(M,g)$ be a compact Riemannian manifold, possibly with boundary. The norms 
$$
\|u\|_{L^p_g(M)}:=\left(\int_M |u|^p d\mu_g\right)^{\frac1p}
\hspace{.5cm}\textrm{and}\hspace{.5cm}
\|u\|_{W^{1,p}_g(M)}:=\left(\int_M |u|^p d\mu_g+\int_M \|\nabla_g u\|_g^p d\mu_g\right)^{\frac1p}
$$
provide norms on $L^p(M)$ and $W^{1,p}(M)$ that are equivalent to the $L^p(M)$- and $W^{1,p}(M)$-norm from Definition \ref{SchauderCh_SobolevSpaceDefinition}. Here $\|\nabla_g u\|_g:=(g(\nabla_g u,\nabla_g u))^ {\frac12}$.
\end{proposition}

For $p=2$, Proposition \ref{SchauderCh_MetricindeucedNormEquivalent} provides us with scalar products on $L^2(M)$ and $W^{1,2}(M)$ that makes these spaces Hilbert spaces. Concretely
$$
\langle u,v\rangle_{L^2_g(M)}:=\int_M uv d\mu_g
\hspace{.5cm}\textrm{and}\hspace{.5cm}
\langle u,v\rangle_{W^{1,2}_g(M)}:=\int_M uv+g(\nabla_g u,\nabla_g v)d\mu_g.
$$

The weak gradient satisfies the following expected property.
\begin{proposition}[Divergence Theorem]
Let $(M,g)$ be a compact Riemannian manifold, possibly with boundary $\partial M$. We denote the induced measure on $M$ by $\mu_g$ and the induced measure on the boundary by $S_g$. For any function $u\in W^{1,p}(M)$ and vector field $Y\in C^1(TM)$ 
\begin{equation}\label{NeumannLaplacianCh_DivergenceTheorem}
\int_M g(\nabla_g u,Y)d\mu_g=-\int_M u\operatorname{div}_g Y d\mu_g+\int_{\partial M}u g(Y,\nu_g)dS_g.
\end{equation}
$\nu_g$ denotes the exterior unit normal along $\partial M$ with respect to $g$ if $\partial M\neq\emptyset$. If $\partial M=\emptyset$, the boundary integral in Equation \eqref{NeumannLaplacianCh_DivergenceTheorem} vanishes.
\end{proposition}
\begin{proof}
    For $u\in C^1(M)$, this result is standard and can, for example, be found in \cite{lee2012smooth} (see Chapter 16, Theorem 16.32). For general $u\in W^{1,p}(M)$, we choose a sequence $(u_k)\subset C^1(M)$ such that $u_k\rightarrow u$ in $W^{1,p}(M)$. Equation \eqref{NeumannLaplacianCh_DivergenceTheorem} holds with $u$ replaced by $u_k$. By the continuity of the trace operator from Proposition \ref{SchauderCh_TraceTheoremStandard}, we can send $k\rightarrow\infty$ and obtain \eqref{NeumannLaplacianCh_DivergenceTheorem} for $u\in W^{1,p}(M)$.
\end{proof}

\begin{lemma}\label{SchauderCh_TraceInequalityLemma}
Let $(M,g)$ be a compact Riemannian manifold with boundary $\partial M\neq\emptyset$. Then, for any $\epsilon>0$ there exists a constant $C(\epsilon)$ such that for all $u\in W^ {1,2}(M)$
$$\|u\|_{L^2_g(\partial M)}^2\leq \epsilon\int_M \|\nabla_g u\|_g^2 d\mu_g+C(\epsilon)\|u\|_{L^2_g(M)}^2.$$
\end{lemma}
\begin{proof}
Let $\nu$ denote the exterior unity normal along $\partial M$ and let $\bar\nu$ denote a smooth vector field on $M$ such that $\bar\nu|_{\partial M}=\nu$. Given any smooth function $u\in C^\infty(M)$ we have 
$$\int_{\partial M} u^2dS_g=\int_{\partial M}g(u^2\bar\nu,\nu)dS_g=\int_{M}\operatorname{div}_g(u^2\bar\nu)d\mu_g=\int_M u^2\operatorname{div}_g\bar\nu+2ug(\bar\nu,\nabla_g u)d\mu_g.$$
Consequently, there exists a constant $C(M)$ such that 
\begin{equation}\label{SchauderCh_TraceEstimateEq1}
\int_{\partial M} u^2 dS_g\leq C\left(\int_M u^2 d\mu_g+\int_M |u|\ \|\nabla_g u\|_gd\mu_g\right).
\end{equation}
By density, Estimate \eqref{SchauderCh_TraceEstimateEq1} is true for all $u\in W^{1,2}(M)$. The lemma follows by employing Young's inequality. 
\end{proof}

\chapter{The Neumann Laplacian on Riemannian Manifolds}\label{SchauderCh_Ch02NeumannLaplacian}
Throughout this chapter, let $(M,g)$ be a smooth compact connected Riemannian manifold, potentially with boundary. We denote the Laplace-Beltrami operator by $\Delta_g$. In local coordinates 
$$\Delta_g u=g^{ij}\left(\partial_{ij} u-\Gamma^k_{\ ij}\partial_k u\right).$$
If $\partial M\neq\emptyset$, we denote the exterior unit normal along $\partial M$ by $\nu$.\\

This chapter aims to establish the spectral decomposition of the Neumann Laplacian on $M$.

\section{Existence of Weak Eigenfunctions}
We consider the elliptic equation
\begin{equation}\label{SchauderCh_WeakProblemMotivation}
\left\{
\begin{aligned}
    -\Delta_g u&=f\hspace{.5cm}\textrm{in $M$},\\
    \frac{\partial u}{\partial\nu}&=h\hspace{.5cm}\textrm{along $\partial M$}.
\end{aligned}
\right.
\end{equation}
We assume that for given $f$ and $h$, we have a solution $u$ to Problem \eqref{SchauderCh_WeakProblemMotivation} and that $u$, $f$ and $h$ are all sufficiently smooth. For $\varphi\in C^\infty(M)$, we then compute 
\begin{align*} 
\int_M f\varphi d\mu_g&=-\int_M \Delta_g u \varphi d\mu_g=\int_M g(\nabla_g u,\nabla_g \varphi)d\mu_g-\int_{\partial M}\frac{\partial u}{\partial\nu}\varphi dS_g\\
&=
\int_M g(\nabla_g u,\nabla_g \varphi)d\mu_g-\int_{\partial M}h\varphi dS_g.
\end{align*}
This formula makes sense for $u,\varphi\in W^{1,2}(M)$, $f\in L^2(M)$ and $h\in L^2(\partial M)$. So, given $f\in L^2(M)$ and $h\in L^2(\partial M)$, we say that $u\in W^{1,2}(M)$ is a weak solution to \eqref{SchauderCh_WeakProblemMotivation}, if for all $\varphi\in W^{1,2}(M)$ 
\begin{equation}\label{SchauderCh_WeakSolutionDefinitionEq}
\int_M g(\nabla_g u,\nabla_g \varphi)d\mu_g=\int_M f\varphi d\mu_g+\int_{\partial M}h\varphi dS_g.
\end{equation}
\begin{lemma}[Poincaré Inequality]\label{SchauderCh_PoincareInequality}
There exists a constant $\Lambda\in(0,1)$ such that for all $u\in W^{1,2}(M)$ with $\int_M ud\mu_g=0$ we have 
$$\int_M \|\nabla u_g\|_g^2d\mu_g\geq\Lambda\int_M u^2d\mu_g.$$
\end{lemma}
\begin{proof}
If this was not the case, we could find a sequence $(u_m)\subset W^{1,2}(M)$ that satisfies  $\int_M u_m d\mu_g=0$ and 
$$\int_M \|\nabla_g u_m\|_g^2 d\mu_g<\frac1m\int_M u_m^2 d\mu_g.$$
We put $v_m:=\|u_m\|_{L^2_g(M)}^{-1}u_m$ and obtain 
$$\int_M \|\nabla_g v_m\|_g^2 d\mu_g<\frac1m,
\hspace{.5cm}
\int_M v_m d\mu_g=0
\hspace{.5cm}\textrm{and}\hspace{.5cm} \|v_m\|_{L^2_g(M)}^2=1.$$
By Proposition \ref{SchauderCh_MetricindeucedNormEquivalent}, $(v_m)\subset W^{1,2}(M)$ is bounded. So, there exists a function $v^*\in W^{1,2}(M)$ such that $v_m\rightarrow v^*$ weakly in $W^{1,2}(M)$ and strongly in $L^2(M)$ along a subsequence which we again denote by $v_m$. As $\|\nabla_g v_m\|_{L^2_g(M)}^2<\frac1m$ we learn $\nabla v^*\equiv 0$. In local coordinates 
$$0=\nabla_g v^*= g^{ij}\partial_i (v^*\circ\phi^{-1})\frac{\partial}{\partial x^j}.$$
So $v^*$ is locally and by connectedness, therefore also globally constant, i.e. $v^*\equiv c$ for some $c\in\R$. Since 
$$|M|c=\int_M v^*d\mu_g=\lim_{m\rightarrow \infty}\int_M v_md\mu_g=0,$$
$c=0$. However, since $v_m\rightarrow v^*$ strongly in $L^2(M)$ and $\|v_m\|_{L^2_g(M)}^2\equiv 1$ we also get $\|v^*\|_{L^2_g(M)}^2=1$. This is a contradiction and so the lemma follows. 
\end{proof}

Next, we establish the existence of weak eigenfunctions.

\begin{lemma}\label{SchauderCh_ExistenceofNeumannEFLemma}
There exists $0=\lambda_0<\lambda_1\leq...\leq\lambda_k\rightarrow\infty$ and pairwise $L^2_g(M)$-orthonormal and $W^{1,2}_g(M)$-orthogonal functions $\phi_k\in W^{1,2}(M)$ that satisfy weakly 
\begin{equation}\label{SchauderCh_EigenFunctionEquation}
\left\{
\begin{array}{rll}
    \displaystyle-\Delta_g \phi_k&\displaystyle=\lambda_k\phi_k&\displaystyle\hspace{.5cm}\textrm{in $M$},\vspace{.1cm}\\   
    \displaystyle\frac{\partial \phi_k}{\partial\nu}&\displaystyle=0&\displaystyle\hspace{.5cm}\textrm{along $\partial M$}.
\end{array}
\right.
\end{equation}
The span of $\phi_k$ is dense in $L^2(M)$ and $W^{1,2}(M)$. Additionally, $\phi_0=|M|^{-\frac12}$. 
\end{lemma}
\begin{proof}
We define $X:=\set{u\in W^{1,2}(M)\ |\ \int_M ud\mu_g=0}$ and consider the bilinear form 
$$B:X\times X\rightarrow \R,\ B[u,v]:=\int_M g(\nabla_g u,\nabla_g v) d\mu_g.$$
By Proposition \ref{SchauderCh_MetricindeucedNormEquivalent}, $B$ is continuous and in view of Lemma \ref{SchauderCh_PoincareInequality}, $B$ defines a scalar product on $X$ that induces the $W^{1,2}(M)$-topology. We equip $X$ with $B$ as a scalar product and denote the induced norm by $\|\cdot\|_X$. The Riesz–Fréchet representation theorem implies that for any $f\in L^2(M)$ there exists a unique $u_f\in X$ such that  
$$\langle f, v\rangle_{L^2_g(M)}=B[u_f,v]
\hspace{.5cm}\textrm{for all $v\in X$}.$$
We define $\tilde K:L^2(M)\rightarrow X$ by mapping $f\mapsto u_f$. It is easy to see that $\tilde K$ is linear. Again using Lemma \ref{SchauderCh_PoincareInequality} we estimate
$$\|\tilde K(f)\|_{X}^2=B[\tilde K(f),\tilde K(f)]=\langle f, \tilde K(f)\rangle_{L^2_g(M)}\leq \|f\|_{L^2_g(M)}\|\tilde K(f)\|_{L^2_g(M)}\leq C\|f\|_{L^2_g(M)}\|\tilde K(f)\|_X.$$
So $\tilde K$ is continuous. Let $\imath:X\hookrightarrow L^2(M)$ denote the inclusion. As $\imath$ is compact, the operator $K:=\imath\circ \tilde K:L^2(M)\rightarrow L^2(M)$ is compact. Next, we compute $\ker(K)$. If $K(f)=0$, then $\tilde K(f)=0$ and for any $v\in X$ 
\begin{equation}\label{NeumannChapter_KernelEq}
\langle f, v\rangle_{L^2_g(M)}=B[\tilde K(f),v]=0.
\end{equation}
Let $\varphi\in C^\infty(M)$ and $\sigma:=\fint_M\varphi d\mu_g$ so that $\varphi-\sigma\in X$. Testing Equation \eqref{NeumannChapter_KernelEq} with $\varphi-\sigma$ gives
$$
\int_M\varphi fd\mu_g
=\int_M \sigma fd\mu_g
=\int_Mf\left(\fint_M \varphi d\mu_g\right)d\mu_g
=\int_M\varphi\left(\fint_M fd\mu_g\right)d\mu_g.
$$
As $\varphi\in C^ \infty(M)$ is arbitrary we deduce $f=\fint fd\mu_g$. So $f$ is a constant and hence $\ker(K)\subset \operatorname{span}(1)$. It is readily seen that $K(1)=0$ and hence $\ker(K)= \operatorname{span}(1)$.
Next, for $f,g\in L^2(M)$ we have
$$
\langle K(f), g\rangle_{L^2_g(M)}
=
\langle g, \tilde K(f)\rangle_{L^2_g(M)}
=
B[\tilde K(g),\tilde K(f)]
=
B[\tilde K(f),\tilde K(g)]
=...=
\langle f, K(g)\rangle_{L^2_g(M)}.$$
So $K$ is a self-adjoint compact operator. By the spectral theorem, there exists $\mu_k\in\R\backslash\set0$ with $1\leq k\in\N$ and pairwise $L^2_g(M)$-orthonormal functions $\phi_k\in L^2(M)$ such that $K(\phi_k)=\mu_k \phi_k$ and such that $\operatorname{span}(\phi_k)\oplus\operatorname{span}(1)$ is dense in $L^2(M)$. Note that 
$$
\mu_k
=
\mu_k\|\phi_k\|_{L^2_g(M)}^2
=
\langle  \phi_k,K(\phi_k)\rangle_{L^2_g(M)}
=
B[\tilde K(\phi_k),\tilde K(\phi_k)]\geq 0.
$$
As $\mu_k\neq 0$, we get $\mu_k>0$ for all $k\geq 1$. So, after possibly relabelling the indices, we may assume that $\mu_1\geq \mu_2\geq...$. Finally, again by the spectral theorem, we have $\mu_k\rightarrow 0$ as $k\rightarrow\infty$. We note
$$\phi_k=\frac1{\mu_k}K(\phi_k)=\frac1{\mu_k}i(\tilde K(\phi_k))=\frac1{\mu_k}\tilde K(\phi_k)\in X.$$
Let $\varphi\in C^\infty(M)$ and put $\sigma:=\fint_M \varphi d\mu_g$ so that $\varphi-\sigma\in X$. Using $\phi_k\in X$ and the definition of $B$, we get 
\begin{align}
\int_M \langle \nabla_g \phi_k,\nabla_g\varphi\rangle d\mu_g
=&
B[\phi_k,\varphi-\sigma]
=
\frac1{\mu_k}B[\tilde K(\phi_k),\varphi-\sigma]\nonumber\\
=&
\frac1{\mu_k}\int_M \phi_k\varphi d\mu_g-\frac\sigma{\mu_k}\int_M \phi_k d\mu_g\nonumber\\
=&
\frac1{\mu_k}\int_M \phi_k\varphi d\mu_g.\label{SchauderCh_SpectralTheoremproofExistence}
\end{align}
For $k\geq 1$ we put $\lambda_k:=\mu_k^{-1}$, define $\lambda_0:=0$ and $\phi_0:=|M|^{-\frac12}$. In view of Equation \eqref{SchauderCh_WeakSolutionDefinitionEq}, Equation \eqref{SchauderCh_SpectralTheoremproofExistence} is precisely the definition of stating that for $k\geq1$ the function $\phi_k$ solves \eqref{SchauderCh_EigenFunctionEquation} weakly with eigenvalue $\lambda_k$. As $\phi_0$ is a constant, this is also true for $k=0$. Since $\mu_k\rightarrow 0^+$ we get $\lambda_k\rightarrow\infty$. Given $u\in C^\infty(M)$ we use Equation \eqref{SchauderCh_SpectralTheoremproofExistence} to get
\begin{equation}\label{NeumannChapter_ScalarProductRelations}
\langle u,\phi_k\rangle_{W^{1,2}_g(M)}=\langle u,\phi_k\rangle_{L^2_g(M)}+\int_M g(\nabla u,\nabla\phi_k)d\mu_g=(1+\lambda_k)\langle u,\phi_k\rangle_{L^2_g(M)}.
\end{equation}
After approximation, this identity is also true for $u\in W^{1,2}(M)$. So if  $\langle u,\phi_k\rangle_{W^{1,2}_g(M)}=0$ for some  $u\in W^{1,2}(M)$ and all $k\geq 0$, we deduce $\langle u,\phi_k\rangle_{L^2_g(M)}=0$ for all $k\geq 0$ and thus $u=0$ as $(\phi_k)_k$ is an $L^2_g(M)$-orthonormal basis of $L^2(M)$. Finally, choosing $u=\phi_l$ in Equation \eqref{NeumannChapter_ScalarProductRelations}, we obtain $\langle \phi_k,\phi_l\rangle_{W^{1,2}_g(M)}=(1+\lambda_k)\delta_{kl}$ and hence the functions $\phi_k$ constitute a $W^{1,2}_g(M)$-orthogonal basis of $W^{1,2}(M)$. 
\end{proof}

\section{Elliptic Regularity Theory for the Neumann Problem}
To deduce higher regularity for the eigenfunctions $\phi_k$ from Lemma \ref{SchauderCh_ExistenceofNeumannEFLemma}, we express Equation \eqref{SchauderCh_EigenFunctionEquation} in charts and use elliptic regularity theory for Neumann type problems in $\R^ n$. These come in the form of interior and boundary estimates. Interior estimates are easily found in the literature. The following theorem is, for example, proven in Evans' book (see Theorem 2 in Chapter 6, Section 3, Subsection 1)
\begin{theorem}[Interior Regularity]\label{SchauderCh_NeumannLaplacianINterior}
Let $\Omega\subset\R^n$ be a bounded domain, $m\in\N_0$,  $a^{ij}$, $b^i$, $c\in C^{m+1}(\bar\Omega)$ and $\theta$, $\Lambda>0$ such that $a^{ij}(x)\xi_i\xi_j\geq\theta|\xi|^2$ for all $x\in\Omega$ and $\|a^{ij}\|_{C^{m+1}(\bar\Omega)}$, $\|b^{i}\|_{C^{m+1}(\bar\Omega)}$, $\|c\|_{C^{m+1}(\bar\Omega)}\leq\Lambda$. Let further $f\in W^{m,2}(\Omega)$ and let $u\in W^{1,2}(\Omega)$ be a weak solution of $Lu:=a^{ij}\partial_{ij}u+b^i \partial_i u+cu=f$ in $\Omega$. Then $u\in W^{m+2,2}(\Omega')$ for all $\Omega'\subset\subset\Omega$ and
$$\|u\|_{W^{m+2,2}(\Omega')}\leq C(\Lambda,\theta, \Omega',\Omega)(\|u\|_{L^2(\Omega)}+\|f\|_{W^{m,2}(\Omega)}).$$
\end{theorem}

A similar statement is true for weak solutions near the boundary. A proof is, however, rather difficult to track down in the literature. For the reader's convenience, we present it here. Let us first introduce some notation. We consider a bounded domain $\Omega\subset\R^{n-1}\times [0,\infty)$ and put 
$$
\Omega^+:=\set{x\in\Omega\ |\ x_n>0}
\hspace{.5cm}\textrm{and}\hspace{.5cm}
\partial^+\Omega:=\Omega\cap(\R^{n-1}\times 0).$$
We assume that $\partial^+\Omega\neq \emptyset$ and study the boundary value problem
\begin{equation}\label{SchauderCh_EllipticEqatBoundary}
\left\{
\begin{array}{ l l}
a^{ij}\partial_{ij}u+b^i \partial_i u+cu=f&\textrm{in $\Omega$},\vspace{.1cm}\\
-a^{nj}\partial_j u=g&\textrm{along $\partial^+\Omega$}.
\end{array}
\right.
\end{equation}
To formulate a weak formulation, we consider test functions from the space
$$W^{k,2}_0(\Omega):=\set{\varphi\in W^{k,2}(\Omega^+)\ |\ \exists V\subset\subset\Omega\textrm{ open such that }\varphi\equiv0\textrm{ on }\Omega\backslash V\textrm{ almost everywhere}}.$$
We stress that $V\subset\Omega$ is open in the relative topology, so $\varphi \in W^{k,2}_0(\Omega)$ need not vanish along $\partial^+\Omega$. Putting $\tilde b^ i:=b^ i-\partial_j a^ {ij}$, we say that a function $u\in W^{1,2}(\Omega^ +)$ is a weak solution of \eqref{SchauderCh_EllipticEqatBoundary} if for all $\varphi\in W^ {1,2}_0(\Omega)$
$$
\int_{\Omega^+}-a^{ij}\partial_i u\partial_ j\varphi+\tilde b^i\partial_i u\varphi+cu\varphi=\int_{\Omega^+} f\varphi-\int_{\partial^+\Omega}g\varphi.
$$

The following analysis requires the concept of difference quotients. Given a function $u:\Omega\rightarrow \R$, $x\in\Omega$, $h\in\R$ and $1\leq l\leq n$, we put 
$$\left(D^ h_l u\right)(x):=\frac{u(x+he_l)-u(x)}h\hspace{.5cm}\textrm{when $x+he_l\in\Omega$}.$$
The following facts are readily available in the literature and we refer to Evans' book \cite{evans} for a proof (see Chapter 5, Section 8, Subsection2, Theorem 3).
\begin{lemma}[Properties of Difference Quotients]\label{SchauderCh_DiffrenceQuotLemma}\ \\
    Let $\Omega\subset\R^n$ be a domain, $\Omega'\subset\subset \Omega$ and $1\leq l\leq n$. 
    \begin{enumerate}[(1)]
        \item There exists $h_0(\Omega,\Omega')>0$ such that $D^h_lu\in L^2(\Omega')$ for all $u\in W^{1,2}(\Omega)$ and $0<|h|<h_0$. Additionally $\|D^h_l u\|_{L^2(\Omega')}\leq C\|\partial_l u\|_{L^2(\Omega)}.$
        \item If $u\in L^2(\Omega)$ and $\|D^h_l u\|_{L^2(\Omega')}\leq C<\infty$ for all $0<|h|<h_0$, then $u$ is weakly differentiable on $\Omega'$ with respect to $x_l$ and $\|\partial_l u\|_{L^2(\Omega')}\leq C$.
    \end{enumerate}
\end{lemma}
We note that Lemma \ref{SchauderCh_DiffrenceQuotLemma} also holds for domains in $\R^{n-1}\times[0,\infty)$ and $1\leq l\leq n-1$. 

\begin{theorem}[Boundary Regularity]\label{SchauderCh_BoundaryRegularityLemmaStep1}
Let $\Omega\subset\R^{n-1}\times[0,\infty)$ be a bounded domain, $a^{ij}\in C^1(\bar\Omega)$, $b^i$, $c\in C^0(\bar \Omega)$ and $\theta$, $\Lambda>0$ such that  $a^{ij}(x)\xi_i\xi_j\geq\theta|\xi|^2$ for all $x\in\Omega$ and $\|a^{ij}\|_{C^1(\bar\Omega)},$ $\|b^i\|_{C^0(\bar \Omega)}$, $\|c\|_{C^0(\bar \Omega)}\leq\Lambda$. Let further $f\in L^2(\Omega^+)$, $g\in W ^{1,2}(\partial^+\Omega)$ and let $u\in W^{1,2}(\Omega^+)$ be a weak solution of 
$$
\left\{\begin{array}{ l l}
a^{ij}\partial_{ij}u+b^i \partial_i u+cu=f&\textrm{in $\Omega^+$},\vspace{.1cm}\\
-a^{nj}\partial_j u=g&\textrm{along $\partial^+\Omega$}.
\end{array}
\right.$$
Then $u\in W^{2,2}(V^+)$  for all $V\subset\subset \Omega$ and
$$
\|u\|_{W^{2,2}(V^+)}\leq C(\Lambda,\theta, V,\Omega)(\|u\|_{L^2(\Omega^+)}+\|f\|_{L^{2}(\Omega^+)}+\|g\|_{W^{1,2}(\partial^+\Omega)}).
$$
\end{theorem}

\begin{proof}
Throughout the proof, we will suppress the dependence of various constants on $\Lambda$ and $\theta$. Let $W\subset\Omega$ such that $V\subset\subset W\subset\subset\Omega$. For any $\varphi\in C^\infty_0(W)$ we have 
\begin{equation}\label{SchauderCh_WeakEquationonTheBoundary}
\int_{\Omega^+}-a^{ij}\partial_i u\partial_ j\varphi+\tilde b^i\partial_i u\varphi+cu\varphi=\int_{\Omega^+} f\varphi-\int_{\partial^+\Omega}g\varphi.
\end{equation}
We choose a domain $\tilde W$ such that $V\subset\subset\tilde W\subset\subset W$, a test function $\zeta\in C^\infty(\Omega)$ such that $\zeta\equiv 1$ on $V$ and $\zeta\equiv 0$ on $\Omega\backslash \tilde W$ and fix $1\leq l\leq n-1$. For small real numbers $|h|\leq h_0(V,W)$ we consider the test function 
$$
\varphi:=D^{-h}_l\left(\zeta^2 D_l^h u\right)
\hspace{.5cm}\textrm{which, for $|h|\leq h_0$, satisfies }\operatorname{supp}\varphi\subset W.
$$
We insert this choice of $\varphi$ into Equation \eqref{SchauderCh_WeakEquationonTheBoundary} and investigate the individual terms by themselves. First, using a discrete integration by parts, we compute 
\begin{align*}
    \int_{\Omega^+}-a^{ij}\partial_i u\partial_ j\varphi =& \int_{\Omega^+}-a^{ij}\partial_i uD^{-h}_l\left[\partial_ j\left(\zeta^2D^h_l u\right)\right]\\
    =&\int_{\Omega^+}D^h_l (a^{ij}\partial_ i u)\zeta^2 D^h_l\left( \partial_j u\right)
    +
    \int_{\Omega^+}D^h_l (a^{ij}\partial_ i u)2\zeta\partial_j\zeta D^h_l u\\
    =&\int_{\Omega^+}a^{ij}(x+he_l)\partial_ i\left(D_l^h u\right)\zeta^2 D^h_l\left( \partial_j u\right)
    +\int_{\Omega^+}\zeta^2D^h_l(a^{ij})\partial_i u D^h_l \left(\partial_ ju\right) \\
    &+2\int_{\Omega^+}D^h_l(a^{ij})\partial_i u\zeta\partial_j\zeta D^h_l u
    +2\int_{\Omega^+}a^{ij}(x+he_l)D^h_l(\partial_i u)\zeta\partial_j \zeta D^h_l u.
\end{align*}
We can lower bound the first term by exploiting the uniform ellipticity of the coefficients $a^{ij}$. To get a lower term for the entire expression, we use the triangle inequality and the regularity of the coefficients to bound $|a^{ij}|+|D^h_l(a^{ij})|\leq C$. So
\begin{align*}
    \int_{\Omega^+}-a^{ij}\partial_iu \partial_ j\varphi 
    \geq &\theta\int_{\Omega^+} |D^h_l\nabla u|^2 \zeta^2 
    -C\int_{\Omega^+}\zeta^2|\nabla u| |D^h_l\nabla u|\\
    &-C\int_{\Omega^+}|\nabla\zeta|\zeta |\nabla u||D^h_l u|
    -C\int_{\Omega^+}\zeta|\nabla\zeta||D^h_l\nabla u| |D^h_l u|.
\end{align*}
Let $\epsilon>0$ be a small parameter. Using Young's inequality, $\operatorname{supp}(\zeta)\subset W$ and $|\nabla\zeta|\leq C(V,W)$ we estimate 
\begin{align}
    \int_{\Omega^+}-a^{ij}\partial_i u\partial_ j\varphi 
    \geq &\theta\int_{\Omega^+} |D^h_l\nabla u|^2 \zeta^2 
    -\epsilon\int_{\Omega^+}\zeta^2 |D^h_l \nabla u|^2-C(\epsilon)\int_{\Omega^+}\zeta^2|\nabla u|^2\nonumber\\
    &-C\int_{\Omega^+}\zeta \left(|\nabla u|^2+|D^h_l u|^2\right)
    -C(\epsilon)\int_{\Omega^+}|\nabla\zeta|^2 |D^h_l u|^2\nonumber\\
    \geq & (\theta-\epsilon)\int_{\Omega^+}\zeta^2 |D^h_l \nabla u|^2-C(V,W,\epsilon)\int_{W^+}|\nabla u|^2\nonumber\\
    & -C(V,W,\epsilon)\int_{\operatorname{supp}(\zeta)^+}|D^h_l u|^2.\label{SchauderCh_EllipticBaseCaseBoundaryEstimateeq1}
\end{align}
For small enough $h$ we may use $\operatorname{supp}(\zeta)\subset W$ and Lemma \ref{SchauderCh_DiffrenceQuotLemma} to estimate
$$\int_{\operatorname{supp}(\zeta)^+}|D^h_l u|^2\leq C\int_{W^+}|\nabla u|^2.$$
Inserting into Estimate \eqref{SchauderCh_EllipticBaseCaseBoundaryEstimateeq1} and choosing $\epsilon=\frac12\theta$ we obtain 
\begin{equation}\label{SchauderCh_BoundaryReglemma1_FirstTerm}
     \int_{\Omega^+}-a^{ij}\partial_i \partial_ j\varphi 
    \geq \frac\theta2\int_{\Omega^+} |D^h_l\nabla u|^2 \zeta^2-C(V,W)\int_{W^+}|\nabla u|^2.
\end{equation}
Using $\operatorname{supp}(\varphi)\subset W$, the continuity of $\tilde b^i$ and $c$ and Young's inequality, we estimate 
\begin{align}
    \left|\int_{\Omega^+}f\varphi-\tilde b^i\partial_i u\varphi-cu\varphi\right|\leq & \epsilon\int_{\Omega^+}\varphi^2 +C(\epsilon)\int_{W^+}|f|^2+|\nabla u|^2+u^2.\label{SchauderCh_EllipticbaseCaseLemmaVarphiL20}
\end{align}
Next, we estimate the $L^2$-norm of $\varphi$. Using $\operatorname{supp}\varphi\subset W$, we may apply Lemma \ref{SchauderCh_DiffrenceQuotLemma} for small enough $h$ and estimate
\begin{align}
    \int_{\Omega^+}\varphi^2  &= \int_{W^+}\left|D^{-h}_l\left(\zeta^2 D_l^h u\right)\right|^2\nonumber\\
     &\leq  C \int_{W^+}\left|\partial_l\left(\zeta^2 D_l^h u\right)\right|^2\nonumber\\
    &\leq  C\int_{W^+}|\zeta\nabla\zeta|^2|D^h_l u|^2+\zeta^2|D^h_l\nabla u|^2\nonumber\\
    &\leq  C(V,W)\left(\int_{W^+}|\nabla u|^2+\int_{W^+}\zeta^2|D^h_l\nabla u|^2\right).\label{SchauderCh_EllipticbaseCaseLemmaVarphiL2}
\end{align}
Combining Estimates \eqref{SchauderCh_EllipticbaseCaseLemmaVarphiL20} and \eqref{SchauderCh_EllipticbaseCaseLemmaVarphiL2}, we obtain
\begin{equation}\label{SchauderCh_BoundaryRegLemma2_SecondTerm}
\left|\int_{\Omega^+}\hspace{-.1cm}f\varphi-\tilde b^i\partial_i u\varphi-cu\varphi\right|\leq C(\epsilon,V,W)\left(\int_{W^+}\hspace{-.1cm}|f|^2+|u|^2+|\nabla u|^2\right)+\epsilon\int_{W^+}\hspace{-.1cm}\zeta^2 |D^h_l\nabla u|^2.
\end{equation}
Finally, we estimate the boundary integral in Equation \eqref{SchauderCh_WeakEquationonTheBoundary}. Since $\partial^+\Omega\subset\R^{n-1}\times 0$ and $1\leq l\leq n-1$, we can move the difference quotient $D^{-h}_l$ from the test function $\varphi$ to $g$. Using $\operatorname{supp}(\varphi)\cap\partial^+\Omega\subset \partial^+W$ and $g\in W^{1,2}(\partial^+\Omega)$ we may then estimate  
\begin{equation}\label{SchauderCh_WeakEquationonTheBoundary29}
\left|\int_{\partial\Omega^+}gD^{-h}_l (\zeta^2 D^h_l u)\right|
=
\left|\int_{\partial\Omega^+}(D^{h}_lg)\zeta^2 D^h_l u\right|
\leq 
C(\epsilon)\int_{\partial^+W}|\nabla_{\R^{n-1}} g|^2+\epsilon\int_{\partial^+\Omega}\zeta^2|D^h_l u|^2.
\end{equation}
Since the trace operator $W^{1,2}(\Omega)\rightarrow L^2(\partial^+\Omega)$ is continuous, we may estimate
\begin{align}
    \int_{\partial\Omega^+}\zeta^2|D^h_l u|^2 \leq &C\int_{\Omega^+}\zeta^2|D^h_l u|^2+\zeta|\nabla\zeta||D^h_l u|^2+\zeta^2|D^h_l u|\ |D^h_l\nabla u |\nonumber\\
    \leq & C(V,W)\left(\int_{\operatorname{supp}(\zeta)^+}|D^h_l u|^2+\int_{\Omega^+}\zeta^2 |D^h_l\nabla u|^2\right).\label{SchauderCh_WeakEquationonTheBoundary40}
\end{align}
Inserting Estimate \eqref{SchauderCh_WeakEquationonTheBoundary40} into Estimate  \eqref{SchauderCh_WeakEquationonTheBoundary29} and using Lemma \ref{SchauderCh_DiffrenceQuotLemma}, we get
\begin{equation}\label{SchauderCh_BoundaryRegLemma1_Estiamte3}
\left|\int_{\partial\Omega^+}gD^{-h}_l (\zeta^2 D^h_l u)\right|
\leq 
C(\epsilon,V,W)\left(\|g\|_{W^{1,2}(\partial^+ W)}^2+\int_{W^+}|\nabla u|^2\right)+\epsilon\int_{\Omega^+}\zeta^2|D^h_l\nabla u|^2.
\end{equation}
Putting Estimates \eqref{SchauderCh_BoundaryReglemma1_FirstTerm}, \eqref{SchauderCh_BoundaryRegLemma2_SecondTerm} and \eqref{SchauderCh_BoundaryRegLemma1_Estiamte3} into Equation \eqref{SchauderCh_WeakEquationonTheBoundary} and choosing $\epsilon(\theta)$ small enough, we obtain 
\begin{equation}\label{SchauderCh_Reglemma1_Estimate4}
\frac\theta4\int_{\Omega^+}\zeta^2|D^h_l\nabla u|^2\leq C(V,W)\left(\|f\|_{L^2(W^+)}^2+\|g\|_{W^{1,2}(\partial^+W)}^2+\|u\|_{W^{1,2}(W^+)}^2\right).
\end{equation}
As this estimate holds uniformly for all small $h$ and $\zeta\equiv 1$ on $V$, we deduce that on $V^+$ the gradient $\nabla u$ is weakly differentiable in $l$-direction and 
\begin{equation}\label{SchauderCh_Reglemma1_Estimate5}
\int_{V^ +}|\partial_l \nabla u|^2\leq C(V,W)\left(\|f\|_{L^2(W^+)}^2+\|u\|_{W^{1,2}(W^+)}^2+\|g\|_{W^{1,2}(\partial^+W)}^2\right).
\end{equation}
Next, for $\epsilon>0$ small we consider $V_\epsilon:=\set{x\in V\ |\ x_n>\epsilon}$. By Theorem \ref{SchauderCh_NeumannLaplacianINterior}, we deduce that $u\in W^{2,2}(V_\epsilon)$ for every small $\epsilon>0$. So, it suffices to derive an estimate for $\|\partial_{ni} u\|_{L^2(V_\epsilon)}$ that is independent of $\epsilon$. First, we note that for $1\leq i\leq n-1$ we have $\|\partial_{ni} u\|_{L^2(V_\epsilon)}\leq \|\partial_i\nabla u\|_{L^2(V^+)}^2$, which is estimated by Estimate \eqref{SchauderCh_Reglemma1_Estimate5}. 
So it remains to establish a uniform estimate for $\|\partial_{nn} u\|_{L^2(V_\epsilon)}$.  As $u\in W^{2,2}(V_\epsilon)$ for all $\epsilon>0$, we deduce that pointwise almost everywhere in $V^+$
$$\partial_{nn} u=\frac1{a^{nn}}\left(f-\sum_{i+j\leq 2n-1}a^{ij}\partial_{ij} u-\sum_{i=1}^ nb^i\partial_i u-cu\right).$$
Again using Estimate \eqref{SchauderCh_Reglemma1_Estimate5}, this implies the uniform estimate
\begin{align*}
    \int_{V_\epsilon}|\partial_{nn} u|^2\leq & C(V,W)\left(\int_{V^+} f^2 +\sum_{i+j\leq 2n-1}|\partial_{ij} u|^2+|\nabla u|^2+u^2\right)\\
    \leq & C(V,W)\left(\|f\|_{L^2(W^+)}^2+\|u\|_{W^{1,2}(W^+)}^2+\|g\|_{W^{1,2}(\partial^+W)}^2\right).
\end{align*}
In total we have shown that $u\in W^{2,2}(V^+)$ and 
\begin{equation}\label{schauderCh_EllipticInductionStartMissing}
\int_{V^+}|\nabla^2 u|^2\leq  C(V,W)\left(\|f\|_{L^2(W^+)}^2+\|u\|_{W^{1,2}(W^+)}^2+\|g\|_{W^{1,2}(\partial^+W)}^2\right).
\end{equation}
This estimate is, however, worse than the one claimed in the theorem. We still need to derive an estimate for $\|u\|_{W^{1,2}(W^+)}$. To do so, let $W\subset\subset Z\subset\subset \Omega$. Using the same arguments as before, we also may prove that $u\in W^{2,2}(W^+)$. Now let $\zeta\in C^\infty_0(\Omega, [0,1])$ be another cutoff function such that $\zeta\equiv 1$ on $W$ and $\zeta\equiv 0$ outside of $Z$. Taking $\varphi=\zeta^2 u$ in Equation \eqref{SchauderCh_WeakEquationonTheBoundary}, we get 
\begin{equation}\label{SchauderCh_Reglemma1_Estimate6}
\int_{\Omega^+}a^{ij}\partial_ iu\partial_j (\zeta^2 u)
=\int_{\Omega^+}(\tilde b^i \partial_i u+cu-f)\zeta^2 u+\int_{\partial^+\Omega} g\zeta^2u.
\end{equation}
We first estimate the term on the left-hand side by writing
\begin{align}
    \int_{\Omega^+}a^{ij}\partial_i u\partial_j(\zeta^2 u)=&\int_{\Omega +}\zeta^2 a^{ij}\partial_i u\partial_j u+2\zeta a^{ij}\partial_i u\partial_j\zeta u\nonumber\\
    \geq & \theta\int_{\Omega^+}\zeta^2|\nabla u|^2-C(W,Z)\int_{Z^+}\zeta|\nabla u|\ |u|\nonumber\\
    \geq &(\theta-\epsilon)\int_{\Omega^+}\zeta^2|\nabla u|^2-C(W,Z,\epsilon)\int_{Z^+} |u|^2.\label{SchauderCh_Reglemma1_Estimate7}
\end{align}
Next, we estimate the right-hand side of Equation \eqref{SchauderCh_Reglemma1_Estimate6}. Using the Cauchy-Schwarz inequality, Young's inequality and the continuity of the trace operator $W^{1,2}(\Omega^+)\rightarrow L^2(\partial^+\Omega)$, we obtain
\begin{align}
    &\left|\int_{\Omega^+}(\tilde b^i \partial_i u+cu-f)\zeta^2 u+\int_{\partial^+\Omega} g\zeta^2 u\right|\nonumber\\
    \leq &C\left(\|f\|_{L^2(\Omega^+)}+\|\zeta\nabla u\|_{L^2(\Omega^+)}+\|u\|_{L^2(\Omega^+)}\right)\|u\|_{L^2(\Omega^+)}+\|g\|_{L^2(\partial^+\Omega)}\|\zeta u\|_{L^2(\partial^+\Omega)}\nonumber\\
    \leq &C\left(\|f\|_{L^2(\Omega^+)}+\|\zeta\nabla u\|_{L^2(\Omega^+)}+\|u\|_{L^2(\Omega^+)}\right)\|u\|_{L^2(\Omega^+)}+\|g\|_{L^2(\partial^+\Omega)}\|\zeta u\|_{W^{1,2}(\Omega^+)}\nonumber\\
    \leq &C(\epsilon)\left(\|f\|_{L^2(\Omega^+)}^2+\|g\|_{L^2(\partial^+\Omega)}^2+\|u\|_{L^2(\Omega^+)}^2\right) 
    +\epsilon\|\zeta\nabla u\|_{L^2(\Omega^+)}^2
    +\epsilon\|\zeta u\|_{W^{1,2}(\Omega    +)}^2.\label{SchauderCh_W21EstimateEq1}
\end{align}
We estimate 
$$\|\zeta u\|_{W^{1,2}(\Omega^+)}^2=\int_{\Omega^+}\zeta^2 u^2 +\int_{\Omega^+}(u\nabla\zeta+\zeta \nabla u)^2\leq C\int_{\Omega^+} u^2+\int_{\Omega^+}\zeta^2|\nabla u|^2.$$
Inserting into Estimate \eqref{SchauderCh_W21EstimateEq1}, we get 
\begin{align}
   & \left|\int_{\Omega^+}(\tilde b^i \partial_i u+cu-f)\zeta^2 u+\int_{\partial^+\Omega} g\zeta^2 u\right|\nonumber\\
    \leq &
    C(\epsilon)\left(\|f\|_{L^2(\Omega^+)}^2+\|g\|_{L^2(\partial^+\Omega)}^2+\|u\|_{L^2(\Omega^+)}^2\right) 
    +\epsilon\|\zeta\nabla u\|_{L^2(\Omega^+)}^2.\label{SchauderCh_Reglemma1_Estimate8}
\end{align}
Putting Estimates \eqref{SchauderCh_Reglemma1_Estimate7} and \eqref{SchauderCh_Reglemma1_Estimate8} into Equation \eqref{SchauderCh_Reglemma1_Estimate6} and choosing $\epsilon(\theta)$ small enough, we obtain
\begin{equation}\label{SchauderCh_Reglemma1_Estimate9}\int_{W^+}|\nabla u|^2\leq C\left(\|f\|_{L^2(\Omega^+)}^2+\|g\|_{L^2(\partial\Omega^+)}^2+\|u\|_{L^2(\Omega^+)}^2\right).
\end{equation}
The theorem follows by combining Estimates \eqref{schauderCh_EllipticInductionStartMissing} and \eqref{SchauderCh_Reglemma1_Estimate9}.
\end{proof}

Next, we prove higher regularity.
\begin{theorem}[Higher Boundary Regularity]\label{SchauderCh_NeumannLapalceBoundaryRegularity}
Let $\Omega\subset\R^{n-1}\times[0,\infty)$ be a bounded domain, $m\in\N_0$,  $a^{ij}$, $b^i$, $c\in C^{m+1}(\bar\Omega)$ and $\theta$, $\Lambda>0$ such that $a^{ij}(x)\xi_i\xi_j\geq\theta|\xi|^2$ for all $x\in\Omega$ and $\|a^{ij}\|_{C^{m+1}(\bar\Omega)}$, $\|b^{i}\|_{C^{m+1}(\bar\Omega)}$, $\|c\|_{C^{m+1}(\bar\Omega)}\leq\Lambda$. Let further $f\in W^{m,2}(\Omega^+)$, $g\in W ^{m+1,2}(\partial^+\Omega)$ and let $u\in W^{1,2}(\Omega^+)$ be a weak solution of 
$$
\left\{\begin{array}{ l l}
a^{ij}\partial_{ij}u+b^i \partial_i u+cu=f&\textrm{in $\Omega^+$},\vspace{.1cm}\\
-a^{nj}\partial_j u=g&\textrm{along $\partial^+\Omega$}.
\end{array}
\right.$$
Then $u\in W^{m+2,2}(V^+)$ for any $V\subset\subset\Omega$ and
$$\|u\|_{W^{m+2,2}(V^+)}\leq C(\Lambda,\theta, V,\Omega)(\|u\|_{L^2(\Omega^+)}+\|f\|_{W^{m,2}(\Omega^+)}+\|g\|_{W^{m+1,2}(\partial^+\Omega)}).$$
\end{theorem}
\begin{proof}
Throughout the proof, we will suppress the dependence of various constants on $\Lambda$, $\theta$, $V$ and $\Omega$. The proof is inductive. $m=0$ is taken care of in Theorem \ref{SchauderCh_BoundaryRegularityLemmaStep1}. We prove $m\rightarrow m+1$. Let $V\subset\subset W\subset\subset \Omega$. As $f\in W^{m+1,2}(\Omega^+)$, the inductive hypothesis implies $u\in W^{m+2,2}(W^+)$ and 
\begin{equation}\label{SchauderCh_HigherRegularityAtBoundaryEq1}
\|u\|_{W^{m+2,2}(W^+)}\leq C(W^+, \Omega^+)\left(\|f\|_{W^{m,2}(\Omega^+)}+\|g\|_{W^{m+1,2}(\partial^+\Omega)}+\|u\|_{L^2(\Omega^+)}\right).
\end{equation}
Let $1\leq a\leq n-1$. It is easy to see that the function $\partial_a u$ solves weakly
$$
\left\{
\begin{array}{ll} 
a^{ij}\partial_{ij}\left(\partial_a u\right)+b^i \partial_i \left(\partial_a u\right)+c\left(\partial_a u\right)=\partial_a f-\partial_aa^{ij}\partial_{ij}u-\partial_a b^i\partial_i u-\partial_a c u&\textrm{in $\Omega^+$,}\vspace{.15cm}\\
-a^{nj}\partial_j\left(\partial_a u\right)=\partial_a g+\partial_a a^{nj}\partial_j u&\textrm{along $\partial^+\Omega$.}
\end{array}\right.
$$
Applying the inductive hypothesis and Estimate \eqref{SchauderCh_HigherRegularityAtBoundaryEq1} to this problem for  $\partial_a u$ and using Estimate \eqref{SchauderCh_HigherRegularityAtBoundaryEq1}, we deduce $\partial_a u\in W^{m+2,2}(V^+)$ and the estimate
\begin{align} 
\|\partial_a u\|_{W^{m+2,2}(V^+)}
\leq &C\left(
\|\partial_a f\|_{W^{m,2}(W^+)}+\|\partial_a g\|_{W^{m1,2}(\partial^+W)}+\|\partial_a u\|_{L^2(W^+)}
\right)\nonumber\\
\leq &C\left(
\|f\|_{W^{m+1,2}(\Omega^+)}+\|g\|_{W^{m+2,2}(\partial^+\Omega)}+\|u\|_{L^2(\Omega^+)}
\right).\label{SchauderCh_BoundaryRegularityInductiveStepEq1}
\end{align}
For small $\epsilon>0$ we put $V_\epsilon:=\set{x\in V\ |\ x_n>\epsilon}$. Using Theorem \ref{SchauderCh_NeumannLaplacianINterior}
we have $\partial_n u\in W^{m+2,2}(V_\epsilon)$ for all $\epsilon>0$. We aim for a uniform estimate for $\|\partial_n u\|_{W^{m+2,2}}$ as $\epsilon\rightarrow 0^+$. To do this, we consider a multiindex $\alpha=(\alpha_1,...,\alpha_n)$ of length $|\alpha|=m+2$. First we suppose $\alpha_n\leq m+1$, and without loss of generality $\alpha_1\neq 0$. We can use $\partial_1u\in W^{m+2,2}(V^+)$, Schwarz's theorem and Estimate \eqref{SchauderCh_BoundaryRegularityInductiveStepEq1} to deduce the uniform estimate 
$$
\|\nabla_\alpha \partial_n u\|_{L^2(V_\epsilon)}\leq \|\partial_1 u\|_{W^{m+2,2}(V^+)}\leq C\left(
\|f\|_{W^{m+1,2}(\Omega^+)}+\|g\|_{W^{m+2,2}(\partial^+\Omega)}+\|u\|_{L^2(\Omega^+)}
\right).
$$
So, we are left with estimating $\partial_n^{m+3} u$ in $L^2(V_\epsilon)$. As $u\in W^{2,2}(V^+)$ we deduce that pointwise, almost everywhere in $V^+$
\begin{equation}\label{SchauderCh_HigherRegularityrewrittenPDE}
\partial_{nn} u=\frac1{a^{nn}}\left(f-\sum_{i+j\leq 2n-1}a^{ij}\partial_{ij} u-\sum_{i=1}^ n b^i\partial_i u-cu\right).
\end{equation}
Let $\epsilon>0$ and $\varphi\in C^\infty_0(V_\epsilon)$. We multiply Equation \eqref{SchauderCh_HigherRegularityrewrittenPDE} with $(-1)^{m+1}\partial_n^{m+1}\varphi$ and integrate. As $u\in W^ {m+3,2}(V_\epsilon)$ we may integrate by parts to obtain
\begin{align*} 
\int_{V^+}\partial_n^{m+3}u\varphi =& (-1)^{m+1}\int_{V^+}\partial_n^2u \partial_n^{m+1}\varphi\\
=& (-1)^{m+1}\int_{V^+}\frac1{a^{nn}}\left(f-\sum_{i+j\leq 2n-1}a^{ij}\partial_{ij} u-\sum_{i=1}^ n b^i\partial_i u-cu\right) \partial_n^{m+1}\varphi.
\end{align*}
Now we use that $f\in W^{m+1,2}(\Omega^+)$ and that for all $1\leq i\leq n-1$ we already have established that $\partial_i u\in W^{m+2,2}(V^+)$ as well as the smoothness of the coefficients to integrate by parts $m+1$-times. We then get that for all small $\epsilon>0$ and $\varphi\in C^\infty_0(V_\epsilon)$ we have 
$$\left|\int_{V_\epsilon}\partial_n^{m+3}u\varphi\right|\leq C\left(\|f\|_{W^{m+1,2}(\Omega^+)}+\sum_{1=1}^{n-1}\|\partial_i u\|_{W^{m+2,2}(\Omega^+)}\right)\|\varphi\|_{L^2(V^+_\epsilon)}.$$
By approximation, this estimate is valid for all $\varphi\in L^2(V_\epsilon)$. In particular, as $\partial_n^{m+3}u\in L^2(V_\epsilon)$ we can take $\varphi=\partial_n^{m+3}u\in L^2(V_\epsilon)$ and find
\begin{equation}\label{SchauderCh_BoundaryRegularityInductiveStepEq2}
\|\partial_n^{m+3}u\|_{L^2(V_\epsilon)}
\leq C\left(\|f\|_{W^{m+1,2}(\Omega^+)}+\sum_{1=1}^{n-1}\|\partial_i u\|_{W^{m+2,2}(\Omega^+)}\right).
\end{equation}
Estimate \eqref{SchauderCh_BoundaryRegularityInductiveStepEq2} is uniform over all small $\epsilon>0$. Combining Estimates \eqref{SchauderCh_BoundaryRegularityInductiveStepEq1} and \eqref{SchauderCh_BoundaryRegularityInductiveStepEq2}, we deduce $u\in W^{m+3,2}(V^+)$ and  
$$\|u\|_{W^{m+3,2}(V^+)}\leq C\left(\|f\|_{W^{m+1,2}(\Omega^+)}+\|g\|_{W^{m+2,2}(\partial^+\Omega)}+\|u\|_{L^2(\Omega^+)}
\right).$$
\end{proof}

\section{Elliptic Regularity on Manifolds}
\begin{lemma}\label{SchauderCh_LiftingRegularitytomfd}
Let $\partial M\neq 0$, $m\in\N_0$, $f\in W^{m,2}(M)$, $h\in W^{m+1,2}(\partial M)$ and and $u\in W^{1,2}(M)$ be a weak solution of 
$$\left\{
\begin{aligned}
    -\Delta_g u&=f\hspace{.5cm}\textrm{in $M$},\\
    \frac{\partial u}{\partial\nu}&=h\hspace{.5cm}\textrm{along $\partial M$}.
\end{aligned}
\right.$$
Then $u\in W^{m+2,2}(M)$ and 
$$\|u\|_{W^{m+2,2}(M)}\leq C(M,m,g)\left(\|f\|_{W^{m,2}(M)}+\|h\|_{W^{m+1,2}(\partial M)}+\|u\|_{L^2(M)}\right).$$
\end{lemma}
\begin{proof}
Let $\mathcal A=\set{(\varphi_i, U_i)\ |\ 1\leq i\leq m}$ be a good atlas. For each $1\leq i\leq m$ let $\tilde U_i\supset\supset  Ui$ such that $\varphi_i$ can be extended to a smooth map onto $\tilde U_i$ and put $\tilde V_i:=\varphi_i(\tilde U_i)$. We choose an open set $U_i'$ such that $U_i\subset\subset U_i'\subset\subset \tilde U_i$, put $V_i':=\varphi_i(U_i')$ and let $\mathcal A':=\set{(\varphi_i', U_i')\ |\ 1\leq i\leq m}$. $\mathcal A'$ is also a good atlas. Finally, we let $u_i:=u\circ\varphi_i^{-1}$, $h_i:=h\circ\varphi_i^{-1}$ and $f_i:= f\circ\varphi_i^{-1}$. For $\psi\in C^\infty_0( V_i')$, we have 
$$\int_M g (\nabla_g u,\nabla_g(\psi\circ\varphi_i))d\mu_g=\int_M f\cdot(\psi\circ\varphi_i) d\mu_g+\int_{\partial M} g\cdot (\psi\circ\varphi_i) d S_g.$$
Expressing these integrals in the chart $V_i'$ and denoting the induced metric on the boundary by $\hat g$, we obtain 
$$\int_{V_i'}g^{ab}\partial_ a u_i\partial_b \psi \sqrt{\det g} dx
=
\int_{V_i'} f_i\psi \sqrt{\det g}dx+\int_{\partial^+V_i'}h _i\psi\sqrt{\det{\hat g}}dS(x).$$
Here $dS(x)$ is the $(n-1)$-dimensional Hausdorff measure on $\R^{n-1}\times 0$. The boundary integral is absent if $\varphi_i$ is an interior chart. By Cramer's rule $\sqrt{g^{nn}}\sqrt{\det g}=\sqrt{\det{\hat g}}$. So, we deduce that for each $1\leq i\leq m$, $u_i$ is a weak solution of
$$
\left\{\begin{array}{lll}
    &\displaystyle -\left(g^{ab}\partial_{ab}-g^{ab}\Gamma^k_{\ ab}\partial_k\right)u_i=f_i\hspace{.5cm}&\displaystyle\textrm{in $V_i'$,}\\
    &\displaystyle-\frac{g^{na}}{\sqrt{g^{nn}}}\partial_a u_i=h_i\hspace{.5cm}&\displaystyle\textrm{along $\partial^+V_i'$ if $U_i$ is a boundary chart.}
\end{array}
\right.$$
Using Theorem \ref{SchauderCh_NeumannLaplacianINterior} or \ref{SchauderCh_NeumannLapalceBoundaryRegularity}, depending on whether $U_i$ is a boundary chart or not, we deduce that $u_i\in W^{m+2,2}(V_i)$ with 
\begin{equation}\label{SchauderCh_RegularityLemmaNeumannproblemonManifoldeq1}
\|u_i\|_{W^{m+2,2}(V_i)}\leq C(V_i,V_i',m,g)\left(\|f_i\|_{W^{m,2}(V_i')}+\|h_i\|_{W^{m+1,2}(\partial^+V_i')}\right).
\end{equation}
Here the last term on the right is absent if $U_i$ is an interior chart.
Summing Estimate \eqref{SchauderCh_RegularityLemmaNeumannproblemonManifoldeq1} over $1\leq i\leq m$,  and using the equivalence of the Sobolev norms induced by the good atlases $\mathcal A$ and $\mathcal A'$ from Proposition \ref{SchauderCh_GoodatlasSobolevNorm}, we deduce the theorem. 
\end{proof}
If $\partial M=\emptyset$, we can follow the proof of Lemma \ref{SchauderCh_LiftingRegularitytomfd} and obtain the following result.
\begin{proposition}\label{SchauderCh_RegularityMdfWITHOUTBOundary}
Let $\partial M=\emptyset$, $m\in\N_0$, $f\in W^{m,2}(M)$ and $u\in W^{1,2}(M)$ be a weak solution of $-\Delta_g u=f$ in $M$. Then $u\in W^{m+2,2}(M)$ and 
$$\|u\|_{W^{m+2,2}(M)}\leq C(M,m,g)\left(\|f\|_{W^{m,2}(M)}+\|u\|_{L^2(M)}\right).$$
\end{proposition}

As an easy consequence, we obtain the following lemma:
\begin{lemma}\label{SchauderCh_Wk2EstimateinTermsOfLaplace_Lemma}
Let $M$ be a compact, $n$-dimensional Riemannian manifold with boundary $\partial M\neq \emptyset$ and $k\geq 0$. If $u\in C^{2k+2}(M)$ satisfies
$\frac{\partial}{\partial\nu}\Delta_g^l u=0$ for all $0\leq l\leq k$, then 
$$\|u\|_{W^{2k,2}(M)}\leq C(M,g,k)\sum_{j=0}^k\|\Delta_g^j u\|_{L^2(M)}.$$
\end{lemma}
\begin{proof}
We prove the lemma by an inductive-type argument over $k$. For $k=0$, the lemma is trivial. Now assume that the lemma is already proven for some $k$ and let $u\in C^{2k+4}(M)$ such that 
$$\frac{\partial}{\partial\nu}\Delta_g^l u=0\hspace{.5cm}\textrm{for all $0\leq l\leq k+1$}.$$
Note that we may apply the inductive hypothesis on the function $\Delta_g u\in C^{2k+2}(M)$. Additionally, we apply lemma \ref{SchauderCh_LiftingRegularitytomfd} to the function $u$ (note that $\partial_\nu u=0$). This gives the following two estimates:
\begin{align*}
    \|\Delta_g u\|_{W^{2k,2}(M)}\leq & C(M,g,k)\sum_{j=0}^k\|\Delta_g^{j+1} u\|_{L^2(M)}\\
    \|u\|_{W^{2k+2,2}(M)}\leq &C(M,g)\left(\|\Delta_g u\|_{W^{2k,2}(M)}+\|u\|_{L^2(M)}\right)
\end{align*}
The lemma follows by combining these two estimates.
\end{proof}

\begin{korollar}\label{SchauderCh_RegularityofEigenfunctions}
Let $(M,g)$ be a compact, $n$-dimensional Riemannian manifold with $\partial M\neq \emptyset$, $\lambda\geq 0$ and $u\in W^{1,2}(M)$ be a weak solution of 
$$\left\{
\begin{aligned}
    -\Delta_g u&=\lambda u\hspace{.5cm}\textrm{in $M$},\\
    \frac{\partial u}{\partial\nu}&=0\hspace{.75cm}\textrm{along $\partial M$}.
\end{aligned}
\right.$$
Then $u\in C^\infty(M)$ and 
$$\|u\|_{W^{2m,2}(M)}\leq C(M,g,m)(1+\lambda^m)\|u\|_{L^2(M)}.$$
\end{korollar}
\begin{proof}
If $u\in W^{k,2}(M)$ for some $k\geq 1$, then, using Theorem \ref{SchauderCh_LiftingRegularitytomfd}, we get $u\in W^{k+2,2}(M)$. Since $u\in W^{1,2}(M)$, we deduce $u\in\bigcap_{l\geq 0}W^{l,2}(M)$ and hence, by Sobolev embedding, $u\in C^\infty(M)$. Therefore we may compute $\partial_\nu\Delta_g^ku=(-1)^k\lambda^k \partial_\nu u=0$ for all $k\in\N_0$. Applying Lemma \ref{SchauderCh_Wk2EstimateinTermsOfLaplace_Lemma}, we get 
$$\|u\|_{W^ {2m,2}(M)}\leq C(M,g,m)\sum_{j=0}^m\lambda^ j\|u\|_{L^2(M)}\leq C(M,g,m)(m+1)\max(1,\lambda^ m)\|u\|_{L^2(M)}.$$
\end{proof}

Combining the results from Lemma \ref{SchauderCh_ExistenceofNeumannEFLemma} and Corollary \ref{SchauderCh_RegularityofEigenfunctions}, we get the following theorem.

\begin{theorem}[Spectral Decomposition of the Neumann Laplacian]\label{SchauderCh_neumannlaplacianResultsTheorem}
Let $(M,g)$ be a smooth compact, $n$-dimensional connected Riemannian manifold. There exists $0=\lambda_0<\lambda_1\leq\lambda_2\leq...\leq\lambda_k\rightarrow\infty$ and smooth $L^2_g(M)$-orthonormal and $W^{1,2}_g(M)$-orthogonal functions $\phi_k\in C^\infty(M)$ that satisfy 
$$\left\{
\begin{array}{rll}
    \displaystyle-\Delta_g \phi_k\hspace{-.2cm}&\displaystyle=\lambda_k \phi_k&\displaystyle\textrm{in $M$},\vspace{.2cm}\\
    \displaystyle\frac{\partial \phi_k}{\partial\nu}&\displaystyle=0\hspace{-.2cm}&\displaystyle\textrm{along $\partial M$}.
\end{array}
\right.$$
The span of the functions $\phi_k$ is dense in $L^2(M)$ and $W^{1,2}(M)$. Additionally, $\phi_0=|M|^{-\frac12}$ and 
$$\|\phi_k\|_{W^{2m,2}(M)}\leq C(M,m)(1+\lambda_k^m).$$
\end{theorem}

\chapter{Schauder Estimates in \texorpdfstring{$\R^n$}{Rn}}\label{SchauderCh_Ch03SchauderEstiamtesRn}
\section{Liouville Type Theorems}
Throughout this section, we fix a strongly elliptic matrix $\left(a^{ij}\right)_{ij}\in\R^{n\times n}$. That is there exists $\theta\in(0,\infty)$ such that
$$\forall\xi\in\R^n:\ a^{ij}\xi_i\xi_j\geq\theta|\xi|^2.$$
Throughout this section, we study 
\begin{align}
    &\dot u+a^{ij}a^{kl}\partial_{ijkl} u=0,\label{SchauderCh_ClassicalEquation}\\
    &a^{ni}\partial_i u=a^{ni}a^{kl}\partial_{ikl} u=0\hspace{.5cm}\textrm{along $x_n=0$},\label{SchauderCh_ClassicalBoundaryCondition}
\end{align}
on various domains and prove that solutions must be polynomials if they satisfy suitable growth estimates. 
\begin{definition}[Parabolic Polynomials]
A function $p:\R^n\times\R\rightarrow\R$ is called \emph{parabolic polynomial} of \emph{parabolic degree} $\operatorname{deg}(p)=d$, if 
$$p(x,t)=\sum_{|\alpha|+4j\leq d}a_{\alpha,j} x^\alpha t^j\hspace{.5cm}\textrm{for real parameters }a_{\alpha,j}.$$
\end{definition}

\begin{definition}[Parabolic Metric]\label{SchauderCh_parabolicMetricDefinition}
    The parabolic metric $d:(\R^n\times\R)\times(\R^n\times\R)\rightarrow [0,\infty)$ is defined as $d((x,t), (y,s)):=\max(|x-y|,|t-s|^{\frac14})$. 
\end{definition}

Note that the parabolic balls $U_\rho(p)$ from Definition \ref{SchauderCh_ParabolicBallDefnition} are precisely the metric balls with respect to the parabolic metric. 

\begin{theorem}\label{SchauderCh_BasicLiouvilleTheorem}
Let $u\in C^\infty(\R^n\times(-\infty,0])$ satisfy $\dot u=-a^{ij}a^{kl}\partial_{ijkl}u$ as well as $|u(x,t)|\leq C(1+d((x,t),0))^{4+\gamma}$ for some $\gamma\in(0,1)$, $C>0$ and all $(x,t)\in \R^n\times(-\infty,0]$. Then $u$ is a parabolic polynomial of parabolic degree $\deg(u)\leq 4$.  
\end{theorem}
\begin{proof}
Throughout this proof, we will suppress the dependence of various constants on $a^{ij}$ and $\theta$. Let $0<\rho_1<{\rho_2}$ and $\eta\in C^\infty_0(U_{\rho_2}(0),[0,1])$ such that $\eta\equiv 1$ on $U_{\rho_1}^-(0)$. We compute
\begin{align}
    0=&\int_{\R^ n\times(-\infty,0]}\eta^4 u(\dot u+a^{ij}a^{kl}\partial_{ijkl}u)dxdt\nonumber\\
        &=\frac12\int_{\R^ n\times(-\infty,0]}\eta^4 \frac \partial{\partial t}(u^2)dxdt+\int_{\R^ n\times(-\infty,0]}\eta^4 ua^{ij}a^{kl}\partial_{ijkl}u dx dt.\label{SchauderCh_Integral1}
\end{align}
We first simplify the first integral. Since $\operatorname{supp}(\eta)\subset U_{\rho_2}(0)$, we get
\begin{align}
    &\int_{U_{\rho_2}^-(0)}\eta^4 \frac \partial{\partial t}(u^2)dxdt\nonumber\\
=&\int_{B_{\rho_2}(0)}\left[\int_{-\rho_2^4}^0 \frac\partial{\partial t}(\eta^4u^2)dt\right]dx-\int_{U_{\rho_2}^-(0)}4\eta^3\dot\eta u^2dxdt\nonumber\\
=&\int_{B_2(0)}\eta^4(x,0)^2u(x,0)^2-\eta(x,-\rho_2^4)^4u(x,-\rho_2^4)^2dx-\int_{U_{\rho_2}^-(0)}4\eta^3\dot\eta u^2dxdt\nonumber\\
=&\int_{B_2(0)}\eta^4(x,0)^2u(x,0)^2dx-\int_{U_{\rho_2}^-(0)}4\eta^3\dot\eta u^2dxdt.
\label{SchauderCh_Integral2}
\end{align}
The last step is justified as $\eta$ has compact support in $U_{\rho_2}(0)$. Next, we simplify the second integral in Equation \eqref{SchauderCh_Integral1}. Since $\eta\equiv 0$ on $\partial B_{\rho_2}(0)\times [-\rho_2^4,\rho_2^4]$, we have
\begin{align}
    &\int_{U_{\rho_2}^-(0)}\eta^4 ua^{ij}a^{kl}\partial_{ijkl }u dx dt\nonumber\\
= & \int_{U_{\rho_2}^-(0)}a^{ij}\partial_{ij}\left(\eta^4 u\right)a^{kl}\partial_{kl } u dx dt\nonumber\\
= & \int_{U_{\rho_2}^-(0)}\eta^4a^{ij}\partial_{ij}u\,a^{kl}\partial_{kl}u dx dt
+\int_{U_{\rho_2}^-(0)}a^{ij}(u\partial_{ij}(\eta^4)+8\eta^3\partial_i\eta\partial_j u) a^{kl}\partial_{kl}u dx dt.\label{SchauderCh_Integral3}
\end{align}
Inserting Equations \eqref{SchauderCh_Integral2} and \eqref{SchauderCh_Integral3} into Equation \eqref{SchauderCh_Integral1} and rearranging, we obtain
\begin{align}
     &\int_{U_{\rho_2}^-(0)}\eta^4a^{ij}\partial_{ij}ua^{kl}\partial_{kl}u dx dt\nonumber\\
=&-\int_{U_{\rho_2}^-(0)}a^{ij}(u\partial_{ij}(\eta^4)+8\eta^3\partial_i\eta\partial_j u) a^{kl}\partial_{kl}u dx dt\nonumber\\
&-\frac12\int_{B_{\rho_2}(0)}\eta^4(x,0)^2u(x,0)^2dx+\int_{U_{\rho_2}^-(0)}2\eta^3\dot\eta u^2dxdt.
\label{SchauderCh_Integral4}
\end{align}
First note that $-\int_{B_{\rho_2}(0)}\eta^4(x,0)^2u(x,0)^2dx\leq 0$ and can therefore be dropped in favor of an estimate. It is precisely this step where the same proof breaks down if we were to work on $\R^n\times[0,\infty)$ instead. Additionally, we estimate 
\begin{equation}\label{SchauderCh_Integral5}
    \left|\int_{U_{\rho_2}^-(0)}2\eta^3\dot\eta u^2\right|
    \leq C(\eta)\int_{U_{\rho_2}^-(0)}u^2.
\end{equation}
Let $\epsilon>0$ be a small but arbitrary parameter. Using Young's inequality, we estimate 
\begin{align}
   &\left| \int_{U_{\rho_2}^-(0)}a^{ij}u\partial_{ij}(\eta^4)a^{kl}\partial_{kl}u dx dt\right|\nonumber\\
   =&\left| \int_{U_{\rho_2}^-(0)}a^{ij}u\left(4\eta^3\partial_{ij}\eta+12\eta^2\partial_i\eta\partial_j\eta\right)
   a^{kl}\partial_{kl}u dx dt\right|\nonumber\\
    =&\left| \int_{U_{\rho_2}^-(0)}a^{ij}u\left(4\eta\partial_{ij}\eta+12\partial_i\eta\partial_j\eta\right)
   \left(\eta^2 a^{kl}\partial_{kl}u\right) dx dt\right|\nonumber\\
\leq & C(\eta,\epsilon)\int_{U_{\rho_2}^-(0)}u^2dxdt+\epsilon\int_{U_{\rho_2}^-(0)}\eta^4\left(a^{ij}\partial_{ij} u\right)^2 dxdt.\label{SchauderChIntegral6}
\end{align}
Using the same arguments, we estimate 
\begin{align}
   &\left| \int_{U_{\rho_2}^-(0)}a^{ij}\eta^3\partial_i\eta\partial_j u a^{kl}\partial_{kl}u dx dt\right|\nonumber\\
  \leq &\left| \int_{U_{\rho_2}^-(0)}a^{ij}\eta\partial_i\eta\partial_j u \left(\eta^2 a^{kl}\partial_{kl}u \right)dx dt\right|\nonumber\\
  \leq &C(\epsilon)\int_{U_{\rho_2}^-(0)}\left(a^{ij}\eta\partial_i\eta\partial_j u \right)^2 +\epsilon\int_{U_{\rho_2}^-(0)}\left(\eta^2 a^{kl}\partial_{kl}u \right)^2dx dt.
   \label{SchauderChIntegral7}
\end{align}
Putting Estimates \eqref{SchauderCh_Integral5}, \eqref{SchauderChIntegral6} and \eqref{SchauderChIntegral7} into Estimate \eqref{SchauderCh_Integral4} and choosing $\epsilon>0$ small enough, we obtain
\begin{equation}\label{SchauderCh_Integral7.5}
\int_{U_{\rho_2}^-(0)}\left(\eta^2 a^{kl}\partial_{kl}u \right)^2dx dt
\leq 
C(\eta)\int_{U_{\rho_2}^-(0)}\hspace{0cm}u^2dxdt+C(\eta)\int_{U_{\rho_2}^-(0)}\hspace{0cm}\eta^2\left(a^{ij}\eta\partial_i\eta\partial_j u \right)^2dxdt.
\end{equation}
We further estimate the second integral. Let $A:=(a^{ij})_{ij}$. $A$ is a positive, symmetric matrix and hence $\langle x,y\rangle_A:=a^{ij}x_iy_j$ is a scalar product. Using the Cauchy-Schwarz inequality, we estimate
\begin{equation}\label{SchauderCh_CSTypeTrick}
\left(a^{ij}\partial_i u\partial_j\eta\right)^2
=\langle \nabla u,\nabla\eta\rangle_A^2
\leq \langle \nabla u,\nabla u\rangle_A
\langle \nabla \eta,\nabla\eta\rangle_A
=\left(a^{ij}\partial_i u\partial_j u\right)\left(a^{kl}\partial_k \eta\partial_l\eta\right).
\end{equation}
We multiply Estimate \eqref{SchauderCh_CSTypeTrick} by $\eta^2$ and integrate to get
\begin{align}
&\int_{U_{\rho_2}^-(0)}\eta^2 \left(a^{ij}\partial_i\eta\partial_ju\right)^2dxdt\nonumber\\
\leq  &
\int_{U_{\rho_2}^-(0)}\eta^2 \left(a^{ij}\partial_iu\partial_j u\right)
\left(a^{kl}\partial_k\eta\partial_l\eta\right)dxdt\nonumber\\
=&-\int_{U_{\rho_2}^-(0)} u\partial_i\left(\eta^2a^{ij}\partial_j ua^{kl}\partial_k\eta\partial_l\eta\right)dxdt\nonumber\\
= & -\int_{U_{\rho_2}^-(0)}\eta^2 a^{ij} a^{kl}\left(u\partial_{ij}u\partial_k\eta\partial_l\eta+2u\partial_j u\partial_{ik}\eta\partial_l\eta\right)dxdt\nonumber\\
&-2\int_{U_{\rho_2}(0)}u\eta\partial_i\eta a^{ij}\partial_j ua^{kl}\partial_k\eta\partial_l\eta dxdt.\label{SchauderChIntegral7AndAHalf}
\end{align}
We first focus on the second integral. Let $\epsilon>0$ be a small parameter. Using Young's inequality, we estimate 
\begin{align*}
    \left|\int_{U_{\rho_2}(0)}u\eta\partial_i\eta a^{ij}\partial_j ua^{kl}\partial_k\eta\partial_l\eta dxdt\right|
    \leq &\left|\int_{U_{\rho_2}(0)}(\eta a^{ij}\partial_i\eta\partial_ ju)(ua^{kl}\partial_k\eta\partial_l\eta)dxdt\right|\\
     &\hspace{-2cm} \leq\epsilon\int_{U_{\rho_2}(0)}\eta^2( a^{ij}\partial_i\eta\partial_j u)^2dxdt+C(\eta,\epsilon)\int_{U_{\rho_2}(0)}u^2dxdt.
\end{align*}
Inserting into Estimate \eqref{SchauderChIntegral7AndAHalf} and choosing $\epsilon>0$ small enough yields
\begin{align}
&\int_{U_{\rho_2}^-(0)}\eta^2 \left(a^{ij}\partial_i\eta\partial_ju\right)^2dxdt\nonumber\\
&\leq C\int_{U_{\rho_2}^-(0)}\eta^2 a^{ij} a^{kl}\left(u\partial_{ij}u\partial_k\eta\partial_l\eta+2u\partial_j u\partial_{ik}\eta\partial_l\eta\right)dxdt+C(\eta)\int_{U_{\rho_2}(0)}u^2dxdt.\label{SchauderChIntegral8}
\end{align}
Let $\epsilon>0$ denote a small parameter. Using Young's inequality, we estimate
\begin{equation}\label{SchauderChIntegral9}
    \left|\int_{U_{\rho_2}^-(0)}\hspace{-.1cm}\eta^2 a^{ij}\partial_{ij}ua^{kl}\partial_k \eta\partial_l \eta u dxdt\right|
    \leq 
    \epsilon\int_{U_{\rho_2}^-(0)}\hspace{-.2cm}\left(\eta^2a^{ij}\partial_{ij} u\right)^2dxdt +C(\eta,\epsilon)\int_{U_{\rho_2}^-(0)}\hspace{-.3cm}u^2dxdt.
\end{equation}
Next, we write $2u\partial_j u=\partial_j (u^2)$. This allows us to perform another partial integration and obtain 
\begin{align}
    \left|-2\int_{U_{\rho_2}^-(0)}\eta^2a^{ij}a^{kl}u\partial_j u\partial_{ik}\eta\partial_l\eta dxdt\right|
    =&\left|-\int_{U_{\rho_2}^-(0)}\eta^2a^{ij}a^{kl}\partial_j (u^2)\partial_{ik}\eta\partial_l\eta dxdt\right|\nonumber\\
    =&\left|\int_{U_{\rho_2}^-(0)}\partial_j\left(\eta^2a^{ij}a^{kl}\partial_{ik}\eta\partial_l\eta\right) (u^2)dxdt\right|\nonumber\\
    \leq &C(\eta)\int_{U_{\rho_2}^-(0)} u^2dxdt.\label{SchauderCh_Integral10}
\end{align}
Putting Estimates \eqref{SchauderChIntegral9} and \eqref{SchauderCh_Integral10} into Estimate \eqref{SchauderChIntegral8} and subsequently into Estimate \eqref{SchauderCh_Integral7.5} and choosing $\epsilon>0$ small enough, we obtain 
\begin{equation}\label{SchauderCh_Integral11}
\int_{U_{\rho_2}^-(0)}\left(\eta^2 a^{kl}\partial_{kl}u \right)^2dx dt
\leq 
C(\eta)\int_{U_{\rho_2}^-(0)}u^2dx dt.
\end{equation}
The matrix $A=\left(a^{ij}\right)_{ij}$ is positive and symmetric and therefore has a positive and symmetric square root $B:=\left(b^{ij}\right)_{ij}$ -- that is $B^2=A$. We estimate 
$$ a^{ij}a^{kl}\partial_{ik}u\partial_{jl}u
  =  a^{ij}b^{kc}b^{lc}\partial_{ik}u\partial_{jl}u
    =  a^{ij}\partial_{i}(b^{kc}\partial_ku)\partial_{j}(b^{lc}\partial_l u).
    $$
As $a^{ij}$ is elliptic with ellipticity constant $\theta>0$, we get 
$$ a^{ij}a^{kl}\partial_{ik}u\partial_{jl}u
\geq\theta\sum_{c=1}^n|\nabla (b^{kc}\partial_k u)|^2
=\theta\sum_{c,s=1}^n|(b^{kc}\partial_{ks} u)|^2.
$$
Note that for fixed indices $s$ and $c$ we have $(b^{kc}\partial_{ks} u)=
  (B\nabla^2 u)^c_{\ s}$ and hence 
  $$ a^{ij}a^{kl}\partial_{ik}u\partial_{jl}u
\geq\theta\sum_{c,s=1}^n(B\nabla^2 u)^c_{\ s}(B\nabla^2 u)^c_{\ s}
=\theta\operatorname{tr}(\left(B\nabla^2 u\right)^TB\nabla^2 u)
=\theta\operatorname{tr}(\nabla^2 u A\nabla^2 u).
$$
Finally, we need to compute $\operatorname{tr}(\nabla^2 u A\nabla^2 u)$. Using the ellipticity of $a^{ij}$, we obtain
\begin{equation}\label{SchauderCh_Ellipticity}
a^{ij}a^{kl}\partial_{ik}u\partial_{jl}u
\geq
\theta \sum_{k=1}^n a^{ij}\partial_{ik}u\partial_{jk} u
\geq 
\theta^2 \sum_{k=1}^n |\partial_k\nabla u|^2
=\theta^2|\nabla^2 u|^2.
\end{equation}
Using $\eta\equiv 1$ on $U_{\rho_1}^-(0)$ and Estimate \eqref{SchauderCh_Ellipticity}, we can use Estimate \eqref{SchauderCh_Integral11} to deduce
\begin{equation}\label{SchauderCh_Integral12}
\int_{U_{\rho_1}^-(0)}|\nabla^2 u|^2 dx dt
\leq 
\int_{U_{\rho_2}^-(0)}\eta^4|\nabla^2 u|^2 dx dt
\leq 
C(\eta)\int_{U_{\rho_2}^-(0)}u^2dxdt.
\end{equation}
Note that $|\nabla u|^2=\frac12\Delta(u^2)-u\Delta u$. Using this identity and partial integration, we may estimate 
\begin{align*}
    \int_{U_{\rho_2}^-(0)}\eta^4 |\nabla u|^2dx dt=&\int_{U_{\rho_2}^-(0)}\eta^4\left(\frac12\Delta (u^2)-u\Delta u\right)dx dt\\
    \leq & \frac12\int_{U_{\rho_2}^-(0)}\Delta(\eta^4) u^2dx dt
    +\frac12\int_{U_{\rho_2}^-(0)}\eta^4 u^2dx dt+\frac12\int_{U_{\rho_2}^-(0)}\eta^4(\Delta u)^2dx dt\\
    \leq & C(\eta)\int_{U_{\rho_2}^-(0)}u^2dx dt +C\int_{U_{\rho_2}^-(0)}\eta^2|\nabla^2 u|^2dx dt.
\end{align*}
Combining this estimate with Estimate \eqref{SchauderCh_Integral12}, we deduce that for $k=0,1,2$
$$\int_{U_{\rho_1}^ -(0)}|\nabla^ k u|^2dx dt\leq C(\rho_1,\rho_2,k,\eta)\int_{U_{{\rho_2}}^ -(0)}u^2dx dt.$$
This estimate is valid for any solution of $\dot u+a^ {ij}a^ {kl}\partial_{ijkl} u=0$. In particular, it is also valid for $u\rightarrow \nabla^m u$. Using an inductive argument and suitable choices for $\rho_1$ and $\rho_2$, we deduce that
\begin{equation}\label{SchauderCh_L2EstimatesforSpatialDerivatives}
\int_{U_1^-(0)}|\nabla^ k u|^2dx dt
\leq C(k)\int_{U_2^-(0)}u^2dx dt.
\end{equation}
Note that by $\partial_t u=-a^ {ij}a^ {kl}\partial_{ijkl}u$, we can trade a time derivative for a fourth-order spatial derivative. Therefore we may use Estimate \eqref{SchauderCh_L2EstimatesforSpatialDerivatives} and obtain
\begin{equation}\label{SchauderCh_GeneralL2EstimateLiouville}
\int_{U_1^ -(0)}|\nabla^ k\partial_t^ l u|^2dx dt\leq 
C(k,l)
\int_{U_{2}^ -(0)}u^2dx dt
\leq C(k,l)\|u\|_{L^\infty(U_2^-(0))}^2.
\end{equation}
Estimate \eqref{SchauderCh_GeneralL2EstimateLiouville} is valid for all solutions of $\dot v+a^{ij}a^{kl}\partial_{ijkl}v=0$. In particular, taking $R>1$ and putting $u_R(x,t):=u(Rx, R^4 t)$, Estimate \eqref{SchauderCh_GeneralL2EstimateLiouville} is also valid for $u_R$. Hence
$$
C(k,l)\|u\|_{L^\infty(U_{2R}^-(0))}
=
C(k,l)\|u_R\|_{L^\infty(U_{2}^-(0))}
\geq 
\|\nabla^k \partial_t^l u_R\|_{L^2(U_1^-(0))}
=
\frac{R^{4l+k}}{R^{\frac{n+4}2}}\|\partial_t^l\nabla^k u\|_{L^2(U_R^-(0))}.
$$
Letting $R\rightarrow\infty$, this shows that $\partial_t^l\nabla^k u\equiv 0$ for $4l+k>\frac n2+2$ and hence $u$ is a parabolic polynomial. Due to the growth estimate $|u(x,t)|\leq C(1+d((x,t),0))^{4+\gamma}$ we deduce that $u$ must have parabolic degree 4 or less. 
\end{proof}

We can extend this theorem to the situation where $u$ is not smooth by using convolution. Let $\phi\in C^\infty(\R^n\times\R)$ be a smooth function with values in $[0,1]$, $\operatorname{supp}(\phi)\subset B_1(0)\times (-1,1)$ and 
$$\int_{B_1(0)\times(-1,1)}\phi(y,s) dyds=1.$$
For $\epsilon>0$ we define $\phi_\epsilon(y,s):=\frac1{\epsilon^{4+n}}\phi\left(\frac y\epsilon, \frac s{\epsilon^4}\right)$. As usual, the mollification of a function $v$ is defined as 
$$(\phi_\epsilon*v)(x,t):=\int_{\R^n\times\R}\phi_\epsilon(x-y,t-s) v(y,s)dyds.$$
We require the following standard lemma:
\begin{lemma}\label{SchauderCh_ConvolutionLemma}
Let $\Omega\subset \R^n\times\R$ be open and $u\in C^0(\Omega)$. Let $\Omega_\epsilon:=\set{p\in\Omega\ |\  U_\epsilon(p)\subset\Omega}$. Then $\phi_\epsilon* u\in C^\infty(\Omega_\epsilon)$ and $\phi_\epsilon* u\rightarrow u$ in $C^0_{\operatorname{loc}}(\Omega)$. Additionally, $\partial_t^k\nabla_\alpha u_\epsilon=\phi_\epsilon*(\partial_t^k\nabla_\alpha u)$ if $\partial_t^k\nabla_\alpha u$ exists and is continuous.
\end{lemma}

\begin{korollar}\label{SchauderCh_LiouvilleHalbraumNachUnten}
Let $u\in C^{4,1}(\R^n\times(-\infty,0])$ satisfy Equation \eqref{SchauderCh_ClassicalEquation} on $\R^n\times(-\infty,0]$ as well as $|u(x,t)|\leq C(1+d((x,t),0))^{4+\gamma}$ for some $\gamma\in(0,1)$, $C>0$ and all $(x,t)\in\R^n\times(-\infty,0]$. Then $u$ is a parabolic polynomial of parabolic degree $\operatorname{deg}(u)\leq 4$. 
\end{korollar}
\begin{proof}
Let $\epsilon\in(0,1)$. For $\delta<\epsilon$ we consider the convolution $u_\delta:\R^n\times(-\infty,-\epsilon]\rightarrow\R,\ u_\delta(x,t):=(\phi_\delta*u)(x,t)$. By Lemma \ref{SchauderCh_ConvolutionLemma}
$u_\delta\in C^\infty(\R^n\times(-\infty,-\epsilon])$ and satisfies $\dot u_\delta=-a^{ij}a^{kl}\partial_{ijkl} u_\delta$. Additionally, since $\delta\in(0,1)$, we may estimate 
\begin{align*} 
|u_\delta(x,t)|\leq &C\int_{U_\delta(0)}\phi_\delta(y,s)(1+d((x,t)- (y,s),0))^{4+\gamma}dyds\\
\leq &C\int_{U_\delta(0)}\phi_\delta(y,s)\left(1+d((x,t),0)\right)^{4+\gamma}dyds\\
\leq & C(1+d((x,t),0))^{4+\gamma}.
\end{align*}
So, by Theorem \ref{SchauderCh_BasicLiouvilleTheorem}, $u_\delta$ is a parabolic polynomial of degree at most 4 for all $\delta<\epsilon$. That is 
$$u_\delta(x,t)=a(\delta) t+\sum_{|\alpha|\leq 4}b_\alpha(\delta) x^\alpha\hspace{.5cm}\textrm{for all $x\in\R^n$ and $t\leq -\epsilon$.}$$
Noting the formulas
\begin{align*}
    b_\alpha(\delta)&=\frac1{\alpha!}\nabla_\alpha u_\delta(0,-2)=\frac1{\alpha!}\int_{U_\delta(0)}\phi_\delta(y,s) \nabla_\alpha u(0-y,-2-s)dyds,\\
    a(\delta)&=\frac \partial{\partial t}u_\delta(0,t)\bigg|_{t=-2}=\int_{U_\delta(0)}\phi_\delta(y,s)  \dot u(0-y,-2-s)dyds,
\end{align*}
it is readily seen that $b_\alpha(\delta)\rightarrow \frac1{\alpha!}\nabla_\alpha u(0,-2)=:b_\alpha^*$ and $a(\delta)\rightarrow \dot u(0,-2)=:a^*$ as $\delta\rightarrow 0^+$.  Consequently $u_{\delta}\rightarrow a^*t+\sum_{|\alpha|\leq 4}b^*_\alpha x^\alpha$ in $C^0_{\operatorname{loc}}(\R^n\times(-\infty,-\epsilon])$. However, we also know that $u_{\delta}\rightarrow u$ in $C^{0}_{\operatorname{loc}}(\R^n\times(-\infty,-\epsilon))$. Therefore
$$u(x,t)=a^*t+\sum_{|\alpha|\leq 4}b^*_\alpha x^\alpha\hspace{.5cm}\textrm{for all }(x,t)\in\R^n\times(-\infty,-\epsilon).$$
Since this is true for all $\epsilon>0$, the corollary follows.
\end{proof}

\begin{korollar}\label{SchauderCh_LiouvilleGanzraum}
Let $u\in C^{4,1}(\R^n\times\R)$ satisfy Equation \eqref{SchauderCh_ClassicalEquation} on $\R^n\times\R$ as well as $|u(x,t)|\leq C(1+d((x,t),0))^{4+\gamma}$ for some $\gamma\in(0,1)$, $C>0$ and all $(x,t)\in\R^n\times\R$. Then $u$ is a parabolic polynomial of parabolic degree $\operatorname{deg}(u)\leq 4$. 
\end{korollar}
\begin{proof}
For $T>0$, we put $v_T:\R^n\times(-\infty,0]\rightarrow \R$, $v_T(x,t):=u(x, t+T)$. It is easy to see that Corollary \ref{SchauderCh_LiouvilleHalbraumNachUnten} is applicable and hence, $v_T$ is a parabolic polynomial of degree at most 4. Consequently there exist $a(T), b_\alpha(T)\in\R$ such that
$$u(x,t)=v(x, t-T)=a(T)t+\sum_{|\alpha|\leq 4}b_\alpha(T)x^\alpha\hspace{.5cm}\textrm{for all }(x,t)\in\R^n\times(-\infty,T].$$
This formula implies $a(T)=\dot u(0,0)$ and $b_\alpha=\frac1{\alpha!}\nabla_\alpha u(0,0)$ and so the coefficients $a(T)$ and $b_\alpha(T)$ do not depend on $T$.  Letting $T\rightarrow\infty$, we deduce the corollary. 
\end{proof}

\begin{korollar}\label{SchauderCh_LiouvilleHalbraumNachOben}
Let $u\in C^{4,1}(\R^n\times[0,\infty))$ satisfy Equation \eqref{SchauderCh_ClassicalEquation} on $\R^n\times[0,\infty)$ and as well as $|u(x,t)|\leq C(1+d((x,t),0))^{4+\gamma}$ for some $\gamma\in(0,1)$, $C>0$ and all $(x,t)\in\R^n\times[0,\infty)$. If $u(\cdot,0)$ is a polynomial of degree 4 or less, then $u$ is a parabolic polynomial of parabolic degree $\operatorname{deg}(u)\leq 4$. 
\end{korollar}
\begin{proof}
We first assume that $u(\cdot,0)=0$. In this case, we define a new function $v:\R^n\times\R\rightarrow\R$ by mapping $v(x,t):=u(x,t)$ for $t\geq 0$ and $v(x,t):=0$ for $t<0$. We claim that $v\in C^{4,1}(\R^n\times\R)$. Suppose for the moment this is already shown. Then clearly $\dot v+a^{ij}a^{kl}\partial_{ijkl} v=0$ on $\R^n\times\R$ and still $|v(x,t)|\leq C(1+d((x,t),0))^{4+\gamma}$. So, by Corollary \ref{SchauderCh_LiouvilleGanzraum}, the corollary follows.\\

To prove $v\in C^{4,1}(\R^n\times\R)$, it suffices to check that for each $x_0\in\R^n$ the function $v$ is $C^{4,1}$ near $(x_0,0)$. To see this first note that $v(x_0+\xi, 0)\equiv 0$, so at $(x_0,0)$ the function $v$ is four times differentiable in space with $\nabla^k v(x_0,0)=0$. For $\epsilon>0$ we compute 
$$\frac{v(x_0,\epsilon)-v(x_0,0)}\epsilon=\frac{u(x_0,\epsilon)-u(x_0,0)}\epsilon=\int_0^1 \dot u(x_0, s\epsilon)ds\rightarrow \dot u(x_0,0)=a^{ij}a^{kl}\partial_{ijkl} u(x_0,0)=0.$$
Also, since $v(x_0,-\epsilon)=v(x_0,0)=0$ we have $\frac{v(x_0,-\epsilon)-v(x_0,0)}{-\epsilon}\equiv 0$.
So, with respect to time $v$ is differentiable at $(x_0,0)$  and $\dot v(x_0,0)=0$.\\

It remains to check that the partial derivatives are continuous. Again it suffices to check this at points $(x_0,0)$. First we use that $u\in C^{4,1}(\R^n\times[0,\infty))$ to deduce
$$\nabla^k v(x_0,0)=\nabla^k u(x_0,0)=\lim_{\substack{(x,t)\rightarrow (x_0, 0)\\ t>0}}\nabla^k u(x,t)=\lim_{\substack{(x,t)\rightarrow (x_0, 0)\\ t>0}}\nabla^k v(x,t).$$
Next we note that $\nabla^k v(x_0,0)=0$ and as $v\equiv 0$ for $t<0$ we clearly also have 
$\nabla^k v(x,t)\rightarrow \nabla^k v(x_0,0)$ when $(x,t)\rightarrow (x_0,0)$ and $t<0$. For the time derivative, we argue similarly. $\dot v(x_0,0)=0$ and hence it is clear that $\dot v(x,t)\rightarrow \dot v(x_0,0)$ when $(x,t)\rightarrow (x_0,0)$ and $t<0$. For $t>0$ we have $\dot v(x,t)=\dot u(x,t)=-a^{ij}a^{kl}\partial_{ijkl} u(x,t)$. Letting $(x,t)\rightarrow (x_0,0)$ we use $u\in C^{4,1}(\R^n\times[0,\infty))$ to deduce 
$a^{ij}a^{kl}\partial_{ijkl} u(x,t)\rightarrow a^{ij}a^{kl}\partial_{ijkl} u(x_0,0)=0$. The last step is justified by $u(x,0)\equiv 0$. So $\dot v(x,t)\rightarrow \dot v(x_0,0)$ when $(x,t)\rightarrow (x_0, 0)$ and $t>0$.\\

Finally, we justify that we can assume $u(\cdot,0)=0$. Indeed, in the general situation, we define
$$p(x):=\sum_{|\alpha|\leq 4}\frac{\nabla_\alpha u(0,0)}{\alpha!}x^\alpha\hspace{.5cm}\textrm{and put}\hspace{.5cm}
\tilde u(x,t):=u(x,t)-p(x)+ta^{ij}a^{kl}\partial_{ijkl} p(x).$$
$\tilde u -u$ is a parabolic polynomial with $\operatorname{deg}(\tilde u -u)\leq 4$. In particular, $\tilde u $ satisfies the same growth estimate as $u$ and if $\tilde u $ is a parabolic polynomial with degree $\operatorname{deg}(\tilde u )\leq 4$ then $u$ is as well. We compute
$$\dot{\tilde u } +a^{ij}a^{kl}\partial_{ijkl} \tilde u =\dot u+a^{ij}a^{kl}\partial_{ijkl} p+a^{ij}a^{kl}\partial_{ijkl}(u-p)=0.$$
Since $u(\cdot,0)$ is a polynomial of degree $4$ or less, the same is true for $\tilde u $. Since additionally $\nabla_\alpha \tilde u (0,0)=0$ for all $|\alpha|\leq 4$, we deduce that 
$$
\tilde u(x,0)=\sum_{\alpha\in\N_0^{n}:\ |\alpha|\leq 4}\frac{\nabla_\alpha \tilde u(0,0)}{\alpha!} x^\alpha=0.$$
So, by the special case considered in the beginning of the proof, $\tilde u $ is a parabolic polynomial of degree $\operatorname{deg}(\tilde u )\leq 4$, which implies the corollary.
\end{proof}
It is possible to extend these Liouville type theorems to the situation where Equation \eqref{SchauderCh_ClassicalEquation} is not satisfied on all of $\R^n$ but only on $\R^{n-1}\times[0,\infty)$ when the boundary conditions \eqref{SchauderCh_ClassicalBoundaryCondition} are imposed. The fundamental observation is the following lemma:

\begin{lemma}\label{SchauderCh_ReflectionLemma}
Let $I\subset\R$ be an interval and $u\in C^{4,1}(\R^{n-1}\times[0,\infty)\times I)$ satisfy Equation \eqref{SchauderCh_ClassicalEquation} and the boundary conditions from Equation \eqref{SchauderCh_ClassicalBoundaryCondition}. Let $v:=\sum_{i=1}^n\frac{a^{ni}}{a^{nn}}e_i\in\R^n$ and
$$
\left\{
\begin{array}{rll}
\tilde u(x,t)\hspace{-.2cm}&:=u(x,t)&\textrm{if $x_n\geq 0$},\\
\tilde u(x,t)\hspace{-.2cm}&:=u(x-2x_n v,t)&\textrm{if $x_n< 0$}.\\
\end{array}\right.
$$
 Then $\tilde u\in C^{4,1}(\R^n\times I)$ and $\tilde u$ is a classical solution to Equation \eqref{SchauderCh_ClassicalEquation}.
\end{lemma}
\begin{proof}
It is clear that $\tilde u\in C^{4,1}(\R^{n-1}\times[0,\infty)\times I)$ and $\tilde u\in C^{4,1}(\R^{n-1}\times(-\infty,0]\times I)$. We may deduce $\tilde u\in C^{4,1,\gamma}(\R^n\times I)$ once we have shown that for $0\leq k\leq 4$
\begin{equation}\label{SchauderCh_DerivativesMatchUpCondition}
\lim_{x_n\rightarrow0^-}\partial_n^k \tilde u(x,t)=\lim_{x_n\rightarrow0^+}\partial_n^k \tilde u(x,t).
\end{equation}
Clearly  $\tilde u$ solves Equation \eqref{SchauderCh_ClassicalEquation} on $\R^{n-1}\times [0,\infty)\times I$. We now prove, that it also satisfies \eqref{SchauderCh_ClassicalEquation} on $\R^{n-1}\times (-\infty,0]\times I$. For $(x,t)\in \R^n\times I$ with $x_n<0$ we compute
\begin{equation}\label{SchauderCh_ReflectionDerivativeFormula}
    \partial_i \tilde u(x,t)=\partial_a u(x-2x_n v, t)\left(\delta^i_a-2\delta^i_n v_a\right).
\end{equation}
Using $v_i=\frac{a^{ni}}{a^{nn}}$, we compute
\begin{align*}
    &a^{ij}\left(\delta^i_a-2\delta^i_n v_a\right)\left(\delta^j_b-2\delta^j_n v_b\right)\\
    =&\left(a^{aj}-2a^{nj}v_a \right)\left(\delta^j_b-2\delta^j_n v_b\right)\\
    =&a^{ab}-2a^{an}v_b-2a^{nb}v_a+4a^{nn}v_a v_b\\
    =&a^{ab}-2a^{an}\frac{a^{nb}}{a^{nn}}-2a^{nb}\frac{a^{na}}{a^{nn}}+4a^{nn}\frac{a^{na}a^{nb}}{(a^{nn})^2}\\
    =&a^{ab}.
\end{align*}
Combining with Equation \eqref{SchauderCh_ReflectionDerivativeFormula}, we obtain 
$a^{ij}\partial_{ij}\tilde u(x,t)=a^{ab}\partial_{ab}u(x-2x_n v,t)$ and hence
$$\dot{\tilde u}(x,t)+a^{ij}a^{kl}\partial_{ijkl} \tilde u(x,t)
=
(\dot u+a^{ij}a^{kl}\partial_{ijkl} u)(x-2x_n v, t)=0\hspace{.5cm}\textrm{in }\R^{n-1}\times(-\infty,0]\times I.$$
It remains to establish Equation \eqref{SchauderCh_DerivativesMatchUpCondition}. The case $k=0$ is trivial. Indeed, if $x_n\rightarrow0^-$ we get
$\tilde u((\vec x, x_n), t)=u((\vec x-2x_n \vec v,-x_n),t)\rightarrow u((\vec x, 0),t)$. Next, we consider 
$$\frac{\partial}{\partial x_n}\tilde u((\vec x, x_n),t)=\frac{\partial}{\partial x_n}u(x-2x_n v, t)=\frac{\partial u}{\partial x_n}(x-2x_n v, t)-2\langle \nabla u(x-2x_n v, t), v\rangle.$$
As $x_n\rightarrow 0^-$, the first term converges to $\partial_n u((\vec x, 0),t)$. Using the first order boundary condition from Equation \eqref{SchauderCh_ClassicalBoundaryCondition}, the second term converges to 
$$-2\langle \nabla u((\vec x, 0),t), v\rangle)=-2\frac{a^{ni}\partial_i u((\vec x, 0),t)}{a^{nn}}=0.$$
For the second derivative, we argue similarly. For $x=(\vec x, x_n)$ with $x_n<0$ we compute 
\begin{align*} 
\frac{\partial^2 \tilde u}{\partial x_n^2}(x,t)&=D^2 u(x-2x_n v,t)[e_n-2v, e_n-2v]\\
&=\frac{\partial ^2 u}{\partial x_n^2}(x-2x_n v,t)+4\left[D^2 u(x-2x_n v)[v,v]-D^2 u(x-2x_n v)[v,e_n]\right]\\
&=\frac{\partial ^2 u}{\partial x_n^2}(x-2x_n v,t)+4D^2 u(x-2x_n v)[v,v-e_n].
\end{align*}
When $x_n\rightarrow 0^-$, the first term converges to $\partial_n^2 u((\vec x, 0),t)$. We put $v^\parallel:=v-e_n\in\R^{n-1}\times 0$. Differentiating the first order boundary condition in Equation \eqref{SchauderCh_ClassicalBoundaryCondition} in direction $v^\parallel$, we obtain
\begin{align*}
    \lim_{x_n\rightarrow0^-}4D^2 u(x-2x_n v)[v,v-e_n]=&4 D^2 u((\vec x, 0),t)(v,v^\parallel)\\
            =&4\nabla_{v^\parallel}\left( (\nabla_v u)((\vec x,0),t) \right)\\
            =&4\nabla_{v^\parallel}\left( \frac{a^{ni}}{a^{nn}}\partial_i u((\vec x,0),t) \right)\\
            =&0.
\end{align*}
Thus, Equation \eqref{SchauderCh_DerivativesMatchUpCondition} has also been established for $k=2$. Next, we take care of $k=3$. Let $\alpha,\beta\in\set{1,...,n-1}$. We differentiate the first order boundary condition $a^{ni}\partial_i u(\vec x,0,t)=0$ with respect to $x_\alpha$ and $x_\beta$ to obtain $a^{ni}\partial_{i\alpha\beta}u(\vec x,0,t)=0$. Inserting into the third order boundary condition from Equation \eqref{SchauderCh_ClassicalBoundaryCondition}, we get
$$0=\sum_{i,k,l=1}^n a^{ni}a^{kl}\partial_{ijk} u
=\sum_{i=1}^n\left[a^{ni}a^{nn}\partial_{inn}u+2\sum_{\alpha=1}^{n-1}a^{ni}a^{n\alpha}\partial_{in\alpha}u\right].$$
We divide by $(a^{nn})^2$ and use that $v^i=\frac{a^{ni}}{a^{nn}}$ as well as $v^\parallel=\sum_{\alpha=1}^{n-1}v^\alpha e_\alpha$. This shows 
\begin{equation}\label{SchauderCh_ThirdDerivativeIdentity1}
\nabla^3 u((\vec x,0),t)[v,e_n,e_n]+2\nabla^3 u((\vec x,0),t)[v,v^\parallel, e_n]=0.
\end{equation}
Equation \eqref{SchauderCh_ThirdDerivativeIdentity1} and the first boundary condition $\nabla_v u((\vec x, 0),t)\equiv0$ imply
\begin{align}
    \nabla^3 u((\vec x,0),t)[v,v,v]&=\nabla^3 u((\vec x,0),t)[v,v^\parallel+e_n,v^\parallel+e_n]\nonumber\\
    &=\nabla^3 u((\vec x,0),t)[v,v^\parallel,v^\parallel]
    +2\nabla^3 u((\vec x,0),t)[v,e_n,v^\parallel]
    +\nabla^3 u((\vec x,0),t)[v,e_n,e_n]\nonumber\\
    &=\nabla_{v^\parallel}\nabla_{v^\parallel}\left((\nabla_v u)((\vec x, 0),t)\right)\nonumber\\
    &=0.\label{SchauderCh_ThirdDerivativeIdentity2}
\end{align}
We rewrite Equation \eqref{SchauderCh_ThirdDerivativeIdentity1} as
\begin{align}
    0&=\nabla^3 u((\vec x,0),t)[v,e_n,e_n]+2\nabla^3 u((\vec x,0),t)[v,v^\parallel, e_n]\nonumber\\
    &=\nabla^3 u((\vec x,0),t)[v,e_n,e_n]+\nabla^3 u((\vec x,0),t)[v,v^\parallel, e_n]+\nabla^3 u((\vec x,0),t)[v,v^\parallel, e_n]\nonumber\\
    &=\nabla^3 u((\vec x,0),t)[v,e_n+v^\parallel,e_n]+\nabla^3 u((\vec x,0),t)[v,v^\parallel, e_n]\nonumber\\
    &=\nabla^3 u((\vec x,0),t)[v,v,e_n]+\nabla^3 u((\vec x,0),t)[v,v^\parallel, e_n].\label{SchauderCh_ThirdDerivativeIdentity3}
\end{align}
Combining Equations \eqref{SchauderCh_ThirdDerivativeIdentity1}, \eqref{SchauderCh_ThirdDerivativeIdentity2} and \eqref{SchauderCh_ThirdDerivativeIdentity3}, we deduce
\begin{align*}
   \lim_{x_n\rightarrow 0^-} \frac{\partial^3 \tilde u}{\partial x_n^3}((\vec x, x_n),t)=&\lim_{x_n\rightarrow 0^-}\bigg(
    \frac{\partial^3  u}{\partial x_n^3}(x-2x_n v,t)
    +3\nabla^3 u(x-2x_n v,t)[e_n,e_n,-2v]\\
    &\hspace{.6cm} +3\nabla^3 u(x-2x_n v,t)[e_n,-2v,-2v]
    +\nabla^3 u(x-2x_n v,t)[-2v,-2v,-2v]\bigg)\\
    =&\frac{\partial^3  u}{\partial x_n^3}((\vec x, 0),t)
    -6\nabla^3 u((\vec x, 0),t)[e_n,e_n,v]
    +12\nabla^3 u((\vec x, 0),t)[e_n,v,v]\\
    &  \hspace{.6cm}   -8\nabla^3 u((\vec x, 0),t)[v,v,v]\\
    =&\frac{\partial^3  u}{\partial x_n^3}((\vec x, 0),t)
    -6\left(\nabla^3 u((\vec x, 0),t)[e_n,e_n,v]
    -2\nabla^3 u((\vec x, 0),t)[e_n,v,v]\right)\\
    =&\frac{\partial^3  u}{\partial x_n^3}((\vec x, 0),t).
\end{align*}
So, Equation \eqref{SchauderCh_DerivativesMatchUpCondition} holds for $k=3$. Finally, for $k=4$ we use that in $\R^{n-1}\times(-\infty,0)\times I$, $\tilde u$ solves
$$\partial_{nnnn}\tilde u(x,t)=-\frac{\dot{\tilde u}(x,t)}{(a^{nn})^2}-\frac1{(a^{nn})^2}\sum_{\substack{i,j,k,l=1\\ i+j+k+l\neq 4n}}^na^{ij}a^{kl}\partial_{ijkl} \tilde u(x,t).$$
As the right-hand side only contains derivatives with respect to $x_n$ of order $\leq 3$, and Equation \eqref{SchauderCh_DerivativesMatchUpCondition} is already proven for $k=0,1,2,3$, we deduce that it also holds for $k=4$. 
\end{proof}

 Lemma \ref{SchauderCh_ReflectionLemma} allows us to deduce the following corollaries to Corollaries \ref{SchauderCh_LiouvilleHalbraumNachUnten},\ref{SchauderCh_LiouvilleGanzraum} and \ref{SchauderCh_LiouvilleHalbraumNachOben}. We only prove the first of the following three corollaries, as the strategy for proving them is always the same. 

\begin{korollar}\label{SchauderCh_LiouvilleGanzraumBC}
Let $u\in C^{4,1}(\R^{n-1}\times[0,\infty)\times\R)$ satisfy Equation \eqref{SchauderCh_ClassicalEquation} as well as  $|u(x,t)|\leq C(1+d((x,t),0))^{4+\gamma}$ for some $\gamma\in(0,1)$, $C>0$ and all $(x,t)\in\R^{n-1}\times[0,\infty)\times\R$. If $a^{ni}\partial_i u((\vec x,0),t)$ and $a^{ni}a^{kl}\partial_{ikl}u((\vec x,0),t)$ are parabolic polynomials of parabolic degrees at most 3 and 1 respectively, then $u$ is a parabolic polynomial of parabolic degree $\operatorname{deg}(u)\leq 4$. 
\end{korollar}
\begin{proof}
Using the same argument as in the proof of Corollary \ref{SchauderCh_LiouvilleHalbraumNachOben}, we may assume without loss of generality that $a^{ni}\partial_i u((\vec x,0),t)=a^{ni}a^{kl}\partial_{ikl}u((\vec x,0),t)\equiv 0$. So, we can apply Lemma \ref{SchauderCh_ReflectionLemma} and define a solution $\tilde u(x,t)\in C^{4,1}(\R^n\times\R)$ of Equation \eqref{SchauderCh_ClassicalEquation}. For $x_n<0$, we estimate 
$$|\tilde u(x,t)=|u(x-2x_n v, t)|\leq C(1+ d((x-2x_nv, t),0))^{4+\gamma}\leq C(1+d((x,t),0))^{4+\gamma}.$$
We may therefore apply Corollary \ref{SchauderCh_LiouvilleGanzraum} and deduce the corollary.

\end{proof}

The following two corollaries can be established by essentially repeating the same proof.
\begin{korollar}\label{SchauderCh_LiouvilleHalbraumNachObenBC}
Let $u\in C^{4,1}(\R^{n-1}\times[0,\infty)\times[0,\infty))$ satisfy Equation \eqref{SchauderCh_ClassicalEquation} and $|u(x,t)|\leq C(1+d((x,t),0))^{4+\gamma}$ for some $\gamma\in(0,1)$, $C>0$ and all $(x,t)\in\R^{n-1}\times[0,\infty)\times[0,\infty)$. If $u(x,0)$ is a polynomial of degree at most 4 and
$a^{ni}\partial_i u((\vec x,0),t)$ and $a^{ni}a^{kl}\partial_{ikl}u((\vec x,0),t)$ are parabolic polynomials of parabolic degrees at most 3 and 1 respectively, then $u$ is a parabolic polynomial of parabolic degree $\operatorname{deg}(u)\leq 4$. 
\end{korollar}

\begin{korollar}\label{SchauderCh_LiouvilleHalbraumNachUntenBC}
Let $u\in C^{4,1}(\R^{n-1}\times[0,\infty)\times(-\infty,0])$ satisfy Equation \eqref{SchauderCh_ClassicalEquation} and  $|u(x,t)|\leq C(1+d((x,t),0))^{4+\gamma}$ for some $\gamma\in(0,1)$, $C>0$ and all $(x,t)\in\R^{n-1}\times[0,\infty)\times(-\infty,0]$. If 
$a^{ni}\partial_i u((\vec x,0),t)$ and $a^{ni}a^{kl}\partial_{ikl}u((\vec x,0),t)$ are parabolic polynomials of parabolic degrees at most 3 and 1 respectively, then $u$ is a parabolic polynomial of parabolic degree $\operatorname{deg}(u)\leq 4$. 
\end{korollar}

\section{Simon's Blow-up Argument}
Throughout this section, let $a^{ij}$ denote the entries of a strictly elliptic matrix. We consider the parabolic operator 
\begin{equation}\label{SchauderCh_BlowUpSectionAssumption01}Pu:=\partial_t u+a^ {ij}a^ {kl}\partial_{ijkl}u.
\end{equation}
Let further $\Lambda>0$ and $\theta>0$ such that 
\begin{equation}\label{SchauderCh_BlowUpSectionAssumption02}
|a^{ij}|\leq\Lambda\textrm{ for all $1\leq i,j\leq n$}
\hspace{.5cm}\textrm{and}\hspace{.5cm}
a^{ij}\xi_i\xi_j\geq\theta|\xi|^2\textrm{ for all }\xi\in\R^n.
\end{equation}

\begin{lemma}\label{SchauderCh_BlowUpGanzraumLemma}
There exists a constant $C(n,\gamma,\Lambda,\theta)$ such that for all constant coefficient operators $P$ of the form \eqref{SchauderCh_BlowUpSectionAssumption01} satisfying Assumption \eqref{SchauderCh_BlowUpSectionAssumption02} and all $u\in C^{4,1,\gamma}(\R^n\times\R)$
$$[D^{4,1}u]_{\gamma,\R^n\times\R}^{(0)}\leq C(n,\gamma, \Lambda,\theta)[Pu]_{\gamma,\R^n\times\R}^{(0)}.$$
\end{lemma}
\begin{proof}
For the duration of the proof, we write $[\cdot]_\gamma:=[\cdot]_{\gamma,\R^n\times\R}^{(0)}$. The proof is by contradiction. If the lemma were false we could find operators $P_k$ of the form \eqref{SchauderCh_BlowUpSectionAssumption01} that satisfy Assumption \eqref{SchauderCh_BlowUpSectionAssumption02} and functions $u_k$ such that $[D^{4,1} u_k]_\gamma>k[P_ku_k]_\gamma$. We define
$$v_k:=\frac{u_k}{[D^{4,1} u_k]_\gamma}\hspace{.5cm}\textrm{so that}\hspace{.5cm}[D^{4,1} v_k]_\gamma=1\textrm{ and }[P_k v_k]_\gamma<\frac1k.$$
$[D^{4,1} v_k]_\gamma=1$ implies the existence of $p_k=(x_k, t_k)\neq q_k=(y_k, s_k)\in\R^n\times\R$ and $\epsilon_0(n)\in(0,1)$ such that 
\begin{equation}\label{SchauderCh_BlowUpGanzraum1}
    |D^{4,1}v_k(p_k)-D^{4,1}v_k(q_k)|\geq\epsilon_0(n)d(p_k,q_k)^\gamma.
\end{equation}
Here $d$ is the parabolic metric from Definition \ref{SchauderCh_parabolicMetricDefinition}. We put $\lambda_k:=d(p_k,q_k)$ and let 
\begin{equation}\label{SchauderCh_BlowUpGanzraum11}
w_k(x,t):=\lambda_k^{-4-\gamma}v_k(q_k+(\lambda_k x,\lambda_k^4 t)).
\end{equation}
Note that $\dot w_k(x,t)=\lambda_k^{-\gamma}\dot v_k(q_k+(\lambda_k x,\lambda_k^4 t))$ and $\nabla_\alpha w_k(x,t)=\lambda_k^{-\gamma}(\nabla_\alpha v_k)(q_k+(\lambda_k x,\lambda_k^4 t))$ for all multi indices $\alpha\in\N_0^n$ of length 4. This implies 
\begin{equation}\label{SchauderCh_BlowUpGanzraum2}
    [D^{4,1} w_k]_\gamma=[D^{4,1}v_k]_\gamma=1
    \hspace{.5cm}\textrm{and}\hspace{.5cm}
    [P_k w_k]_\gamma=[P_kv_k]_\gamma<\frac1k.
\end{equation}
Let $\tilde p_k:=(\lambda_k^{-1}(x_k-y_k),\lambda_k^{-4}(t_k-s_k))$. In view of Estimate \eqref{SchauderCh_BlowUpGanzraum1} and the definition of $w_k$ we get 
\begin{equation}\label{SchauderCh_BlowUpGanzraum3}
    |D^{4,1} w_k(\tilde p_k)- D^{4,1}w_k(0)|\geq\epsilon_0(n).
\end{equation}
In a final step, we modify $w_k$ by subtracting the parabolic Taylor polynomial 
$$\tilde w_k(x,t):=w_k(x,t)-\dot w_k(0,0)t-\sum_{|\alpha|\leq 4}\frac{\nabla_\alpha w_k(0,0)}{\alpha!}x^\alpha.$$
This ensure that $\partial_t^ k\nabla_\alpha \tilde w_k(0,0)=0$ whenever $0\leq 4k+|\alpha|\leq 4$. Note that $D^{4,1}\tilde w_k$ differs from $D^{4,1}w_k$ only by a constant. So Estimates \eqref{SchauderCh_BlowUpGanzraum2} and \eqref{SchauderCh_BlowUpGanzraum3} transform into 
\begin{equation}\label{SchauderCh_BlowUpGanzraum4}
    [D^{4,1}\tilde w_k]_\gamma=1,\hspace{.5cm}
    [P_k\tilde w_k]_\gamma<\frac1k\hspace{.5cm}\textrm{and}\hspace{.5cm}
    |D^{4,1}\tilde w_k(\tilde p_k) |\geq\epsilon_0(n).
\end{equation}
We claim that $(\tilde p_k)\subset\R^n\times\R$ is bounded and that $(\tilde w_k)\subset C^{4,1,\gamma}(\R^n\times\R)$ is locally bounded. The first claim follows from
$$d(\tilde p_k, 0)=d\left(\left(
\frac{x_k-y_k}{\lambda_k},\frac{t_k-s_k}{\lambda_k^4}
,0\right)\right)
=\frac1{\lambda_k}d(( x_k-y_k,t_k-s_k),0)
=\frac1{\lambda_k}d(p_k,q_k)
=1.$$
Next, note that $[D^{4,1}\tilde w_k]_\gamma=1$ and $D^{4,1}\tilde w_k(0,0)=0$ imply
\begin{equation}\label{SchauderCh_D41BoundedMISS} 
|D^{4,1}\tilde w_k(x,t)|\leq |D^{4,1}\tilde w_k(0,0)|+C(|x|^\gamma+|t|^{\frac\gamma4})\leq C(|x|+|t|^{\frac14})^\gamma.
\end{equation}
Using $\dot{\tilde w}_k(0,0)=\nabla_\alpha \tilde w_k(0,0)=0$ for $0\leq |\alpha|\leq 4$ and the intermediate value theorem, we deduce that for any multi-index $|\alpha|\leq 3$
\begin{equation}\label{SchauderCh_D41BoundedMISS2} 
    |\nabla_\alpha\tilde w_k(x,t)|\leq C(n)(|x|+|t|^{\frac14})^{4-|\alpha|+\gamma}.
\end{equation}
Estimates \eqref{SchauderCh_D41BoundedMISS} and \eqref{SchauderCh_D41BoundedMISS2} imply that $(w_k)\subset C^{4,1,\gamma}(K\times[0,T])$ is bounded for every compact set $K\times[0,T]\subset\R^n\times \R$ and hence we can pass to a subsequence along which $\tilde p_k\rightarrow\tilde p$ and $\tilde w_k\rightarrow\tilde w$ in $C^{4,1}_{\operatorname{loc}}$. As the operators $P_k$ satisfy Assumption \eqref{SchauderCh_BlowUpSectionAssumption02}, we may, after potentially passing to another subsequence, assume that $P_k\rightarrow P$, where $P$ is an operator of the form \eqref{SchauderCh_BlowUpSectionAssumption01} that satisfies Assumption \eqref{SchauderCh_BlowUpSectionAssumption02}. As  $\dot {\tilde w}_k(0,0)=\nabla_\alpha \tilde w_k(0,0)=0$ for all $|\alpha|\leq 4$ the same remains true for $\tilde w$. Taking $k\rightarrow\infty$ in Estimate \eqref{SchauderCh_BlowUpGanzraum4}, we deduce 
\begin{align}
    &[P\tilde w]_\gamma\leq \liminf_{k\rightarrow\infty}[P_k\tilde w_k]_\gamma=0\hspace{.5cm}\textrm{ and, using $P\tilde w(0)=0$, we get }P\tilde w=0,\\
    &|\tilde w(x,t)|=\lim_{k\rightarrow\infty}|\tilde w_k(x,t)|\leq C(n)(|x|+|t|^{\frac14})^{4+\gamma}.
\end{align}
By Corollary \ref{SchauderCh_LiouvilleGanzraum}, $\tilde w$ is a parabolic polynomial of degreee $\operatorname{deg}(\tilde w)\leq 4$. In particular $D^{4,1}\tilde w$ is a constant and therefore $D^{4,1}\tilde w(\tilde p)=D^{4,1}\tilde w(0)=0$. However, taking $k\rightarrow\infty$ in Estimate \eqref{SchauderCh_BlowUpGanzraum4}, we find
$$|D^{4,1}\tilde w(\tilde p)|=\lim_{k\rightarrow\infty}|D^{4,1}\tilde w_k(\tilde p_k)|\geq\epsilon_0(n).$$
Thus we have arrived at a contradiction and the lemma follows. 
\end{proof}

Following essentially the same argument, we can obtain similar estimates when we have either a temporal or a spatial boundary or both. In the following, we give the respective results and comment on the required changes in the proofs.
\begin{lemma}\label{SchauderCh_BlowUpHalbraumLemma}
There exists a constant $C(n,\gamma,\Lambda,\theta)$ such that for all constant coefficient operators $P$ of the form \eqref{SchauderCh_BlowUpSectionAssumption01} satisfying Assumption \eqref{SchauderCh_BlowUpSectionAssumption02} and all $u\in C^{4,1,\gamma}(\R^n\times[0,\infty))$
$$[D^{4,1}u]_{\gamma, \R^n\times[0,\infty)}^{(0)}\leq C(n,\gamma, \Lambda,\theta)\left([Pu]_{\gamma,\R^n\times[0,\infty)}^{(0)}+[\nabla^4u(\cdot,0)]_{\gamma,\R^n}\right).$$
\end{lemma}
\begin{proof}
Throughout the proof, we only write $[\cdot]_\gamma$ for the various Hölder seminorms. We follow the proof of Lemma \ref{SchauderCh_BlowUpGanzraumLemma} and assume the lemma was false. Then we could find operators $P_k$ of the form \eqref{SchauderCh_BlowUpSectionAssumption01} that satisfy Assumption \eqref{SchauderCh_BlowUpSectionAssumption02} and functions $u_k$ such that 
$$[D^{4,1}u_k]_{\gamma}>k\left([P_ku_k]_{\gamma}+[\nabla^4 u_k(\cdot,0)]_{\gamma}\right).$$
Next, we choose $p_k=(x_k,t_k)$, $q_k=(y_k,s_k)$  as in the proof of Lemma \ref{SchauderCh_BlowUpGanzraumLemma}. Without loss of generality $t_k\geq s_k$. We define $\tilde p_k$ and the functions $\tilde w_k$ as we did in the proof of Lemma \ref{SchauderCh_BlowUpGanzraumLemma}. However, in view of Equation \eqref{SchauderCh_BlowUpGanzraum11}, as $u_k$ are only defined on $\R^n\times [0,\infty)$, the functions $w_k$ and $\tilde w_k$ are only defined for $t\geq -\lambda_k^{-4}s_k$. Following the arguments in the proof of Lemma \ref{SchauderCh_BlowUpGanzraumLemma}, we obtain
\begin{align}
    &[D^{4,1}\tilde w_k]_\gamma=1,\hspace{.5cm} |\tilde w_k(x,t)|\leq C(n)(|x|+|t|^{\frac14})^{4+\gamma},\label{SchauderCh_BlowUpObererHalbraum0}\\
    &|D^{4,1}\tilde w_k(\tilde p_k)-D^{4,1}\tilde w_k(0) |\geq\epsilon_0(n),\hspace{.5cm}\tilde p_k\in\R^n\times[0,\infty),\label{SchauderCh_BlowUpObererHalbraum1}\\
    &[\nabla^4\tilde w_k(\cdot,-\lambda_k^{-4}s_k)]_\gamma<\frac1k\hspace{.5cm}\textrm{and}\hspace{.5cm}    [P_k\tilde w_k]_\gamma<\frac1k,\label{SchauderCh_BlowUpObererHalbraum2}\\
    &\partial_t^l\nabla_\alpha\tilde w_k(0,0)=0\hspace{.5cm}\textrm{whenever}\hspace{.5cm}   0\leq 4l+|\alpha|\leq 4.\label{SchauderCh_BlowUpObererHalbraum3}
\end{align}
 After passing to a subsequence, we may assume that $\lambda_k^{-4}s_k\rightarrow\tau\in[0,\infty]$, $\tilde p_k\rightarrow\tilde p$, $P_k\rightarrow P$ and $\tilde w_k\rightarrow \tilde w$ in $C^{4,1}_{\operatorname{loc}}$. Now there are two cases to consider: First, if $\tau=\infty$, the limiting function $\tilde w$ is defined on $\R^n\times \R$ and we can derive the same contradiction as in the proof of Lemma \ref{SchauderCh_BlowUpGanzraumLemma}. If, however, $\tau<\infty$, we obtain a limiting function $\tilde w\in C^{4,1}(\R^n\times[-\tau,\infty))$. Letting $k\rightarrow\infty$ in \eqref{SchauderCh_BlowUpObererHalbraum3} yields $\partial_t^l\nabla_\alpha\tilde w_k(0,0)=0$ whenever $0\leq 4l+|\alpha|\leq 4$. Letting $k\rightarrow\infty$ in Estimates \eqref{SchauderCh_BlowUpObererHalbraum0}, \eqref{SchauderCh_BlowUpObererHalbraum1} and \eqref{SchauderCh_BlowUpObererHalbraum2}, we get 
\begin{align*}
    &[P\tilde w]_\gamma\leq \liminf_{k\rightarrow\infty}[P_k\tilde w_k]_\gamma=0\hspace{.5cm}\textrm{ and, using $P\tilde w(0)=0$, we get }P\tilde w=0,\\
    &[\nabla^4\tilde w(\cdot, -\tau)]_\gamma\leq \liminf_{k\rightarrow\infty}[\nabla^4\tilde w_k(\cdot,-\lambda_k^{-4}s_k)]_\gamma=0\hspace{.25cm}\textrm{ and hence $\tilde w(\cdot, -\tau)$ is a polynomial of degree $\leq 4$,}\\
    &|\tilde w(x,t)|=\lim_{k\rightarrow\infty}|\tilde w_k(x,t)|\leq C(n)(|x|+|t|^{\frac14})^{4+\gamma}\hspace{.5cm}\textrm{and}\hspace{.5cm}|D^{4,1}\tilde w(\tilde p)-D^{4,1}\tilde w(0,0)|\geq\epsilon_0(n).
\end{align*}
Now we may apply Corollary \ref{SchauderCh_LiouvilleHalbraumNachOben} to deduce that $\tilde w$ is a parabolic polynomial of degree at most 4. We can now produce the same contradiction as in the proof of Lemma \ref{SchauderCh_BlowUpGanzraumLemma}: By Corollary \ref{SchauderCh_LiouvilleHalbraumNachOben} $D^{4,1}\tilde w$ must be a constant, which contradicts $D^{4,1}\tilde w(\tilde p)\neq D^{4,1}\tilde w(0)$.
\end{proof}

\begin{lemma}\label{SchauderCh_BlowUpHalbRaumUntenLemma}
 There exists a constant $C(n,\gamma,\Lambda,\theta)$ such that for all constant coefficient operators $P$ of the form \eqref{SchauderCh_BlowUpSectionAssumption01} satisfying Assumption \eqref{SchauderCh_BlowUpSectionAssumption02} and all $u\in C^{4,1,\gamma}(\R^n\times(-\infty,0])$
$$[D^{4,1}u]^{(0)}_{\gamma, \R^n\times(-\infty,0]}\leq C(n,\gamma, \Lambda,\theta)[Pu]^{(0)}_{\gamma, \R^n\times(-\infty,0]}.$$
\end{lemma}
\begin{proof}
Throughout the proof, we write $[\cdot]_\gamma$ for the Hölder seminorms. We follow the proof of Lemma \ref{SchauderCh_BlowUpGanzraumLemma} and assume the lemma was false. Then there exist operators $P_k$ of the form \eqref{SchauderCh_BlowUpSectionAssumption01} that satisfy Assumption \eqref{SchauderCh_BlowUpSectionAssumption02} and functions $u_k$ such that 
$$[D^{4,1}u_k]_{\gamma}>k[P_ku_k]_\gamma.$$
Next, we define $p_k=(x_k,t_k)$, $q_k=(y_k,s_k)$ as we did in the proof of Lemma \ref{SchauderCh_BlowUpGanzraumLemma}. Additionally, we assume without loss of generality that $t_k\leq s_k$. We define $\tilde p_k$ and the functions $\tilde w_k$ as we did in the proof of Lemma \ref{SchauderCh_BlowUpGanzraumLemma}. In view of Equation \eqref{SchauderCh_BlowUpGanzraum11} the functions $\tilde w_k$ are only defined for $t\leq -\lambda_k^{-4}s_k$. Following the arguments in the proof of Lemma \ref{SchauderCh_BlowUpGanzraumLemma}, we obtain
\begin{align}
    &[D^{4,1}\tilde w_k]_\gamma=1,\hspace{.5cm} |\tilde w_k(x,t)|\leq C(n)(|x|+|t|^{\frac14})^{4+\gamma},\hspace{.5cm} [P_k\tilde w_k]_\gamma<\frac1k,\label{SchauderCh_BlowUpUntererHalbraum1}\\
    &|D^{4,1}\tilde w_k(\tilde p_k)-D^{4,1}\tilde w_k(0) |\geq\epsilon_0(n),\hspace{.5cm}\tilde p_k\in\R^n\times(-\infty,0],\label{SchauderCh_BlowUpUntererHalbraum2}\\
    &\partial_t^l\nabla_\alpha\tilde w_k(0,0)=0\hspace{.5cm}\textrm{whenever}\hspace{.5cm}   0\leq 4l+|\alpha|\leq 4.\label{SchauderCh_BlowUpUntererHalbraum3}
\end{align}
 After passing to a subsequence, we may assume that $\lambda_k^{-4}s_k\rightarrow\tau\in[-\infty,0]$, $\tilde p_k\rightarrow\tilde p$, $P_k\rightarrow P$ and $\tilde w_k\rightarrow \tilde w$ in $C^{4,1}_{\operatorname{loc}}$. Now there are two cases to consider: First, if $\tau=-\infty$, the limiting function $\tilde w$ is defined on $\R^n\times \R$ and we can derive the same contradiction as in the proof of Lemma \ref{SchauderCh_BlowUpGanzraumLemma}. If, however, $\tau>-\infty$, we obtain a limiting function $\tilde w\in C^{4,1}(\R^n\times(-\infty,-\tau])$. Letting $k\rightarrow\infty$ in \eqref{SchauderCh_BlowUpUntererHalbraum3} yields $\partial_t^l\nabla_\alpha\tilde w(0,0)=0$ whenever $0\leq 4l+|\alpha|\leq 4$ and letting $k\rightarrow\infty$ in Estimates \eqref{SchauderCh_BlowUpUntererHalbraum1} and \eqref{SchauderCh_BlowUpUntererHalbraum2}, we get 
\begin{align*}
    &[P\tilde w]_\gamma\leq \liminf_{k\rightarrow\infty}[P_k\tilde w_k]_\gamma=0\hspace{.5cm}\textrm{ and, using $P\tilde w(0)=0$, we get }P\tilde w=0,\\
    &|\tilde w(x,t)|=\lim_{k\rightarrow\infty}|\tilde w_k(x,t)|\leq C(n)(|x|+|t|^{\frac14})^{4+\gamma}\hspace{.5cm}\textrm{and}\hspace{.5cm}|D^{4,1}\tilde w(\tilde p)-D^{4,1}\tilde w(0)|\geq\epsilon_0(n).
\end{align*}
By Corollary \ref{SchauderCh_LiouvilleHalbraumNachUnten}, we deduce that $D^{4,1}\tilde w$ is constant, which contradicts $D^{4,1}\tilde w(\tilde p)\hspace{-.05cm}\neq\hspace{-.05cm} D^{4,1}\tilde w(0)$. 
\end{proof}

Next, we formulate similar lemmas that allow for a spatial boundary. For that sake, we introduce the constant coefficient boundary operators \emph{induced by $P$} from \eqref{SchauderCh_BlowUpSectionAssumption01}  
\begin{equation}\label{SchauderCh_BlowUpSectionAssumption03}
    B_1u:=a^{ni}\partial_i u\hspace{.5cm}\textrm{and}\hspace{.5cm}
    B_2 u=a^{ni}a^{kl}\partial_{ikl}u.
\end{equation}

\begin{lemma}\label{SchauderCh_BlowUpGanzraumLemmaBC}
 There exists a constant $C(n,\gamma,\Lambda,\theta)$ such that for all constant coefficient operators $P$ of the form \eqref{SchauderCh_BlowUpSectionAssumption01} satisfying Assumption \eqref{SchauderCh_BlowUpSectionAssumption02} with induced boundary operators $B_1$ and $B_2$ and all $u\in C^{4,1,\gamma}(\R^{n-1}\times[0,\infty)\times\R)$
$$[D^{4,1}u]_{\gamma,\R^{n-1}\times[0,\infty)\times\R}^{(0)}\leq C(n,\gamma, \Lambda,\theta)\left(
[Pu]_{\gamma,\R^{n-1}\times[0,\infty)\times\R}^{(0)}+[B_1 u]^{(3)}_{\gamma,\R^{n-1}\times0\times\R}+[B_2u]^{(1)}_{\gamma,\R^{n-1}\times0\times\R}
\right).$$
\end{lemma}
The proof is essentially the same as for Lemma \ref{SchauderCh_BlowUpGanzraumLemma}. The important new observation is that $[a^{ni}\partial_i u]^{(3)}_{\gamma,\R^{n-1}\times0\times\R}$ and $[a^{ni}a^{kl}\partial_{ikl} u]^{(1)}_{\gamma,\R^{n-1}\times0\times\R}$ scale just like $[Pu]_{\gamma,\R^{n-1}\times[0,\infty)\times\R}^{(0)}$ under the parabolic scaling $(x,t)\mapsto (\lambda x,\lambda^4 t)$. We have shown this in the motivation directly after Definition \ref{SchauderCh_DefinitionofParabolicHölderSpaces}.
\begin{proof}
Throughout the proof, we write $[\cdot]_\gamma$ for the Hölder seminorms. We follow the proof of Lemma \ref{SchauderCh_BlowUpGanzraumLemma} and assume the lemma was false. Then there exist operators $P_k$ of the form \eqref{SchauderCh_BlowUpSectionAssumption01} that satisfy Assumption \eqref{SchauderCh_BlowUpSectionAssumption02} with associated boundary operators $B_{1k}$, $B_{2k}$ as in \eqref{SchauderCh_BlowUpSectionAssumption03} and functions $u_k$ such that 
$$[D^{4,1}u_k]_{\gamma}>k\left([P_ku_k]_\gamma+[B_{1k}u_k]_\gamma+[B_{2k}u_k]_\gamma\right).$$
We choose $p_k=(\vec x_k, \zeta_k, t_k)\in\R^{n-1}\times[0,\infty)\times\R$ and $q_k=(\vec y_k,\xi_k, s_k)\in\R^{n-1}\times[0,\infty)\times\R$ as in the proof of Lemma \ref{SchauderCh_BlowUpGanzraumLemma}. Additionally, we assume without loss of generality that $\zeta_k\geq \xi_k$. We define $\tilde p_k\in\R^{n-1}\times[0,\infty)\times\R$ and $\tilde w_k$ as in the proof of Lemma \ref{SchauderCh_BlowUpGanzraumLemma}. As $u_k$ are only defined for $x_n\geq 0$, the functions $\tilde w_k$ are only defined for $x_n\geq -\lambda_k^{-1}\xi_k$. Following the proof of Lemma \ref{SchauderCh_BlowUpGanzraumLemma}, we get
\begin{align}
    &[D^{4,1}\tilde w_k]_\gamma=1,\hspace{.5cm} |\tilde w_k(x,t)|\leq C(n)(|x|+|t|^{\frac14})^{4+\gamma},\label{SchauderCh_BlowUpGanzraumBC0}\\
    & [P_k\tilde w_k]_\gamma+[B_{1k}\tilde w_k]_\gamma+[B_{2k}\tilde w_k]_\gamma <\frac1k,\label{SchauderCh_BlowUpGanzraumBC1}\\
    &|D^{4,1}\tilde w_k(\tilde p_k)-D^{4,1}\tilde w_k(0) |\geq\epsilon_0(n),\hspace{.5cm}\tilde p_k\in\R^{n-1}\times[0,\infty)\times\R ,\label{SchauderCh_BlowUpGanzraumBC2}\\
    &\partial_t^l\nabla_\alpha\tilde w_k(0,0)=0\hspace{.5cm}\textrm{whenever}\hspace{.5cm}   0\leq 4l+|\alpha|\leq 4.\label{SchauderCh_BlowUpGanzraumBC3}
\end{align}
As $k\rightarrow\infty$ we have $\lambda_k^{-1}\xi_k\rightarrow z\in[0,\infty]$, $\tilde p_k\rightarrow\tilde p$, $P_k\rightarrow P$ and $\tilde w_k\rightarrow \tilde w$ in $C^{4,1}_{\operatorname{loc}}$. Additionally, by Assumption \eqref{SchauderCh_BlowUpSectionAssumption02} and Equation \eqref{SchauderCh_BlowUpSectionAssumption03}, we deduce that $B_{1k}\rightarrow B_1$ and $B_{2k}\rightarrow B_2$, where $B_1$ and $B_2$ are the boundary operators induced by $P$. If $z=\infty$, the limiting function $\tilde w$ is defined on all of $\R^n\times\R$ and we deduce the same contradiction as in the proof of Lemma \ref{SchauderCh_BlowUpGanzraumLemma}.  If $z<\infty$, we obtain a limiting function $\tilde w\in C^{4,1}(\R^{n-1}\times[-z,\infty)\times\R)$. Letting $k\rightarrow\infty$ in Estimates \eqref{SchauderCh_BlowUpGanzraumBC0}, \eqref{SchauderCh_BlowUpGanzraumBC1} and \eqref{SchauderCh_BlowUpGanzraumBC2}, and using \eqref{SchauderCh_BlowUpGanzraumBC3}, we deduce that $\tilde w$ satisfies 
\begin{align*}
    &[P\tilde w]_\gamma\leq \liminf_{k\rightarrow\infty}[P_k\tilde w_k]_\gamma=0\hspace{.5cm}\textrm{ and hence }P\tilde w=0,\\
    &[B_1\tilde w(\cdot, -z,\cdot)]_\gamma+[B_2\tilde w(\cdot, -z,\cdot)]_\gamma
    \leq \liminf_{k\rightarrow\infty}\left(
    [B_1\tilde w_k(\cdot,-\lambda_k^{-1}\xi_k,\cdot)]_\gamma+[B_2\tilde w_k(\cdot,-\lambda_k^{-1}\xi_k,\cdot)]_\gamma\right)
    =0,\\
    &|\tilde w(x,t)|=\lim_{k\rightarrow\infty}|\tilde w_k(x,t)|\leq C(n)(|x|+|t|^{\frac14})^{4+\gamma}\hspace{.5cm}\textrm{and}\hspace{.5cm}|D^{4,1}\tilde w(\tilde p)-D^{4,1}\tilde w(0)|\geq\epsilon_0(n).
\end{align*}
We claim that $B_1 w$ and $B_2 w$ are parabolic polynomials of degree 3 or less and 1 or less, respectively.  Once this is shown, we can use Corollary \ref{SchauderCh_LiouvilleGanzraumBC} to deduce that $\tilde w$ is a parabolic polynomial of degree at most 4. This produces the usual contradiction of $D^{4,1}\tilde w$ being constant while simultaneously $|D^{4,1}\tilde w(\tilde p)-D^{4,1}\tilde w(0)|\geq\epsilon_0(n)$.\\

We prove that $B_1\tilde w$ is a parabolic polynomial of degree 3 or less. Proving that $B_2 \tilde w$ is a parabolic polynomial of degree 1 or less is then achieved by a similar argument. By definition of the $[\cdot]_\gamma^{(3)}$-seminorm $[\nabla^3 B_1\tilde w]_{\gamma, \R^{n-1}\times (-z)\times\R}^{\operatorname{space}}=0$. So $\nabla^3 B_1w(\vec x, -z,0)$ is constant and hence $B_1w(\vec x, -z,0)$ a polynomial in $\vec x$ of degree 3 or less. Also by definition of the $[\cdot]_\gamma^{(3)}$-seminorm, we deduce that $[B_1\tilde w]^{\operatorname{time}}_{\frac{3+\gamma}4,\R^{n-1}\times(-z)\times\R}=0$ and thus $B_1\tilde w(\vec x, -z,t)=B_1\tilde w(\vec x, -z,0)$. Combining these two observations, we deduce that $B_1\tilde w(\vec x, -z,t)$ is a parabolic polynomial in $(\vec x, t)$ of degree 3 or less. 
\end{proof}

Finally, we consider the cases where we have a spatial and a temporal boundary. 
\begin{lemma}\label{SchauderCh_BlowUpHalbraumObenLemmaBC}
 There exists a constant $C(n,\gamma,\Lambda,\theta)$ such that for all constant coefficient operators $P$ of the form \eqref{SchauderCh_BlowUpSectionAssumption01} satisfying Assumption \eqref{SchauderCh_BlowUpSectionAssumption02} with induced boundary operators $B_1$ and $B_2$ and all $u\in C^{4,1,\gamma}(\R^{n-1}\times[0,\infty)\times[0,\infty))$
\begin{align*} 
[D^{4,1}u]^{(0)}_{\gamma,\R^{n-1}\times[0,\infty)\times[0,\infty)}\leq 
C(n,\gamma, \Lambda,\theta)\big(&
[Pu]_{\gamma,\R^{n-1}\times[0,\infty)\times[0,\infty)}^{(0)}
+[B_1 u]^{(3)}_{\gamma,\R^{n-1}\times0\times[0,\infty)}\\
&+[B_2u]^{(1)}_{\gamma,\R^{n-1}\times0\times[0,\infty)}
+[\nabla^4 u(\cdot,0)]_{\gamma,\R^{n-1}\times[0,\infty)}
\big).
\end{align*}
\end{lemma}
\begin{proof}
Throughout the proof, we write $[\cdot]_\gamma$ for the Hölder seminorms. We follow the proof of Lemma \ref{SchauderCh_BlowUpGanzraumLemma} and assume the lemma was false. Then there exist operators $P_k$ of the form \eqref{SchauderCh_BlowUpSectionAssumption01} that satisfy Assumption \eqref{SchauderCh_BlowUpSectionAssumption02} with associated boundary operators $B_{1k}$, $B_{2k}$ as in \eqref{SchauderCh_BlowUpSectionAssumption03} and functions $u_k$ such that
$$[D^{4,1}u_k]_{\gamma}>k\left([P_ku_k]_\gamma+[B_{1k} u_k]_\gamma+[B_{2k} u_k]_\gamma+[\nabla^4 u_k(\cdot,0)]_\gamma\right).$$
Next, we define $p_k=(\vec x_k,\zeta_k,t_k)$, $q_k=(\vec y_k,\xi_k,s_k)$ as we did in the proof of Lemma \ref{SchauderCh_BlowUpGanzraumLemma}. Without loss of generality, we assume that $\zeta_k\geq\xi_k$. We define $\tilde p_k$ and the functions $\tilde w_k$ as we did in the proof of Lemma \ref{SchauderCh_BlowUpGanzraumLemma}. In view of Equation \eqref{SchauderCh_BlowUpGanzraum11} the functions $\tilde w_k$ are only defined on $\R^{n-1}\times[-\lambda_k^{-1}\xi_k, \infty)\times[-\lambda_k^{-4}s_k,\infty)$. After passing to some subsequence we may assume that $\lambda_k^{-1}\xi_k\rightarrow z\in[0,\infty]$, $\lambda_k^{-4}s_k\rightarrow \tau\in[0,\infty]$, $\tilde p_k\rightarrow\tilde p$, $P_k\rightarrow P$, $B_{1k}\rightarrow B_1$, $B_{2k}\rightarrow B_2$ and $\tilde w_k\rightarrow \tilde w$ in $C^{4,1}_{\operatorname{loc}}$. We have the following cases:

\begin{enumerate}[(1)]
    \item If $\tau=z=\infty$, we obtain a limit function $\tilde w\in C^{4,1}(\R^n\times \R)$ and arrive at the same contradiction as in the proof of Lemma \ref{SchauderCh_BlowUpGanzraumLemma}.
    \item If $z=\infty$ but $\tau<\infty$, we obtain a limit function $\tilde w\in  C^{4,1}(\R^n\times [-\tau,\infty))$ and arrive at the same contradiction as in the proof of Lemma \ref{SchauderCh_BlowUpHalbraumLemma}.
    \item If $\tau=\infty$ but $z<\infty$, we obtain a limit function $\tilde w\in C^{4,1}(\R^{n-1}\times[-z,\infty)\times \R)$ and arrive at the same contradiction as in the proof of Lemma \ref{SchauderCh_BlowUpGanzraumLemmaBC}.
    \item If $z,\tau<\infty$, we obtain a limit function $\tilde w\in C^{4,1}(\R^{n-1}\times [-z,\infty)\times[-\tau,\infty))$. Following the arguments in the proofs of Lemmas \ref{SchauderCh_BlowUpGanzraumLemma},  \ref{SchauderCh_BlowUpHalbraumLemma} and \ref{SchauderCh_BlowUpGanzraumLemmaBC}, the limit $\tilde w$ satisfies:
    $$\begin{array}{lll}
        &P\tilde w=0&\textrm{ in }\R^{n-1}\times[-z,\infty)\times[-\tau,\infty),\\
        &B_1 \tilde w(\vec x, -z, t)&\textrm{ is a parabolic polynomial in $(\vec x, t)$ of degree 3 or less,}\\
        &B_2 \tilde w(\vec x, -z, t)&\textrm{ is a parabolic polynomial in $(\vec x, t)$ of degree 1 or less,}\\
        &\tilde w(x,-\tau)&\textrm{ is a polynomial of degree 4 or less,}
    \end{array}
    $$
    $$\textrm{as well as}\hspace{.5cm}
    |D^{4,1}\tilde w(\tilde p)-D^{4,1}\tilde w(0)|\geq\frac12.\hspace{4.2cm}
    $$
    Using Corollary \ref{SchauderCh_LiouvilleHalbraumNachObenBC}, $\tilde w$ is a parabolic polynomial of degree 4 or less and hence $D^{4,1}\tilde w$ is constant, which contradicts $D^{4,1}\tilde w(0)\neq D^{4,1}\tilde w(\tilde p)$.
\end{enumerate}
\end{proof}

Using the same arguments but replacing 
Lemma \ref{SchauderCh_BlowUpHalbraumLemma}
with Lemma \ref{SchauderCh_BlowUpHalbRaumUntenLemma}
in the second and fourth case and
Corollary \ref{SchauderCh_LiouvilleHalbraumNachObenBC}
with Corollary \ref{SchauderCh_LiouvilleHalbraumNachUntenBC} in the fourth case,
we can derive the following lemma:

\begin{lemma}\label{SchauderCh_BlowUpHalbraumUntenLemmaBC}
 There exists a constant $C(n,\gamma,\Lambda,\theta)$ such that for all constant coefficient operators $P$ of the form \eqref{SchauderCh_BlowUpSectionAssumption01} satisfying Assumption \eqref{SchauderCh_BlowUpSectionAssumption02} with induced boundary operators $B_1$ and $B_2$ and all $u\in C^{4,1,\gamma}(\R^{n-1}\times[0,\infty)\times(-\infty,0] )$
\begin{align*} 
[D^{4,1}u]^{(0)}_{\gamma,\R^{n-1}\times[0,\infty)\times(-\infty,0]}\leq 
C(n,\gamma, \Lambda,\theta)\big(&
[Pu]_{\gamma,\R^{n-1}\times[0,\infty)\times(-\infty,0]}^{(0)}
+[B_1 u]^{(3)}_{\gamma,\R^{n-1}\times0\times(-\infty,0]}\\
&+[B_2u]^{(1)}_{\gamma,\R^{n-1}\times0(-\infty,0]}
\big).
\end{align*}
\end{lemma}

\section{Localized Schauder Estimates}\label{SchauderCh_LocalizedEstimatesSection}
To localize the results from Lemmas 
\ref{SchauderCh_BlowUpGanzraumLemma},
\ref{SchauderCh_BlowUpHalbraumLemma},
\ref{SchauderCh_BlowUpHalbRaumUntenLemma},
\ref{SchauderCh_BlowUpGanzraumLemmaBC},
\ref{SchauderCh_BlowUpHalbraumObenLemmaBC} and
\ref{SchauderCh_BlowUpHalbraumUntenLemmaBC},
we require a result due to Simon \cite{simon}. To formulate it, we recall the following definition.
\begin{definition}
    Let $\mathcal C$ be a class of sets. We say that a map $S:\mathcal C\rightarrow[0,\infty)$ is 
    \begin{enumerate}[(1)]
        \item \emph{monotone}, if $S(A)\leq S(B)$ whenever $A,B\in\mathcal C$ with $A\subset B$.
        \item \emph{subadditive}, if 
    $S(A)\leq\sum_{i=1}^N S(A_i)$
    whenever $A,A_1,...,A_N\in\mathcal C$ with $A\subset\bigcup_{i=1}^N A_i$.
    \end{enumerate}
\end{definition}
\begin{theorem}[Simon's Absorption Lemma]\label{SchauderCh_SimonsCoveringLemma}\ \\
Let $R>0$, $p_0\in\R^n\times\R$, $A\subset\R^n\times\R$ be convex, put $\Omega_\rho(p):=U_\rho(p)\cap A$ and let $S$ be a nonnegative, monotone and subadditive function on the class $\set{\Omega_\rho(p)\ |\ \rho>0\textrm{ and }p\in\R^n\times\R}$. Let $k\in\R$, $\nu\in(0,1]$, $\theta\in(0,1)$ and $E>0$. There exists $\delta(n,\theta)>0$ and a constant $C=C(n,k,\nu,\theta)$ with the following property:\\

If  for all $\Omega_\rho(y)\subset \Omega_R(p_0)$  with radius $\rho\leq\nu R$
$$\rho^k S(\Omega_{\theta\rho}(y))\leq \delta\rho^k S(\Omega_\rho(y))+E,\hspace{.5cm}\textrm{then}\hspace{.5cm}
R^k S(\Omega_{\theta R}(p_0))\leq C E.$$
\end{theorem}
\begin{tcolorbox}[colback=white!20!white,colframe=black!100!white,sharp corners, breakable]
The set $A$ should be thought of as $\R^n\times\R$, $R^{n-1}\times[0,\infty)\times(-\infty,0]$ etc. so that each $\Omega_\rho(p)$ is a truncated parabolic ball as in Definition \ref{SchauderCh_ParabolicBallDefnition}. 
\end{tcolorbox}
\footnotetext{Open means open in the relative topology.}
\begin{proof}
Let $p\in\R^n$ and $r_2>r_1>0$. We claim that there exists $N(n, \frac{r_2}{r_1})$ and points $q_1,...,q_N\in  \Omega_{r_2}(p)$ such that $\Omega_{r_2}(p)\subset\bigcup_{i=1}^N \Omega_{r_1}(q_i)$.\\

To see that this is the case, first, note that there are points $\tilde q_1,..., \tilde q_{\tilde N}\in \bar U_{r_2}(p)$ such that $U_{r_2}(p)\subset\bigcup_{i=1}^{\tilde N}U_{\frac{r_1}2}(\tilde q_i)$ where $\tilde N=\tilde N(n, \frac{r_2}{r_1})$. For each $1\leq i\leq \tilde N$ we do the following: If $\tilde q_i\in A$, we put $q_i:=\tilde q_i$. If $\tilde q_i\not\in A$ but $U_{\frac{r_1}2}(\tilde q_i)\cap A\neq\emptyset$ we select any $q_i\in U_{\frac{r_1}2}(\tilde q_i)\cap A$. Finally, if $U_{\frac{r_1}2}(\tilde q_i)\cap A=\emptyset$ we don't define a point $q_i$. Note that for any $1\leq i\leq \tilde N$ that falls into either the first or the second case, we have $U_{\frac{r_1}2}(\tilde q_i)\subset U_{r_1}(q_i)$.\\

This produces $N\leq \tilde N(n,\frac{r_2}{r_1})$ and points $q_1,..., q_N\in \Omega_{r_2}(p)$ such that $\Omega_{r_2}(p)\subset\bigcup_{i=1}^N\Omega_{r_1}(q_i)$.\\

Having finished this preparatory consideration, we now start the proof of the theorem.  Note that for all $p\in\R^n$ and $\rho\leq\nu R$ we have $S(\Omega_\rho(p))\leq S(\Omega_R(p_0))<\infty$ since $S$ is monotone and hence
$$Q:=\sup\set{\rho^k S(\Omega_{\theta\rho}(p))\ |\ \Omega_\rho(p)\subset \Omega_R(p_0)\textrm{ and }\rho\leq\nu R}<\infty.$$
Let $\delta>0$ and assume that for all $\rho\leq \nu R$ and $\Omega_\rho(p)\subset \Omega_R(p_0)$ we have 
\begin{equation}\label{SchauderCh_SimonCovering0}
    \rho^kS(\Omega_{\theta\rho}(p))\leq \delta\rho^kS(\Omega_\rho(p))+E.    
\end{equation}
Later in the proof, we will derive a suitable condition on the smallness of $\delta$. Given any $q\in \Omega_R(p_0)$ and $\rho\in(0,\nu R]$ such that $\Omega_\rho(q)\subset \Omega_R(p_0)$ we can apply Estimate \eqref{SchauderCh_SimonCovering0} with $\Omega_{\theta\rho}(q)$ and obtain 
$$\left(\theta\rho\right)^kS(\Omega_{\theta^2\rho}(q))
\leq \delta \left(\theta\rho\right)^kS(\Omega_{\theta\rho}(q))+E
= \delta\theta^k\rho^kS(\Omega_{\theta\rho}(q))+E
\leq \delta\theta^k Q+E.$$
The last step is justified by the definition of $Q$. Rearranging, we have proven that for any $\Omega_\rho(q)\subset \Omega_R(p_0)$ with $\rho\leq\nu R$
\begin{equation}\label{SchauderCh_SimonCovering1}
\rho^kS(\Omega_{\theta^2\rho}(q))\leq \delta Q+\theta^{-k}E.
\end{equation}
Now let $p\in \Omega_R(p_0)$ and $\rho\in(0,\nu R]$ such that $\Omega_{\rho}(p)\subset \Omega_R(p_0)$. By the preparatory consideration, we may cover $\Omega_{\theta\rho}(p)$ with balls $\Omega_{\theta^2\rho}(q_i)$ where $q_i\in \Omega_{\theta\rho}(p)$ and $1\leq i\leq N(n,\theta)$. By first using that $S$ is subadditive and then applying Estimate \eqref{SchauderCh_SimonCovering1}, we get
$$
\rho^k S(\Omega_{\theta\rho}(p))
\leq \sum_{i=1}^N\rho^kS(\Omega_{\theta^2\rho}(q_i))
\leq N(n,\theta)(\delta Q+ \theta^{-k}E).
$$
We now demand that $\delta N(n,\theta)\leq\frac12$. This imposes a constraint on $\delta$ that only depends on $n$ and $\theta$. Taking the supremum over all $\Omega_\rho(p)$ we get
\begin{equation}\label{SchauderCh_SimonLemmaMISSING}
Q\leq \frac12 Q+\frac{N(n,\theta)}{\theta^k}E\hspace{.5cm}\textrm{and hence}\hspace{.5cm} Q\leq \frac{2 N(n,\theta)}{\theta^k}E.
\end{equation}
By the preparatory consideration, there exists $M=M(n,\nu,\theta)\in\N$ and $\tilde q_1$,...., $\tilde q_M\in  \Omega_{\theta R}(p_0)$ such that $\bigcup_{i=1}^M \Omega_{\theta\nu R}(\tilde q_i)\supset \Omega_{\theta R}(p_0)$. Using the subadditivity of $S$ and Estimate \eqref{SchauderCh_SimonLemmaMISSING}, we get 
$$R^k S(\Omega_{\theta R}(p_0))
\leq \left(\frac R{\nu R}\right)^k\sum_{i=1}^M (\nu R)^k S(\Omega_{\theta\nu R}(\tilde q_i))
\leq \nu^{-k}M(n,\nu,\theta)Q
\leq  C(n,\nu,\theta)E.$$
\end{proof}

\subsection{Localized Estimates with Constant Coefficients}
Throughout this subsection, let $a^{ij}$ denote a strictly elliptic matrix and $\theta>0$, $\Lambda>0$ such that 
$$|a^{ij}|\leq\Lambda
\hspace{.5cm}
\textrm{and}
\hspace{.5cm}
a^{ij}\xi_i\xi_j\geq\theta|\xi|^2\textrm{ for all $\xi\in\R^n$.}
$$
Additionally, we define the following operators
$$
P:=\partial_t +a^{ij}a^{kl}\partial_{ijkl},\hspace{.5cm}
B_1:=a^{ni}\partial_i\hspace{.5cm}\textrm{and}\hspace{.5cm}
B_2:=a^{ni}a^{kl}\partial_{ikl}.
$$
\begin{lemma}
\label{SchauderCh_GanzraumLokalisierungFixeCoeff}
  Let $q\in\R^n\times\R$, $R>0$, $U_R:=U_R(q)$ and $u\in C^{4,1,\gamma}(U_R(q))$. Then 
 $$[D^{4,1}u]_{\gamma,U_{\frac R2}}^{(0)}\leq C(\Lambda,\theta,n,\gamma)\left([Pu]_{\gamma,U_R}^{(0)}+R^{-4-\gamma}\|u\|_{L^\infty(U_R)}\right).$$
\end{lemma}
\begin{proof}
Throughout this proof, we will suppress the dependence of various constants $C$ on $\Lambda$, $\theta$, $n$ and $\gamma$. Let $p\in U$ and $\rho>0$ such that $U_\rho(p)\subset U_R(q)$. Let $\eta\in C^\infty_0(U_\rho(p),[0,1])$ with $\eta\equiv 1$ on $U_{\frac\rho2}(p)$ such that for $j\in\N$ and $\alpha\in\N_0^n$ 
\begin{equation}\label{SchauderCh_CutOffLocaliteEstimates}
\|\partial_t^j\nabla_\alpha \eta\|_{L^\infty}\leq C\rho^{-4j-|\alpha|}
\hspace{.5cm}\textrm{and}\hspace{.5cm}
[\partial_t^j\nabla_\alpha \eta]_{\gamma,\R^n\times\R}^{(0)} \leq C\rho^{-4j-|\alpha|-\gamma}.
\end{equation}
Extending by zero outside of $U_\rho(p)$, we define $v:=\eta u\in C^{4,1,\gamma}(\R^n\times\R)$. By Lemma \ref{SchauderCh_BlowUpGanzraumLemma}, we have 
\begin{equation}\label{SchauderCh_ExtendedLemmaToLocalizeEquation}
[D^{4,1} v]_{\gamma,\R^n\times\R}^{(0)}\leq C[Pv]_{\gamma,\R^n\times\R}^{(0)}.
\end{equation}
To translate this back into an estimate for $u$, we compute 
\begin{equation}\label{SchauderCh_PvFormula}
Pv=\eta Pu+\dot\eta u+\sum_{|\alpha|\leq 3}\sum_{\substack{|\beta|\leq 4\\ |\alpha|+|\beta|=4}}c(\alpha,\beta)\nabla_\alpha u\nabla_\beta \eta.
\end{equation}
For functions $f$ and $g$ we have the estimate 
\begin{equation}\label{SchauderCh_LocalizationProductRuileEstimate}
[fg]_{\gamma,\R^n\times\R}^{(0)}\leq C\left(
\|f\|_{L^\infty(\R^n\times\R)}[g]_{\gamma,\R^n\times\R}^{(0)}
+
[f]_{\gamma,\R^n\times\R}^{(0)}\|g\|_{L^\infty(\R^n\times\R)}
\right).
\end{equation}
Using the formula for $Pv$ in Equation \eqref{SchauderCh_PvFormula} and Estimate \eqref{SchauderCh_CutOffLocaliteEstimates}, we deduce
\begin{align*}
    [Pv]_{\gamma,\R^n\times\R}^{(0)}
    \leq & C\bigg([ Pu]_{\gamma,U_\rho(p)}^{(0)}+\rho^{-\gamma}\|Pu\|_{L^\infty(U_\rho(p))}
            +\rho^{-4}[ u]_{\gamma,U_\rho(p)}^{(0)}+\rho^{-4-\gamma}\|u\|_{L^\infty(U_\rho(p))}\\
            &\hspace{2cm}+\sum_{|\alpha|\leq 3}\sum_{\substack{|\beta|\leq 4\\ |\alpha|+|\beta|=4}}
                    \rho^{-|\beta|}[\nabla_\alpha u]_{\gamma,U_\rho(p)}^{(0)} 
                      +    \rho^{-|\beta|-\gamma}\|\nabla_\alpha u\|_{L^\infty(U_\rho(p))} \bigg)\\
     &\hspace{-1.5cm}\leq  C\bigg([ Pu]_{\gamma,U_\rho(p)}^{(0)}+\rho^{-\gamma}\|\dot u\|_{L^\infty(U_\rho)}+\sum_{|\alpha|\leq 4}\rho^{-4+|\alpha|-\gamma}\|\nabla_\alpha u\|_{L^\infty(U_\rho)}+\sum_{|\alpha|\leq 3}\rho^{-4+|\alpha|}[\nabla_\alpha u]_{\gamma,U_\rho(p)}^{(0)}\bigg).
\end{align*}
Inserting into Estimate \eqref{SchauderCh_ExtendedLemmaToLocalizeEquation}, recalling that $v\equiv u$ on $U_{\frac\rho2}(p)$ and using the interpolation estimates from Theorem \ref{SchauderCh_InterpolationEstimates}, we obtain 
\begin{equation}\label{SchauderCh_LocalizedEstimateConstCoeffToSimon}
   \rho^{4+\gamma}[D^{4,1} u]_{\gamma,U_{\frac\rho2}(p)}^{(0)}\leq C(\epsilon)\bigg(\rho^{4+\gamma}[ Pu]_{\gamma,U_{\rho}(p)}^{(0)}+\|u\|_{L^\infty(U_\rho(p))}\bigg)+C\epsilon \rho^{4+\gamma}[D^{4,1} u]_{\gamma,U_{\rho(p)}}^{(0)}.
\end{equation}
This estimate is true for all $U_\rho(p)\subset U_R(q)$. We put $A:=\R^n\times\R$, $\theta_S:=\frac12$, $\nu:=1$, $k:=4+\gamma$ and 
$$E(\epsilon):= C(\epsilon)\bigg(R^{4+\gamma}[ Pu]_{\gamma,U_R(q)}^{(0)}+\|u\|_{L^\infty(U_R(q))}\bigg).$$
Finally, we define $S(U_\rho(p)):=[D^{4,1}u]_{\gamma,U_\rho(p)}^{(0)}$. By Proposition \ref{SchauderCh_SemiNormSubbadditiveonConvexSpaces_Parabolic}, $S$ is subadditive and Estimate \eqref{SchauderCh_LocalizedEstimateConstCoeffToSimon} becomes
$$\rho^kS(U_{\theta_S\rho}(p))\leq \rho^k\epsilon S(U_\rho(p))+E(\epsilon)\hspace{.5cm}\textrm{for all }U_\rho(p)\subset U_R(q).$$
We now fix $\epsilon=\epsilon(n,k):=\frac12\delta(n,k)$ from Simon's Absorption Lemma \ref{SchauderCh_SimonsCoveringLemma} and may deduce $R^kS(U_{\theta_SR}(0))\leq C(n,k,\theta_S,\nu)E(\epsilon(n,k))$. Recalling the definition of $S$, $E$, $\theta_S$ and $k$, we obtain 
$$R^{4+\gamma}[D^{4,1}u]_{\gamma,U_{\frac R2}(0)}^{(0)}\leq C\left(R^{4+\gamma}[D^{4,1}u]_{\gamma,U_R(q)}^{(0)}+\|u\|_{L^\infty(U_R(q))}\right),$$
where $C=C(n,\gamma,\Lambda,\theta)$. 
\end{proof}

\begin{lemma}\label{SchauderCh_HalbraumObenLokalisierungFixeCoeff}
 Let $q\in\R^n\times[0,\infty)$, $R>0$, $U_R^+:=U_R^+(q)$,  $u\in C^{4,1,\gamma}(U_R^+)$ and $u_0:=u(\cdot,0)$. Then 
 $$[D^{4,1}]^{(0)}_{\gamma,U_{\frac R2}^+(q)}\leq C(\Lambda,\theta,n,\gamma)\left([Pu]^{(0)}_{\gamma,U_R^+(q)}+[\nabla^4 u_0]_{\gamma, \partial_F U_R^+(q)}+R^{-4-\gamma}\|u\|_{L^\infty(U_R^+(q))}\right).$$
\end{lemma}
 \begin{tcolorbox}[colback=white!20!white,colframe=black!100!white,sharp corners, breakable]
If $\partial_F U_R^+(q)=\emptyset$, the estimate holds with $[\nabla^4 u_0]_{\gamma, \partial_F U_R^+(q)}=0$. Indeed, when $\partial_F U_R^+(q)=\emptyset$, we deduce $U_R(q)\subset\R^n\times(0,\infty)$ and we are in the position where Lemma \ref{SchauderCh_GanzraumLokalisierungFixeCoeff} applies. 
\end{tcolorbox}
\begin{proof}
Throughout this proof, we will suppress the dependence of various constants $C$ on $\Lambda$, $\theta$, $n$ and $\gamma$. Let $p\in U_R^+(q)$ and $\rho>0$ such that $U_\rho^+(p)\subset U_R^+(q)$. We wish to prove that for any $\epsilon>0$
 \begin{align}
 [D^{4,1}u]^{(0)}_{\gamma,U_{\frac \rho2}(p)^+}\leq & C(\epsilon)\left([Pu]_{\gamma,U_\rho^+(p)}^{(0)}+[\nabla^4 u_0]_{\gamma, \partial_F U_\rho^+(p)}+\rho^{-4-\gamma}\|u\|_{L^\infty(U_\rho^+)}\right)\nonumber\\
 &+C\epsilon[D^{4,1}u]^{(0)}_{\gamma,U_{ \rho}^+(p)}.\label{SchauderCh_LokhalbRaumOben01}
 \end{align}
Once this is established, the lemma follows from Simons's Absorption Lemma \ref{SchauderCh_SimonsCoveringLemma} with $A=\R^n\times[0,\infty)$ by the same arguments that we gave in the proof of Lemma \ref{SchauderCh_GanzraumLokalisierungFixeCoeff}. To establish Estimate \eqref{SchauderCh_LokhalbRaumOben01}, we essentially repeat the argument from Lemma \ref{SchauderCh_GanzraumLokalisierungFixeCoeff}. We choose a cutoff function $\eta\in C^\infty_0(U_\rho(p),[0,1])$ with $\eta\equiv 1$ on $U_{\frac\rho2}(p)$ that satisfies Estimate \eqref{SchauderCh_CutOffLocaliteEstimates} and define $v:=\eta u\in C^{4,1,\gamma}(\R^n\times[0,\infty))$ by extending with $0$ outside of $U_\rho^+(p)$. By Lemma \ref{SchauderCh_BlowUpHalbraumLemma} we have 
\begin{equation}\label{SchauderCh_LokhalbRaumOben03}
[D^{4,1}v]^{(0)}_{\gamma,\R^n\times[0,\infty)}\leq C\left([Pv]^{(0)}_{\gamma, \R^n\times[0,\infty)}+[\nabla^4 v(\cdot,0)]_{\gamma,\R^n}\right).
\end{equation}
Following the derivation of Estimate \eqref{SchauderCh_LocalizedEstimateConstCoeffToSimon}, we obtain 
\begin{equation}\label{SchauderCh_LokhalbRaumOben04}
[P v]^{(0)}_{\gamma, \R^n\times[0,\infty)}\leq C(\epsilon)\left([Pu]^{(0)}_{\gamma, U_\rho^+(p)}+\rho^{-4-\gamma}\|u\|_{L^\infty(U_\rho^+(p))}\right)+C\epsilon[D^{4,1}u]_{\gamma, U_\rho^+(p)}^{(0)}.
\end{equation}
Next, using the analogue of Estimate \eqref{SchauderCh_LocalizationProductRuileEstimate} for the seminorm $[\cdot]_{\gamma,\R^n}$, the formula 
$$\nabla^4 (\eta u_0)=\eta \nabla^4 u_0+\sum_{|\alpha|\leq 3}\sum_{\substack{|\beta|\leq 4\\ |\alpha|+|\beta|=4}}c(\alpha,\beta)\nabla_\beta\eta\nabla_\alpha u_0$$
and Estimate \eqref{SchauderCh_CutOffLocaliteEstimates}, we estimate 
\begin{align}
    [\nabla ^4 (\eta u_0)]_{\gamma,\R^n}\leq C\bigg(
     & [\nabla^4 u_0]_{\gamma, \partial_FU_\rho^+(p)}
    +
    \rho^{-\gamma} \|\nabla^4 u_0\|_{L^\infty(\partial_F U_\rho^+(p))}\nonumber\\
    &\hspace{-2cm}+\sum_{|\alpha|\leq 3}\sum_{\substack{|\beta|\leq 4\\ |\alpha|+|\beta|=4}}
    \rho^{-|\beta|}[\nabla_\alpha u_0]_{\gamma, \partial_FU_\rho^+(p)}
    +
    \rho^{-|\beta|-\gamma}\|\nabla_\alpha u_0\|_{L^\infty(\partial_F U_\rho^+(p))}\bigg).\label{SchauderCh_InitialvalueLocalFixCoefEst1}
\end{align}
Inserting the interpolation estimates from Theorem \ref{SchauderCh_InterpolationEstimates}, we get 
\begin{equation}\label{SchauderCh_LokhalbRaumOben05}
     [\nabla ^4 (\eta u_0)]_{\gamma, \R^n}\leq
     C(\epsilon)\left([\nabla^4 u_0]_{\gamma, \partial_F U_\rho^+(p)}+\rho^{-4-\gamma}\|u\|_{L^\infty(U_\rho^+(p))}
     \right)+C\epsilon[D^{4,1}u]_{\gamma, U_\rho^+(p)}^{(0)}.
\end{equation}
Inserting Estimates \eqref{SchauderCh_LokhalbRaumOben04} and \eqref{SchauderCh_LokhalbRaumOben05} into Estimate \eqref{SchauderCh_LokhalbRaumOben03} and using that $v\equiv u$ on $U_{\frac\rho2}^+(p)$, we obtain Estimate \eqref{SchauderCh_LokhalbRaumOben01} and thereby the lemma.
\end{proof}

\begin{lemma}\label{SchauderCh_HalbraumUntenLokalisierungFixeCoeff}
Let $q\in\R^n\times(-\infty,0]$, $R>0$, $U_R^-:=U_R^-(q)$ and  $u\in C^{4,1,\gamma}(U_R^-)$. Then
 $$[D^{4,1}u]^{(0)}_{\gamma,U_{\frac R2}^-}\leq C(\Lambda,\theta,n,\gamma)\left([Pu]^{(0)}_{\gamma,U_R^-}+R^{-4-\gamma}\|u\|_{L^\infty(U_R^-)}\right).$$
\end{lemma}
\begin{proof}
Throughout this proof, we will suppress the dependence of various constants $C$ on $\Lambda$, $\theta$, $n$ and $\gamma$. Let $p\in U_R^-(q)$ and $\rho>0$ such that $U_\rho^-(p)\subset U_R^-(q)$. Let $\eta\in C^\infty_0(U_\rho(0),[0,1])$ with $\eta\equiv 1$ on $U_{\frac\rho2}(0)$ that satisfies Estimate \eqref{SchauderCh_CutOffLocaliteEstimates}.
We define $v:=\eta u\in C^{4,1,\gamma}(\R^n\times(-\infty,0])$ by extending with $0$ outside of $U_\rho^-(p)$. Then by Lemma \ref{SchauderCh_BlowUpHalbRaumUntenLemma}
$$[D^{4,1} v]_{\gamma,\R^n\times(-\infty,0]}^{(0)}\leq C[Pv]_{\gamma,\R^n\times(-\infty,0]}^{(0)}.$$ 
By following the estimates in the proof of Lemma \ref{SchauderCh_GanzraumLokalisierungFixeCoeff}, we get that for all small $\epsilon>0$
 \begin{equation}\label{SchauderCh_LokhalbRaumUnten01}
 [D^{4,1}u]_{\gamma,U_{\frac \rho2}^-(p)}^{(0)}\leq C(\epsilon)\left([Pu]_{\gamma,U_\rho^-(p)}^{(0)}+\rho^{-4-\gamma}\|u\|_{L^\infty(U_\rho^-(p))}\right)+C\epsilon[D^{4,1}u]_{\gamma,U_{ \rho}^-(p)}^{(0)}.
 \end{equation}
Simons's Absorption Lemma \ref{SchauderCh_SimonsCoveringLemma} with $A=\R^n\times(-\infty,0]$ implies the lemma.
\end{proof}

\begin{lemma}\label{SchauderCh_HalbraumBCLokalisierungFixeCoeff}
 Let $q\in\R^{n-1}\times[0,\infty)\times\R$, $R>0$, $U_{R+}:=U_{R+}(q)$ and  $u\in C^{4,1,\gamma}(U_{R+})$. Then
 $$
 [D^{4,1}u]_{\gamma,U_{\frac R2+}}^{(0)}\leq C(\Lambda,\theta,n,\gamma)\left([Pu]_{\gamma,U_{R+}}^{(0)}
 +[B_1 u]^{(3)}_{\gamma,\partial_S U_{R+}}
 +[B_2 u]^{(1)}_{\gamma,\partial_S U_{R+}}
 +R^{-4-\gamma}\|u\|_{L^\infty(U_{R+})}\right).
 $$
\end{lemma}
 \begin{tcolorbox}[colback=white!20!white,colframe=black!100!white,sharp corners, breakable]
If $\partial_S U_{R+}=\emptyset$, the estimate holds with $[B_1 u]^{(3)}_{\gamma,\partial_S U_{R+}}
 =[B_2 u]^{(1)}_{\gamma,\partial_S U_{R+}}=0$.
 \end{tcolorbox}
\begin{proof}
Throughout this proof, we will suppress the dependence of various constants $C$ on $\Lambda$, $\theta$, $n$ and $\gamma$. Let $p\in U_{R+}(q)$ and $\rho>0$ such that $U_{\rho+}(p)\subset U_{R+}(q)$. Let $\eta\in C^\infty_0(U_\rho(p),[0,1])$ with $\eta\equiv 1$ on $U_{\frac\rho2}(p)$ that satisfies Estimate \eqref{SchauderCh_CutOffLocaliteEstimates}.
We define $v:=\eta u\in C^{4,1,\gamma}(\R^{n-1}\times[0,\infty)\times\R)$ by extending with zero outside of $U_{\rho+}(p)$. By Lemma \ref{SchauderCh_BlowUpGanzraumLemmaBC}
 \begin{equation}\label{SchauderCh_LokhalbRaumBC01}
 \hspace{-.1cm}[D^{4,1} v]_{\gamma,\R^{n-1}\times[0,\infty)\times\R}^{(0)}\leq C\left(
[Pv]_{\gamma,\R^{n-1}\times[0,\infty)\times\R}^{(0)}+[B_1 v]^{(3)}_{\gamma,\R^{n-1}\times0\times\R}
 +[B_2v]^{(1)}_{\gamma,\R^{n-1}\times0\times\R}
\right).
\end{equation}
Following the derivation of Estimate \eqref{SchauderCh_LocalizedEstimateConstCoeffToSimon}, we get that for all small $\epsilon>0$
 \begin{equation}\label{SchauderCh_LokhalbRaumBC02}
 [Pv]_{\gamma,\R^{n-1}\times[0,\infty)\times\R}^{(0)} \leq C(\epsilon)
 \left([Pu]^{(0)}_{\gamma, U_{\rho+}(p)}+\rho^{-4-\gamma}\|u\|_{L^\infty(U_{\rho+}(p))}\right)+C\epsilon[D^{4,1}u]_{\gamma,U_{\rho+}(p)}^{(0)}.
 \end{equation}
 To estimate $[B_1 v]^{(3)}_{\gamma,\R^{n-1}\times0\times\R}$, we note the formulas 
 \begin{align}
 [f]_{\gamma, \R^{n-1}\times0\times\R}^{(3)}&=[\nabla^3 f]_{\gamma,\R^{n-1}\times 0\times\R}^{\operatorname{space}}+\sum_{k=0}^3[\nabla^k f]_{\frac{3+\gamma-k}4,\R^{n-1}\times 0\times\R}^{\operatorname{time}},\label{SchauderCh_LokBCFicCoeffHelp01}\\
\nabla^k(B_1 v)&=\eta\nabla^k (B_1 u)+\sum_{|\alpha|\leq k}\sum_{\substack{|\beta|\leq k+1\\ |\alpha|+|\beta|=k+1}}c(\alpha,\beta, a^{ij})\nabla_\beta \eta\nabla_\alpha u\label{SchauderCh_LokBCFicCoeffHelp02}.
 \end{align}
We insert Equation \eqref{SchauderCh_LokBCFicCoeffHelp02} into Equation \eqref{SchauderCh_LokBCFicCoeffHelp01}. Using the analogue of Estimate \eqref{SchauderCh_LocalizationProductRuileEstimate} for the various Hölder seminorms in Equation \eqref{SchauderCh_LokBCFicCoeffHelp01} and Estimate \eqref{SchauderCh_CutOffLocaliteEstimates}, we estimate
\begin{align*}
    &[a^{ni}\partial_i v]^{(3)}_{\gamma,\R^{n-1}\times0\times\R}\\
    \leq &C\bigg( [\nabla^3(a^{ni}\partial_i  u)]_{\gamma,\partial_S U_{\rho+}(p)}^{\operatorname{space}}
    +\rho^{-\gamma}\|\nabla^4u\|_{L^\infty(U_{\rho+}(p))}\\
    &\hspace{1cm}+\sum_{|\alpha|\leq 3}\sum_{\substack{|\beta|\leq 4\\ |\alpha|+|\beta|= 4}}\rho^{-|\beta|-\gamma}\|\nabla_\alpha u\|_{L^\infty(U_{\rho+}(p))}+\rho^{-|\beta|}[\nabla_\alpha u]_{\gamma,U_{\rho+}(p)}^{\operatorname{space}}\\
    &\hspace{.5cm}+ \sum_{k=0}^3[\nabla^k(a^{ni}\partial_i  u)]_{\frac{3+\gamma-k}4,\partial_S U_{\rho+}(p)}^{\operatorname{time}}
    +\rho^{-3-\gamma+k}\|\nabla^{k+1}u\|_{L^\infty(U_{\rho+}(p))}\\ 
    &\hspace{1cm}+
    \sum_{k=0}^3\sum_{|\alpha|\leq k}\sum_{\substack{|\beta|\leq k+1\\ |\alpha|+|\beta|=k+1}}
    \rho^{-|\beta|-3-\gamma+k}\|\nabla_\alpha u\|_{L^\infty(U_{\rho+}(p))}
    +
    \rho^{-|\beta|}[\nabla_\alpha u]_{\frac{3+\gamma-k}4,\partial_S U_{\rho+}(p)}^{\operatorname{time}}\bigg).
    \end{align*}
Using the definition of the $[\cdot]^{(3)}_{\gamma,\R^{n-1}\times0\times\R}$ seminorm and collecting terms, we obtain
    \begin{align*}
   &[a^{ni}\partial_i v]^{(3)}_{\gamma,\R^{n-1}\times0\times\R}\\
    \leq &C\bigg(
    [a^{ni}\partial_i u]_{\gamma,\partial_SU_{\rho+}(p)}^{(3)}+\sum_{k=0}^4\rho^{-4-\gamma+k}\|\nabla^{k}u\|_{L^\infty(U_{\rho+}(p))}\\
    &\hspace{.5cm}+\sum_{|\alpha|\leq 3}\sum_{\substack{|\beta|\leq 4\\ |\alpha|+|\beta|= 4}}
    \rho^{-|\beta|}[\nabla_\alpha u]_{\gamma,U_{\rho+}(p)}^{\operatorname{space}}
    +
    \sum_{k=0}^3\sum_{|\alpha|\leq k}\sum_{\substack{|\beta|\leq k+1\\ |\alpha|+|\beta|=k+1}}
    \rho^{-|\beta|}[\nabla_\alpha u]_{\frac{3+\gamma-k}4, U_{\rho+}(p)}^{\operatorname{time}}\bigg).
\end{align*}
We insert the interpolation estimates from Theorem \ref{SchauderCh_InterpolationEstimates} and Lemma \ref{SchauderCh_TemporalInterpolationLemma}. This shows that for all small $\epsilon$
 \begin{align}
[a^{ni}\partial_i v]^{(3)}_{\gamma,\R^{n-1}\times0\times\R}
    \leq &C(\epsilon)\left([a^{ni}\partial_i u]_{\gamma, \partial_SU_{\rho+}(p)}^{(3)}+\rho^{-4-\gamma}\|u\|_{L^\infty(U_{\rho+}(p))}\right)\nonumber\\
    &+C\epsilon[D^{4,1}u]_{\gamma,U_{\rho+}(p)}^{(0)}.\label{SchauderCh_LokhalbRaumBC03}
 \end{align}
 To estimate $[B_2 v]^{(1)}_{\gamma,\R^{n-1}\times0\times\R}$, we essentially repeat the same estimates. First, we note the formulas 
 \begin{align}
 [f]_{\gamma, \R^{n-1}\times0\times\R}^{(1)}&=[\nabla f]_{\gamma,\R^{n-1}\times 0\times\R}^{\operatorname{space}}+\sum_{k=0}^1[\nabla^k f]_{\frac{1+\gamma-k}4,\R^{n-1}\times 0\times\R}^{\operatorname{time}},\label{SchauderCh_LokBCFicCoeffHelp03}\\
\nabla^k(B_2 v)&=\eta\nabla^k (B_2 u)+\sum_{|\alpha|\leq k+2}\sum_{\substack{|\beta|\leq k+3\\ |\alpha|+|\beta|=k+3}}c(\alpha,\beta, a^{ij})\nabla_\beta \eta\nabla_\alpha u\label{SchauderCh_LokBCFicCoeffHelp04}.
 \end{align}
We insert Equation \eqref{SchauderCh_LokBCFicCoeffHelp04} into Equation \eqref{SchauderCh_LokBCFicCoeffHelp03} and use the analog of Estimate \eqref{SchauderCh_LocalizationProductRuileEstimate} for the various Hölder seminorms in Equation \eqref{SchauderCh_LokBCFicCoeffHelp03} and Estimate \eqref{SchauderCh_CutOffLocaliteEstimates}. Afterwards, we insert the interpolation estimates from Theorem \ref{SchauderCh_InterpolationEstimates} and Lemma \ref{SchauderCh_TemporalInterpolationLemma} and eventually obtain that for all small $\epsilon$
 \begin{align}
 [a^{ni}a^{kl}\partial_{ikl} v]^{(1)}_{\gamma,\R^{n-1}\times0\times\R}
 \leq &C(\epsilon)\left([a^{ni}a^{kl}\partial_{ikl} u]^{(1)}_{\gamma,\partial_SU_{\rho+}(p)}+\rho^{-4-\gamma}\|u\|_{L^\infty(U_{\rho+}(p))}\right)\nonumber\\
 &+C\epsilon[D^{4,1}u]_{\gamma,U_{\rho+}(p)}^{(0)}.\label{SchauderCh_LokhalbRaumBC04}
 \end{align}
 Inserting Estimates \eqref{SchauderCh_LokhalbRaumBC02}, \eqref{SchauderCh_LokhalbRaumBC03} and \eqref{SchauderCh_LokhalbRaumBC04} into Estimate \eqref{SchauderCh_LokhalbRaumBC01} and using that $v\equiv u$ on $U_{\frac\rho2+}(p)$ we get that for all small $\epsilon$
\begin{align*}
     [D^{4,1}u]_{\gamma, U_{\frac\rho2+}(p)}^{(0)}\leq &C(\epsilon)\left(
     [Pu]_{\gamma, U_{\rho+}(p)}^{(0)}
     +[B_1u]_{\gamma,\partial_S U_{\rho+}(p)}^{(3)}
     +[B_2u]_{\gamma,\partial_S U_{\rho+}(p)}^{(1)}
     +\rho^{-4-\gamma}\|u\|_{L^\infty(U_{\rho+}(p))}
     \right)\\
     &+C\epsilon [D^{4,1}u]_{\gamma, U_{\rho+}(p)}^{(0)}.
\end{align*}
 Using the same arguments as in the proof of Lemma \ref{SchauderCh_GanzraumLokalisierungFixeCoeff}, the lemma follows by Simon's Absorption Lemma \ref{SchauderCh_SimonsCoveringLemma} with $A=\R^{n-1}\times[0,\infty)\times\R$.
\end{proof}

\begin{lemma}\label{SchauderCh_UntenHalbraumBCLokalisierungFixeCoeff}
Let $q\in\R^{n-1}\times[0,\infty)\times(-\infty,0]$, $R>0$, $U_{R+}^-:=U_{R+}^-(q)$ and $u\in C^{4,1,\gamma}(U_{R+}^-)$. Then
 $$[D^{4,1}u]_{\gamma,U_{\frac R2+}^-}^{(0)}\leq C(\Lambda,\theta,n,\gamma)\left([Pu]^{(0)}_{\gamma,U_{R+}^ -}
 +[B_1 u]^{(3)}_{\gamma,\partial_S U_{R+}^ -}
 +[B_2 u]^{(1)}_{\gamma,\partial_S U_{R+}^ -}
 +R^{-4-\gamma}\|u\|_{L^\infty(U_{R+}^ -)}\right).$$
\end{lemma}
 \begin{tcolorbox}[colback=white!20!white,colframe=black!100!white,sharp corners, breakable]
If $\partial_S U_{R+}^-=\emptyset$, the estimate holds with $[B_1 u]^{(3)}_{\gamma,\partial_S U_{R+}^-}
 =[B_2 u]^{(1)}_{\gamma,\partial_S U_{R+}^-}=0$.
 \end{tcolorbox}
\begin{proof}
Throughout this proof, we will suppress the dependence of various constants $C$ on $\Lambda$, $\theta$, $n$ and $\gamma$. Let $p\in U_{R+}^ -(q)$ and $\rho>0$ such that $U_{\rho+}^ -(p)\subset U_{R+}^-(q)$. Let $\eta\in C^\infty_0(U_\rho(p),[0,1])$ with $\eta\equiv 1$ on $U_{\frac\rho2}(p)$ that satisfies Estimate \eqref{SchauderCh_CutOffLocaliteEstimates}.
We define $v:=\eta u\in C^{4,1,\gamma}(\R^{n-1}\times[0,\infty)\times(-\infty,0])$ by extending with $0$ outside of $U_{\rho+}^-(p)$. Then by Lemma \ref{SchauderCh_BlowUpHalbraumUntenLemmaBC}
 \begin{align}
 [D^{4,1} v]_{\gamma,\R^{n-1}\times[0,\infty)\times(-\infty,0]}^{(0)}\leq
 C\bigg(
&[Pv]_{\gamma,\R^{n-1}\times[0,\infty)\times(-\infty,0]}^{(0)}+[B_1 v]^{(3)}_{\gamma,\R^{n-1}\times0\times(-\infty,0]}\nonumber\\
& \hspace{.5cm}+[B_2 v]^{(1)}_{\gamma,\R^{n-1}\times0\times(-\infty,0]}
\bigg).\label{SchauderCh_LokUntenhalbRaumBC01}
\end{align}
By following the estimates from Lemma \ref{SchauderCh_GanzraumLokalisierungFixeCoeff}, the estimates that lead to estimate \eqref{SchauderCh_LokhalbRaumBC03} and the estimates that lead to Estimate \eqref{SchauderCh_LokhalbRaumBC04}, we deduce that for small $\epsilon>0$
 \begin{align}
   &\hspace{-.51cm}[Pv]^{(0)}_{\gamma,\R^{n-1}\times[0,\infty)\times(-\infty,0]}
  \leq 
  C(\epsilon)\left([Pu]_{\gamma,U_{\rho+}^ -(p)}^{(0)}+\rho^{-4-\gamma}\|u\|_{L^\infty(U_{\rho+}^ -(p))}\right)+C\epsilon[D^{4,1}u]^{(0)}_{\gamma,U_{\rho+}^ -(p)},\label{SchauderCh_LokFixCeoffThree01}\\
 &\hspace{-.51cm}[B_1 v]^{(3)}_{\gamma,\R^{n-1}\times0\times(-\infty,0]}
    \leq 
    C(\epsilon)\left([B_1 u]^{(3)}_{\gamma, U_{\rho+}^ -(p)}+\rho^{-4-\gamma}\|u\|_{L^\infty(U_{\rho+}^ -(p))}\right)+C\epsilon[D^{4,1}u]^{(0)}_{\gamma,U_{\rho+}^ -(p)},\label{SchauderCh_LokFixCeoffThree02}\\
&\hspace{-.51cm} [B_2 v]^{(1)}_{\gamma,\R^{n-1}\times0\times(-\infty,0]}
 \leq 
 C(\epsilon)\left([B_2 u]^{(1)}_{\gamma,U_{\rho+}^ -(p)}+\rho^{-4-\gamma}\|u\|_{L^\infty(U_{\rho+}^ -(p))}\right)+C\epsilon[D^{4,1}u]^{(0)}_{\gamma,U_{\rho+}^ -(p)}.\label{SchauderCh_LokFixCeoffThree03}
 \end{align}
 The lemma follows by inserting these Estimates \eqref{SchauderCh_LokFixCeoffThree01}, \eqref{SchauderCh_LokFixCeoffThree02} and \eqref{SchauderCh_LokFixCeoffThree03} into Estimate \eqref{SchauderCh_LokUntenhalbRaumBC01}, using $v\equiv u$ on $U_{\frac\rho2+}^-(p)$ and applying Simon's Absorption Lemma \ref{SchauderCh_SimonsCoveringLemma} with $A=\R^{n-1}\times[0,\infty)\times(-\infty,0]$.
\end{proof}

\begin{lemma}\label{SchauderCh_ObenHalbraumBCLokalisierungFixeCoeff}
Let $q\in\R^{n-1}\times[0,\infty)\times[0,\infty)$, $R>0$, $U_{R+}^+:=U_{R+}^+(q)$, $u\in C^{4,1,\gamma}(U_{R+}^+)$ and $u_0:=u(\cdot,0)$. Then
 \begin{align*} 
 \hspace{-.5cm}[D^{4,1}u]^{(0)}_{\gamma,U^+_{\frac R2+}}\leq C(\Lambda,\theta,n,\gamma)&\bigg([Pu]_{\gamma,U_{R+}^ +(0)}
 +[B_1 u]^{(3)}_{\gamma,\partial_S U_{R+}^ +}
 +[B_2 u]^{(1)}_{\gamma,\partial_S U_{R+}^ +}\\
 &\hspace{1cm}+[\nabla^4 u_0]_{\gamma,\partial_F U_{R+}^ +}
 +R^{-4-\gamma}\|u\|_{L^\infty(U_{R+}^+)}\bigg).
 \end{align*}
\end{lemma}
 \begin{tcolorbox}[colback=white!20!white,colframe=black!100!white,sharp corners, breakable]
If $\partial_S U_{R+}^+=\emptyset$ or $\partial_F U_{R+}^+=\emptyset$, the estimate holds with $[B_1 u]^{(3)}_{\gamma,\partial_S U_{R+}^+}
 =[B_2 u]^{(1)}_{\gamma,\partial_S U_{R+}^+}=0$ or $[\nabla^4 u_0]_{\gamma,\partial_F U_{R+}^ +}=0$ respectively.
 \end{tcolorbox}
\begin{proof}
Throughout this proof, we will suppress the dependence of various constants $C$ on $\Lambda$, $\theta$, $n$ and $\gamma$. Let $p\in U_{R+}^ +(q)$ and $\rho>0$ such that $U_{\rho+}^+ (p)\subset U_{R+}^+(q)$. Let $\eta\in C^\infty_0(U_\rho(p),[0,1])$ with $\eta\equiv 1$ on $U_{\frac\rho2+}^+(p)$ that satisfies Estimate \eqref{SchauderCh_CutOffLocaliteEstimates}. We define $v:=\eta u\in C^{4,1,\gamma}(\R^{n-1}\times[0,\infty)\times[0,\infty))$ by extending with zero outside of $U_{\rho+}^+(p)$. By Lemma \ref{SchauderCh_BlowUpHalbraumObenLemmaBC}
 \begin{align}
 [D^{4,1} v]_{\gamma,\R^{n-1}\times[0,\infty)\times[0,\infty)}^{(0)}\leq &C\bigg(
[Pv]_{\gamma,\R^{n-1}\times[0,\infty)\times[0,\infty)}^{(0)}
+[B_1 v]^{(3)}_{\gamma,\R^{n-1}\times0\times[0,\infty)}
\nonumber\\
 &\hspace{1cm} +[B_2 v]^{(1)}_{\gamma,\R^{n-1}\times0\times[0,\infty)}
 +[\nabla^4 v(\cdot,0)]_{\gamma,\R^{n-1}\times[0,\infty)}\bigg).\label{SchauderCh_LokObenhalbRaumBC01}
\end{align}
As in the proof of Lemma \ref{SchauderCh_UntenHalbraumBCLokalisierungFixeCoeff}, we may deduce the analog of Estimates \eqref{SchauderCh_LokFixCeoffThree01}, \eqref{SchauderCh_LokFixCeoffThree02} and \eqref{SchauderCh_LokFixCeoffThree03}. Additionally, following the estimates that lead to Estimate \eqref{SchauderCh_LokhalbRaumOben05}, we may derive the analog estimate on $U_{\rho+}^+(p)$.
Inserting these into Estimate \eqref{SchauderCh_LokObenhalbRaumBC01}, the lemma follows by Simon's Absorption Lemma \ref{SchauderCh_SimonsCoveringLemma} with $A=\R^{n-1}\times[0,\infty)\times[0,\infty)$.
\end{proof}

\subsection{Localized Estimates with Variable Coefficients}
Throughout this subsection let $a^{ij}\in C^{0,0,\gamma}(\Omega_R(q))$, where depending on the situation on $\Omega_R(q)=U_R(q)$, $U_R^{\pm}(q)$, $U_{R+}^\pm(q)$ for suitable $q\in\R^n\times\R$ and $R>0$.  We assume that for some $\theta>0$ 
$$
a^{ij}(x,t)\xi_i\xi_j\geq\theta|\xi|^2\hspace{.5cm}\textrm{for all $(x,t)\in \Omega_R(q)$ and $\xi\in\R^n$}.
$$
For a multiindex $\alpha\in\N_0^n$ of length $|\alpha|=4$, we define $A^\alpha$ by $\sum_{|\alpha|=4}A^\alpha\nabla_\alpha=a^{ij}a^{kl}\partial_{ijkl}$. Additionally, for $|\alpha|\leq 3$ let $A^\alpha\in C^{0,0,\gamma}(\Omega_R(q))$ be given functions. We assume that for some $M>0$ and all $|\alpha|\leq 4$
\begin{equation}\label{SchauderCh_VariableCoeff_Assumption01}
\|A^\alpha\|_{L^\infty(\Omega_R(q))}\leq M R^{-4+|\alpha|}
\hspace{.5cm}\textrm{and}\hspace{.5cm}
[A^\alpha]_{\gamma,\Omega_R(q)}^{(0)}\leq MR^{-4-\gamma+|\alpha|}.
\end{equation}
In the situations, where we $\Omega_\rho(q)$ has a spatial boundary --  that is for $U_{R+}(q)$ and $U_{R+}^\pm(q)$, we also require the following: For a multiindex $\alpha\in\N_0^n$ of length $|\alpha|=3$, we define $b^\alpha$ by $\sum_{|\alpha|= 3}b^\alpha\nabla_\alpha:=a^{ni}a^{kl}\partial_{ikl}$ and assume that for $|\alpha|\leq 2$, we are given functions $b^\alpha\in C^{1,0,\gamma}(\partial_S \Omega_{R}(q))$. We further assume that for all $|\alpha|\leq 3$, $k\in\set{0,1}$ and $k\leq l\leq 1$
\begin{align}
&\|\nabla ^k b^\alpha\|_{L^\infty(\partial_S \Omega_{R}(q))}\leq M R^{-3-k+|\alpha|},\hspace{.5cm}[\nabla ^kb^\alpha]_{\gamma, \partial_S \Omega_{R}(q)}^{(0)}\leq M R^{-3-k-\gamma+|\alpha|},\label{SchauderCh_VariableCoeff_Assumption02}\\
&\textrm{and}\hspace{.5cm}[\nabla^k b^\alpha]_{\frac{1-l+\gamma}4,\partial_S \Omega_{R}(q)}^{\operatorname{time}}\leq MR^{-4-\gamma-k+l+|\alpha|}.\label{SchauderCh_VariableCoeff_Assumption03}
\end{align}
Similarly,  for a multiindex $\alpha\in\N_0^n$ of length $|\alpha|=1$, we define $c^\alpha$ by $\sum_{|\alpha|=1}c^\alpha\nabla_\alpha:=a^{ni}\partial_{i}$ and assume that for $|\alpha|<1$ (that is $|\alpha|=0$), we are given functions (a function) $c^\alpha\in C^{3,0,\gamma}(\partial_S \Omega_{R}(q))$. We further assume that for $|\alpha|\leq 1$, $k\in\set{0,1,2,3}$ and $k\leq l\leq 3$
\begin{align}
&\|\nabla ^k c^\alpha\|_{L^\infty(\partial_S\Omega_R(q))}\leq M R^{-1-k+|\alpha|} ,\hspace{.5cm}
[\nabla ^kc^\alpha]_{\gamma, \partial_S\Omega_R(q)}^{(0)}\leq MR^{-1-k-\gamma+|\alpha|}\label{SchauderCh_VariableCoeff_Assumption04},\\
&\textrm{and}\hspace{.5cm}[\nabla^k c^\alpha]_{\frac{3-l+\gamma}4,\partial_S\Omega_R(q)}^{\operatorname{time}}\leq MR^{-4-\gamma-k+l+|\alpha|}.\label{SchauderCh_VariableCoeff_Assumption05}
\end{align}

Finally, we define the operators 
$$
P:=\partial_t+A^\alpha\nabla_\alpha,\hspace{.5cm}
B_1:=c^\beta\nabla_\beta\hspace{.5cm}\textrm{and}\hspace{.5cm}
B_2:=b^\alpha\nabla_\alpha.
$$

\begin{bemerkung}
Assumptions \eqref{SchauderCh_VariableCoeff_Assumption01}-\eqref{SchauderCh_VariableCoeff_Assumption05} should rather be seen as the definition of $M$. Indeed, since we assume $A^\alpha\in C^{0,0,\gamma}(\Omega_R(q))$, $b^\alpha\in C^{1,0,\gamma}(\partial_S\Omega_R(q))$ and $c^\alpha\in C^{3,0,\gamma}(\partial_S\Omega_R(q))$, all left hand sides in Assumptions \eqref{SchauderCh_VariableCoeff_Assumption01}-\eqref{SchauderCh_VariableCoeff_Assumption05} are finite and hence \eqref{SchauderCh_VariableCoeff_Assumption01}-\eqref{SchauderCh_VariableCoeff_Assumption05} are always satisfied for some $M$.
\end{bemerkung}

\begin{lemma}\label{SchauderCh_GanzraumLokalisierungVariableCoeff}
Let $q\in\R^n\times\R$, $R>0$, $U_R:=U_R(q)$ and $u\in C^{4,1,\gamma}(U_R)$. There exists a constant $C=C(M,\gamma, n,\theta)$ such that  
\begin{align*}
    [D^{4,1}u]_{\gamma,U_{\frac R2}}^{(0)}&\leq C\left([Pu]_{\gamma,U_R}^{(0)}+R^{-4-\gamma}\|u\|_{L^\infty(U_R)}\right),\\
    \|u\|_{C^{4,1,\gamma}(U_{\frac R2})}&\leq C(R)\left([Pu]_{\gamma,U_R}^{(0)}+\|u\|_{L^\infty(U_R)}\right).
\end{align*}

\end{lemma}
\begin{proof}
Throughout the proof, we will suppress the dependence of various constants $C$ on $M$, $\gamma$, $n$ and $\theta$. Let $\nu\in(0,1]$, $p_0=(x_0,t_0)\in U_R(q)$ and $\rho\in(0,\nu R)$ such that $U_\rho(p_0)\subset U_R(q)$. We define the operator $P_0:=\partial_t+\sum_{|\alpha|= 4} A^\alpha(p_0)\nabla_\alpha$.
Using Lemma \ref{SchauderCh_GanzraumLokalisierungFixeCoeff}, we have 
\begin{align} 
[D^{4,1}u]_{\gamma,U_{\frac\rho2}(p_0)}^{(0)}&\leq C\left([P_0 u]^{(0)}_{\gamma,U_\rho(p_0)}+\rho^{-4-\gamma}\|u\|_{L^\infty(U_\rho(p_0))}\right)\nonumber\\
&\leq C\left([P u]^{(0)}_{\gamma,U_\rho(p_0)}+[(P-P_0) u]_{\gamma,U_\rho(p_0)}^{(0)}+\rho^{-4-\gamma}\|u\|_{L^\infty(U_\rho(p_0))}\right).\label{SchauderCh_GanzraumLokVarCoeff1}
\end{align}
We derive a suitable estimate for $[(P-P_0) u]_{\gamma,U_\rho(p_0)}^{(0)}$. To do so, we note
\begin{equation}\label{SchauderCh_LoaklGanzRaumMissingEst00}
(P-P_0) u=\sum_{|\alpha|= 4}(A^\alpha-A^\alpha(p_0))\nabla_\alpha u+\sum_{|\alpha|\leq 3}A^\alpha\nabla_\alpha u.
\end{equation}
For a multiindex $\alpha$ of length $|\alpha|\leq 3$, we use Estimate \eqref{SchauderCh_LocalizationProductRuileEstimate} and Assumption \eqref{SchauderCh_VariableCoeff_Assumption01} to estimate 
\begin{align}
    [A^\alpha\nabla_\alpha u]_{\gamma,U_\rho(p_0)}^{(0)}
    \leq & C\left(
    \|A^\alpha\|_{L^\infty(U_\rho(p_0))}[\nabla_\alpha u]_{\gamma,U_\rho(p_0)}^{(0)}
    +
    [A^\alpha]_{\gamma,U_\rho(p_0)}^{(0)}\|\nabla_\alpha u\|_{L^\infty(U_\rho(p_0))}\right)\nonumber\\
    \leq & CM\left(R^{-4+|\alpha|} [\nabla_\alpha u]_{\gamma,U_\rho(p_0)}^{(0)}
    +R^{-4-\gamma+|\alpha|}\|\nabla_\alpha u\|_{L^\infty(U_\rho(p_0))}\right).\label{SchauderCh_LoaklGanzRaumMissingEst01}
\end{align}
Similarly, if $|\alpha|=4$, we estimate 
\begin{align}
    [(A^\alpha-A^\alpha(p_0))\nabla_\alpha u]_{\gamma,U_\rho(p_0)}^{(0)}
    \leq & C\left(
  \|A^\alpha-A^\alpha(p_0)\|_{L^\infty(U_\rho(p_0))}  [\nabla_\alpha u]_{\gamma,U_\rho(p_0)}^{(0)}
    +
    [A^\alpha]_{\gamma,U_\rho(p_0)}^{(0)}\|\nabla_\alpha u\|_{L^\infty(U_\rho(p_0))}\right)\nonumber\\
    \leq & C\left( [A^\alpha]_{\gamma, U_R}^{(0)}\rho^\gamma [\nabla_\alpha u]_{\gamma,U_\rho(p_0)}^{(0)}
    +
    MR^{-\gamma}\|\nabla_\alpha u\|_{L^\infty(U_\rho(p_0))}\right)\nonumber\\
    \leq & CM\left( R^{-\gamma}\rho^\gamma [D^{4,1}u]_{\gamma,U_\rho(p_0)}^{(0)}
    +
    R^{-\gamma}\|\nabla_\alpha u\|_{L^\infty(U_\rho(p_0))}\right).\label{SchauderCh_LoaklGanzRaumMissingEst02}
\end{align}
Considering Equation \eqref{SchauderCh_LoaklGanzRaumMissingEst00}, we use Estimates \eqref{SchauderCh_LoaklGanzRaumMissingEst01} and \eqref{SchauderCh_LoaklGanzRaumMissingEst02} and the interpolations estimates from Theorem \ref{SchauderCh_InterpolationEstimates}, to deduce
\begin{align*}
     &[(P-P_0) u]_{\gamma,U_\rho(p_0)}^{(0)} \nonumber\\
     \leq &C\bigg[MR^{-\gamma}\rho^\gamma [D^{4,1}u]_{\gamma,U_\rho(p_0)}^{(0)}
     +
     \sum_{|\alpha|=4}MR^{-\gamma}\rho^{-4}\left(\epsilon \rho^{4+\gamma}[D^{4,1}u]_{\gamma,U_\rho(p_0)}^{(0)}+C(\epsilon)\|u\|_{L^\infty(U_\rho(p_0))}\right)
      \\
      & +\sum_{|\alpha|\leq 3} MR^{-4+|\alpha|}\rho^{-|\alpha|-\gamma} \left(\epsilon \rho^{4+\gamma}[D^{4,1}u]_{\gamma,U_\rho(p_0)}^{(0)}+C(\epsilon)\|u\|_{L^\infty(U_\rho(p_0))}\right)\\
        & + \sum_{|\alpha|\leq 3}MR^{-4-\gamma+|\alpha|}\rho^{-|\alpha|}\left(\epsilon \rho^{4+\gamma}[D^{4,1}u]_{\gamma,U_\rho(p_0)}^{(0)}+C(\epsilon)\|u\|_{L^\infty(U_\rho(p_0))}\right)        
        \bigg].
\end{align*}
To estimate the first term, we use $\rho\leq\nu R$ to estimate $R^{-\gamma}\rho^\gamma\leq\nu^\gamma$. In all other terms, we estimate the terms $R^{-\kappa}$ with $\kappa\geq 0$ by $R^{-\kappa}\leq\rho^{-\kappa}$. Hence
\begin{equation}\label{SchauderCh_ganzraumVariableCoeff_100}
[(P-P_0) u]_{\gamma,U_\rho(p_0)}^{(0)} \leq CM(\epsilon+\nu^\gamma)[D^{4,1}u]_{\gamma,U_\rho(p_0)}^{(0)}+C(\epsilon)\rho^{-4-\gamma}\|u\|_{L^\infty(U_\rho(p_0))}.
\end{equation}
Inserting this estimate into Estimate \eqref{SchauderCh_GanzraumLokVarCoeff1} we get that for all small $\epsilon$ and $\nu$ 
\begin{align}
\rho^{4+\gamma}[D^{4,1}u]_{\gamma,U_{\frac\rho2}(p_0)}^{(0)}
\leq &
C(\epsilon,\nu)\left(\rho^{4+\gamma}[Pu]_{\gamma,U_\rho(p_0)}^{(0)}+\|u\|_{L^\infty(U_\rho(p_0))}\right)\nonumber\\
&+
CM(\epsilon+\nu^\gamma)\rho^{4+\gamma}[D^{4,1}u]^{(0)}_{U_{\rho}(p_0)}.\label{SchauderCh_LokGanzraumMissing04}
\end{align}
Estimate \eqref{SchauderCh_LokGanzraumMissing04} is analogue to Estimate \eqref{SchauderCh_LocalizedEstimateConstCoeffToSimon}. Choosing $\epsilon$ and $\nu$ small enough, we may apply Simon's Absorption Lemma\footnote{This is possible since the parameter $\delta$ from Simon's Absorption Lemma does not depend on $\nu$.} and deduce the first estimate claimed in the lemma. 
The second estimate is a direct consequence of the interpolation estimates in Theorem \ref{SchauderCh_InterpolationEstimates}
\end{proof}

\begin{lemma}\label{SchauderCh_HalbraumObenLokalisierungVariableCoeff}
Let $R>0$, $q\in\R^n\times[0,\infty)$, $U_R^+:=U_R^+(q)$, $u\in C^{4,1,\gamma}(U_R^+)$ and $u_0:=u(\cdot,0)$. There exists a constant $C=C(M,\gamma, n,\theta)$ such that 
\begin{align*}
[D^{4,1}u]_{\gamma,U^+_{\frac R2}}^{(0)}&\leq C\left([Pu]_{\gamma,U_R^+}^{(0)}+[\nabla^4 u_0]_{\gamma,\partial_F U_R^+}+R^{-4-\gamma}\|u\|_{L^\infty(U_R^+)}\right),\\
\|u\|_{C^{4,1,\gamma}(U_{\frac R2})}&\leq C(R)\left([Pu]_{\gamma,U_R^+}^{(0)}+[\nabla^4 u_0]_{\gamma,\partial_F U_R^+}+\|u\|_{L^\infty(U_R^+)}\right).
\end{align*}
\end{lemma}
\begin{tcolorbox}[colback=white!20!white,colframe=black!100!white,sharp corners, breakable]
If $\partial_F U_R^+(q)=\emptyset$, the estimate holds with $[\nabla^4 u_0]_{\gamma, \partial_F U_R^+(q)}=0$.
\end{tcolorbox}
\begin{proof}
Throughout the proof, we will suppress the dependence of various constants $C$ on $M$, $\gamma$, $n$ and $\theta$. Let $\nu\in(0,1]$ $p_0=(x_0,t_0)\in U_R^+(q)$ and $\rho\in(0,\nu R)$ such that $U_\rho^+(p_0)\subset U_R^+(q)$. We establish that for small $\epsilon>0$ 
\begin{align}
[D^{4,1}u]^{(0)}_{\gamma,U_{\frac \rho2}^+(p_0)}\leq &C(\epsilon,\nu)\left([Pu]^{(0)}_{\gamma,U_{ \rho}^+(p_0)}
+[\nabla^4 u_0]_{\gamma,\partial_F U_\rho^+(p_0)}
+\rho^{-4-\gamma}\|u\|_{L^\infty(U_\rho^+(p_0))}\right)\nonumber\\
&+CM(\epsilon+\nu^\gamma)[D^{4,1}u]_{\gamma,U_\rho^+(p_0)}^{(0)}.\label{SchauderCh_HalbraumObenLokalisierungVariableCoeff01}
\end{align}
To prove Estimate \eqref{SchauderCh_HalbraumObenLokalisierungVariableCoeff01}, we essentially follow the proof of Lemma \ref{SchauderCh_GanzraumLokalisierungVariableCoeff}.
We define the operator $P_0:=\partial_t+\sum_{|\alpha|= 4} A^\alpha(p_0)\nabla_\alpha$. 
Using Lemma \ref{SchauderCh_HalbraumObenLokalisierungFixeCoeff}, we have 
\begin{align} 
[D^{4,1}u]^{(0)}_{\gamma,U_{\frac\rho2}^+(p_0)}&\leq C\left([P_0 u]^{(0)}_{\gamma, U_\rho^+(p_0)}+[\nabla^4 u_0]_{\gamma, \partial_F U_\rho^+(p_0)}+\rho^{-4-\gamma}\|u\|_{L^\infty(U_\rho^+(p_0))}\right)\nonumber\\
&\hspace{-2.4cm}\leq C\bigg([P u]^{(0)}_{\gamma, U_\rho^+(p_0)}+[(P-P_0) u]^{(0)}_{\gamma, U_\rho^+(p_0)}+[\nabla^4 u_0]_{\gamma, \partial_F U_\rho^+(p_0)}+\rho^{-4-\gamma}\|u\|_{L^\infty(U_\rho^+(p_0))}\bigg).\label{SchauderCh_halbraumLokVarCoeff1}
\end{align}
Following the estimates that lead to Estimate \eqref{SchauderCh_ganzraumVariableCoeff_100}, we can derive an analogue estimate $[(P-P_0) u]^{(0)}_{\gamma, U_\rho^+(p_0)}$ and deduce Estimate \eqref{SchauderCh_HalbraumObenLokalisierungVariableCoeff01}. The first estimate claimed in the lemma follows from Simon's Absorption Lemma \ref{SchauderCh_SimonsCoveringLemma} by the same argument as in the proof of Lemma \ref{SchauderCh_GanzraumLokalisierungVariableCoeff}. The second estimate is a direct consequence of the interpolation estimates in Theorem \ref{SchauderCh_InterpolationEstimates}
\end{proof}

\begin{lemma}\label{SchauderCh_HalbraumUntenLokalisierungVariableCoeff}
Let $q\in\R^n\times(-\infty,0]$, $R>0$, $U_R^-:=U_R^-(q)$ and $u\in C^{4,1,\gamma}(U_R^-)$. There exists a constant $C=C(M,\gamma, n,\theta)$ such that 
\begin{align*} 
[D^{4,1}u]^{(0)}_{\gamma,U_{\frac R2}^-}&\leq C\left([Pu]_{\gamma,U_R^-}^{(0)}+R^{-4-\gamma}\|u\|_{L^\infty(U_R^-)}\right),\\
\|u\|_{C^{4,1,\gamma}(U_{\frac R2}^-)}&\leq C(R)\left([Pu]_{\gamma,U_R^-}^{(0)}+\|u\|_{L^\infty(U_R^-)}\right).
\end{align*}
\end{lemma}
\begin{proof}
Throughout the proof, we will suppress the dependence of various constants $C$ on $M$, $\gamma$, $n$ and $\theta$. The proof follows the proof of Lemma \ref{SchauderCh_GanzraumLokalisierungVariableCoeff}.
 Let $\nu\in(0,1]$, $p_0=(x_0,t_0)\in U_R^-(q)$ and $\rho\in(0,\nu R)$ such that $U_\rho^-(p_0)\subset U_R^-(q)$. We define the operator $P_0:=\partial_t+\sum_{|\alpha|= 4} A^\alpha(p_0)\nabla_\alpha$.
Using Lemma \ref{SchauderCh_HalbraumUntenLokalisierungFixeCoeff}, we deduce
\begin{align} 
[D^{4,1}u]^{(0)}_{\gamma,U_{\frac\rho2}^-(p_0)}&\leq C\left([P_0 u]^{(0)}_{\gamma,U_\rho^-(p_0)}+\rho^{-4-\gamma}\|u\|_{L^\infty(U_\rho^-(p_0))}\right)\nonumber\\
&\leq C\left([P u]^{(0)}_{\gamma,U_\rho^-(p_0)}+[(P-P_0) u]^{(0)}_{\gamma,U_\rho^-(p_0)}+\rho^{-4-\gamma}\|u\|_{L^\infty(U^-_\rho(p_0))}\right).\nonumber
\end{align}
Following the estimates the lead to Estimate \eqref{SchauderCh_ganzraumVariableCoeff_100}, we can derive an analogue estimate for $[(P-P_0) u]^{(0)}_{\gamma,U_\rho^-(p_0)}$ and get
$$
[D^{4,1}u]^{(0)}_{\gamma,U_{\frac\rho2}^-(p_0)}
\leq 
C(\epsilon,\nu)\left([P u]^{(0)}_{\gamma,U_\rho^-(p_0)}+\rho^{-4-\gamma}\|u\|_{L^\infty(U_\rho^-(p_0))}\right)
+
CM(\epsilon+\nu^\gamma)[D^{4,1}u]^{(0)}_{\gamma,U_{\rho}^-(p_0)}.$$
As in the proof of Lemma \ref{SchauderCh_GanzraumLokalisierungVariableCoeff}, the first estimate claimed in the lemma follows from Simon's Absorption Lemma \ref{SchauderCh_SimonsCoveringLemma}. The second estimate is a direct consequence of the interpolation estimates in Theorem \ref{SchauderCh_InterpolationEstimates}
\end{proof}

\begin{lemma}\label{SchauderCh_HalbraumBCLokalisierungVariableCoeff}
Let $q\in\R^{n-1}\times[0,\infty)\times\R$, $R>0$, $U_{R+}:=U_{R+}(q)$ and $u\in C^{4,1,\gamma}(U_{R+})$. There exists a constant $C=C(M,\gamma, n,\theta)$ such that 
\begin{align*} 
[D^{4,1}u]^{(0)}_{\gamma, U_{\frac R2+}}\leq &C\left([Pu]^{(0)}_{\gamma, U_{R+}}+[B_1u]^{(3)}_{\gamma, \partial_S U_{R+}}+[B_2u]^{(1)}_{\gamma, \partial_S U_{R+}}+R^{-4-\gamma}\|u\|_{L^\infty(U_{R+})}\right),\\
\|u\|_{C^{4,1,\gamma}(U_{\frac R2+})}\leq &C(R)\left([Pu]^{(0)}_{\gamma, U_{R+}}+[B_1u]^{(3)}_{\gamma, \partial_S U_{R+}}+[B_2u]^{(1)}_{\gamma, \partial_S U_{R+}}+\|u\|_{L^\infty(U_{R+})}\right).
\end{align*}
\end{lemma}
\begin{tcolorbox}[colback=white!20!white,colframe=black!100!white,sharp corners, breakable]
If $\partial_S U_{R+}=\emptyset$, the estimate holds with $[B_1 u]^{(3)}_{\gamma,\partial_S U_{R+}}
 =[B_2 u]^{(1)}_{\gamma,\partial_S U_{R+}}=0$.
 \end{tcolorbox}
\begin{proof}
Throughout the proof, we will suppress the dependence of various constants $C$ on $M$, $\gamma$, $n$ and $\theta$. We follow the proof of Lemma \ref{SchauderCh_GanzraumLokalisierungVariableCoeff}. Let $\nu\in(0,1]$, $p_0=(x_0,t_0)\in U_{R+}(q)$ and $\rho\in(0,\nu R)$ such that $U_{\rho+}(p_0)\subset U_{R+}(q)$. We put $A^\alpha_0:=A^\alpha(p_0)$, $b^\alpha_0:=b^\alpha(p_0)$ and $c^\alpha_0:=c^\alpha(p_0)$ and consider the constant coefficient operators $P_0:=\sum_{|\alpha|=4} A^\alpha_0\nabla_\alpha$, $B_{1,0}:=\sum_{|\alpha|=1}c_0^\alpha\nabla_\alpha$ and $B_{2,0}:=\sum_{|\alpha|= 3}b_0^\alpha\nabla_\alpha$. By Lemma \ref{SchauderCh_HalbraumBCLokalisierungFixeCoeff} we get 
\begin{align}
    \hspace{-.3cm}[D^{4,1} u]^{(0)}_{\gamma,U_{\frac\rho2+}(p_0)}
    \leq &C\left(
[P_0 u]_{\gamma,U_{\rho+}(p_0)}^{(0)}+[B_{1,0} u]_{\gamma,\partial_S U_{\rho+}(p_0)}^{(3)}
+[B_{2,0} u]_{\gamma,\partial_S U_{\rho+}(p_0)}^{(1)}+\rho^ {-4-\gamma}\|u\|_{L^ \infty(U_{\rho+}(p_0))}\right)\nonumber\\
&  \hspace{-2.4cm}\leq C\left(
[Pu]_{\gamma,U_{\rho+}(p_0)}^{(0)}+[B_1 u]_{\gamma,\partial_S U_{\rho+}(p_0)}^{(3)}+[B_2 u]_{\gamma,\partial_S U_{\rho+}(p_0)}^{(1)}
+\rho^ {-4-\gamma}\|u\|_{L^ \infty(U_{\rho+}(p_0))}\right)\nonumber
\\& \hspace{-2.2cm}+
C\left(
[(P-P_0)u]_{\gamma,U_{\rho+}(p_0)}^{(0)}+[(B_1-B_{1,0}) u]_{\gamma,\partial_S U_{\rho+}(p_0)}^{(3)}+[(B_2-B_{2,0}) u]_{\gamma,\partial_S U_{\rho+}(p_0)}^{(1)}
\right).\label{SchauderCh_HalbraumBCLokalisierungVariableCoeff_Est01Anfang}
\end{align}
We can use the analysis leading to Estimate \eqref{SchauderCh_ganzraumVariableCoeff_100} to estimate $[(P-P_0)u]_{\gamma,U_{\rho+}(p_0)}^{(0)}$. Next, we estimate $ [(B_1-B_{1,0}) u]_{\gamma,\partial_S U_{\rho+}(p_0)}^{(3)}$. To do so, we note that for $1\leq N\leq 3$
\begin{equation}\label{SchauderCh_BoundaryCOpProductRule}
\begin{aligned}
\nabla^N\left[(B_1-B_{1,0})u\right]
=&
\sum_{|\alpha|=1}(c^\alpha-c_0^\alpha)\nabla^N\nabla_\alpha u
+
\sum_{|\alpha|<1}c^\alpha\nabla^N\nabla_\alpha u\\
&+\sum_{|\alpha|\leq 1}\sum_{|\beta|\leq N-1}\sum_{\substack{|\kappa|\leq N\\ |\beta|+|\kappa|=N}}c(\alpha,\beta,\kappa)\nabla_\kappa c^\alpha\nabla_{\alpha+\beta} u.
\end{aligned}
\end{equation}
First, we consider $[\nabla^3 ((B_1-B_{1,0})u)]_{\gamma,\partial_S U_{\rho+}(p_0)}^{\operatorname{space}}$. We estimate the leading term from Equation \eqref{SchauderCh_BoundaryCOpProductRule}. To do so, let $\alpha$ be a multi-index of length $|\alpha|=1$. We estimate $\|c^\alpha-c^\alpha_0\|_{L^\infty(\partial_S U_{\rho+}(p_0))}\leq [c^\alpha]_{\gamma,\partial_S U_{\rho+}(p_0))}^{\operatorname{space}}\rho^\gamma$. Additionally, we use Assumption \eqref{SchauderCh_VariableCoeff_Assumption04}, $\rho\leq \nu R$ and an analogue of Estimate \eqref{SchauderCh_LocalizationProductRuileEstimate}, to estimate 
\begin{align}
    &[(c^\alpha-c_0^\alpha)\nabla^3\nabla_\alpha u]_{\gamma,\partial_S U_{\rho+}(p_0)}^{\operatorname{space}}\nonumber\\
    \leq & C\left(\|c^\alpha-c_0^\alpha\|_{L^\infty(\partial_S U_{\rho+}(p_0))}
    [\nabla^4 u]_{\gamma,\partial_S U_{\rho+}(p_0)}^{\operatorname{space}}
    +[c^{\alpha}]_{\gamma,\partial_S U_{\rho+}(p_0)}^{\operatorname{space}}\|\nabla^4 u\|_{L^\infty(\partial_S U_{\rho+}(p_0))}\right)\nonumber\\
    \leq & C\left( [c^{\alpha}]_{\gamma,\partial_S U_{\rho+}(p_0)}^{(0)} \rho^\gamma[\nabla^4 u]_{\gamma,\partial_S U_{\rho+}(p_0)}^{\operatorname{space}}
    +[c^{\alpha}]_{\gamma,\partial_S U_{\rho+}(p_0)}^{\operatorname{space}}\|\nabla^4 u\|_{L^\infty(\partial_S U_{\rho+}(p_0))}\right)\nonumber\\
    \leq & CM\left(R^{-\gamma}\rho^\gamma [\nabla^4 u]_{\gamma,\partial_S U_{\rho+}(p_0)}^{\operatorname{space}}+R^{-\gamma}\|\nabla^4 u\|_{L^\infty(\partial_S U_{\rho+}(p_0))}\right)\nonumber\\
    \leq & CM\left(\nu^\gamma [D^{4,1}u]_{\gamma, U_{\rho+}(p_0)}^{(0)}+R^{-\gamma}\|\nabla^4 u\|_{L^\infty( U_{\rho+}(p_0))}\right).\label{SchauderCh_variableBCEstimate01}
\end{align}
Next, for a multiindex $\alpha$ of length $|\alpha|<1$ (that is $|\alpha|=0)$, we use Assumption \eqref{SchauderCh_VariableCoeff_Assumption04} and an analogue of Estimate \eqref{SchauderCh_LocalizationProductRuileEstimate} to estimate
\begin{align}
    &[c^\alpha\nabla^3\nabla_\alpha u]_{\gamma,\partial_S U_{\rho+}(p_0)}^{\operatorname{space}}\nonumber\\
    \leq & C\left(\|c^\alpha\|_{L^\infty(\partial_S U_{\rho+}(p_0))}
    [\nabla^{3+|\alpha|} u]_{\gamma,\partial_S U_{\rho+}(p_0)}^{\operatorname{space}}
    +[c^{\alpha}]_{\gamma,\partial_S U_{\rho+}(p_0)}^{\operatorname{space}}\|\nabla^{3+|\alpha|} u\|_{L^\infty(\partial_S U_{\rho+}(p_0))}\right)\nonumber\\
    \leq & CM\left(R^{-1+|\alpha|} [\nabla^{3+|\alpha|} u]_{\gamma,\partial_S U_{\rho+}(p_0)}^{\operatorname{space}}+R^{-1+|\alpha|-\gamma}\|\nabla^{3+|\alpha|} u\|_{L^\infty( U_{\rho+}(p_0))}\right)\label{SchauderCh_variableBCEstimate01MISSING}.
\end{align}
Next, we estimate the remaining terms in Equation \eqref{SchauderCh_BoundaryCOpProductRule} ($N=3$). Let $|\alpha|\leq 1$, $|\beta|\leq 3-1$ and $|\kappa|=3-|\beta|$. Using Assumption \eqref{SchauderCh_VariableCoeff_Assumption04}, we estimate
\begin{align}
    &[\nabla_\kappa c^\alpha\nabla_{\alpha+\beta} u]^{\operatorname{space}}_{\gamma,\partial_S U_{\rho+}(p_0)} \nonumber\\
    \leq & C\left(
    \|\nabla_\kappa c^\alpha\|_{L^\infty(\partial_S U_{\rho+}(p_0) )} [\nabla_{\alpha+\beta} u]^{\operatorname{space}}_{\gamma,\partial_S U_{\rho+}(p_0)} 
    +
    [\nabla_\kappa c^\alpha]^{\operatorname{space}}_{\gamma,\partial_S U_{\rho+}(p_0)} \|\nabla_{\alpha+\beta} u\|_{L^\infty(\partial_S U_{\rho+}(p_0) )}\right)\nonumber\\
    \leq & CM\left(
    R^{-1-|\kappa|+|\alpha|} [\nabla^{|\alpha|+|\beta|} u]^{\operatorname{space}}_{\gamma,\partial_S U_{\rho+}(p_0)} 
    +
     R^{-1-|\kappa|-\gamma+|\alpha|} \|\nabla^{|\alpha|+|\beta|} u\|_{L^\infty(U_{\rho+}(p_0) )}\right).\label{SchauderCh_variableBCEstimate02}
\end{align}
Considering the terms $R^{(...)}$ appearing in Estimates \eqref{SchauderCh_variableBCEstimate01}, \eqref{SchauderCh_variableBCEstimate01MISSING} and \eqref{SchauderCh_variableBCEstimate02} , we see that $(...)\leq 0$ in all cases. Hence $R^{(...)}\leq \rho^{(...)}$. Inserting the interpolation estimates from Theorem \ref{SchauderCh_InterpolationEstimates} as we did in the derivation of Estimate \eqref{SchauderCh_ganzraumVariableCoeff_100}, we deduce
\begin{equation}\label{SchauderChBT1VariableCoef01}
[\nabla^3 \big((B_1-B_{1,0}) u\big)]_{\gamma,\partial_S U_{\rho+}(p_0)}^{\operatorname{space}}
\leq 
CM(\nu^\gamma+\epsilon)[D^{4,1}u]_{\gamma, U_{\rho+}(p_0)}^{(0)}+C(\epsilon)\|u\|_{L^\infty(U_{\rho+}(p_0))}.
\end{equation}
For $0\leq k\leq 3$, we consider $[\nabla^k ((B_1-B_{1,0})u)]_{\frac{3-k+\gamma}4,\partial_S U_{\rho+}(p_0)}^{\operatorname{time}}$. We first estimate the leading term from Equation \eqref{SchauderCh_BoundaryCOpProductRule}. Let $|\alpha|=1$. Using Assumptions \eqref{SchauderCh_VariableCoeff_Assumption04} and \eqref{SchauderCh_VariableCoeff_Assumption05}, $\rho\leq\nu R$ and an analogue of Estimate \eqref{SchauderCh_LocalizationProductRuileEstimate}, we estimate 
\begin{align}
    &[(c^\alpha-c_0^\alpha)\nabla^k\nabla_\alpha u]_{\frac{3-k+\gamma}4,\partial_S U_{\rho+}(p_0)}^{\operatorname{time}}\nonumber\\
    \leq & C\left(\|c^\alpha-c_0^\alpha\|_{L^\infty(\partial_S U_{\rho+}(p_0))}[\nabla^{k+1} u]_{\frac{3-k+\gamma}4,\partial_S U_{\rho+}(p_0)}^{\operatorname{time}}
    +
    [c^{\alpha}]_{\frac{3-k+\gamma}4,\partial_S U_{\rho+}(p_0)}^{\operatorname{time}}\|\nabla^{k+1} u\|_{L^\infty(\partial_S U_{\rho+}(p_0))}\right)\nonumber\\
    \leq & C\left( [c^{\alpha}]_{\gamma, \partial_S U_{\rho+}(p_0)}^{(0)}\rho^\gamma[\nabla^{k+1} u]_{\frac{3-k+\gamma}4,\partial_S U_{\rho+}(p_0)}^{\operatorname{time}}
    +
    [c^{\alpha}]_{\frac{3-k+\gamma}4,\partial_S U_{\rho+}(p_0)}^{\operatorname{time}}\|\nabla^{k+1} u\|_{L^\infty(\partial_S U_{\rho+}(p_0))}\right)\nonumber\\
    \leq & CM\left(R^{-\gamma}\rho^\gamma [\nabla^{k+1} u]_{\frac{3-k+\gamma}4,\partial_S U_{\rho+}(p_0)}^{\operatorname{time}}+R^{-3+k-\gamma}\|\nabla^{k+1} u\|_{L^\infty(\partial_S U_{\rho+}(p_0))}\right)\nonumber\\
    \leq & CM\left(\nu^\gamma [D^{4,1}u]_{\gamma, U_{\rho+}(p_0)}^{(0)}+R^{-3+k-\gamma}\|\nabla^{k+1} u\|_{L^\infty(\partial_S U_{\rho+}(p_0))}\right).\label{SchauderCh_BT1VariableTempporalSemiNorm}
\end{align}
Next, for $|\alpha|\leq 1$, $|\beta|\leq k-1$ and $|\kappa|=k-|\beta|$, we use Assumptions \eqref{SchauderCh_VariableCoeff_Assumption04} and \eqref{SchauderCh_VariableCoeff_Assumption05}, to estimate
\begin{align}
    &[\nabla_\kappa c^\alpha\nabla_{\alpha+\beta} u]^{\operatorname{time}}_{\frac{3+\gamma-k}4,\partial_S U_{\rho+}(p_0)}\nonumber \\
    \leq & C\left(
    \|\nabla_\kappa c^\alpha\|_{L^\infty(\partial_S U_{\rho+}(p_0) )} [\nabla_{\alpha+\beta} u]^{\operatorname{time}}_{\frac{3+\gamma-k}4,\partial_S U_{\rho+}(p_0)} 
    +
    [\nabla_\kappa c^\alpha]^{\operatorname{time}}_{\frac{3+\gamma-k}4,\partial_S U_{\rho+}(p_0)} \|\nabla_{\alpha+\beta} u\|_{L^\infty(\partial_S U_{\rho+}(p_0) )}\right)\nonumber\\
    \leq & CM\left(
    R^{-1-|\kappa|+|\alpha|} [\nabla^{|\alpha|+|\beta|} u]^{\operatorname{time}}_{\frac{3+\gamma-k}4,\partial_S U_{\rho+}(p_0)} 
    +
     R^{-4-\gamma-|\kappa|+|\alpha|+k} \|\nabla^{|\alpha|+|\beta|} u\|_{L^\infty(\partial_S U_{\rho+}(p_0) )}\right)\nonumber\\
    \leq & CM\bigg(
    R^{-1-k+|\alpha|+|\beta|} [\nabla^{|\alpha|+|\beta|} u]^{\operatorname{time}}_{\frac{3+\gamma-k}4,\partial_S U_{\rho+}(p_0)}\nonumber 
    \\
    &\hspace{2cm}+
     R^{-4-\gamma+|\alpha|+|\beta|} \|\nabla^{|\alpha|+|\beta|} u\|_{L^\infty(\partial_S U_{\rho+}(p_0) )}\bigg).\label{SchauderCh_BT1VariableTempporalSemiNorm2}
\end{align}
Using the same arguments; for $|\alpha|<1$ (that is $|\alpha|=0)$, we estimate 
\begin{align}
    &[c^\alpha\nabla^k\nabla_\alpha u]_{\frac{3-k+\gamma}4,\partial_S U_{\rho+}(p_0)}^{\operatorname{time}}\nonumber\\
    \leq & C\left(\|c^\alpha\|_{L^\infty(\partial_S U_{\rho+}(p_0))}[\nabla^{k} u]_{\frac{3-k+\gamma}4,\partial_S U_{\rho+}(p_0)}^{\operatorname{time}}
    +
    [c^{\alpha}]_{\frac{3-k+\gamma}4,\partial_S U_{\rho+}(p_0)}^{\operatorname{time}}\|\nabla^{k} u\|_{L^\infty(\partial_S U_{\rho+}(p_0))}\right)\nonumber\\
    \leq & CM\left(R^{-1} [\nabla^{k} u]_{\frac{3-k+\gamma}4,\partial_S U_{\rho+}(p_0)}^{\operatorname{time}}+R^{-4+k-\gamma}\|\nabla^{k} u\|_{L^\infty(\partial_S U_{\rho+}(p_0))}\right).\label{SchauderCh_BT1VariableTempporalSemiNormNEO}
\end{align}
Considering the terms $R^{(...)}$ appearing in Estimates \eqref{SchauderCh_BT1VariableTempporalSemiNorm}, \eqref{SchauderCh_BT1VariableTempporalSemiNorm2} and \eqref{SchauderCh_BT1VariableTempporalSemiNormNEO}, we see that $(...)\leq 0$ in all cases. Hence $R^{(...)}\leq \rho^{(...)}$. Inserting the interpolation estimates from Theorem \ref{SchauderCh_InterpolationEstimates} and Lemma \ref{SchauderCh_TemporalInterpolationLemma}, we deduce
\begin{equation}\label{SchauderChBT1VariableCoef03}
\hspace{-.3cm}\sum_{k=0}^3[\nabla^k\left( (B_1-B_{1,0})u\right)]^{\operatorname{time}}_{\frac{3-k+\gamma}4,\partial_S U_{\rho+}(p_0)}
\leq\hspace{-.1cm}
CM(\nu^\gamma+\epsilon)[D^{4,1}u]_{\gamma, U_{\rho+}(p_0)}^{(0)}+C(\epsilon)\|u\|_{L^\infty(U_{\rho+}(p_0))}.
\end{equation}
Combining Estimates \eqref{SchauderChBT1VariableCoef01} and \eqref{SchauderChBT1VariableCoef03}, we obtain
\begin{equation}\label{SchauderChBT1VariableCoef03EndResult}
    [(B_1-B_{1,0}) u]_{\gamma,\partial_S U_{\rho+}(p_0)}^{(3)}\leq CM(\nu^\gamma+\epsilon)[D^{4,1}u]_{\gamma,U_{\rho+}(p_0)}^{(0)}+C(\epsilon)\rho^{-4-\gamma}\|u\|_{L^\infty(U_{\rho+}(p_0))}.
\end{equation}
Following the exact same arguments, the following estimate can be established by using Assumptions \eqref{SchauderCh_VariableCoeff_Assumption02} and \eqref{SchauderCh_VariableCoeff_Assumption03}:
\begin{equation}\label{SchauderChBT2VariableCoef03}
     [(B_2-B_{2,0}) u]_{\gamma,\partial_S U_{\rho+}(p_0)}^{(1)}
    \leq CM(\nu^\gamma+\epsilon)[D^{4,1}u]_{\gamma,U_{\rho+}(p_0)}^{(0)}+C(\epsilon)\rho^{-4-\gamma}\|u\|_{L^\infty(U_{\rho+}(p_0))}.
\end{equation}
Inserting the analogue of Estimate \eqref{SchauderCh_ganzraumVariableCoeff_100} and Estimates \eqref{SchauderChBT1VariableCoef03EndResult} and \eqref{SchauderChBT2VariableCoef03} into Estimate \eqref{SchauderCh_HalbraumBCLokalisierungVariableCoeff_Est01Anfang}, we have shown that for any $p_0\in U_{R+}(q)$ and $\rho\leq\nu R$ such that $U_{\rho+}(p_0)\subset U_{R+}(q)$, we have 
\begin{align*}
[D^{4,1} u]_{\gamma,U_{\frac\rho2+}(p_0)}^{(0)}\leq & C\left(
[Pu]_{\gamma,U_{\rho+}(p_0)}^{(0)}+[B_1 u]_{\gamma,\partial_S U_{\rho+}(p_0)}^{(3)}+[B_2 u]_{\gamma,\partial_S U_{\rho+}(p_0)}^{(1)}
+\rho^ {-4-\gamma}\|u\|_{L^ \infty(U_{\rho+}(p_0))}\right)\\
&+CM\left(\epsilon+\nu^\gamma\right)[D^{4,1} u]^{(0)}_{\gamma,U_{\rho+}(p_0)}+C(\epsilon)\rho^{-4-\gamma}\|u\|_{L^\infty(U_{\rho+}(p_0))}.
\end{align*}
The first estimate in the lemma follows from Simon's Absorption Lemma \ref{SchauderCh_SimonsCoveringLemma}. The second estimate is a direct consequence of the interpolation estimates in Theorem \ref{SchauderCh_InterpolationEstimates}
\end{proof}

\begin{lemma}\label{SchauderCh_HalbraumUntenBCLokalisierungVariableCoeff}
Let $q\in\R^{n-1}\times[0,\infty)\times(-\infty,0]$, $R>0$, $U_{R+}^-:=U_{R+}^-(q)$ and  $u\in C^{4,1,\gamma}(U_{R+}^-)$.
There exists a constant $C=C(M,\gamma, n,\theta)$ such that
\begin{align*} 
[D^{4,1}u]_{\gamma,U_{\frac R2+}^ -}^{(0)}\leq &C\left([Pu]^{(0)}_{\gamma,U_{R+}^-}+[B_1u]^{(3)}_{\gamma, \partial_SU_{R+}^-}+[B_2u]^{(1)}_{\gamma, \partial_S U_{R+}^-}+R^{-4-\gamma}\|u\|_{L^\infty(U_{R+}^-)}\right),\\
\|u\|_{C^{4,1,\gamma}(U_{\frac R2+}^ -)}\leq &C(R)\left([Pu]^{(0)}_{\gamma,U_{R+}^-}+[B_1u]^{(3)}_{\gamma, \partial_SU_{R+}^-}+[B_2u]^{(1)}_{\gamma, \partial_S U_{R+}^-}+\|u\|_{L^\infty(U_{R+}^-)}\right).
\end{align*}
\end{lemma}
\begin{tcolorbox}[colback=white!20!white,colframe=black!100!white,sharp corners, breakable]
If $\partial_S U_{R+}^-=\emptyset$, the estimate holds with $[B_1 u]^{(3)}_{\gamma,\partial_S U_{R+}^-}
 =[B_2 u]^{(1)}_{\gamma,\partial_S U_{R+}^-}=0$.
 \end{tcolorbox}
\begin{proof}
Throughout the proof, we will suppress the dependence of various constants $C$ on $M$, $\gamma$, $n$ and $\theta$. Let $\nu\in(0,1]$, $p_0=(x_0,t_0)\in U_{R+}^ -(q)$ and $\rho\in (0,\nu R)$ such that $U_{\rho+}^ -(p_0)\subset U_{R+}^ -(q)$. We put $A^\alpha_0:=A^\alpha(p_0)$, $b^\alpha_0:=b^\alpha(p_0)$ and $c^\alpha_0:=c^\alpha(p_0)$. Additionally, we put $P_0:=\sum_{|\alpha|=4} A^\alpha_0\nabla_\alpha$, $B_{1,0}:=\sum_{|\alpha|= 1}c_0^\alpha\nabla_\alpha$ and $B_{2,0}:=\sum_{|\alpha|=3}b_0^\alpha\nabla_\alpha$. By Lemma \ref{SchauderCh_UntenHalbraumBCLokalisierungFixeCoeff} we get 
\begin{align}
    [D^{4,1} u]^{(0)}_{\gamma,U_{\frac\rho2+}^-(p_0)}\leq &C\left(
[P_0 u]_{\gamma,U_{\rho+}^-(p_0)}^{(0)}+[B_{1,0} u]_{\gamma,\partial_S U_{\rho+}^-(p_0)}^{(3)}+[B_{2,0} u ]_{\gamma,\partial_S U_{\rho+}^-(p_0)}^{(1)}
+\rho^ {-4-\gamma}\|u\|_{L^ \infty(U_{\rho+}^-(p_0))}\right)\nonumber\\
& \hspace{-2.4cm}\leq C\left(
[Pu]_{\gamma,U_{\rho+}^-(p_0)}^{(0)}+[B_1 u]_{\gamma,\partial_S U_{\rho+}^-(p_0)}^{(3)}+[B_2 u]_{\gamma,\partial_S U_{\rho+}^-(p_0)}^{(1)}
+\rho^ {-4-\gamma}\|u\|_{L^ \infty(U_{\rho+}^-(p_0))}
\right)\nonumber
\\
&\hspace{-2.2cm}+
C\left(
[(P-P_0)u]_{\gamma,U_{\rho+}^-(p_0)}^{(0)}\hspace{-.2cm}+[(B_1-B_{1,0}) u]_{\gamma,\partial_S U_{\rho+}^-(p_0)}^{(3)}\hspace{-.2cm}+[(B_2-B_{2,0}) u]_{\gamma,\partial_S U_{\rho+}^-(p_0)}^{(1)}
\right).\label{SchauderCh_UntenHalbraumBCLokalisierungVariableCoeff_Est01Anfang}
\end{align}

Following the arguments that lead to Estimates \eqref{SchauderCh_ganzraumVariableCoeff_100},
\eqref{SchauderChBT1VariableCoef03EndResult} and
\eqref{SchauderChBT2VariableCoef03}, we obtain suitable estimates for 
$[(P-P_0)u]_{\gamma,U_{\rho+}^-(p_0)}^{(0)}$, 
$[(B_1-B_{1,0}) u]_{\gamma,\partial_S U_{\rho+}^-(p_0)}^{(3)}$ and 
$[(B_2-B_{2,0}) u]_{\gamma,\partial_S U_{\rho+}^-(p_0)}^{(1)}$. 
Inserting these into Estimate \eqref{SchauderCh_UntenHalbraumBCLokalisierungVariableCoeff_Est01Anfang}, we get
\begin{align}
[D^{4,1} u]^{(0)}_{\gamma,U_{\frac\rho2+}^-(p_0)}
\leq &C(\epsilon)\left(
[P u]_{\gamma,U_{\rho+}^-(p_0)}^{(0)}+[B_1 u]_{\gamma,\partial_S U_{\rho+}^-(p_0)}^{(3)}+[B_2 u]_{\gamma,\partial_S U_{\rho+}^-(p_0)}^{(1)}+\rho^{-4-\gamma}\|u\|_{L^\infty(U_{\rho+}^-(p_0))}\right)\nonumber\\
&\hspace{.5cm}+CM\left(\epsilon+\nu^\gamma\right)[D^{4,1} u]^{(0)}_{U_{\rho+}^-(p_0)}.\nonumber
\end{align}
The first estimate claimed in the lemma follows from Simon's Absorption Lemma \ref{SchauderCh_SimonsCoveringLemma}. The second estimate is a direct consequence of the interpolation estimates in Theorem \ref{SchauderCh_InterpolationEstimates}
\end{proof}

Following the proof of Lemma \ref{SchauderCh_HalbraumUntenBCLokalisierungVariableCoeff} and swapping Lemma \ref{SchauderCh_UntenHalbraumBCLokalisierungFixeCoeff} for Lemma \ref{SchauderCh_ObenHalbraumBCLokalisierungFixeCoeff}, the following lemma can be established.
\begin{lemma}\label{SchauderCh_HalbraumObenBCLokalisierungVariableCoeff}
Let $q\in\R^{n-1}\times[0,\infty)\times[0,\infty)$, $R>0$, $U_{R+}^+:=U_{R+}^+(q)$, $u\in C^{4,1,\gamma}(U_{R+}^+)$ and $u_0:=u(\cdot,0)$. There exists a constant $C=C(M,\gamma, n,\theta)$ such that 
\begin{align*} 
[D^{4,1}u]^{(0)}_{\gamma,U_{\frac R2+}^ +}\leq& C\left([Pu]^{(0)}_{\gamma,U_{R+}^+}+[B_1u]^{(3)}_{\gamma, \partial_SU_{R+}^+}+[B_2u]^{(1)}_{\gamma, \partial_S U_{R+}^+}
+[\nabla^4 u_0]_{\gamma,\partial_F U_{R+}^ +(0)}
+R^{-4-\gamma}\|u\|_{L^\infty(U_{R+}^+)}\right),\\
\|u\|_{C^{4,1,\gamma}(U_{\frac R2+}^ +)}\leq& C(R)\left([Pu]_{U_{R+}^+}^{(0)}+[B_1u]^{(3)}_{\gamma, \partial_SU_{R+}^+}+[B_2u]^{(1)}_{\gamma, \partial_S U_{R+}^+}
+[\nabla^4 u_0]_{\gamma,\partial_F U_{R+}^ +(0)}
+\|u\|_{L^\infty(U_{R+}^+)}\right).
\end{align*}
\end{lemma}
 \begin{tcolorbox}[colback=white!20!white,colframe=black!100!white,sharp corners, breakable]
If $\partial_S U_{R+}^+=\emptyset$ or $\partial_F U_{R+}^+=\emptyset$, the estimate holds with $[B_1 u]^{(3)}_{\gamma,\partial_S U_{R+}^+}
 =[B_2 u]^{(1)}_{\gamma,\partial_S U_{R+}^+}=0$ or $[\nabla^4 u_0]_{\gamma,\partial_F U_{R+}^ +}=0$ respectively.
 \end{tcolorbox}

\section{Covering Arguments}
In this section, we provide suitable covering arguments that allow us to extend the results from Section \ref{SchauderCh_LocalizedEstimatesSection} to a more general setting. 

\begin{lemma}\label{SchauderCh_InteriorCoveringLemma}
Let $V\subset\subset  V'\subset\R^n$ be open and bounded sets, $T_0>0$ and $T\geq T_0$. There exists $\rho>0$, $x_1,...,x_N\in V$ and $0=t_0<t_1<...<t_{k_0-1}<t_{k_0}=T$ such that the sets 
$$U_{jk,r}:=U_r(x_j, t_k)\cap (\R^n\times[0,T])
\hspace{.5cm}\textrm{where $1\leq j\leq N$, $0\leq k\leq k_0$ and $r>0$}$$
have the following properties:
\begin{enumerate}[(1)]
    \item For all $1\leq j\leq N$ and $0\leq k\leq k_0$ we have $U_{jk,\rho}\subset U_{jk,2\rho}\subset V'\times[0,T]$.
    \item Each $U_{jk, 2\rho}$ is of one of the following forms: $U_{2\rho}(x_j, t_k)$, $U^+_{2\rho}(x_j, t_k)$ or $U_{2\rho}^-(x_j, t_k)$.
    \item $V\times[0,T]\subset \bigcup_{j=1}^N\bigcup_{k=0}^{k_0}U_{jk,\rho}$.
\end{enumerate}
The parameters $N$ and $\rho$ depend only on the sets $V$, $ V'$ and the parameter $T_0$ but not on $T$. Moreover, for a function $\varphi\in C^{0,0,\gamma}( V'\times[0,T])$ we have 
\begin{align*} 
&[\varphi]^{\operatorname{space}}_{\gamma, V}\leq N\max_{j,k}[\varphi]^{\operatorname{space}}_{\gamma, U_{jk,\rho}}+C(\rho,\gamma)\max_{j,k}\|\varphi\|_{L^\infty(U_{jk,\rho})},\\
&[\varphi]^{\operatorname{time}}_{\frac\gamma4, V}\leq 4\max_{j,k}[\varphi]^{\operatorname{time}}_{\frac\gamma4, U_{jk,\rho}}+C(\rho,\gamma)\max_{j,k}\|\varphi\|_{L^\infty(U_{jk,\rho})}.
\end{align*}
\end{lemma}
\begin{proof}
Let $\rho\in(0,1)$ be a fixed number such that $1000\rho^4<T_0$ and such that for all $x\in \bar V$ we have $B_{2\rho}(x)\subset V'$.
By compactness we can select finitely many points $x_j\in V$ with $1\leq j\leq N(V, V', \rho(V,V',T_0))$ such that $V\subset \bigcup_{j=1}^N B_{\rho}(x_j)$.

\paragraph{Definition of the Covering}\ \\
Let $k_0\in\N$ be the unique natural number that satisfies $k_0\rho^4 \leq T< (k_0+1)\rho^4$. By definition of $\rho$, we have $k_0\geq 1000$. We define $t_{k_0}:=T$ and $t_k:= k\rho^4$ for all $0\leq k\leq k_0-1$. Clearly $0=t_0<t_1<...<t_{k_0}=T$. For $1\leq k\leq k_0-1$, we define
\begin{align*}
    I_0&:=(t_0-\rho^4,t_0+\rho^4)\cap[0,T]=[0,\rho^4),\\
    I_{k_0}&:=(t_{k_0}-\rho^4,t_{k_0}+\rho^4)\cap[0,T]=(T-\rho^4, T],\\
    I_{k}&:=(t_{k}-\rho^4,t_{k}+\rho^4)\cap[0,T]=((k-1)\rho^4,(k+1)\rho^4).
\end{align*}
By definition of $k_0$, we have $[0,T]=I_0\cup...\cup I_{k_0}$.
For $1\leq j\leq N$, $0\leq k\leq k_0$ and $r=\rho,2\rho$ we define
\begin{align*} 
U_{jk,r}&:=U_r(x_j, t_k)\cap(\R^ n\times[0,T])=B_r(x_j)\times (t_k-r^4, t_k+r^4).
\end{align*}
By construction, the sets $U_{jk,\rho}$ cover $V\times[0,T]$ and $U_{jk,2\rho}\subset V'\times[0,T]$. Finally, since $k_0\geq 1000$, $U_{jk,2\rho}$ cannot have a nonempty intersection with $\R^n\times 0$ and $\R^n\times T$. So, sets $U_{jk,r}$ satisfy the claimed properties.

\paragraph{Spatial Hölder Norm}\ \\
Let $x\neq y\in V$ and $t\in[0,T]$. Let $0\leq i\leq k_0$ such that $t\in I_i$. Further, we assume without loss of generality that $x\in B_{\rho}(x_1)$. If $y\not\in B_{2\rho}(x_1)$, then $|x-y|\geq\rho$ and hence 
\begin{equation}\label{SchauderCh_CoveringLemmaFullSpace1}
\frac{|\varphi(x,t)-\varphi(y,t)|}{|x-y|^\gamma}\leq 
\frac{1}{\rho^\gamma}(|\varphi(x,t)|+|\varphi(y,t)|)\leq 
\frac{2}{\rho^\gamma}\max_{j,k}\|\varphi\|_{L^\infty(U_{jk,\rho})}.
\end{equation}
If $y\in B_{2\rho}(x_1)$, we use Lemma \ref{SchauderCh_SemiNormSubbadditiveonConvexSpaces} Indeed $B_{2\rho}(x_1)$ is convex and covered by the convex sets $B_{2\rho}(x_1)\cap B_{\rho}(x_j)$ with $1\leq j\leq N$. So, applying Lemma \ref{SchauderCh_SemiNormSubbadditiveonConvexSpaces} to the function $\varphi(\cdot, t)$ gives 
\begin{equation}\label{SchauderCh_CoveringLemmaFullSpace2}
\frac{|\varphi(x,t)-\varphi(y,t)|}{|x-y|^\gamma}
\leq
\sum_{j=1}^N[\varphi(\cdot,t)]_{\gamma, B_{\rho}(x_j)}
\leq
N\max_{j,k}[\varphi(\cdot,t)]_{\gamma, B_{\rho}(x_j)}
\leq
N\max_{j,k}[\varphi]_{\gamma,U_{jk,\rho}}^{\operatorname{space}}.
\end{equation}
Combining Estimates \eqref{SchauderCh_CoveringLemmaFullSpace1} and \eqref{SchauderCh_CoveringLemmaFullSpace2}, we obtain the claimed estimate.

\paragraph{Temporal seminorm}\ \\
Let $x\in V$ and $0\leq t\neq s\leq T$. Without loss of generality we assume that $x\in B_{\rho}(x_1)$ and $t<s$. First, we assume that $|t-s|\geq \rho^4$. Then 
\begin{equation}\label{SchauderCh_CoveringLemmaFullSpace3}
\frac{|\varphi(x,t)-\varphi(x,s)|}{|t-s|^{\frac\gamma4}}\leq \frac{1}{\rho^{\gamma}}(|\varphi(x,t)|+|\varphi(x,s)|)\leq \frac{2}{\rho^\gamma} \max_{j,k}\|\varphi\|_{L^\infty(U_{jk,\rho})}.
\end{equation}
In the last step, we used that the sets $U_{jk,\rho}$ cover $V\times[0,T]$. 
Now we consider the case $|t-s|<\rho^4$. Let $0\leq k\leq k_0$ be the largest $k$ such that $t\in I_k$ and $l\geq k$ be the smallest $l\geq k$ such that $s\in I_l$. Note that the convex set $\bigcup_{a=k}^lI_a$ is covered by the convex sets $I_a$ with $k\leq a\leq l$. So, applying Lemma \ref{SchauderCh_SemiNormSubbadditiveonConvexSpaces} to the function $\varphi(x,\cdot)$, we obtain 
\begin{equation}\label{SchauderCh_CoveringLemmaFullSpace4}
\frac{|\varphi(x,t)-\varphi(x,s)|}{|t-s|^{\frac\gamma4}}\leq
[\varphi(x,\cdot)]_{\frac\gamma4,\bigcup_{a=k}^l I_a}\leq
\sum_{a=k}^l[\varphi(x,\cdot)]_{\frac\gamma4,I_a}\leq 
(l-k+1)\max_{j,k}[\varphi]_{\frac\gamma4, U_{jk,\rho}}^{\operatorname{time}}.
\end{equation}
If $k\leq k_0-4$, then $t\in I_k$ implies $s\in I_{k}\cup I_{k+1}\cup I_{k+2}$, since $\sup I_k=(k+1)\rho^4$ and $\inf I_{k+3}=(k+2)\rho^4$. So $l-k+1\leq 3$. If $k\geq k_0-3$, then $l\in\set{k_0-3,k_0-2,k_0-1,k_0}$ and so $l-k+1\leq 4$. 
Combining Estimates \eqref{SchauderCh_CoveringLemmaFullSpace3}, \eqref{SchauderCh_CoveringLemmaFullSpace4} and $l-k+1\leq 4$ we deduce the claimed estimate.
\end{proof}

\begin{bemerkung}[On the Lower Bound $T_0$]
    In the proof of Lemma \ref{SchauderCh_InteriorCoveringLemma}, we have used the lower bound $T_0$ to ensure that $k_0\geq 1000$. This is because we need to avoid the situation where a parabolic ball has nonempty intersection with $\R^n\times 0$ and $\R^n\times T$. 
\end{bemerkung}

We can repeat the same arguments to obtain a good covering for domains in $\R^{n-1}\times[0,\infty)$. 

\begin{lemma}\label{SchauderCh_BoundaryCoveringLemma}
Let $V\subset\subset  V'\subset\R^{n-1}\times[0,\infty)$ be open and bounded sets such that $V\cap(\R^{n-1}\times 0)\neq\emptyset$, $T_0>0$ and $T\geq T_0$. There exists $\rho>0$, $x_1,...,x_N$ and $0=t_0<t_1<...<t_{k_0}=T$ such that the sets 
$$U_{jk,r}:=U_r(x_j, t_k)\cap(\R^ {n-1}\times[0,\infty)\times[0,T])\hspace{.5cm}\textrm{where $1\leq j\leq N$, $0\leq k\leq k_0$ and $r>0$}$$
have the following properties:
\begin{enumerate}[(1)]
    \item For all $1\leq j\leq N$ and $0\leq k\leq k_0$ we have $U_{jk,\rho}\subset U_{jk,2\rho}\subset V'\times[0,T]$.
    \item Each $U_{jk, 2\rho}$ is of one of the following forms: $U_{2\rho}(x_j, t_k)$, $U^+_{2\rho}(x_j, t_k)$, $U_{2\rho}^-(x_j, t_k)$, $U_{2\rho+}(x_j, t_k)$, $U^+_{2\rho+}(x_j, t_k)$ or $U_{2\rho+}^-(x_j, t_k)$.
    \item $V\times[0,T]\subset \bigcup_{j=1}^N\bigcup_{k=0}^{k_0}U_{jk,\rho}$.
\end{enumerate}
The parameters $N$ and $\rho$ depend only on the sets $V$ and $ V'$ and the parameter $T_0$ but not on $T$. Moreover, for a function $\varphi\in C^{0,0,\gamma}( V'\times[0,T])$ we have 
\begin{align*} 
&[\varphi]^{\operatorname{space}}_{\gamma, V}\leq N\max_{j,k}[\varphi]^{\operatorname{space}}_{\gamma, U_{jk,\rho}}+C(\rho,\gamma)\max_{j,k}\|\varphi\|_{L^\infty(U_{jk,\rho})},\\
&[\varphi]^{\operatorname{time}}_{\frac\gamma4, V}\leq 4\max_{j,k}[\varphi]^{\operatorname{time}}_{\frac\gamma4, U_{jk,\rho}}+C(\rho,\gamma)\max_{j,k}\|\varphi\|_{L^\infty(U_{jk,\rho})}.
\end{align*}
\end{lemma}

Using Lemmas \ref{SchauderCh_InteriorCoveringLemma} and \ref{SchauderCh_BoundaryCoveringLemma}, we can now extend the estimated from Section \ref{SchauderCh_LocalizedEstimatesSection} to general domains. 

\begin{korollar}\label{SchauderCh_InteriorSchauderEstimateCorollary}
Let $V\subset\subset V'\subset\R^ n$ be open and bounded sets, $T_0>0$, $T\geq T_0$ and $\Lambda,\theta>0$. There exists a constant $C= C(V,V', T_0,\gamma, \Lambda,\theta)$ such that for all 
$$P:=\partial_t+a^{ij}a^{kl}\partial_{ijkl}+\sum_{|\alpha|\leq 3}A^\alpha\nabla_\alpha $$
with coefficients $a^{ij},A^\alpha\in C^{0,0,\gamma}(V'\times[0,T])$ satisfying 
\begin{align}
    &a^{ij}(x,t)\xi_i\xi_j\geq\theta|\xi|^2\hspace{.5cm}\textrm{for all $(x,t)\in V'\times[0,T]$ and $\xi\in\R^n$},\label{SchauderCh_interiorgloablieAss01}\\
    &\|a^{ij}\|_{C^{0,0,\gamma}(V'\times[0,T])},\ \|A^\alpha\|_{C^{0,0,\gamma}(V'\times[0,T])}\leq \Lambda,\label{SchauderCh_interiorgloablieAss02}
\end{align}
and all $u\in C^ {4,1,\gamma}(V'\times[0,T])$ 
$$\|u\|_{C^ {4,1,\gamma}(V\times[0,T])}\leq C\left(\|Pu\|_{C^ {0,0,\gamma}(V'\times[0,T])}+\|u(\cdot,0)\|_{C^ {4,\gamma}(V')}+\|u\|_{L^ \infty(V'\times[0,T])}\right).$$
\end{korollar}
 \begin{tcolorbox}[colback=white!20!white,colframe=black!100!white,sharp corners, breakable]
The important observation in Lemma \ref{SchauderCh_InteriorSchauderEstimateCorollary} lies in the fact that the constant $C(V,V',T_0,\gamma, \Lambda,\theta)$ does not depend on $T$.
\end{tcolorbox}
\begin{proof}
Throughout the proof, we suppress the dependence of constants on $V$, $V'$, $T_0$, $\gamma$, $\Lambda$ and $\theta$. We choose $\rho>0$ and $U_{jk,r}$ as in Lemma \ref{SchauderCh_InteriorCoveringLemma}. Each $U_{jk,2\rho}$ is of the form $U_{2\rho}(q)$, $U_{2\rho}^+(q)$ or $U_{2\rho}^-(q)$ for some $q$. Depending on which is the case we can use Lemmas \ref{SchauderCh_GanzraumLokalisierungVariableCoeff}, \ref{SchauderCh_HalbraumObenLokalisierungVariableCoeff} or \ref{SchauderCh_HalbraumUntenLokalisierungVariableCoeff} to  get the following estimates for all $1\leq j\leq N$ and $0\leq k\leq k_0$:
\begin{align}
   \|u\|_{C^ {4,1,\gamma}(U_{jk,\rho})}
    \leq &C\left(\|Pu\|_{C^ {0,0,\gamma}(U_{jk,2\rho})}+\|u(\cdot,0)\|_{C^ {4,\gamma}(\partial_F U_{jk,2\rho})}+\|u\|_{L^ \infty(U_{jk,2\rho})}\right)\label{SchauderCh_CoverInteriorBottom}
\end{align}
Here we have condensed notation by putting $\|u(\cdot,0)\|_{C^ {4,\gamma}(\partial_F U_{jk,2\rho})}:=0$, if $\partial_F U_{jk,2\rho}=\emptyset$. First, as $U_{jk,\rho}$ cover $V\times[0,T]$ and $U_{jk,2\rho}\subset V'\times[0,T]$ we have the estimate 
\begin{align} 
\|u\|_{C^ {4,1}(V\times[0,T])}\leq &\max_{j,k}\|u\|_{C^ {4,1}(U_{jk,\rho})}\nonumber\\
&\hspace{-2cm}\leq
C(\rho)\left(\|Pu \|_{C^ {0,0,\gamma}(V'\times[0,T])}+\|u(\cdot,0)\|_{C^ {4,\gamma}(V')}+\|u\|_{L^ \infty(V'\times[0,T])}\right).\label{SchauderCh_CoveringInteriorResult1}
\end{align}
To estimate the Hölder seminorms, we use the estimates for the spatial and temporal seminorms from Lemma \ref{SchauderCh_InteriorCoveringLemma}. Inserting  Estimate \eqref{SchauderCh_CoverInteriorBottom}, we deduce
\begin{align}
    [D^ {4,1} u]^{(0)}_{\gamma,V\times[0,T]} &\leq C(\rho)\left(\max_{j,k}[D^{4,1} u]^{(0)}_{\gamma, U_{jk, \rho}}+\max_{j,k}\|u\|_{L^\infty(U_{jk,\rho})}\right)\nonumber\\
  &\hspace{-2cm} \leq  C(\rho)\left(\|Pu\|_{C^ {0,0,\gamma}(V'\times[0,T])}+\|u(\cdot,0)\|_{C^ {4,\gamma}(V')}+\|u\|_{L^ \infty(V'\times[0,T])}\right).\label{SchauderCh_CoveringInteriorResult2}
\end{align} 
The corollary follows by combining Estimates \eqref{SchauderCh_CoveringInteriorResult1} and \eqref{SchauderCh_CoveringInteriorResult2} and recalling that $\rho$ from Lemma \ref{SchauderCh_InteriorCoveringLemma} only depends on $V$, $V'$ and $T_0$.
\end{proof}

We establish the analogue fact for $V\subset\subset V'\subset \R^{n-1}\times[0,\infty)$

\begin{korollar}\label{SchauderCh_BoundarySchauderEstimateCorollary}
Let $V\subset\subset V'\subset\R^{n-1}\times[0,\infty)$ be open and bounded with $V\cap(\R^{n-1}\times 0)\neq\emptyset$, $T_0>0$, $T\geq T_0$ and $\Lambda,\theta>0$. We put $\partial_S V:=V\cap(\R^{n-1}\times 0)$, $\partial_S V':=V'\cap(\R^{n-1}\times 0)$. There exists a constant $C=C(V,V',T_0,\gamma, \Lambda,\theta)$ such that for all
\begin{align*} 
&P:=\partial_t+a^{ij}a^{kl}\partial_{ijkl}+\sum_{|\alpha|\leq 3}A^\alpha\nabla_\alpha \\
&B_1:=a^{ni}\partial_i +\sum_{|\alpha|\leq 0}c^\alpha\nabla_\alpha
\hspace{.5cm}\textrm{and}\hspace{.5cm}
B_2:=a^{ni}a^{kl}\partial_{ikl} +\sum_{|\alpha|\leq2}b^\alpha\nabla_\alpha
\end{align*}
with coefficients $a^{ij},A^\alpha\in C^{0,0,\gamma}(V'\times[0,T])$, $a^{ij}, c^\alpha\in C^{3,0,\gamma}(\partial_S V'\times[0,T])$ and $b^\alpha\in C^{1,0,\gamma}(\partial_S V'\times[0,T])$ satisfying 
\begin{align}
    &a^{ij}(x,t)\xi_i\xi_j\geq\theta|\xi|^2\hspace{.5cm}\textrm{for all $(x,t)\in V'\times[0,T]$ and $\xi\in\R^n$},\label{SchauderCh_BoundarygloablieAss01}\\
    &\|a^{ij}\|_{C^{0,0,\gamma}(V'\times[0,T])},\ \|A^\alpha\|_{C^{0,0,\gamma}(V'\times[0,T])}\leq \Lambda,\label{SchauderCh_BoundarygloablieAss02}\\
    &\|a^{ij}\|_{C^{3,0,\gamma}(\partial_S V'\times[0,T])},
    \ \|c^\alpha\|_{C^{3,0,\gamma}(\partial_S V'\times[0,T])},
    \ \|b^\alpha\|_{C^{1,0,\gamma}(\partial_S V'\times[0,T])}
    \leq \Lambda,\label{SchauderCh_BoundarygloablieAss03}
\end{align}
and all $u\in C^ {4,1,\gamma}(V'\times[0,T])$ 
\begin{align*} 
\|u\|_{C^ {4,1,\gamma}(V\times[0,T])}\leq C&\left(
\|Pu\|_{C^ {0,0,\gamma}(V'\times[0,T])}
+\|u(\cdot,0)\|_{C^ {4,\gamma}(V')}
+\|B_1 u\|_{C^{3,0,\gamma}(\partial_S V'\times[0,T])}\right.\\
&\left.+\|B_2 u\|_{C^{1,0,\gamma}(\partial_S V'\times[0,T])}
+\|u\|_{L^ \infty(V'\times[0,T])}
\right).
\end{align*}
\end{korollar}
 \begin{tcolorbox}[colback=white!20!white,colframe=black!100!white,sharp corners, breakable]
The important observation in Lemma \ref{SchauderCh_BoundarySchauderEstimateCorollary} lies in the fact that the constant $C(V,V',T_0,\gamma, \Lambda,\theta)$ does not depend on $T$.
\end{tcolorbox}
\begin{proof}
The proof is the same as the one for Corollary \ref{SchauderCh_InteriorSchauderEstimateCorollary}. Instead of Lemma \ref{SchauderCh_InteriorCoveringLemma} we use the covering provided by Lemma \ref{SchauderCh_BoundaryCoveringLemma} and then use either Lemma 
\ref{SchauderCh_GanzraumLokalisierungVariableCoeff},
\ref{SchauderCh_HalbraumObenLokalisierungVariableCoeff}, 
\ref{SchauderCh_HalbraumUntenLokalisierungVariableCoeff},
\ref{SchauderCh_HalbraumBCLokalisierungVariableCoeff},
\ref{SchauderCh_HalbraumUntenBCLokalisierungVariableCoeff}
or
\ref{SchauderCh_HalbraumObenBCLokalisierungVariableCoeff}
 depending on whether $U_{kj,2\rho}$ is of the form $U_{2\rho}(q)$, $U_{2\rho}^+(q)$, $U_{2\rho}^-(q)$, $U_{2\rho+}(q)$, $U_{2\rho+}^+(q)$ or $U_{2\rho+}^-(q)$.
\end{proof}

\chapter{Schauder Theory on Manifolds}\label{SchauderCh_Ch04SchauderTheorMfd}
Throughout this chapter, let $M$ be a smooth, compact, connected $n$-dimensional manifold with boundary $\partial M\neq\emptyset$.

\section{Schauder Estimates on Manifolds}\label{SchauderCh_Ch4Section1}
We first establish the definition of elliptic and parabolic operators and suitable boundary conditions on $M$. To do so, we fix a good atlas $\mathcal A=\set{(\varphi_\nu, U_\nu)\ |\ 1\leq \nu\leq m}$ on $M$. Additionally, we put $V_\nu:=\varphi_\nu(U_\nu)$. The following definitions make reference to this choice. However, following arguments similar to the ones in the proofs of Lemmas \ref{SchauderCh_EquivalenceofHolderSpaceDefinition} and \ref{SchauderCh_NormEquivalenceParabolicHolderSpace}, it is easy to see that these definitions are independent of the choice of good atlas. 

\begin{definition}[2-Elliptic Operators on $M$]
    Let $L:C^ 4(M)\rightarrow C^ 0(M)$ be a linear, differential operator. We say that $L$ is a \emph{2-elliptic operator of class $C^{0,\gamma}(M)$}, if for all $u\in C^4(M)$
\begin{equation}\label{SchauderCh_LDefEq}
(Lu)\circ\varphi_\nu^{-1}=a_{\varphi_\nu}^ {ij}a_{\varphi_\nu}^ {kl}\partial_{ijkl}(u\circ\varphi_\nu^{-1})+\sum_{|\alpha|\leq 3}A_{\varphi_\nu}^ \alpha\nabla_\alpha (u\circ\varphi_\nu^{-1})\vspace{-.2cm}
\end{equation}
with coefficients $a^{ij}_{\varphi_\nu}, A^\alpha_{\varphi_\nu}\in C^{0,\gamma}(V_\nu)$ and strictly elliptic coefficients $a^ {ij}_{\varphi_\nu}$. 
\end{definition}
In the following, we require a measure for the \emph{quality} of 2-elliptic operators. For that purpose, we define 
\begin{align*} 
&\|L\|_{C^{0,\gamma}_{\mathcal A}(M)}:=\max_{1\leq \nu\leq m}\left[\max_{|\alpha|\leq 3}\|A^\alpha_{\varphi_\nu}\|_{C^{0,\gamma}(V_\nu)}+\max_{1\leq i,j\leq n}\|a^{ij}_{\varphi_\nu}\|_{C^{0,\gamma}(V_\nu)}\right],\\
&\Theta_{\mathcal A}(L):=\sup\{\theta>0\ |\ a^{ij}_{\varphi_\nu}(x)\xi_i\xi_j\geq\theta|\xi|^2\textrm{ for all }1\leq \nu\leq m,\ x\in V_\nu,\ \xi\in\R^n\}.
\end{align*}
Following arguments similar to the proof of Lemma \ref{SchauderCh_EquivalenceofHolderSpaceDefinition}, it is easy to see that for two good atlases $\mathcal A$ and $\mathcal B$ there is a constant $C=C(\mathcal A,\mathcal B,\gamma)>0$ such that 
\begin{equation}\label{SchauderCh_QualityEstimat01}
\|L\|_{C^{0,\gamma}_{\mathcal B}(M)}\leq C\|L\|_{C^{0,\gamma}_{\mathcal A}(M)}
\hspace{.5cm}\textrm{and}\hspace{.5cm}
\Theta_{\mathcal B}(L)\geq \frac1C\Theta_{\mathcal A}(L).
\end{equation}
While e.g. stating $\Theta(L)\geq 1$ does not make sense without making reference to a concrete good atlas $\mathcal A$, it is a meaningful thing to write $\Theta(L_s)\geq \theta_0>0$ for a family of operators $L_s$ with $s\in[0,1]$ without having to resort to a good atlas. However, it has to be understood that the concrete value of $\theta_0$ depends on the choice of good atlas.\\ 

Next, we introduce $L$-compatible boundary conditions. To motivate this definition, we consider a model problem. Assuming we are given a Riemannian metric $g$ on $M$, we can consider the operator $L=\Delta_g^2$. Denoting the exterior unit normal along $\partial M$ by $\nu$ and the Laplace-Beltrami operator by $\Delta_g$, we can define natural boundary operators by
$$B_1u:=\frac{\partial u}{\partial\nu}\hspace{.5cm}\textrm{and}\hspace{.5cm}B_2 u:=\frac{\partial\Delta_g u}{\partial\nu}.$$
In local coordinates, they are given by the formulas 
\begin{equation}\label{SchauderCh_canonicalExample}
B_1u=-\left(g^{nn}\right)^{-\frac12}g^{ni}\partial_i u
\hspace{.5cm}\textrm{and}\hspace{.5cm}
B_2 u=-\left(g^{nn}\right)^{-\frac12}g^{ni}g^{kl}\partial_{ikl}u+\textrm{`lower order terms'}.
\end{equation}
Finally, we recall that for a domain $V\subset\R^{n-1}\times[0,\infty)$ we put $\partial_SV:=V\cap(\R^{n-1}\times 0)$. 

\begin{definition}[$L$-compatible Boundary Conditions]
Let $L$ be a 2-elliptic operator of class $C^{0,\gamma}(M)$. We say that a linear operator $B=(B_1,B_2):C^ 4(M)\rightarrow C^ 3(\partial M)\times C^ 1(\partial M)$ defines \emph{$L$-compatible boundary conditions of class} $C^{3,\gamma}(\partial M)\times C^{1,\gamma}(\partial M)$, if for all $u\in C^4(M)$ 
\begin{align*} 
(B_1 u)\circ\varphi_\nu^{-1}&=\left(a^{nn}_{\varphi_\nu}\right)^{-\frac12}a^ {ni}_{\varphi_\nu}\partial_i (u\circ\varphi_\nu^{-1})+c_{\varphi_\nu} (u\circ\varphi_\nu^{-1}),\\
(B_2 u)\circ\varphi_\nu^{-1}&=\left(a^{nn}\right)^{-\frac12}a^ {ni}a^ {kl}\partial_{ikl} (u\circ\varphi_\nu^{-1})+\sum_{|\alpha|\leq 2}b_{\varphi_\nu}^ \alpha \nabla_\alpha (u\circ\varphi_\nu^{-1}),
\end{align*}
with $a^{ij}_{\varphi_\nu}$ as in Equation \eqref{SchauderCh_LDefEq} and $a^{ij}_{\varphi_\nu}, c_{\varphi_\nu}\in C^{3,\gamma}(\partial_S V_\nu)$ and $b^\alpha_{\varphi_\nu}\in C^{1,\gamma}(\partial_S V_\nu)$ .
\end{definition}

We introduce a measure for the quality of such an operator $B$ by defining 
\begin{align*} 
&\|B_1\|_{C^{3,\gamma}_{\mathcal A}(\partial M)}:=\max_{1\leq\nu\leq m}\left[\max_{1\leq i\leq n}\|\left(a^{nn}_{\varphi_\nu}\right)^{-\frac12}a^ {ni}_{\varphi_\nu}\|_{C^{3,\gamma}(\partial_SV_\nu)}+\|c_{\varphi_\nu}\|_{C^{3,\gamma}(\partial_SV_\nu)}\right],\\
&\|B_2\|_{ C^{1,\gamma}_{\mathcal A}(\partial M)}:=\max_{1\leq\nu\leq m}\left[\max_{1\leq i,k,l\leq n}\|\left(a^{nn}_{\varphi_\nu}\right)^{-\frac12}a^ {ni}_{\varphi_\nu}a^ {kl}_{\varphi_\nu}\|_{C^{1,\gamma}(\partial_SV_\nu)}+\max_{|\alpha|\leq 2}\|b^\alpha_{\varphi_\nu}\|_{C^{1,\gamma}(\partial_SV_\nu)}\right].
\end{align*}
Additionally, we note that the analog of the first estimate in Estimate \eqref{SchauderCh_QualityEstimat01} is true.\\

Given a $2$-elliptic operator $L$, we can consider the associated parabolic operator $P:=\partial_t+L$. However, we also want to allow for parabolic operators with time-dependent coefficients. 

\begin{definition}[2-Parabolic Operators on \textrm{$M\times[0,T]$}]
Let $T>0$ and $P:C^ {4,1}(M\times[0,T])\rightarrow C^{0,0}(M\times[0,T])$ be a linear, differential operator. We say that $P$ is a \emph{2-parabolic operator of class} $C^{0,0,\gamma}(M\times[0,T])$, if for all $u\in C^{4,1}(M)$
\begin{equation}\label{SchauderCh_PDefEq}
(Pu)\circ\varphi_\nu^{-1}=
\partial_t(u\circ\varphi_\nu^{-1})
+
a_{\varphi_\nu}^ {ij}a_{\varphi_\nu}^ {kl}\partial_{ijkl}(u\circ\varphi_\nu^{-1})+\sum_{|\alpha|\leq 3}A_{\varphi_\nu}^ \alpha\nabla_\alpha (u\circ\varphi_\nu^{-1})
\end{equation}
with coefficients $a^{ij}_{\varphi_\nu}, A^\alpha_{\varphi_\nu}\in C^{0,0,\gamma}(V_\nu\times[0,T])$ and strictly elliptic $a^ {ij}_{\varphi_\nu}$.
\end{definition}
As in the elliptic case, we introduce a measure for the quality of a 2-parabolic operator. 
\begin{align*} 
&\|P\|_{C^{0,0,\gamma}_{\mathcal A}(M\times[0,T])}:=\max_{1\leq \nu\leq m}\left[\max_{|\alpha|\leq 3}\|A^\alpha_{\varphi_\nu}\|_{C^{0,0,\gamma}(V_\nu\times[0,T])}+\max_{1\leq i,j\leq n}\|a^{ij}_{\varphi_\nu}\|_{C^{0,0,\gamma}(V_\nu\times[0,T])}\right],\\
&\Theta_{\mathcal A}(P):=\sup\{\theta>0\ |\ a^{ij}_{\varphi_\nu}(x,t)\xi_i\xi_j\geq\theta|\xi|^2\textrm{ for all }1\leq \nu\leq m,\ (x,t)\in V_\nu\times[0,T],\ \xi\in\R^n\}.
\end{align*}
Additionally, we note that the analogs of the estimates in Estimate \eqref{SchauderCh_QualityEstimat01} are true. As in the elliptic case, we now introduce $P$-compatible boundary conditions. 

\begin{definition}[$P$-compatible Boundary Conditions]
Let $T>0$ and $P$ be a 2-parabolic operator. We say that a linear operator $B=(B_1,B_2):C^ {4,1}(M\times[0,T])\rightarrow C^ {3,0}(\partial M\times[0,T])\times C^{1,0}(\partial M\times[0,T])$ defines $P$-compatible boundary conditions of class $C^{3,0,\gamma}(\partial M\times[0,T])\times C^{1,0,\gamma}(\partial M\times[0,T])$, if for all $u\in C^{4,1}(M\times[0,T])$ 
\begin{align*} 
(B_1 u)\circ\varphi_\nu^{-1}&=\left(a^{nn}_{\varphi_\nu}\right)^{-\frac12}a^ {ni}_{\varphi_\nu}\partial_i (u\circ\varphi_\nu^{-1})+c_{\varphi_\nu} (u\circ\varphi_\nu^{-1}),\\
(B_2 u)\circ\varphi_\nu^{-1}&=\left(a^{nn}_{\varphi_\nu}\right)^{-\frac12}a^ {ni}_{\varphi_\nu}a^ {kl}_{\varphi_\nu}\partial_{ikl} (u\circ\varphi_\nu^{-1})+\sum_{|\alpha|\leq 2}b_{\varphi_\nu}^ \alpha \nabla_\alpha (u\circ\varphi_\nu^{-1}),
\end{align*}
with $a^{ij}_{\varphi_\nu}$ as in Equation \eqref{SchauderCh_PDefEq} and $a^{ij}_{\varphi_\nu}, c_{\varphi_\nu}\in C^{3,0,\gamma}(\partial_S V_\nu\times[0,T])$, $b^\alpha_{\varphi_\nu}\in C^{1,0,\gamma}(\partial_S V_\nu\times[0,T])$.
\end{definition}

Again, we introduce a measure for the quality of such a boundary operator $B$ by defining 
\begin{align*} 
&\|B_1\|_{C^{3,0,\gamma}_{\mathcal A}(\partial M\times[0,T])}:=\max_{1\leq\nu\leq m}\left[\max_{1\leq i\leq n}\|\left(a^{nn}_{\varphi_\nu}\right)^{-\frac12}a^ {ni}_{\varphi_\nu}\|_{C^{3,0,\gamma}(\partial_SV_\nu\times[0,T])}+\|c_{\varphi_\nu}\|_{C^{3,0,\gamma}(\partial_SV_\nu\times[0,T])}\right],\\
&\|B_2\|_{ C^{1,0,\gamma}_{\mathcal A}(\partial M\times[0,T])}:=\max_{1\leq\nu\leq m}\left[\max_{1\leq i,k,l\leq n}\|\left(a^{nn}_{\varphi_\nu}\right)^{-\frac12}a^ {ni}_{\varphi_\nu}a^ {kl}_{\varphi_\nu}\|_{C^{1,0,\gamma}(\partial_SV_\nu\times[0,T])}+\max_{|\alpha|\leq 2}\|b^\alpha_{\varphi_\nu}\|_{C^{1,0,\gamma}(\partial_SV_\nu\times[0,T])}\right].
\end{align*}
Additionally, we note that the analog of the first estimate in Estimate \eqref{SchauderCh_QualityEstimat01} is true.\\

To proceed, we require Ehrling's Lemma. For a proof, we refer to \cite{renardy2006introduction}, Chapter 7, Section 2, Theorem 7.30. 

\begin{proposition}[Ehrling's Lemma]    
Let $X$, $Y$ and $Z$ be Banach spaces such that $X\hookrightarrow\hookrightarrow Y$ compactly and $Y\hookrightarrow Z$ continuously. Then for all $\epsilon>0$ there exists a constant $C(\epsilon)>0$ such that for all $x\in X$
    $$\|x\|_Y\leq \epsilon \|x\|_X+C(\epsilon) \|x\|_Z.$$
\end{proposition}

Using Corollaries \ref{SchauderCh_InteriorSchauderEstimateCorollary} and \ref{SchauderCh_BoundarySchauderEstimateCorollary}, we can now deduce Schauder estimates on $M$. The following estimate contains an $L^2$-norm. Since we have not equipped $M$ with a Riemannian metric, we stress that this norm should be understood in the sense of Definition \ref{SchauderCh_SobolevNormDefGoodTalas} and not the $L^2$-norm induced by a Riemannian metric. 
\begin{theorem}\label{SchauderCh_GoodAPrioriEstimateOnManifold}
Let $T_0>0$, $T\geq T_0$, $\mathcal A$ be a good atlas on $M$, $P$ be a 2-parabolic operator of class $C^{0,0,\gamma}(M\times[0,T])$ and let $B$ define $P$-compatible boundary conditions of class $C^{3,0,\gamma}(\partial M\times[0,T])\times C^{1,0,\gamma}(\partial M\times[0,T])$. Then, for any $u\in C^{4,1,\gamma}(M\times[0,T])$
\begin{align*} 
\|u\|_{C^{4,1,\gamma}_{\mathcal A}(M\times[0,T])}
\leq &
C\bigg(\|Pu\|_{C^{0,0,\gamma}_{\mathcal A}(M\times[0,T])}+\|u(\cdot,0)\|_{C^{4,\gamma}_{\mathcal A}(M)}+\|B_1u\|_{C^{3,0,\gamma}_{\mathcal A}(\partial M\times[0,T])}\\
&\hspace{2cm}+\|B_2u\|_{C^{1,0,\gamma}_{\mathcal A}(\partial M\times[0,T])}+\sup_{0\leq t\leq T}\|u(\cdot,t)\|_{L^2_{\mathcal A}(M)}\bigg).
\end{align*}
Here $C$ depends on $M$, $T_0$, $\gamma$, $\Lambda:=\|P\|_{C^{0,0,\gamma}_{\mathcal A}(M\times[0,T])}+\|B_1\|_{C^{3,0,\gamma}_{\mathcal A}(\partial M\times[0,T])}+\|B_2\|_{C^{1,0,\gamma}_{\mathcal A}(\partial M\times[0,T])}$ and $\theta:=\Theta_{\mathcal A}(P)$ but not on $T$.
\end{theorem}

\begin{proof}
Let $\mathcal A=\set{(\varphi_\nu,U_\nu)\ |\ 1\leq\nu\leq m}$. By definition, there exist open sets $\tilde U_\nu$ such that $U_\nu\subset\subset \tilde U_\nu$ and such that $\varphi_\nu$ has a smooth extension onto $\tilde U_\nu$. For each $1\leq\nu\leq m$, we choose an open set $U_\nu'$ such that $U_\nu\subset\subset U_\nu'\subset\subset \tilde U_\nu$. Note that $\mathcal A':=\set{(\varphi_\nu,U_\nu')\ |\ 1\leq\nu\leq m}$ is a good atlas on $M$. By Estimate \eqref{SchauderCh_QualityEstimat01} and the analog estimates, the operator $P$ and the boundary operator $B$ have norms bonded by $C\Lambda$ and an ellipticity constant bounded from below by $C^{-1}\theta$ with respect to $\mathcal A'$. Here $C$ depends only on $\mathcal A$, $\mathcal A'$ and $\gamma$. We put $V_\nu:=\varphi_\nu(U_\nu)$ and $V_\nu':=\varphi_\nu(U_\nu')$.\\

Let $u\in C^{4,1,\gamma}(M\times[0,T])$. For $1\leq\nu\leq m$ we put $v_\nu:=u\circ\varphi_\nu^{-1}$ and denote the local representation of $P$ by $P_\nu$. If $U_\nu$ is a boundary chart, we denote the local representation of the boundary operators by $(B_{1\nu},B_{2\nu})$.\\

Depending on whether $U_\nu$ is an interior or a boundary chart, we use Corollary \ref{SchauderCh_InteriorSchauderEstimateCorollary} or \ref{SchauderCh_BoundarySchauderEstimateCorollary} respectively to deduce that with a constant $C=C(V_\nu, V_\nu', T_0,\Lambda,\theta,\gamma)$ 
\begin{align*} 
\|v_\nu\|_{C^{4,1,\gamma}(V_\nu\times[0,T])}
\leq &C
\left(
\|P_\nu v_\nu\|_{C^{0,0,\gamma}(V_\nu'\times[0,T])}+\|v_\nu(\cdot, 0)\|_{C^{4,\gamma}(V_\nu')}+\|v_\nu\|_{L^\infty(V_\nu'\times[0,T])}
\right),\\
\|v_\nu\|_{C^ {4,1,\gamma}(V_\nu\times[0,T])}\leq &C\left(
\|P_\nu v_\nu\|_{C^ {0,0,\gamma}(V_\nu'\times[0,T])}
+\|v_\nu(\cdot,0)\|_{C^ {4,\gamma}(V_\nu')}
\right.\\
&\left.+\|B_{1\nu}v_\nu\|_{C^{3,0,\gamma}(\partial_S V_\nu'\times[0,T])}+\|B_{2\nu}v_\nu\|_{C^{1,0,\gamma}(\partial_S V_\nu'\times[0,T])}
+\|v_\nu\|_{L^ \infty(V_\nu'\times[0,T])}
\right).
\end{align*}
Taking the maximum over all $1\leq\nu\leq m$, we deduce with $C=C(\mathcal A,\mathcal A', T_0,\Lambda,\theta,\gamma)$
\begin{align*} 
\|u\|_{C^ {4,1,\gamma}_{\mathcal A}(M\times[0,T])}
\leq&C\left(
\|Pu\|_{C^ {0,0,\gamma}_{\mathcal A'}(M\times[0,T])}
+\|u(\cdot,0)\|_{C^ {4,\gamma}_{\mathcal A'}(M)}
+\|B_1 u\|_{C^{3,0,\gamma}_{\mathcal A'}(\partial M\times[0,T])}\right.\\
&\left.+\|B_2 u\|_{C^{1,0,\gamma}_{\mathcal A'}(\partial M\times[0,T])}
+\|u\|_{L_{\mathcal A'}^ \infty(M\times[0,T])}
\right).
\end{align*}
As $\mathcal A$ and $\mathcal A'$ are both good atlases, we can use Lemmas \ref{SchauderCh_EquivalenceofHolderSpaceDefinition} and \ref{SchauderCh_NormEquivalenceParabolicHolderSpace} to deduce that with $C=C(\mathcal A,\mathcal A', T_0, \Lambda,\theta,\gamma)$ \emph{independent of $T$}
\begin{align} 
\|u\|_{C^ {4,1,\gamma}_{\mathcal A}(M\times[0,T])}
\leq C&\left(
\|Pu\|_{C^ {0,0,\gamma}_{\mathcal A}(M\times[0,T])}
+\|u(\cdot,0)\|_{C^ {4,\gamma}_{\mathcal A}(M)}
+\|B_1 u\|_{C^{3,0,\gamma}_{\mathcal A}(\partial M\times[0,T])}\right.\nonumber\\
&\left.+\|B_2 u\|_{C^{1,0,\gamma}_{\mathcal A}(\partial M\times[0,T])}
+\|u\|_{L_{\mathcal A}^ \infty(M\times[0,T])}\label{SchauderCh_AlmostLiftedEstimat}
\right).
\end{align}
As $\mathcal A'$ depends only on $\mathcal A$, we can drop the $\mathcal A'$ dependence in the constant. Finally, we note that $C_{\mathcal A}^{4,\gamma}(M)$ embeds compactly into $L_{\mathcal A}^\infty(M)$ and $L_{\mathcal A}^\infty(M)$ embeds continuously into $L_{\mathcal A}^2(M)$. So, by Ehrling's lemma, we know that for any small $\epsilon>0$
$$
\|u\|_{L^\infty_{\mathcal A}(M\times[0,T])}
=\sup_{t\in[0,T]}\left(\|u\|_{L^\infty_{\mathcal A}(M\times t)}\right)
\leq \sup_{t\in[0,T]}\left(\epsilon \|u(\cdot, t)\|_{C^{4,\gamma}_{\mathcal A}(M)}+C(\epsilon)\|u(\cdot, t)\|_{L^2_{\mathcal A}(M)}\right).
$$
Inserting into Estimate \eqref{SchauderCh_AlmostLiftedEstimat} and choosing $\epsilon$ small enough we can absorb $\sup_{t\in[0,T]}\|u(\cdot, t)\|_{C^{4,\gamma}_{\mathcal A}(M)}$ to the left side and the theorem follows.
\end{proof}

\section{Canonical Boundary Conditions}\label{SchauderCh_Ch4Sec2}
Throughout this section, we assume that $L$ is a $2$-elliptic operator with smooth coefficients and that $B$ defines $L$-compatible boundary conditions with smooth coefficients. Considering Theorem \ref{SchauderCh_GoodAPrioriEstimateOnManifold}, our aim is to derive an a priori estimate of the form 
\begin{align*}
\sup_{0\leq t\leq T}\|u(\cdot,t)\|_{L^2_{\mathcal A}(M)}\leq &
C\bigg(\|\dot u+Lu\|_{C^{0,0,\gamma}_{\mathcal A}(M\times[0,T])}+\|u(\cdot,0)\|_{C^{4,\gamma}_{\mathcal A}(M)}+\|B_1u\|_{C^{3,0,\gamma}_{\mathcal A}(\partial M\times[0,T])}\\
&\hspace{2cm}+\|B_2u\|_{C^{1,0,\gamma}_{\mathcal A}(\partial M\times[0,T])}\bigg).
\end{align*}
To do this, we will specialize to \emph{canonical boundary conditions}. To derive these, we collect some well-known facts from Riemannian geometry. For a thorough introduction, we refer to Lee's book \cite{LeeRiem} (see Chapter 4 and Appendix B).
\begin{bemerkung}[Tensor Fields]
Let $(N,s)$ be a Riemannian manifold. 
\begin{enumerate}[(1)]
\item A smooth $(p,q)$ tensor field is a $C^\infty(N)$-multilinear map 
$$T:\underbrace{T^*N\times...\times T^*N}_{\textrm{$p$ copies}}\times \underbrace{TN\times...\times TN}_{\textrm{$q$ copies}}\rightarrow C^\infty(N).$$
Such a tensor field $T$ has the local representation
$$T=T^{i_1,...,i_p}_{\ \ \ \ \ \ \ k_1...k_q}\ \frac{\partial }{\partial x^{i_1}}\otimes...\otimes \frac{\partial}{\partial x^{i_p}}\otimes dx^{k_1}\otimes...\otimes dx^{k_q}.$$
\item Let $T$ be a smooth $(p,q)$ tensor field on $N$ and let $\nabla$ denote the Levi-Cevita connection induced by $s$. Then $\nabla T$ is a smooth $(p,q+1)$ tensor field defined by 
$$(\nabla T)(\omega^1,...,\omega^p, X_1,...,X_q,Y)=(\nabla_Y T)(\omega^1,...,\omega^p,X_1,...,X_q).$$
The divergence $\nabla^*T$ of $T$ is a smooth $(p,q-1)$ tensor field defined by
$$\nabla^*T(\omega^1,...,\omega^p,X_1,...,X_{q-1}):=s^{ij}\nabla_{\frac\partial{\partial x ^i}}T(\omega^1,...,\omega^p, X_1,...,X_{q-1},\frac{\partial}{\partial x^j}).$$
Here $\frac\partial{\partial x^i}$ is the basis induced by an arbitrary chart and $s^{ij}$ is the inverse metric tensor expressed in the same chart. 
\item If $T$ and $S$ are smooth tensor fields of types $(p,0)$ and $(0,p)$ respectively, we define their \emph{pairing} as $\langle T,S\rangle\in C^\infty(N)$ by the local formula
$$\langle T,S\rangle=T^{i_1...i_p}S_{i_1...i_p}.$$
\end{enumerate}
\end{bemerkung}
After this preparation, we present the following lemma.
\begin{lemma}\label{SchauderCh_LGeometricFormula}
There exists a smooth metric $g$ on $M$, a smooth $(3,0)$-tensor field $J$ on $M$ and a second order linear elliptic operator $\mathcal L$ with smooth coefficients such that
$$Lu=\Delta_g^2u+\langle J,\nabla^3u\rangle+\mathcal Lu .$$
\end{lemma}
\begin{proof}
Let $\varphi:U\rightarrow V$ and $\tilde\varphi:\tilde U\rightarrow\tilde V$ be good charts such that $U\cap\tilde U\neq\emptyset$. For $u\in C^\infty(M)$, we put $v:=u\circ\varphi^{-1}$, $\tilde v:=u\circ\tilde \varphi^{-1}$ and $\chi:=\tilde\varphi\circ\varphi^{-1}:\varphi(U\cap\tilde U)\rightarrow\tilde\varphi(U\cap\tilde U)$. Denoting derivatives of $v$ and $\tilde v$ of order $\leq 3$ by `lower order terms', we compute
\begin{align*}
    a_\varphi^{ij}a_\varphi^{kl}\partial_{ijkl} v +\textrm{`lower order terms'}
    =&(Lu)\circ\varphi^{-1}\\
    =& (Lu)\circ(\tilde \varphi^{-1}\circ\chi)\\
    =&\left( a^{ij}_{\tilde\varphi} a^{kl}_{\tilde\varphi}\partial_{ijkl}\tilde v +\textrm{`lower order terms'}\right)\circ\chi
\end{align*}
Using $v=\tilde v\circ\chi$, we deduce 
$$ a^{ij}_{\tilde\varphi} a^{kl}_{\tilde\varphi}\partial_{ijkl}\tilde v=\partial_i\chi^a\partial_j\chi^b\partial_k\chi^c\partial_l\chi^da^{ij}_{\varphi}a^{kl}_{\varphi}\partial_{abcd}\tilde v+\textrm{`lower order terms'}.$$
As $\tilde v$ is arbitrary we deduce that $ a_{\tilde\varphi}^{ij}=\partial_a\chi^i\partial_b\chi^j a_\varphi^{ab}$ and hence we can define a smooth $(2,0)$-tensor field in local coordinates by $g^{ij}:=a^{ij}_\varphi$. As $a^{ij}_\varphi$ is an elliptic matrix, we can define a Riemannian metric using the inverse $g_{ij}$. Since in local coordinates $\Delta_g=g^{ij}\partial_{ij}+\textrm{`lower order terms'}$, we deduce that $L=\Delta_g^2+\textrm{`lower order terms'}$. We consider the operator $T:=L-\Delta_g^2$. This is a third-order operator and so locally $T=J^{ijk}\partial_{ijk} +\textrm{`lower order terms'}$. Using the formula $T=L-\Delta_g^2$, we also see that $T$ is a geometric object. 
Using the same arguments we have used before, we can use $J^{ijk}$ to define a smooth $(3,0)$-tensor field $J$.
Using that in local coordinates 
$(\nabla_g^3 u)_{ijk}=\partial_{ijk}u+\textrm{`lower order terms'}$, we deduce
$$Tu=J^{ijk}(\nabla^3 u)_{ijk}+\textrm{lower order terms}=\langle J,\nabla^3 u\rangle +\textrm{lower order terms}.$$
As $T$ is a geometric operator and $\langle J,\nabla^3 u\rangle$ is geometric, we deduce that $\mathcal Lu:=Tu-\langle J,\nabla^3 u\rangle$ defines a geometric second order operator and so the lemma follows. 
\end{proof}

Using Lemma \ref{SchauderCh_LGeometricFormula}, we can now formulate \emph{canonical boundary conditions} for $L$.

\begin{definition}[Canonical Boundary Conditions]\label{SchauderCh_CanonicalBCDef}
    Let $g$ and $J$ be as in Lemma \ref{SchauderCh_LGeometricFormula} and denote by $\nu$ the exterior unit normal along $\partial M$ with respect to $g$. The boundary conditions 
    $$
\left\{
\begin{aligned}
    B_1(u):=&\frac{\partial u}{\partial\nu},\\
    B_2(u):=&\frac{\partial\Delta_g u}{\partial\nu}+\langle J,\nabla^2 u\otimes\nu^\flat\rangle,
\end{aligned}
\right.
$$
are called \emph{canonical boundary conditions associated to $L$}. Here $\nu^\flat:=g(\nu,\cdot)\in T^*\partial M$.
\end{definition}

To justify the name canonical, we recall the following fact from Riemannian geometry. For a guided proof, we refer to an exercise in Lee's book \cite{LeeRiem} (see Chapter 5, Problem 5-16). 

\begin{proposition}
    Let $(N,s)$ be a Riemannian manifold, $\nu$ denote the exterior unit normal along $\partial M$ and let $T$ and $S$ be smooth tensor fields of types $(p,0)$ and $(0,p-1)$ respectively. Then 
    $$\int_N \langle T,\nabla S\rangle d\mu_s=-\int_N\langle \nabla^*T,S\rangle d\mu_s+\int_{\partial N}\langle T,S\otimes\nu^\flat\rangle dS_s.$$
    Here $\nu^\flat:=s(\nu,\cdot)\in T^*\partial N$.
\end{proposition}

\begin{lemma}\label{SchauderCh_CanonicalBDLemma}
The canonical boundary conditions $B=(B_1,B_2)$ are $L$-compatible. Additionally, putting 
\begin{equation}\label{SchauderCh_BetaDefinition}
\beta[u,v]:=\int_M \Delta_gv \Delta_g ud\mu_g-\int_M\langle \nabla_g^*(vJ),\nabla^2 u\rangle d\mu_g+\int_M v\mathcal L ud\mu_g,
\end{equation}
for $u, v\in C^4(M)$, we have the formula
$$\int_M vLu d\mu_g=\beta[u,v]+\int_{\partial M}B_2(u)v-\Delta_g v B_1(u)dS_g.$$
\end{lemma}
\begin{proof}
The fact that the canonical boundary conditions are $L$-compatible follows from the local formulas in Equation \eqref{SchauderCh_canonicalExample}. To prove the formula, we use Lemma \ref{SchauderCh_LGeometricFormula} and compute
    \begin{align*}
    \int_M v Lud\mu_g=&\int_M v \Delta_g^2 ud\mu_g+\int_M\langle vJ,\nabla^3 u\rangle d\mu_g+\int_M v\mathcal L ud\mu_g\\
    =&\int_M \Delta_gv \Delta_g ud\mu_g-\int_M\langle \nabla_g^*(vJ),\nabla^2 u\rangle d\mu_g+\int_M v\mathcal L ud\mu_g\\
    &+\int_{\partial M}\left(v\frac{\partial\Delta_g u}{\partial\nu}-\Delta_g v\frac{\partial u}{\partial\nu}\right) dS_g+\int_{\partial M}\langle v J, \nabla^2 u\otimes\nu^\flat \rangle dS_g.
\end{align*}
\end{proof}

Next, we establish an estimate for $\beta[u,u]$ that we will use as a form of a coercivity estimate. To do this, we require the following definition.

\begin{definition}
    Let $T$ be a $(p,q)$ tensor field on $M$ of class $C^k$ and $\mathcal A=\set{(\varphi_\nu, U_\nu)\ |\ 1\leq\nu\leq m}$ be a good atlas on $M$. Let $T^{i_1...i_p}_{\varphi_\nu,\ \ \ j_1...j_q}$ be the coordinates of $T$ in the chart $\varphi_\nu$. We define 
    $$\|T\|_{C^k_{\mathcal A}(M)}:=\max_{1\leq\nu\leq m}\max_{1\leq i_1,...,i_p\leq n}\max_{1\leq j_1,...,j_q\leq n}\|T^{i_1...i_p}_{\varphi_\nu,\ \ \ j_1...j_q}\|_{C^k(V_\nu)}.$$
\end{definition} 

\begin{tcolorbox}[colback=white!20!white,colframe=black!100!white,sharp corners, breakable]
 \begin{bemerkung}[Conventions for the Rest of this Section]\label{SchauderCh_ConventionsL2AprioriSection}
For the rest of this section, we fix a good atlas $\mathcal A$, denote by $B=(B_1,B_2)$ the canonical boundary conditions associated to $L$ and by $J$ the $(3,0)$ tensor field provided by Lemma \ref{SchauderCh_LGeometricFormula}. We define
\begin{align*}
    &\Lambda:=\|L\|_{C^{0,\gamma}_{\mathcal A}(M)}+\|B_1\|_{C^{3,\gamma}_{\mathcal A}(M)}+\|B_2\|_{C^{1,\gamma}_{\mathcal A}(M)}+\|J\|_{C^1_{\mathcal A}(M)}.
\end{align*}
Additionally, we denote the metric from Lemma \ref{SchauderCh_LGeometricFormula} by $g$. Note that $\Theta_{\mathcal A}(L)$ depends only on $g$.
\end{bemerkung}
\end{tcolorbox}

\begin{lemma}[Approximate Coercivity]\label{SchauderCh_ApproximateCoercivityLemma}
There exist $C=C(M,\gamma,g,\Lambda)>0$ and $\kappa=\kappa(M,\gamma,g,\Lambda)>0$ such that for all $u\in C^4(M)$
$$
\beta[u,u]\geq\kappa \|u\|_{W^{2,2}_{\mathcal A}(M)}^2-C\left(\|u\|_{L^2_{\mathcal A}(M)}^2+\|B_1u\|_{L^2_{\mathcal A}(\partial M)}^2+\|u\|_{C^2_{\mathcal A}(M)}\|B_1u\|_{C^1_{\mathcal A}(\partial M)}\right).
$$
\end{lemma}
\begin{proof}
Throughout the proof, let $C$ denote a constant that is allowed to depend only on $M$,$\gamma$, $g$ and $\Lambda$. Throughout the proof, we will drop the subscript $\mathcal A$ on all norms. Additionally, we will use the equivalence of the norms to use the norms $\|\cdot\|_{L^2_{\mathcal A}(M)}$ and $\|\cdot\|_{L^2_g(M)}$ interchangeably. Let $u\in C^4(M)$ and $h:=\frac{\partial u}{\partial\nu}$. Noting that $\mathcal L$ is a second-order operator, we estimate  
$$
\left|
\int_M\langle\nabla_g^*(uJ),\nabla^2 u\rangle-u\mathcal Lud\mu_g
\right|
\leq 
C\|u\|_{W^{1,2}(M)}\|u\|_{W^{2,2}(M)}
$$
and deduce 
\begin{equation}\label{SchauderCh_BetaEstimate01}
\beta[u,u]\geq \int_M(\Delta_ gu)^2d\mu_g-C\|u\|_{W^{1,2}(M)}\|u\|_{W^{2,2}(M)}.
\end{equation}
We now derive suitable estimates for $\|u\|_{W^{1,2}(M)}$ and $\|u\|_{W^{2,2}(M)}$. Using Lemma \ref{SchauderCh_TraceInequalityLemma}, we estimate 
\begin{align*}
    \int_M|\nabla u|^2d\mu_g&=\int_{\partial M}\frac{\partial u}{\partial\nu} udS_g-\int_M u\Delta_g u d\mu_g\\
    &\leq \|h\|_{L^2(\partial M)}\|u\|_{L^2(\partial M)}+\|u\|_{L^2(M)}\|\Delta_g u\|_{L^2(M)}\\
    &\leq C\left(\|h\|_{L^2(\partial M)}^2+\|u\|_{L^2(M)}\|\Delta_g u\|_{L^2(M)}\right)+C(\epsilon)\|u\|_{L^2( M)}^2+\epsilon\|\nabla u\|_{L^2(M)}^2.
\end{align*}
Absorbing $\|\nabla u\|_{L^2(M)}^2$ to the left, we deduce 
\begin{equation}\label{SchauderCh_GradientL2Estimate}
\|u\|_{W^{1,2}(M)}^2
\leq C\left(\|h\|_{L^2(\partial M)}^2+\|u\|_{L^2( M)}^2+\|u\|_{L^2(M)}\|\Delta_g u\|_{L^2(M)}\right).
\end{equation}
Next, we derive an estimate for $\|u\|_{W^{2,2}(M)}^2$. To do this, we use Bochner's identity\footnote{For a guided proof, we refer to Lee's book \cite{LeeRiem} (see Chapter 7, Problem 7-7)} 
$$\frac12\Delta_g |\nabla u|^2=g(\nabla \Delta_g u,\nabla u)+|\nabla^2 u|^2+\operatorname{Ric}_g(\nabla u,\nabla u).$$
Integrating over $M$ and recalling $h:=\frac{\partial u}{\partial\nu}$, we get 
\begin{equation}\label{SchauderCh_Coercivity01}
\int_M(\Delta_g u)^2d\mu_g
=
\int_M|\nabla^2 u|^2+\operatorname{Ric}_g(\nabla u,\nabla u)d\mu_g+\int_{\partial M}\Delta_g uh-g(\nabla_\nu\nabla u,\nabla u)dS_g.
\end{equation}
In local coordinates $\nabla^2 u=\partial_{ij} u-\Gamma^k_{\ ij}\partial_k u$ and $\operatorname{Ric}_g(\nabla u,\nabla u)=\operatorname{Ric}_g^{ij}\partial_ iu\partial_j u$. Hence, there exists (a perhaps small) constant $c_0>0$ and a constant $C>0$ such that 
\begin{align} 
&\int_M|\nabla^2 u|^2d\mu_g\geq c_0\|u\|_{W^{2,2}(M)}^2-C\|u\|_{W^{1,2}(M)}^2,\label{SchauderCh_Coercivity02}\\
&\left|\int_M\operatorname{Ric}_g(\nabla u,\nabla u)d\mu_g\right|\leq C\int_M|\nabla u|^2 d\mu_g \leq C\|u\|_{W^{1,2}(M)}^2,\label{SchauderCh_Coercivity03}\\
&\left|\int_{\partial M}\Delta_g u hdS_g\right|\leq C\|u\|_{C^2(M)}\|h\|_{L^2(\partial M)}.\label{SchauderCh_Coercivity04}
\end{align}
Next, we estimate the remaining boundary integral in Equation \eqref{SchauderCh_Coercivity01}. Let $p\in\partial M$ and denote the $g_p$-orthogonal projection from $T_p M$ onto $T_p\partial M$ by $\pi_{T_p\partial M}$. Using $\frac{\partial u}{\partial\nu}=h$, we write
\begin{align} 
\nabla u(p)=&\nabla u(p)-g_p(\nabla u(p),\nu(p))\nu(p) +g(\nabla u(p),\nu(p))\nu(p)\nonumber\\
=&\pi_{T_p\partial M}(\nabla u(p))+\frac{\partial u}{\partial\nu}\bigg|_p\nu(p)\nonumber\\
=&\nabla^{\partial M} u(p)+h(p)\nu(p).\label{SchauderCh_TangentialGradientComputation}
\end{align}
We recall the formulas $\nabla^2 u[X,Y]=g(\nabla_X\nabla u, Y)=\nabla_X(\nabla_Y u) -du(\nabla_XY)$. Since $\nabla^{\partial M}u\in T\partial M$ is a tangential vector field along $\partial M$, we can use Equation \eqref{SchauderCh_TangentialGradientComputation} to compute
\begin{align*} 
g(\nabla_\nu\nabla u,\nabla u)=&g(\nabla_\nu\nabla u,\nabla^{\partial M}u)+h g(\nabla_\nu\nabla u,\nu)\\
=&\nabla^2 u[\nu,\nabla^{\partial M}u]+h \nabla^2 u[\nu,\nu]\\
=&\nabla_{\nabla^{\partial M}u}\nabla_\nu u-du(\nabla_{\nabla^{\partial M} u}\nu)+h\nabla^2 u[\nu,\nu]\\
=&\nabla_{\nabla^{\partial M}u}h -du(\nabla_{\nabla^{\partial M} u}\nu)+h\nabla^2 u[\nu,\nu].
\end{align*}
The first and third terms contain products of derivatives of $u$ and $h$ or orders at most 2 and 1, respectively. The second term is quadratic in $\nabla u$. Hence
\begin{align}
    \int_{\partial M}|g(\nabla_\nu\nabla u,\nabla u)|dS_g\leq & C\|h\|_{C^1(\partial M)}\|u\|_{C^2(M)}+C\int_{\partial M}|\nabla u|^2 dS_g.\label{SchauderCh_Coercivity05}
\end{align}
Inserting Estimates \eqref{SchauderCh_Coercivity02}, \eqref{SchauderCh_Coercivity03}, \eqref{SchauderCh_Coercivity04} and \eqref{SchauderCh_Coercivity05} into Estimate \eqref{SchauderCh_Coercivity01} and subsequently inserting the trace estimate from Lemma \ref{SchauderCh_TraceInequalityLemma}, we deduce that for all small $\epsilon>0$ 
\begin{align*}
    c_0\|u\|_{W^{2,2}(M)}^2\leq &\|\Delta_gu\|_{L^2(M)}^2+C\left(\|u\|_{W^{1,2}(M)}^2+\|u\|_{C^2(M)}\|h\|_{C^1(\partial M)}+\|\nabla u\|_{L^2(\partial M)}^2\right)\\
    \leq & \|\Delta_gu\|_{L^2(M)}^2+C\|u\|_{C^2(M)}\|h\|_{C^1(\partial M)}+C(\epsilon)\|u\|_{W^{1,2}(M)}^2+\epsilon\|u\|_{W^{2,2}(M)}^2.
\end{align*}
For small enough $\epsilon$, we deduce that with a (perhaps small) constant $c_1>0$
\begin{equation}\label{SchauderCh_Coercivity06}
    \|\Delta_gu\|_{L^2(M)}^2\geq c_1 \|u\|_{W^{2,2}(M)}^2-C\left(\|u\|_{C^2(M)}\|h\|_{C^1(\partial M)}+\|u\|_{W^{1,2}(M)}^2\right).
\end{equation}
Let $\delta>0$ be a small parameter. Using Estimates \eqref{SchauderCh_BetaEstimate01} and \eqref{SchauderCh_Coercivity06} and choosing $\delta>0$ small enough, we estimate  
\begin{align*} 
\beta[u,u]\geq&\int_M(\Delta_gu)^2d\mu_g -C(\delta)\|u\|_{W^{1,2}(M)}^2-\delta\|u\|_{W^{2,2}(M)}^2\\
\geq&c_1 \|u\|_{W^{2,2}(M)}^2-C\|u\|_{C^2(M)}\|h\|_{C^1(\partial M)}-C(\delta)\|u\|_{W^{1,2}(M)}^2-\delta\|u\|_{W^{2,2}(M)}^2\\
\geq&\frac{2c_1}3 \|u\|_{W^{2,2}(M)}^2-C\|u\|_{C^2(M)}\|h\|_{C^1(\partial M)}-C\|u\|_{W^{1,2}(M)}^2.
\end{align*}
Let $\mu>0$ be another small parameter. Inserting Estimate \eqref{SchauderCh_GradientL2Estimate}, using $\|\Delta_g u\|_{L^2(M)}^2\leq C\|u\|_{W^{2,2}(M)}^2$ and choosing $\mu$ small enough, we deduce 
\begin{align*} 
\beta[u,u]
\geq&\frac{2c_1}3 \|u\|_{W^{2,2}(M)}^2-C\left(\|u\|_{C^2(M)}\|h\|_{C^1(\partial M)}+\|h\|_{L^2(\partial M)}^2\right)-C(\mu)\|u\|_{L^2(M)}^2-\mu\|\Delta_g u\|_{L^2(M)}^2\\
\geq&\frac{c_1}2 \|u\|_{W^{2,2}(M)}^2-C\left(\|u\|_{C^2(M)}\|h\|_{C^1(\partial M)}+\|h\|_{L^2(\partial M)}^2+\|u\|_{L^2(M)}^2\right).
\end{align*}
\end{proof}

Using the approximate coercivity of $\beta$, we can now eliminate the $L^2$-norm of $u$ in Theorem \ref{SchauderCh_GoodAPrioriEstimateOnManifold}.
\begin{korollar}\label{SchauderCh_SchauderEstimateonmanifolds}
Let $T>0$, $L$ be a 2-elliptic operator with smooth coefficients and $P:=\partial_t+L$. Let further $B=(B_1,B_2)$ denote the canonical boundary conditions associated to $L$. Using the conventions from Remark \ref{SchauderCh_ConventionsL2AprioriSection}, there exists a constant $C=C(M,T,\gamma,g,\Lambda)$ such that for all $u\in C^{4,1,\gamma}(M\times[0,T])$
\begin{align*} 
\|u\|_{C^{4,1,\gamma}_{\mathcal A}(M\times[0,T])}
\leq &
C\bigg(\|Pu\|_{C^{0,0,\gamma}_{\mathcal A}(M\times[0,T])}+\|u(\cdot,0)\|_{C^{4,\gamma}_{\mathcal A}(M)}+\|B_1u\|_{C^{3,0,\gamma}_{\mathcal A}(\partial M\times[0,T])}\\
&\hspace{2cm}+\|B_2u\|_{C^{1,0,\gamma}_{\mathcal A}(\partial M\times[0,T])}\bigg).
\end{align*}
 
\end{korollar}
 \begin{tcolorbox}[colback=white!20!white,colframe=black!100!white,sharp corners, breakable]
By eliminating the $L^2$-norm from Theorem \ref{SchauderCh_GoodAPrioriEstimateOnManifold}, the constant $C$ in Corollary \ref{SchauderCh_SchauderEstimateonmanifolds} does generally depend on $T$ and not only on a lower bound $T_0$. Also note, that $C$ depends on $\Theta_{\mathcal A}(L)$ via $g$.
\end{tcolorbox}
\begin{proof}
Throughout the proof, we will only make the dependence on $T$ explicit in the notation. Also we will use the equivalence of the norms $\|\cdot\|_{L^2_{\mathcal A}(M)}$ and $\|\cdot\|_{L^2_g(M)}$ to use them interchangeably. Finally, we will drop the subscripts $\mathcal A$ or $g$ on all norms.\\

We define $f:=\dot u+Lu$, $h_1:=B_1 u$ and $h_2:=B_2 u$. Using Lemmas \ref{SchauderCh_CanonicalBDLemma} and \ref{SchauderCh_ApproximateCoercivityLemma}, we estimate 
\begin{align*}
&\frac12\frac d{dt} \int_M u(t)^2 d\mu_g=\int_M u(t)(f(t)-Lu(t))d\mu_g\\
    =&\int_M f(t)u(t) d\mu_g-\beta[u(t),u(t)]-\int_{\partial M}h_2(t)u(t) -h_1(t)\Delta_g u(t)dS_g\\
    \leq & -\int_{\partial M}u(t)h_2(t) -h_1(t)\Delta_g u(t)dS_g-\kappa\|u(t)\|_{W^{2,2}(M)}^2\\
    &+C\bigg(\|f(t)\|_{L^2(M)}^2+\|u(t)\|_{L^2(M)}^2+\|h_1(t)\|_{L^2(\partial M)}^2+\|u(t)\|_{C^2(M)}\|h_1(t)\|_{C^1(\partial M)}\bigg).
\end{align*}
Let $\epsilon>0$ be a small parameter and $Q_0:=\|f\|_{C^{0,0,\gamma}(M\times[0,T])}+\|h_1\|_{C^{3,0,\gamma}(\partial M\times[0,T])}+\|h_2\|_{C^{1,0,\gamma}(\partial M\times[0,T])}$. Using Young's inequality we get that for some constant $C_1$ independent of $\epsilon$ and another constant $C(\epsilon)$
$$\frac12 \frac d{dt}\|u(t)\|_{L^2(M)}^2\leq C_1\|u(t)\|_{L^2(M)}^2+C(\epsilon)Q_0^2+\epsilon\|u\|_{C^{4,1,\gamma}(M\times[0,T])}^2.$$
Multiplying with the integrating factor $e^{-2C_1 t}\leq 1$, we deduce
\begin{align*} 
\frac{d}{dt}\left[e^{-2C_1 t}\|u(t)\|_{L^2(M)}^2\right]
\leq &2\left(C(\epsilon)Q_0^2+\epsilon\|u\|_{C^{4,1,\gamma}(M\times[0,T])}^2\right).
\end{align*}
Integrating this inequality, we obtain the estimate
$$
\sup_{t\in[0,T]}\|u(t)\|_{L^2(M)}^2\leq e^{2C_1 T}\|u_0\|_{L^2(M)}^2+T e^{2C_1T}C(\epsilon)Q_0^2+\epsilon Te^{2C_1T}\|u\|_{C^{4,1,\gamma}(M\times[0,T])}^2.
$$
Inserting into the estimate from Theorem \ref{SchauderCh_GoodAPrioriEstimateOnManifold}, we get
\begin{align*} 
\|u\|_{C^{4,1,\gamma}(M\times[0,T])}\leq & C(T,\epsilon)\bigg(\|Pu\|_{C^{0,0,\gamma}(M\times[0,T])}+\|u(\cdot,0)\|_{C^{4,\gamma}(M)}+\|B_1u\|_{C^{3,0,\gamma}(\partial M\times[0,T])}\\
&\hspace{2cm}+\|B_2u\|_{C^{1,0,\gamma}(\partial M\times[0,T])}\bigg)+\epsilon C(T)\|u\|_{C^{4,1,\gamma}(M\times[0,T])}.
\end{align*}
Choosing $\epsilon=\epsilon(T)$ small enough, we can absorb $\|u\|_{C^{4,1,\gamma}}$ to the left-hand side and deduce the corollary. 
\end{proof}

\section{Existence of Solutions}\label{SchauderCh_Ch04Sec3}
Throughout this section, we assume that $g$ is a smooth  Riemannian metric on $M$. We denote the Laplace-Beltrami operator by $\Delta_g$ and the exterior unit normal along $\partial M$ by $\nu$.
\subsection{The Galerkin Approximation}
We wish to construct solutions by expanding into Neumann eigenfunctions. To deduce convergence of the resulting sequences, we require the Arzelà-Ascoli theorem. This result is well-known in the literature. As an example, we point to \cite{hirzebruch}, Chapter 1, Section 3, Theorem 3.10.
\begin{theorem}[Arzelà-Ascoli Theorem]
Let $(A,d)$ be a compact metric space and $(B,d')$ be a metric space. Let $f_k:A\rightarrow B$ be a sequence of maps and assume the following: 
\begin{enumerate}[(1)]
    \item $f_k$ is uniformly bounded. That is there exists $b_0\in B$ and $R>0$ such that $f_k(A)\subset B_R^{d'}(b_0)$ for all $k$.
    \item $f_k$ is equicontinuous. That is: For every $\epsilon>0$, there exists $\delta>0$ such that whenever $d(a,\tilde a)<\delta$ we have $d'(f_k(a), f_k(\tilde a))<\epsilon$ for all $k$.
    \item For each $a\in A$ the sequence $f_k(a)$ has a convergent subsequence.
\end{enumerate}
Then there exists $f\in C^0(A,B)$ such that along a subsequence $f_k\rightarrow f$ in $C^0(A,B)$ as $k\rightarrow\infty$. 
\end{theorem}

We require the following corollary:
\begin{korollar}\label{SchauderCh_CompactEmbeddingCorollary}
Let $X$ and $Y$ be Banach spaces such that $X\hookrightarrow\hookrightarrow Y$ compactly, $T>0$ and $m\in\N_0$. Then $C^{m+1}([0,T],X)\hookrightarrow\hookrightarrow C^m([0,T], Y)$ compactly.
\end{korollar}
\begin{proof}
We first consider the case $m=0$. Let $f_k\subset C^1([0,T], X)$ be bounded. We prove that the sequence $(f_k)\subset C^0([0,T],Y)$ satisfies the conditions from the Arzelà-Ascoli theorem. Clearly $(f_k)\subset C^0([0,T],Y)$ is uniformly bounded. Let $s,t\in[0,T]$. We estimate 
$$\|f_k(t)-f_k(s)\|_Y\leq C\|f_k(t)-f_k(s)\|_X\leq \sup_{t\in[0,T]}\|f_k'(t)\|_X |t-s|\leq C\sup_{k\in\N}\|f_k\|_{C^1([0,T],X)}|t-s|.$$
Consequently, the uniform continuity of $f_k$ follows. Finally, let $t\in[0,T]$. Since $f_k(t)$ is bounded in $X$ and $X\hookrightarrow\hookrightarrow Y$ compactly, there exists a subsequence $k_l$ such that $f_{k_l}(t)$ converges in $Y$.\\

We now argue by induction. If the corollary is established for some $m$ and $(f_k)\subset C^{m+2}([0,T],X)$ is bounded then $(f_k)\subset C^{m}([0,T],X)$ is bounded and $(f_k')\subset C^{m}([0,T],X)$ is also bounded. By the inductive hypothesis, we can pass to a subsequence that we again denote by $k$ such that $f_k\rightarrow \tilde f$ in $C^{m}([0,T], Y)$ and $f_k'\rightarrow \tilde g$ in $C^{m}([0,T], Y)$. We now prove that $\tilde f'=\tilde g$. Let $\varphi\in Y'$ with $\|\varphi\|_{Y'}=1$. Since $m\geq 0$ we have $\varphi\circ f_k\in C^2([0,T],\R)$ with $(\varphi\circ f_k)'=\varphi\circ f_k'$ and $(\varphi\circ f_k)''=\varphi\circ f_k''$. Given $t\in (0,T)$ and small $h\in\R$ we compute
\begin{align} 
\left|\varphi\left(\frac{f_k(t+h)-f_k(t)}h-f_k'(t)\right)\right|=&\left|\int_0^1\int_0^{sh} \varphi (f_k''(t+r))drds\right|\nonumber\\
\leq& \|\varphi\|_{Y'}\sup_{t\in[0,T]}\|f_k''(t)\|_Y\left|\int_0^1\int_0^{sh} drds\right|\nonumber\\
=&\frac12 \sup_{t\in[0,T]}\|f_k''(t)\|_Y|h|.\label{SchauderCh_ArzelaAscoliCorollaryeq1}
\end{align}
Note that $\sup_{t\in[0,T]}\|f_k''(t)\|_Y\leq\sup_k\|f_k\|_{C^2([0,T],Y)}=:\Lambda<\infty$. So, Estimate \eqref{SchauderCh_ArzelaAscoliCorollaryeq1} is uniform over all $k$. Taking the supremum over all $\varphi\in Y'$ with $\|\varphi\|_{Y'}=1$ and subsequently letting $k\rightarrow\infty$ shows 
$$\left\|\frac{\tilde f(t+h)-\tilde f(t)}h-\tilde g(t)\right\|_Y=\lim_{k\rightarrow\infty}\left\|\frac{f_k(t+h)-f_k(t)}h-f_k'(t)\right\|_Y\leq \Lambda |h|.$$
This shows that as a map from $(0,T)$ into $Y$, $\tilde f$ is differentiable with $\tilde f'=\tilde g$. Since $\tilde g\in C^{m}([0,T], Y)$ we deduce that $\tilde f\in C^{m+1}([0,T], Y)$. 
Finally, using that 
$f_k\rightarrow \tilde f$ in $C^{m}([0,T], Y)$ and $f_k'\rightarrow \tilde g=\tilde f'$ in $C^{m}([0,T], Y)$ we deduce that $f_k\rightarrow\tilde f$ in $C^{m+1}([0,T], Y)$.
\end{proof}

\begin{theorem}\label{SchauderCh_GalerkinExistenceTheorem}
There exists $N(n)$ with the following property: For all $f\in C^\infty(M\times[0,T])$ satisfying
$$\frac{\partial}{\partial\nu}\Delta_g^m f=0\hspace{.5cm}\textrm{for all $0\leq m\leq N$}$$
there exists a function $u\in C^{4,1,\gamma}(M\times[0,T])$ such that 
\begin{equation}\label{SchauderCh_GalerkinProjectedProblem}
\left\{
\begin{array}{ll}
\displaystyle\dot u+\Delta_g^2u=f&\displaystyle\textrm{in $M\times[0,T]$},\vspace{.2cm}\\
\displaystyle\frac{\partial u}{\partial\nu}=\frac{\partial\Delta_g u}{\partial\nu}=0&\displaystyle\textrm{along $\partial M\times[0,T]$},\vspace{.2cm}\\
\displaystyle u(\cdot,0)=0& \displaystyle\textrm{on $M$}.
\end{array}
\right.
\end{equation}
\end{theorem}
\begin{proof}
By Sobolev embedding, we can fix $N(n)$ such that $W^{2N,2}(M)\hookrightarrow\hookrightarrow C^{0,\gamma}(M)$ compactly. Using Corollary \ref{SchauderCh_CompactEmbeddingCorollary} we get $C^2([0,T], W^{2N,2})\hookrightarrow\hookrightarrow C^1([0,T], C^{0,\gamma}(M))$ compactly and by the definition of the space $C^{0,0,\gamma}(M\times[0,T])$ we have the continuous inclusion $C^1([0,T], C^{0,\gamma}(M))\hookrightarrow C^{0,0,\gamma}(M\times[0,T])$. Hence
\begin{equation}\label{SchauderCh_CompactEmbeddinggParabolic}
C^2([0,T], W^{2N,2}(M))\hookrightarrow\hookrightarrow C^{0,0,\gamma}(M\times[0,T])\hspace{.5cm}\textrm{compactly}.
\end{equation}
Let $\lambda_k$ and $\phi_k$ denote the Neumann eigenvalues and eigenfunctions from Theorem \ref{SchauderCh_neumannlaplacianResultsTheorem}. For $k\geq 0$, we define $c_k(t):=\langle \phi_k, f(t)\rangle_{L^2_g(M)}$ and for $\mu\in\N$ we define $f_\mu:=\sum_{j=0}^\mu c_j(t)\phi_j$ as well as
$$u_\mu(p,t):=\sum_{j=0}^\mu e^{-\lambda_j^2 t}\int_0^t e^{\lambda_j^2 s}c_j(s)ds\phi_j(p).$$
A direct computation shows that
\begin{equation}\label{SchauderCh_ParabolicApproxProb}
\left\{
\begin{array}{ll}
\displaystyle\dot u_\mu+\Delta_g^2u_\mu=f_\mu&\displaystyle\textrm{in $M\times[0,T]$},\vspace{.2cm}\\
\displaystyle\frac{\partial u_\mu}{\partial\nu}=\frac{\partial\Delta_g u_\mu}{\partial\nu}=0&\displaystyle\textrm{along $\partial M\times[0,T]$},\vspace{.2cm}\\
\displaystyle u_\mu(\cdot,0)=0& \displaystyle\textrm{on $M$}.
\end{array}
\right.
\end{equation}
We note that $B_1u:=\frac\partial{\partial\nu}$ and $B_2u:=\frac{\partial\Delta_g u}{\partial\nu}$ are the canonical boundary conditions associated to $\Delta_g^2$. We may therefore apply the Schauder estimate from Corollary \ref{SchauderCh_SchauderEstimateonmanifolds} to the functions $u_\mu-u_\nu$ and get
$$\|u_\mu-u_\nu\|_{C^{4,1,\gamma}(M\times[0,T])}
\leq C\|f_\mu-f_\nu\|_{C^{0,0,\gamma}(M\times[0,T])}.$$
We prove that up to a subsequence, which we again denote by $\mu$, $f_\mu\rightarrow f$ in $C^{0,0,\gamma}(M\times[0,T])$. Once this is shown, we deduce that $u_\mu$ is Cauchy in $C^{4,1,\gamma}(M\times[0,T])$ and hence convergent in $C^{4,1,\gamma}(M\times[0,T])$. Putting $u:=\lim_{\mu\rightarrow\infty} u_\mu$ and passing to the limit in Equation \eqref{SchauderCh_ParabolicApproxProb}, we verify that $u$ is indeed a solution of Equation \eqref{SchauderCh_GalerkinProjectedProblem}.\\

First note that for each fixed $t\in[0,T]$ we may use Theorem \ref{SchauderCh_neumannlaplacianResultsTheorem} to deduce that $f_\mu(\cdot,t)\rightarrow f(\cdot, t)$ in $L^2(M)$. So, if along some subsequence $f_\mu$ is convergent in $C^{0,0,\gamma}(M\times[0,T])$, the limit has to be $f$ as well. Using \eqref{SchauderCh_CompactEmbeddinggParabolic}, it is sufficient to show that $f_\mu$ is bounded in $C^2([0,T],W^{2N,2}(M))$ to deduce that $f_\mu$ is convergent along a subsequence in $C^{0,0,\gamma}(M\times[0,T])$. For $l=0,1,2$ and $k\leq N$ we may use the assumption on $f$ and integration by parts to compute
\begin{align*}
    \lambda_j^{2k}\partial_t^l c_j(t)=&\langle (-\Delta_g)^ k\phi_j, \partial_t^lf(t)\rangle_{L^2_g(M)}
    =\langle \phi_j, (-\Delta_g)^ k\partial_t^lf(t)\rangle_{L^2_g(M)}.
\end{align*}
Consequently, we learn that for all $k\leq N$ and $l=0,1,2$ we have the estimate 
\begin{align}
\sum_{j=0}^ \infty \sup_{t\in[0,T]}|\lambda_j^ {2k}\partial_t^lc_j(t)|^2\nonumber
\leq & \sup_{t\in[0,T]}\sum_{l=0}^2\sum_{j=0}^\infty \langle \phi_j, (-\Delta_g)^ k\partial_t^lf(t)\rangle_{L^2_g(M)}^2\nonumber\\
= & \sup_{t\in[0,T]}\sum_{l=0}^2 \|\Delta_g^k\partial_t^l f(t)\|_{L^2_g(M)}^2\nonumber\\
\leq &C(M,k)\|f\|_{C^2([0,T], W^{2N,2}(M))}^2.\label{SchauderCh_ParabolicGalerkinlambdapowerestimate}
\end{align}
For $0\leq k\leq N$ and $l=0,1,2$ we can use Estimate \eqref{SchauderCh_ParabolicGalerkinlambdapowerestimate} and the fact that $\phi_j$ is an $L^2$-orthonormal system to estimate 
$$
\|\partial_t^l\Delta_g^k f_\mu(t)\|_{L^2(M)}^2
=\|\sum_{j=0}^\mu(-\lambda_j)^k\partial_t^lc_l(t) \phi_j\|_{L^2(M)}^2=\sum_{j=0}^\mu\lambda_j^{2k}(\partial_t^l c_j(t))^2 \leq C(M,N)\|f\|_{C^2([0,T], W^{2N,2}(M))}^2.
$$
Since $\phi_j\in C^\infty(M)$ and $f\in C^\infty(M\times[0,T])$ we get $f_\mu\in C^\infty(M\times[0,T])$. Additionally, we can use that  $\partial_\nu\Delta^k\phi_j=0$ for all $j,k\in\N_0$ to deduce that $\partial_\nu\Delta^k\partial_t^l f_\mu=0$ for all $k,l\in\N_0$ and $\mu\in\N$. Hence, we can use Lemma \ref{SchauderCh_Wk2EstimateinTermsOfLaplace_Lemma} to estimate
$$ 
\|\partial_t^l f_\mu(t)\|_{W^{2N,2}(M)}
\leq C(M,N)\sum_{k=0}^N \|\partial_t^l \Delta_g^kf_\mu(t)\|_{L^2(M)}\leq C(M,N)\|f\|_{C^2([0,T], W^{2N,2}(M))}^2.
$$
Therefore $f_\mu$ is bounded in $C^2([0,T], W^{2N,2}(M))$ and the theorem follows. 
\end{proof}

\subsection{Initial and Boundary Values}\label{SchauderCh_SubSection_InitialAndBoundaryValues}
Next, we wish to extend the existence result from Theorem \ref{SchauderCh_GalerkinExistenceTheorem} to allow for general $f\in C^{0,0,\gamma}(M\times[0,T])$ as well as initial and boundary values. To do this, we require some preparational results. The following theorem can, for example, be found in \cite{lee2012smooth} (see Chapter 9, Theorem 9.25).

\begin{theorem}[Collar Neighbourhood Theorem]\label{SchauderCh_CollarNeighbourhoodTheorem}
There exists a neighbourhood $U\subset M$ of $\partial M$ and a smooth diffeomorphism $F:\partial M\times[0,1)\rightarrow U$ with $\Psi(p)=(p,0)$ for $p\in\partial M$.
\end{theorem}

\begin{lemma}\label{SchauderCh_ExistenceofphijfunctionsLemma}
For each $\mu\in\N_0$ there exists a smooth function $\varphi_\mu\in C^\infty(M)$ such that 
$$\frac{\partial}{\partial\nu}\Delta_g^\mu\varphi_\mu=1\hspace{.5cm}\textrm{and}\hspace{.5cm}\nabla^k\varphi_\mu|_{\partial M}\equiv 0\hspace{.5cm}\textrm{for all $k\leq 2\mu$}.$$
\end{lemma}
\begin{proof}
Let $U\subset M$ be a chart with $U\cap\partial M\neq\emptyset$ and $\phi:U\rightarrow V\subset \R^{n-1}\times[0,\infty)$ be the chart map. On $V$ we define 
$$f_\mu:V\rightarrow\R,\ f_\mu(x):=\frac{-1}{(g^{nn})^{\mu+\frac12}}\frac{x_n^{2\mu+1}}{(2\mu+1)!}.$$
Using the formulas 
$\frac{\partial}{\partial\nu}=-\frac{g^{ni}}{\sqrt{g^{nn}}}\partial_i$ and $\Delta_g=g^{ij}\partial_{ij}-g^{ij}\Gamma^k_{\ ij}\partial_k$, it is straight forward to check that $f_\mu\circ \phi\in C^\infty(U,\R)$ has the desired properties on $U\cap\partial M$.\\

Let $U_1,...,U_m$ be charts that intersect and cover $\partial M$, $\eta_i$ be a subordinate partition of unity and put $U:=\bigcup_{i=1}^m U_i$. For $1\leq j\leq m$ we construct $\varphi^{(j)}_\mu:=f^{(j)}_\mu\circ \phi_j$ as above and define
$$\varphi_\mu:=\sum_{j=1}^{m}\eta_j \varphi^{(j)}_\mu\in C^\infty(U,\R).$$
Finally, let $U'\subset\subset U$ be a neighbourhood of $\partial M$ and $\zeta\in C^\infty(M)$ such that $\operatorname{supp}(\zeta)\subset U$ and $\zeta\equiv 1$ on $U'$. Then, by extending with $0$ on $M\backslash U$, the function $\zeta \varphi_\mu\in C^\infty(M)$ has the desired properties. 
\end{proof}

\begin{lemma}[Existence of an Extension Operator]\label{SchauderCh_SmoothExistenceOperatorExistenceLemma}
There exists a linear operator $E:C^\infty(\partial M)\rightarrow C^\infty(M)$ such that $E[f]\big|_{\partial M}=f$ for any $f\in C^\infty(\partial M)$.
\end{lemma}
\begin{proof}
Using the Collar Neighbourhood Theorem \ref{SchauderCh_CollarNeighbourhoodTheorem}, there exists a neighbourhood $U$ of $\partial M$ and a diffeomorphism $F:U\rightarrow\partial M\times[0,1)$ with $F(p)=(p,0)$ for all $p\in\partial M$. We write $F(p)=(y(p), z(p))$. Given $f\in C^\infty(\partial M)$ we first define 
$$\tilde f\in C^\infty(U,\R),\ \tilde f(x):=f(y(x)).$$
Now, let $U'\subset\subset U$ be a neighbourhood of $\partial M$ and $\zeta\in C^\infty(M,[0,1])$ with $\zeta\equiv 1$ on $U'$ and $\operatorname{supp}(\zeta)\subset U$. We define 
$E[f]:=\zeta \tilde f$ by extending with $0$ outside of $U$.
\end{proof}

We now to extend the existence result from Theorem \ref{SchauderCh_GalerkinExistenceTheorem} to general $f\in C^{0,0,\gamma}(M\times[0,T])$.
\begin{korollar}\label{SchauderCh_ApproximationforgeneralfParabolic}
For every $f\in C^{0,0,\gamma}(M\times[0,T])$ there exists a function $u\in C^{4,1,\gamma}(M\times[0,T])$ such that 
\begin{equation}\label{SchauderCh_ExistenceLemmaProblem}
\left\{
\begin{array}{ll}
\displaystyle\dot u+\Delta_g^2u=f&\displaystyle\textrm{in $M\times[0,T]$},\vspace{.2cm}\\
\displaystyle\frac{\partial u}{\partial\nu}=\frac{\partial\Delta_g u}{\partial\nu}=0&\displaystyle\textrm{along $\partial M\times[0,T]$},\vspace{.2cm}\\
\displaystyle u(\cdot,0)=0& \displaystyle\textrm{on $M$}.
\end{array}
\right.
\end{equation}
\end{korollar}
\begin{proof}
First we consider the case where $f\in C^ \infty(M\times[0,T])$. Let $N\in\N$ be as in Theorem \ref{SchauderCh_GalerkinExistenceTheorem}. For $0\leq j\leq N$, we use Lemma \ref{SchauderCh_ExistenceofphijfunctionsLemma} and choose smooth functions $\varphi_j$ that satisfy 
$$\frac{\partial}{\partial\nu}\Delta^ j\varphi_j=1\hspace{.5cm}\textrm{and}\hspace{.5cm}\nabla^k\varphi_j\big|_{\partial M}=0\hspace{.5cm}\textrm{for all $0\leq k\leq 2j$.}$$
Using the extension operator $E$ from Lemma \ref{SchauderCh_SmoothExistenceOperatorExistenceLemma}, we define a function $T_N\in C^\infty(M)$ recursively by
$$T_0:=\varphi_0E\left[\frac{\partial f}{\partial\nu}\right]
\hspace{.5cm}\textrm{and}\hspace{.5cm}
T_{k+1}:=T_k+\varphi_{k+1}E\left(\frac{\partial\Delta^ {k+1}f}{\partial\nu}-\frac{\partial \Delta^ {k+1}T_k}{\partial\nu}\right).$$
By construction, $T_N$ is smooth and satisfies
$$\frac{\partial\Delta^ j T_N}{\partial\nu}=\frac{\partial\Delta^ jf}{\partial\nu}\hspace{.5cm}\textrm{for all $0\leq j\leq N$}
\hspace{.5cm}\textrm{and}\hspace{.5cm}
T_N|_{\partial M}=0.$$
Let $\Omega$ be a collar neighbourhood of $\partial M$, i.e. a neighbourhood such that there exists a diffeomorphism $\Psi=(\Psi_1,\Psi_2):\Omega\rightarrow \partial M\times [0,1)$ with $\Psi(p)=(p,0)$ for $p\in\partial M$. For small $\epsilon>0$ we choose a cutoff function $\zeta_\epsilon\in C^ \infty([0,1),[0,1])$ with the properties 
$$\zeta_\epsilon\equiv 1\textrm{ on $[0,\epsilon]$},\hspace{.5cm}\zeta_\epsilon\equiv 0 \textrm{ on $[2\epsilon,1]$}\hspace{.5cm}\textrm{and}\hspace{.5cm}\|\zeta_\epsilon'\|_{C^1([0,1])}\leq \frac{C}\epsilon.$$
On $\Omega$ we define $\eta_\epsilon(p):=\zeta_\epsilon(\Psi_2(p))$. Extending $\eta_\epsilon$ by zero to all of $M$, we obtain a smooth function, again denoted by $\eta_\epsilon$, on $M$ with support in a small neighborhood $\Omega_\epsilon$ of the boundary. We put 
$$F_{\epsilon}:=f-\eta_\epsilon T_N$$
and claim that
$F_\epsilon\rightarrow f$ in $C^{0,0,\gamma}(M\times[0,T])$. Showing this is, however, quite involved and we postpone it to the end of the proof. Since $\eta_\epsilon\equiv 1$ on a neighbourhood of $\partial M$, we have
$$\frac{\partial}{\partial\eta}\Delta^ j F_{\epsilon}=0\hspace{.5cm}\textrm{for all $0\leq j\leq N$}.$$
Using Theorem \ref{SchauderCh_GalerkinExistenceTheorem}, there exist functions $u_\epsilon\in C^{4,1,\gamma}(M\times[0,T])$ such that $u_\epsilon$ solves problem \eqref{SchauderCh_ExistenceLemmaProblem} with $f$ replaced by $F_\epsilon$. 
Since $F_\epsilon\rightarrow f$ in $C^{0,0,\gamma}(M\times[0,T])$, we can use the Schauder estimate from Corollary \ref{SchauderCh_SchauderEstimateonmanifolds} to deduce that $(u_\epsilon)$ is Cauchy in $C^{4,1,\gamma}(M\times[0,T])$. We then obtain a solution to problem \eqref{SchauderCh_ExistenceLemmaProblem} as the limit $u:=\lim_{\epsilon\rightarrow 0^+} u_\epsilon$.\\

For general $f\in C^{0,0,\gamma}(M\times[0,T])$, we choose a sequence of smooth functions $f_k$ such that $f_k\rightarrow f$ in $C^{0,0,\gamma}(M\times[0,T])$. By the first part of the proof, we get a sequence $(u_k)\subset C^{4,1,\gamma}(M\times[0,T])$ such that $u_k$ solves Problem \ref{SchauderCh_ExistenceLemmaProblem} with $f$ replaced by $f_k$. Using the Schauder estimate from Corollary \ref{SchauderCh_SchauderEstimateonmanifolds}, we deduce that $(u_k)\subset C^{4,1,\gamma}(M\times[0,T])$ is Cauchy and hence convergent. The limit $u:=\lim_{k\rightarrow \infty}u$ is the sought-after solution.\\
 
To see that $F_\epsilon\rightarrow f$ in $C^{0,0,\gamma}(M\times[0,T])$ it suffices to show that $T_N\eta_\epsilon\rightarrow 0$ in $C^{0,0,\gamma}(M)$. Let $\mathcal A=\set{(\phi_\nu,U_\nu)\ |\ 1\leq\nu\leq m}$ be a good atlas on $M$ and $V_\nu:=\phi_\nu(U_\nu)$. Denoting the $x$-argument by $\cdot_1$ and the time argument by $\cdot_2$, we need to prove that for all $1\leq\nu\leq m$
\begin{equation}\label{SchauderCh_ApproximationsAreRealyApproximation}
\lim_{\epsilon\rightarrow0^+}\|T_N(\phi_\nu^{-1}(\cdot_1),\cdot_2)\eta_\epsilon(\phi_\nu^{-1}(\cdot_1 ))\|_{C^{0,0,\gamma}(V_\nu\times[0,T])}\rightarrow0.
\end{equation}
If $U_\nu$ is an interior chart, establishing Equation \eqref{SchauderCh_ApproximationsAreRealyApproximation} is trivial as $U_\nu\cap \operatorname{supp}(\eta_\epsilon)=\emptyset$ for all small enough values of $\epsilon$. Next, we consider a boundary chart. For the sake of simplicity in the representation, we drop the index $\nu$ and consider a chart $\phi:U\rightarrow V$ where $U\cap\partial M\neq\emptyset$ and $V\subset\R^{n-1}\times[0,\infty)$ is open with $\partial_SV:=V\cap(\R^{n-1}\times 0)\neq\emptyset$. We write $J(x,t):=T_N((\phi^{-1})(x),t)$ and note that $J$ is a smooth function with bounded derivatives that satisfies $J(x,t)=0$ when $x_n=0$. Additionally, we write $\xi_\epsilon(x):=\eta_\epsilon(\phi^{-1}(x))$.\\

We start by establishing a technical estimate. Note that if $p\in\partial M\cap U$ and $v\in T_pM$ but $v\not\in T_p\partial M$, then $D\Psi_2(p) v\neq 0$ as otherwise $D\Psi(p) T_pM\subset T_{p}\partial M\times0$, which cannot be as $\Psi$ is a diffeomorphism. Next, note that for any $x=(\vec x,0)$ we have  $\frac{\partial\phi^{-1}(\vec x,z)}{\partial z}\big|_{z=0}\not\in T_{\phi^{-1}(\vec x,0)}\partial M$. Since $\Psi_2(\phi^{-1}(\vec x,z))\geq 0$ and $\Psi_2(\phi^{-1}(\vec x,0))= 0$, we obtain
$$\frac d{dz}\bigg|_{z=0}\Psi_2(\phi^{-1}(\vec x,z))=D\Psi_2(\phi^{-1}(\vec x,0))\frac{\partial\phi^{-1}(x,z)}{\partial z}\bigg|_{z=0}>0.$$
As $\phi:U\rightarrow V$ is a good chart, we deduce that there exist $\theta>0$ and $z_0>0$ such that whenever $(\vec x,z)\in V$ and $0\leq z\leq z_0$ we have
$$\frac d{dz}\Psi_2(\phi^{-1}(\vec x,z))=D\Psi_2(\phi^{-1}(\vec x,z))\frac{\partial\phi^{-1}(x,z)}{\partial z}\geq \theta>0.$$
This implies $\Psi_2(\phi^{-1}(\vec x,z))\geq\theta z$ whenever $z\in[0,z_0]$. There exists $\epsilon_0>0$ such that for $\epsilon<\epsilon_0$ we have $\phi(U\cap \operatorname{supp}(\eta_\epsilon))\subset V\cap[0\leq z\leq z_0]$. As a consequence, note that when $2\theta^{-1}\epsilon<z<z_0$ we obtain $\Psi_2(\phi^{-1}(\vec x,z))>2\epsilon$ and hence $\xi_\epsilon(\vec x,z)$=0. By construction, if $\epsilon\in(0,\epsilon_0)$ and $z\geq z_0$ we have $\xi_\epsilon(\vec x,z)=0$. Thus there exists a constant $c_0>0$ such that if $\epsilon<\epsilon_0$ we have $\xi_\epsilon((\vec x,z))=0$ whenever $z\geq c_0\epsilon$. From now on we always assume $\epsilon\in(0,\epsilon_0)$.\\

Let $x=(\vec x, x_n)\in V$. If $x_n\geq c_0\epsilon$ we have $\xi_\epsilon(x)J(x,t)=0$. When $0\leq x_n<c_0\epsilon$ we may use $|\xi_\epsilon(x)|\leq 1$ and $J((\vec x,0),t)=0$ to estimate 
$$|J(x,t)\xi_\epsilon(x)|\leq | J(x,t)|\leq \|\nabla J\|_{C^0(V\times[0,T])}x_n\leq C\epsilon.$$
Since the constant in the above estimate is universal we have shown that $J\xi_\epsilon\rightarrow 0$ in $C^0(V\times[0,T])$.\\

Now let $t_1\neq t_2\in[0,T]$ and $x\in V$. If $x_n\geq c_0\epsilon$ we have $|(J\xi_\epsilon)(x,t_2)-(J\xi_\epsilon)(x,t_1)|=0$. If $x_n\leq c_0\epsilon$ we use $J((\vec x,0), t)=0$ to estimate
\begin{align*} 
&|\xi_\epsilon(x) J(x,t_2)-\xi_\epsilon(x) J(x,t_1)|\\
\leq &|J(x,t_2)-J(x,t_1)|\\
\leq &|t_2-t_1|\ \sup_{t\in[0,T]}|\partial_t J(x,t)|\\
\leq &|t_2-t_1|\sup_{t\in[0,T],y\in V}|\nabla\partial_t J(y,t)\||x_n|\\
\leq &C\epsilon|t_2-t_1|.
\end{align*}
Again, as the constant $C$ in the above estimate is universal we obtain that $[\xi_\epsilon J]^{\operatorname{time}}_{\frac\gamma4,V\times[0,T]}\rightarrow 0$ as $\epsilon\rightarrow 0^+$.\\

To estimate the spatial Hölder seminorm, we first make the following observation. By definition $\xi_\epsilon:=\zeta_\epsilon\circ\Psi_2\circ\phi^{-1}$ and hence $\|\nabla\xi_\epsilon\|_{C^0(V)}\leq \frac C\epsilon$. Let $x\neq y\in V$ and $t\in[0,T]$. If $x_n,y_n\geq C\epsilon$ we have $|\xi_\epsilon(x) J(x,t)-\xi_\epsilon(y) J(y,t)|=0$. If $x_n\geq c_0\epsilon>y_n$ we use $\xi_\epsilon(x)=0$ and $J((\vec y, 0),t)=0$ to estimate 
\begin{align*} 
|\xi_\epsilon(x)J(x,t)-\xi_\epsilon(y)J(y,t)|\leq & |\xi_\epsilon(x)-\xi_\epsilon(y)|\ |J(y,t)|\\
\leq &\left(\frac{|\xi_\epsilon(x)-\xi_\epsilon(y)|}{|x-y|}\right)^\gamma |x-y|^\gamma (|\xi_\epsilon(x)|+|\xi_\epsilon(y)|)^{1-\gamma}\ \|\nabla J\|_{C^0(V\times[0,T])}|y_n|\\
\leq &C\|\nabla\xi_\epsilon\|_{C^0(V)}^\gamma\epsilon|x-y|^\gamma\\
\leq &C\epsilon^{1-\gamma}|x-y|^\gamma.
\end{align*}
  Finally, if $x_n,y_n\leq C\epsilon$ we first use $|\xi_\epsilon|\leq 1$ to estimate 
$$|\xi_\epsilon(x)J(x,t)-\xi_\epsilon(y)J(y,t)|\leq |J(x,t)-J(y,t)|+|J(x,t)| |\xi_\epsilon(x)-\xi_\epsilon(y)|.$$
The second term can be estimated just as in the previous case and is therefore bounded by $C\epsilon^{1-\gamma}|x-y|^\gamma$. Using $J((\vec x,0),t)=J((\vec y,0),t)=0$, we estimate 
\begin{align*} 
|J(x,t)-J(y,t)|=&\left(\frac{|J(x,t)-J(y,t)|}{|x-y|}\right)^\gamma (|J(x,t)|+|J(y,t)|)^{1-\gamma}|x-y|^\gamma\\
\leq &C\|\nabla J\|_{C^0(V\times[0,T])}^\gamma\|\nabla J\|_{C^0(V\times[0,T])}^{1-\gamma}(|x_n|+|y_n|)^{1-\gamma}|x-y|^\gamma\\
\leq &C\epsilon^{1-\gamma}|x-y|^\gamma.
\end{align*}
So in total we have shown that $[\xi_\epsilon J]_{\gamma, V\times[0,T]}^{\operatorname{space}}\leq C\epsilon^{1-\gamma}$.\\

Combining the estimates from above, we have verified \eqref{SchauderCh_ApproximationsAreRealyApproximation} also for boundary charts. 
\end{proof}

We now extend the analysis to general boundary data. 
\begin{lemma}\label{SchauderCh_ExttendingtoIVandBV_Lemma}
Let $f\in C^{0,0,\gamma}(M\times[0,T])$, $h_1\in C^{3,0,\gamma}(\partial M\times[0,T])$, $h_2\in C^{1,0,\gamma}(\partial M\times[0,T])$ and $u_0\in C^{4,\gamma}(M)$ such that the following compatibility conditions are satisfied:
$$\frac{\partial u_0}{\partial\nu}=h_1(\cdot,0)
\hspace{.5cm}\textrm{and}\hspace{.5cm}
\frac{\partial\Delta_g u_0}{\partial\nu}=h_2(\cdot, 0)$$
Then there exists a unique solution $u\in C^{4,1,\gamma}(M\times[0,T])$ to 
$$
\left\{
\begin{array}{ll}
\displaystyle\dot u+\Delta_g^2u=f&\displaystyle\textrm{in $M\times[0,T]$},\vspace{.2cm}\\
\displaystyle\frac{\partial u}{\partial\nu}=h_1,\ \frac{\partial\Delta_g u}{\partial\nu}=h_2&\displaystyle\textrm{along $\partial M\times[0,T]$},\vspace{.2cm}\\
\displaystyle u(\cdot,0)=u_0& \displaystyle\textrm{on $M$}.
\end{array}
\right.
$$
\end{lemma}
\begin{proof}
Uniqueness follows immediately from the Schauder estimate in Corollary \ref{SchauderCh_SchauderEstimateonmanifolds}. To establish existence, we first assume that $f,g,h$ and $u_0$ are smooth. For $k=0,1$, we use Lemma \ref{SchauderCh_ExistenceofphijfunctionsLemma} to choose $\varphi_k\in C^\infty(M)$ that satisfy
$$\frac{\partial\Delta_g ^k \varphi_k}{\partial\nu}=1
\hspace{.5cm}\textrm{and}\hspace{.5cm}
\nabla^l \varphi_k=0\hspace{.2cm}\textrm{for all $0\leq l\leq 2k$}.$$
Using the extension operator $E$ from Lemma \ref{SchauderCh_SmoothExistenceOperatorExistenceLemma}, we define 
$$T_0:=\varphi_0 E[h_1]\hspace{.5cm}\textrm{and}\hspace{.5cm}
w:=T_0+\varphi_1 E\left[h_2-\frac{\partial \Delta_g T_0}{\partial\nu}\right].$$
By construction, $w$ is smooth and satisfies
$$\frac{\partial}{\partial\nu} w=h_1\hspace{.5cm}\textrm{and}\hspace{.5cm}\frac{\partial}{\partial\nu}\Delta_g w=h_2.$$
We define $w_0:=w(\cdot,0)$. Using Theorem \ref{SchauderCh_GalerkinExistenceTheorem}, we deduce that there exists a function $v\in C^{4,1,\gamma}(M\times[0,T])$ such that
$$\left\{
\begin{array}{ll}
\displaystyle\dot v+\Delta_g^2v=f-\dot w-\Delta_g ^2w-\Delta_g^2(u_0-w_0)&\displaystyle\textrm{in $M\times[0,T]$},\vspace{.2cm}\\
\displaystyle\frac{\partial v}{\partial\eta}=\frac{\partial\Delta_g v}{\partial\eta}=0&\displaystyle\textrm{along $\partial M\times[0,T]$},\vspace{.2cm}\\
\displaystyle v(\cdot,0)=0&\displaystyle\textrm{on $M$}.
\end{array}
\right.
$$
 We define $u:=v+w+(u_0-w_0)$. Clearly $u(\cdot,0)=u_0$ and 
$$\dot u+\Delta_g^2 u=f-\dot w-\Delta_g ^2w-\Delta_g^2(u_0-w_0)+\dot w+\Delta_g^2 w+\Delta_g^2 (u_0-w_0)=f.$$
Finally, we use the compatibility conditions to check that 
$$
\frac{\partial u}{\partial\nu}=0+h_1+\frac{\partial u_0}{\partial\nu}-h_1(\cdot,0)=h_1
\hspace{.5cm}\textrm{and}\hspace{.5cm}
\frac{\partial \Delta_g u}{\partial\nu}=0+h_2+\frac{\partial \Delta_g u_0}{\partial\nu}-h_2(\cdot,0)=h_2.
$$
So, for smooth data, the theorem is proven. Now for the general case. Let $f_\epsilon$, $h_{1,\epsilon}$, $h_{2,\epsilon}$ and $u_{0,\epsilon}$ be smooth approximations of $f,h_1,h_2$ and $u_0$  satisfying 
\begin{align*}
    &\|f-f_\epsilon\|_{C^{0,0,\gamma}(M\times[0,T])}<\epsilon,\hspace{.2cm}
\|h_1-h_{1,\epsilon}\|_{C^{3,0,\gamma}(M\times[0,T])}<\epsilon,\hspace{.2cm}
\|h_2-h_{2,\epsilon}\|_{C^{1,0,\gamma}(M\times[0,T])}<\epsilon,\\
&\hspace{.2cm}\textrm{and}\hspace{.2cm}
\|u_0-u_{0,\epsilon}\|_{C^{4,\gamma}(M)}<\epsilon.
\end{align*}
In general, these approximations will not satisfy the compatibility conditions, so we need to modify them a little. 
We define 
$$
    \hat h_{1,\epsilon}:=h_{1,\epsilon}+\left(\frac{\partial u_{0,\epsilon}}{\partial\nu}-h_{1,\epsilon}(\cdot,0)\right)
    \hspace{.5cm}\textrm{and}\hspace{.5cm}
    \hat h_{2,\epsilon}:=h_{2,\epsilon}+\left(\frac{\partial \Delta_g u_{0,\epsilon}}{\partial\nu}-h_{2,\epsilon}(\cdot,0)\right).
$$
Clearly $\hat h_{1,\epsilon}(\cdot,0)=\frac{\partial u_{0,\epsilon}}{\partial\nu}$ and $\hat h_{2,\epsilon}(\cdot,0)=\frac{\partial \Delta_g u_{0,\epsilon}}{\partial\nu}$. However, we need to verify that $\hat h_{1,\epsilon}\rightarrow h_1$ in $C^{3,0,\gamma}(\partial M\times[0,T])$ and $\hat h_{2,\epsilon}\rightarrow h_2$ in $C^{1,0,\gamma}(\partial M\times[0,T])$. To do so we use the compatibility condition $\partial_\nu u_0=h_1(\cdot, 0)$ to estimate 
\begin{align*}
    \|\hat h_{1,\epsilon}-h_{1,\epsilon}\|_{C^{3,0,\gamma}(\partial M\times[0,T])} &\leq \left\| \frac{\partial u_{0,\epsilon}}{\partial\nu}-h_{1,\epsilon}(\cdot,0)\right\|_{C^{3,0,\gamma}(\partial M\times[0,T])}\\
    &\leq \left\| \frac{\partial u_{0,\epsilon}}{\partial\nu}-\frac{\partial u_{0}}{\partial\nu}\right\|_{C^{3,0,\gamma}(\partial M\times[0,T])} +\|h_1(\cdot,0)-h_{1,\epsilon}(\cdot,0)\|_{C^{3,0,\gamma}(\partial M\times[0,T])}\\
    &\leq \| u_{0,\epsilon}-u_0\|_{C^{4,\gamma}( M)} +\|h_1-h_{1,\epsilon}\|_{C^{3,0,\gamma}(\partial M\times[0,T])}\\
    &\leq 2\epsilon.
\end{align*}
Proving $\hat h_{2,\epsilon}\rightarrow h_2$ in $C^{1,0,\gamma}(\partial M\times[0,T])$ is achieved by repeating essentially the same argument. Since $u_{0,\epsilon}$, $\hat h_{1,\epsilon}$ and $\hat h_{2,\epsilon}$ are smooth functions that satisfy the compatibility conditions, the first part of the proof implies the existence of a unique solution  $u_\epsilon$ to the problem
\begin{equation}\label{SchauderCh_GeneralProblemApproximateProblem}
\left\{
\begin{array}{ll}
    \displaystyle\dot u_\epsilon+\Delta_g^2 u_\epsilon=f_\epsilon&\displaystyle\textrm{in $M\times[0,T]$},\vspace{.2cm}\\
    \displaystyle\frac{\partial u_\epsilon}{\partial\nu}= \hat h_{1,\epsilon},\ \frac{\partial \Delta_g u_\epsilon}{\partial\nu}= \hat h_{2,\epsilon}&\displaystyle\textrm{along $\partial M\times[0,T]$},\vspace{.2cm}\\
    \displaystyle u_\epsilon(\cdot, 0)=u_{0,\epsilon}& \displaystyle\textrm{on $M$}.
\end{array}
\right.
\end{equation}
Using the Schauder estimate from Corollary \ref{SchauderCh_SchauderEstimateonmanifolds} it follows that $u_\epsilon$ is Cauchy in $C^{4,1,\gamma}(M\times[0,T])$ and thus converges to a function $u$ in $C^{4,1,\gamma}(M\times[0,T])$. Passing to the limit $\epsilon\rightarrow 0^+$ in Equation \eqref{SchauderCh_GeneralProblemApproximateProblem}, we deduce that $u$ is the sought after solution.\\
\end{proof}

\subsection{The Method of Continuity}
The following is a standard result and a proof can, for example, be found in Gilbarg's and Trudinger's \cite{gilbarg1977elliptic} (see Chapter 5, Section 2, Theorem 5.2). 
\begin{theorem}[The Method of Continuity]\label{SchauderCh_MethodofCOntinuity}\ \\
Let $B$ be a Banach space, $V$ be a normed vector space and let $L_0,L_1$ be bounded linear operators from $B$ into $V$. For each $t\in(0,1)$ put 
$L_t:=(1-t)L_0+tL_1$
and suppose that there is a constant $C$ such that 
$\|x\|_B\leq C\|L_t x\|_V$
for all $x\in B$ and $t\in[0,1]$.
Then, $L_0$ is surjective if and only if $L_1$ is surjective. 
\end{theorem}

Combining Lemma \ref{SchauderCh_ExttendingtoIVandBV_Lemma} and Theorem \ref{SchauderCh_MethodofCOntinuity}, we obtain the following existence result.
\begin{lemma}\label{Schauder_ExistenceMethodContiuityLemma}
Let $L$ be a 2-elliptic operator with smooth coefficients on $M$ and denote by $B$ the canonical boundary conditions\footnotemark. For every $f\in C^{0,0,\gamma}(M\times[0,T])$, $h_1\in C^{3,0,\gamma}(\partial M\times[0,T])$ and $h_2\in C^{1,0,\gamma}(\partial M\times[0,T])$ satisfying $h_1(\cdot,0)=0$ and $h_2(\cdot, 0)=0$, there exists a unique solution $u\in C^{4,1,\gamma}(M\times[0,T])$ of 
$$
\left\{
\begin{array}{ll}
\displaystyle\dot u+Lu=f&\displaystyle\textrm{in $M\times[0,T]$},\vspace{.2cm}\\
\displaystyle B_1u=h_1,\ B_2 u=h_2&\displaystyle\textrm{along $\partial M\times[0,T]$},\vspace{.2cm}\\
\displaystyle u(\cdot,0)=0& \displaystyle\textrm{on $M$}.
\end{array}
\right.
$$
\end{lemma}
\footnotetext{These automatically also have smooth coefficients.}
\begin{proof}
By Lemma \ref{SchauderCh_LGeometricFormula}, the operator $L$ induces a smooth metric $g_{L}$ on $M$. We have $L_1:=L=\Delta_{g_L}^2+Hu$ where $H$ is a third-order, linear differential operator with smooth coefficients. Let $L_0:=\Delta_{g_L}^2$ and, for $s\in(0,1)$, put
\begin{equation}\label{SchauderCh_Continuity01}
 L_s:=(1-s)L_0+sL_1=\Delta_{g_L}^2+sH.
\end{equation}
We denote by $\mathcal B_s$ the canonical boundary conditions induced by $ L_s$. Note that 
\begin{equation}\label{SchauderCh_Continuity02}
\mathcal B_su=\left(\frac{\partial u}{\partial\nu}, \frac{\partial\Delta_{g_L} u}{\partial\nu}+s\langle J,\nabla^2 u\otimes\nu^\flat\rangle\right)
=(1-s)\mathcal B_0 u+s\mathcal B_1 u.
\end{equation}
Here $J$ is the $(3,0)$-tensor field from  Lemma \ref{SchauderCh_LGeometricFormula}. For $k=1,3$ we put $C^{k,0,\gamma}_0(\partial M\times[0,T]):=\set{v\in C^{k,0,\gamma}(\partial M\times[0,T])\ |\ v(\cdot,0)=0}$ and define the spaces
\begin{align*}
    X&:=\set{u\in C^{4,1,\gamma}(M\times[0,T])\ |\ u(\cdot,0)=0},\vspace{-.2cm}\\
    Y&:=C^{0,0,\gamma}(M\times[0,T])\times C^{3,0,\gamma}_0(\partial M\times[0,T])\times C^{1,0,\gamma}_0(\partial M\times[0,T]).
\end{align*}
We equip $X$ with the $\|\cdot\|_{C^{4,1,\gamma}(M\times[0,T])}$ norm and $Y$ with the product norm of $\|\cdot\|_{C^{0,0,\gamma}(M\times[0,T])}$, $\|\cdot\|_{C^{3,0,\gamma}(\partial M\times[0,T])}$ and $\|\cdot\|_{C^{1,0,\gamma}(\partial M\times[0,T])}$. In this way, $X$ and $Y$ become Banach spaces. Clearly $T_0:=L_0\times\mathcal  B_0$ and $T_1:=L_1\times \mathcal B_1$ are linear and continuous maps from $X$ into $Y$. For $s\in(0,1)$ we put $T_s:=(1-s)T_0+sT_1$. Translating the Schauder estimate from Corollary \ref{SchauderCh_SchauderEstimateonmanifolds} into this new notation, we get
\begin{equation}\label{SchauderCh_Continuity03}
    \|u\|_X\leq C(s)\|T_su\|_Y.
\end{equation}
In view of Equations \eqref{SchauderCh_Continuity01} and \eqref{SchauderCh_Continuity02}, the constant $C(s)$ is actually independent of $s$. Indeed, the coefficients of the operator $T_s$ are bounded independent of $s$ and the metric $g_s$ induced by $L_s$ is constant, i.e. $g_s=g_L$ for all $s\in[0,1]$. So, by the method of continuity, the lemma follows once we have shown that $T_0 $ is surjective. This amounts to showing that for $f\in C^{0,0,\gamma}(M\times[0,T])$, $h_1\in C^{3,0,\gamma}(\partial M\times[0,T])$ and $h_2\in C^{1,0,\gamma}(\partial M\times[0,T])$ satisfying $h_1(\cdot,0)=h_2(\cdot, 0)=0$, we can solve the problem
$$
\left\{
\begin{array}{ll}
\displaystyle\dot u+\Delta_{g_L}^2u=f&\displaystyle\textrm{in $M\times[0,T]$},\vspace{.2cm}\\
\displaystyle\frac{\partial u}{\partial\nu}=h_1,\ \frac{\partial\Delta_{g_L} u}{\partial\nu}=h_2&\displaystyle\textrm{along $\partial M\times[0,T]$},\vspace{.2cm}\\
\displaystyle u(\cdot,0)=0& \displaystyle\textrm{on $M$}.
\end{array}
\right.
$$
This is true by Lemma \ref{SchauderCh_ExttendingtoIVandBV_Lemma} as the necessary compatibility conditions are satisfied.
\end{proof}

Using Lemma \ref{Schauder_ExistenceMethodContiuityLemma}, we now generalize to nonzero initial values.

\begin{theorem}\label{SchauderCh_FinalExistenceThmParabolic}
Let $M$ be a smooth, compact connected manifold with boundary $\partial M\neq\emptyset$ and $T>0$. Let $L$ be a 2-elliptic operator with smooth coefficients and denote by $B$ the canonical boundary conditions. Let $f\in C^{0,0,\gamma}(M\times[0,T])$, $h_1\in C^{3,0,\gamma}(\partial M\times[0,T])$, $h_2\in C^{1,0,\gamma}(\partial M\times[0,T])$ and $u_0\in C^{4,\gamma}(M)$ such that the compatibility conditions 
$$B_1 u_0=h_1(\cdot,0)
\hspace{.5cm}\textrm{and}\hspace{.5cm}
B_2 u_0=h_2(\cdot,0)
$$
are satisfied. Then, there exists a unique solution $u\in C^{4,1,\gamma}(M\times[0,T])$ to the problem
$$
\left\{
\begin{array}{ll}
\displaystyle\dot u+Lu=f&\displaystyle\textrm{in $M\times[0,T]$},\vspace{.2cm}\\
\displaystyle B_1u=h_1,\ B_2u=h_2&\displaystyle\textrm{along $\partial M\times[0,T]$},\vspace{.2cm}\\
\displaystyle u(\cdot,0)=u_0& \displaystyle\textrm{on $M$}.
\end{array}
\right.
$$
\end{theorem}
\begin{proof}
We put $\tilde f:=f-Lu_0$, $\tilde h_1:=h_1-B_1 u_0$ and $\tilde h_2:=h_2-B_2 u_0$. Then, $\tilde h_1(\cdot,0)=0$ and $\tilde h_2(\cdot,0)=0$. So, by Lemma \ref{Schauder_ExistenceMethodContiuityLemma}, there exists a solution $v\in C^{4,1,\gamma}(M\times[0,T])$ to the problem 
$$
\left\{
\begin{array}{ll}
\displaystyle\dot v+Lv=\tilde f&\displaystyle\textrm{in $M\times[0,T]$},\vspace{.2cm}\\
\displaystyle B_1v=\tilde h_1,\ B_2v=\tilde h_2&\displaystyle\textrm{along $\partial M\times[0,T]$},\vspace{.2cm}\\
\displaystyle v(\cdot,0)=0& \displaystyle\textrm{on $M$}.
\end{array}
\right.
$$
We get a solution to the original problem by defining $u:=v+u_0$. Finally, the uniqueness follows from the Schauder estimate in Corollary \ref{SchauderCh_SchauderEstimateonmanifolds}.
\end{proof}

\section{An Elliptic Problem}\label{SchauderCh_Ch04Sec05}
Using the parabolic Schauder estimates, we now derive elliptic Schauder estimates. Throughout this section, we assume that $g$ is a smooth Riemannian metric on $M$.

\begin{lemma}\label{SchauderCh_EllipticSchauderEstimate}
Let $\mathcal A$ be a good atlas on $M$, $L$ be a 2-elliptic operator of class $C^{0,\gamma}(M)$ on $M$ and let $B=(B_1,B_2)$ define $L$-compatible boundary conditions of class $C^{3,\gamma}(\partial M)\times C^{1,\gamma}(\partial M)$. Then for any $u\in C^{4,\gamma}(M)$
\begin{equation}\label{SchauderCh_EllipticSchauderEstimateEq}\|u\|_{C^{4,\gamma}(M)}\leq C\left(
\|Lu\|_{C^{0,\gamma}(M)}+\|B_1 u\|_{C^{3,\gamma}(\partial M)}+\|B_2 u\|_{C^{1,\gamma}(\partial M)}+\|u\|_{L^2(M)}
\right)
\end{equation}
Here $C$ depends only on $M$, $\gamma$, $\Lambda:=\|L\|_{C^{0,\gamma}_{\mathcal A}(M)}+\|B_1\|_{C^{3,\gamma}_{\mathcal A}(\partial M)}+\|B_2 \|_{C^{1,\gamma}_{\mathcal A}(\partial M)}$ and $\theta:=\Theta_{\mathcal A}(L)$.
\end{lemma} 
\begin{proof}
We define $U\in C^{4,1,\gamma}(M\times[0,1])$ by letting $U(p,t):=tu(p)$. Note $U(\cdot, 0)=0$ and $U(\cdot, 1)=u$. So, using the Schauder estimate from Theorem \ref{SchauderCh_GoodAPrioriEstimateOnManifold}, we find 
\begin{equation}\label{SchauderCh_EllipticSchauderEstiamteProofeq1}
\begin{aligned}
     \|u\|_{C^{4,\gamma}_{\mathcal A}(M)}\leq \|U\|_{C^{4,1,\gamma}_{\mathcal A}(M\times[0,1])}\leq C\big(&\|\dot U+LU\|_{C^{0,0,\gamma}_{\mathcal A}(M\times[0,1])}+\|B_1U\|_{C^{3,0,\gamma}_{\mathcal A}(\partial M\times[0,1])}\\
     &+\|B_2 U\|_{C^{1,0,\gamma}_{\mathcal A}(\partial M\times[0,1])}+\sup_{t\in[0,1]}\|U(\cdot, t)\|_{L^2_{\mathcal A}(M)}\big).
     \end{aligned}
\end{equation}
Since $U(p,t)=tu(p)$, we have the estimates
\begin{align}
    &\|B_1 U\|_{C^{3,0,\gamma}_{\mathcal A}(\partial M\times[0,1])}\leq C\|B_1 u\|_{C^{3,\gamma}_{\mathcal A}(\partial M)}\label{SchauderCh_EllipticSchauderEstiamteProofeq2},\\
    &\|B_2 U\|_{C^{1,0,\gamma}_{\mathcal A}(\partial M\times[0,1])}\leq C\|B_2 u\|_{C^{1,\gamma}_{\mathcal A}(\partial M)}\label{SchauderCh_EllipticSchauderEstiamteProofeq3},\\
    &\sup_{t\in[0,1]}\|U(\cdot, t)\|_{L^2_{\mathcal A}(M)}\leq \|u\|_{L^2_{\mathcal A}(M)}\label{SchauderCh_EllipticSchauderEstiamteProofeq4},\\
    &\|\dot U+LU\|_{C^{0,0,\gamma}_{\mathcal A}(M\times[0,1])}\leq C(\|u\|_{C^{0,\gamma}_{\mathcal A}(M)}+\|Lu\|_{C^{0,\gamma}_{\mathcal A}(M)})\label{SchauderCh_EllipticSchauderEstiamteProofeq5}.
\end{align}
Inserting Estimates \eqref{SchauderCh_EllipticSchauderEstiamteProofeq2}, \eqref{SchauderCh_EllipticSchauderEstiamteProofeq3}, \eqref{SchauderCh_EllipticSchauderEstiamteProofeq4} and \eqref{SchauderCh_EllipticSchauderEstiamteProofeq5} into Estimate \eqref{SchauderCh_EllipticSchauderEstiamteProofeq1}, we obtain
\begin{equation}\label{SchauderCh_EllipticSchauderEstiamteProofeq7}
\begin{aligned}
\|u\|_{C^{4,\gamma}(M)}\leq C\big(&\|Lu\|_{C^{0,\gamma}(M)}+\|u\|_{C^{0,\gamma}(M)}+\|B_1 u\|_{C^{3,\gamma}(\partial M)}\\
&\hspace{.5cm}+\|B_2 u\|_{C^{1,\gamma}(\partial M)}+\|u\|_{L^2(M)}\big).
\end{aligned}
\end{equation}
By Ehrling's lemma $\|u\|_{C^{0,\gamma}(M)}\leq \epsilon\|u\|_{C^{4,\gamma}(M)}+C(\epsilon)\|u\|_{L^2(M)}$. Inserting into Estimate \eqref{SchauderCh_EllipticSchauderEstiamteProofeq7} and choosing $\epsilon$ small enough, the lemma follows. 
\end{proof}
Unlike in the parabolic case, it is not possible to remove the $L^2(M)$-norm on the right-hand side of Estimate \eqref{SchauderCh_EllipticSchauderEstimateEq}. To see this, we consider the elliptic equation
\begin{equation}\label{SchauderCh_BiLaplcianEq}
\left\{
\begin{aligned}
&\Delta_g^2 u=0\hspace{1.85cm}\textrm{in $M$,}\\
&\frac{\partial u}{\partial\nu}=\frac{\partial\Delta_g u}{\partial\nu}=0\hspace{.5cm}\textrm{along $\partial M$.}
\end{aligned}
\right.
\end{equation}
This equation has the non-trivial solution $u=1$ and hence we can not drop the term $\|u\|_{L^2(M)}$ in Lemma \ref{SchauderCh_EllipticSchauderEstimate}. This is, however, the only issue. We formalize this statement in the following corollary.

\begin{korollar}\label{SchauderCh_EllipticSchauderEstimate_Korollar}
Under the assumptions of Lemma \ref{SchauderCh_EllipticSchauderEstimate}, there exists a constant $C=C(M,\gamma, \Lambda,\theta)$ such that for every $u\in C^{4,\gamma}(M)$ 
$$\|u\|_{C^{4,\gamma}_{\mathcal A}(M)}\leq C\left(
\|\Delta_g^2 u\|_{C^{0,\gamma}_{\mathcal A}(M)}+\left\|\frac{\partial u}{\partial\nu}\right\|_{C^{3,\gamma}_{\mathcal A}(\partial M)}+\left\|\frac{\partial\Delta_g u}{\partial\nu}\right\|_{C^{1,\gamma}_{\mathcal A}(\partial M)}+\left|\int_M ud\mu_g\right|
\right).$$
\end{korollar}
\begin{proof}
Throughout the proof, we will not make the dependence of $C$ on $M$, $\gamma$, $\Lambda$ and $\theta$ explicit. Let $u\in C^{4,\gamma}(M)$, $f:=\Delta_g^2 u$, $h_1:=\frac{\partial u}{\partial\nu}$ and $h_2:=\frac{\partial\Delta_g u}{\partial\nu}$. We compute
$$\int_M fu d\mu_g=\int_M u\Delta_g^2 ud\mu_g=\int_M (\Delta_g u)^2 d\mu_g+\int_{\partial M}u\frac{\partial\Delta_g u}{\partial\nu}-\Delta_g u\frac{\partial u}{\partial\nu} dS_g.$$
Rearranging this identity, introducing a small parameter $\epsilon>0$ and using Young's inequality, we get
\begin{equation}\label{SchauderCh_LapceL2Estimate}
\int_M (\Delta_g u)^2 d\mu_g\leq C(\epsilon)\left(\|f\|_{C^{0,\gamma}_{\mathcal A}(M)}^2+\|h_1\|_{C^{3,\gamma}_{\mathcal A}(\partial M)}^2+\|h_2\|_{C^{1,\gamma}_{\mathcal A}(\partial M)}^2\right)+\epsilon\|u\|_{C^{4,\gamma}_{\mathcal A}(M)}^2.
\end{equation}
We put $\bar u:=\fint_M ud\mu_g$. Clearly $\int_M u-\bar u d\mu_g=0$. So, using Lemma \ref{SchauderCh_PoincareInequality} and Young's inequality with a small parameter $\delta>0$, we may estimate 
\begin{align*}
    \int_M u^2 d\mu_g=&\int_M u^2d\mu_g-|M|\bar u^2+|M|\bar u^2\\
            =&\int_M (u-\bar u)^2 d\mu_g+|M|\bar u^2\\
            \leq & C\left(\int_M |\nabla u|^2 d\mu_g+\bar u^2\right)\\
            =& C\left(\int_{\partial M}u\frac{\partial u}{\partial\nu}dS_g-\int_M u\Delta_g u d\mu_g+\bar u^2\right)\\
            \leq &  C(\delta)\left(\int_{\partial M}uh_1dS_g+\|\Delta_gu\|_{L^2(M)}^2 +\bar u^2\right)+\delta\|u\|_{L^2(M)}^2.
\end{align*}
Choosing $\delta>0$ small enough, we may absorb $\|u\|_{L^2(M)}^2$ to the left-hand side. Inserting Estimate \eqref{SchauderCh_LapceL2Estimate} and yet again using Young's inequality with the already introduced small parameter $\epsilon$, we obtain 
$$\|u\|_{L^2(M)}^2\leq C(\epsilon)\left(
\|f\|_{C^{0,\gamma}(M)}^2+\|h_1\|_{C^{3,\gamma}(\partial M)}^2+\|h_2\|_{C^{1,\gamma}(\partial M)}^2+
\bar u^2\right)+\epsilon\|u\|_{C^{4,\gamma}(M)}^2.$$
Inserting this estimate into the estimates from Lemma \ref{SchauderCh_EllipticSchauderEstimate}, we obtain 
$$\|u\|_{C^{4,\gamma}(M)}\leq C(\epsilon)\left(\|f\|_{C^{0,\gamma}(M)}+\|h_1\|_{C^{3,\gamma}(\partial M)}+\|h_2\|_{C^{1,\gamma}(\partial M)}+
|\bar u|\right)+\epsilon\|u\|_{C^{4,\gamma}(M)}.$$
Choosing $\epsilon>0$ small enough, we may absorb $\epsilon\|u\|_{C^{4,\gamma}(M)}$ to the left-hand side and deduce the corollary.
\end{proof}

Using a similar argument as in the proof of Theorem \ref{SchauderCh_GalerkinExistenceTheorem}, we now establish the following existence theorem.
\begin{theorem}\label{SchauderCh_EllipticGalerkinLemma}
There exists $N(n)$ with the following property: For all $f\in C^\infty(M)$ satisfying
$$\int_M fd\mu_g=0\hspace{.5cm}\textrm{and}\hspace{.5cm}
\frac{\partial}{\partial\nu}\Delta_g^m f=0\hspace{.5cm}\textrm{for all $0\leq m\leq N$}$$
there exists a function $u\in C^{4,\gamma}(M)$ such that 
\begin{equation}\label{SchauderCh_GalerkinProjectedProblem_Elliptic}
\left\{
\begin{aligned}
&\Delta_g^2u=f\hspace{1.85cm}\textrm{in $M$,}\\
&\frac{\partial u}{\partial\nu}=\frac{\partial\Delta_g u}{\partial\nu}=0\hspace{.5cm}\textrm{along $\partial M$,}\\
&\int_M ud\mu_g=0.
\end{aligned}
\right.
\end{equation}
\end{theorem}
\begin{proof}
By Sobolev embedding, we can fix $N(n)$ such that $W^{2N,2}(M)\hookrightarrow\hookrightarrow C^{0,\gamma}(M)$ compactly. 
Let $\lambda_k$ and $\phi_k$ denote the Neumann eigenvalues and eigenfunctions from Theorem \ref{SchauderCh_neumannlaplacianResultsTheorem}. For $k\geq 0$, we define $c_k:=\langle \phi_k, f\rangle_{L^2_g(M)}$ and note that by assumption $c_0=0$. For $\mu\in\N$ we define $f_\mu:=\sum_{j=1}^\mu c_j\phi_j$ as well as
$$u_\mu:=\sum_{j=1}^\mu \frac{c_j}{\lambda_j^2}\phi_j.$$
From direct computation it follows that the functions $u_\mu$ satisfy the problem
\begin{equation}\label{SchauderCh_GalerkinApproxEllipticEqu}
\left\{
\begin{aligned}
    &\Delta_g^2 u_\mu= f_\mu\hspace{1.85cm}\textrm{in $M$,}\\
    &\frac{\partial u_\mu}{\partial\nu}=\frac{\partial\Delta_g u_\mu}{\partial\nu}=0\hspace{.5cm}\textrm{along $\partial M$,}\\
    &\int_M u_\mu d\mu_g=0.
\end{aligned}
\right.
\end{equation}
Applying the Schauder estimate from Corollary \ref{SchauderCh_EllipticSchauderEstimate_Korollar} to the functions $u_\mu-u_\nu$  shows
$$\|u_\mu-u_\nu\|_{C^{4,\gamma}(M)}
\leq C\|f_\mu-f_\nu\|_{C^{0,\gamma}(M)}.$$
We prove that up to a subsequence, which we again denote by $\mu$, $f_\mu\rightarrow f$ in $C^{0,\gamma}(M)$. Once this is shown, we deduce that $u_\mu$ is Cauchy in $C^{4,\gamma}(M)$ and thus convergent in $C^{4,\gamma}(M)$. Putting $u:=\lim_{\mu\rightarrow\infty} u_\mu$ and passing to the limit in Equation \eqref{SchauderCh_GalerkinApproxEllipticEqu}, we verify that $u$ is a solution of Equation \eqref{SchauderCh_GalerkinProjectedProblem_Elliptic}.\\

First, note that we may use Theorem \ref{SchauderCh_neumannlaplacianResultsTheorem} to deduce that $f_\mu\rightarrow f$ in $L^2(M)$. So, if along some subsequence $f_\mu$ is convergent in $C^{0,\gamma}(M)$, the limit has to be $f$ as well. To show that $f_\mu$ is convergent in $C^{0,\gamma}(M)$, it is sufficient to show that $f_\mu$ is bounded in $W^{2N,2}(M)$ due to the choice of $N$. For $k\leq N$, we may use the assumption on $f$ and integration by parts to compute
\begin{align*}
    \lambda_j^{2k}c_j=&\langle (-\Delta_g)^ k\phi_j, f\rangle_{L^2_g(M)}=\langle \phi_j, (-\Delta_g)^ kf\rangle_{L^2_g(M)}.
\end{align*}
Consequently, we learn that for all $k\leq N$, we have the estimate 
\begin{equation}\label{SchauderCh_EllipticGalerkinlambdapowerestimate}
\sum_{j=1}^ \infty|\lambda_j^ {2k}c_j|^2=\sum_{j=1}^\infty \langle \phi_j, (-\Delta_g)^ kf\rangle_{L^2_g(M)}^2=\|\Delta_g^k f\|_{L^2(M)}^2\leq  C(M,k)\|f\|_{W^{2N,2}(M)}^2.
\end{equation}
Using Estimate \eqref{SchauderCh_EllipticGalerkinlambdapowerestimate}, we deduce that for $0\leq k\leq N$
$$\|\Delta_g^k f_\mu\|_{L^2(M)}^2=\|\sum_{j=0}^\mu (-\lambda_j)^{2k}c_j\phi_j\|_{L^2(M)}^2=\sum_{j=0}^\mu \lambda_j^{2k}c_j^2 \leq C(M,k)\|f\|_{W^{2N,2}(M)}^2.$$
Note $\phi_j\in C^\infty(M)$ and $\partial_\nu\Delta_g^k\phi_j=0$ for all $j\in\N$ and $k\in\N_0$. So, for all $\mu\in\N$ we have $f_\mu\in C^\infty(M)$ and $\partial_\nu\Delta_g^kf_\mu=0$ for all $k\in\N_0$. Using Lemma \ref{SchauderCh_Wk2EstimateinTermsOfLaplace_Lemma}, we may estimate
$$\|f_\mu\|_{W^{2N,2}(M)}\leq C(N)\sum_{k=0}^N\|\Delta_g^k f_\mu\|_{L^2(M)}\leq\sum_{k=0}^NC(M,N,k)\|f\|_{W^{2N,2}(M)}^2\leq C(M,N)\|f\|_{W^{2N,2}(M)}^2 .$$
Hence $f_\mu$ is bounded in $W^{2N,2}(M)$ and the theorem follows. 
\end{proof}

We now have to essentially repeat the analysis we provided in Subsection \ref{SchauderCh_SubSection_InitialAndBoundaryValues} with minor tweaks to adapt to the elliptic case.

\begin{korollar}\label{SchauderCh_EllipticExistenceCor1}
Let $f\in C^{0,\gamma}(M)$ such that $\int_M f d\mu_g=0$.
Then, there exists a function $u\in C^{4,\gamma}(M)$ such that 
\begin{equation}\label{SchauderCh_ExistenceLemmaProblemElliptic}
\left\{
\begin{aligned}
&\Delta_g^2u=f\hspace{1.85cm}\textrm{in $M$,}\\
&\frac{\partial u}{\partial\nu}=\frac{\partial\Delta_g u}{\partial\nu}=0\hspace{.5cm}\textrm{along $\partial M$,}\\
&\int_M ud\mu_g=0.
\end{aligned}
\right.
\end{equation}
\end{korollar}
\begin{proof}
We first consider the case where $f\in C^\infty(M)$. Following the proof of Corollary \ref{SchauderCh_ApproximationforgeneralfParabolic} we construct a function $T_N\in C^\infty(M)$ that satisfies 
$$\frac{\partial\Delta_g^ j T_N}{\partial\nu}=\frac{\partial\Delta_g^ jf}{\partial\nu}\hspace{.5cm}\textrm{for all $0\leq j\leq N$}
\hspace{.5cm}\textrm{and}\hspace{.5cm}T_N\big|_{\partial M}=0.$$
We define $\eta_\epsilon$ as in the proof of Corollary \ref{SchauderCh_ApproximationforgeneralfParabolic}. Recall that $\eta_\epsilon=1$ on a small neighborhood of $\partial M$. We define
$$F_{\epsilon}:=f-\eta_\epsilon T_N+\fint_M \eta_\epsilon T_N d\mu_g.$$
Using the properties of $T_N$ and $\eta_\epsilon$ it follows that $F_\epsilon$ satisfies
$$\int_M F_\epsilon d\mu_g=0
\hspace{.5cm}\textrm{and}\hspace{.5cm}
\frac{\partial}{\partial\eta}\Delta_g^ j F_{\epsilon}=0\hspace{.5cm}\textrm{for all $0\leq j\leq N$}.$$
By Corollary \ref{SchauderCh_EllipticExistenceCor1}, there exist functions $u_\epsilon\in C^{4,\gamma}(M)$ such that $u_\epsilon$ solves Problem \eqref{SchauderCh_ExistenceLemmaProblemElliptic} with $f$ replaced by $F_\epsilon$. Following the argument in the proof of Corollary \ref{SchauderCh_ApproximationforgeneralfParabolic}, we get $F_\epsilon\rightarrow f$ in $C^{0,\gamma}(M)$. The Schauder estimate from Corollary \ref{SchauderCh_EllipticSchauderEstimate_Korollar} implies that $(u_\epsilon)$ is Cauchy in $C^{4,\gamma}(M)$ and so we obtain a solution to Problem \eqref{SchauderCh_ExistenceLemmaProblemElliptic} as the limit $u:=\lim_{\epsilon\rightarrow 0^+} u_\epsilon$. In general, we choose smooth functions $f_k\in C^\infty(M)$ such that $f_k\rightarrow f$ in $C^{0,\gamma}(M)$. By the first part of the proof, there exist solutions $u_k$ to Problem \eqref{SchauderCh_ExistenceLemmaProblemElliptic} with $f$ replaced by $f_k$. The Schauder estimate from Corollary \ref{SchauderCh_EllipticSchauderEstimate_Korollar} implies that $(u_k)$ is Cauchy in $C^{4,\gamma}(M)$ and so we obtain a solution to Problem \eqref{SchauderCh_ExistenceLemmaProblemElliptic} as the limit $u:=\lim_{k\rightarrow\infty} u_k$.
\end{proof}

We now extend the analysis to general boundary data. 
\begin{korollar}\label{SchauderCh_EllipticProblemExistenceKorollar}
Let $f\in C^{0,\gamma}(M)$, $h_1\in C^{3,\gamma}(\partial M)$, $h_2\in C^{1,\gamma}(\partial M)$ such that the following compatibility condition is satisfied:
\begin{equation}\label{SchauderCh_EllipticFullCC}
\int_M fd\mu_g=\int_{\partial M}h_2 dS_g
\end{equation}
Then there exists a unique solution $u\in C^{4,\gamma}(M)$ to 
\begin{equation}\label{SchauderCh_EllipticFullProblem}
\left\{
\begin{aligned}
&\Delta_g^2u=f\hspace{2.73cm}\textrm{in $M$,}\\
&\frac{\partial u}{\partial\nu}=h_1,\ \frac{\partial\Delta_g u}{\partial\nu}=h_2\hspace{.5cm}\textrm{along $\partial M$,}\\
&\int_M ud\mu_g=0.
\end{aligned}
\right.
\end{equation}
\end{korollar}
\begin{proof}
Uniqueness follows from the Schauder estimate from Corollary \ref{SchauderCh_EllipticSchauderEstimate_Korollar}. For existence, we follow the proof of Lemma \ref{SchauderCh_ExttendingtoIVandBV_Lemma}. First, we assume that $f$, $h_1$ and $h_2$ are smooth. For $k=0,1$ let $\varphi_k\in C^\infty(M)$ such that 
$$\frac{\partial\Delta_g ^k \varphi_k}{\partial\nu}=1
\hspace{.5cm}\textrm{and}\hspace{.5cm}
\nabla^l \varphi_k\big|_{\partial M}=0\hspace{.2cm}\textrm{for all $0\leq l\leq 2k$}.$$
Using the extension operator $E$ from Lemma \ref{SchauderCh_SmoothExistenceOperatorExistenceLemma}, we define 
$$T_0:=\varphi_0 E[h_1],\hspace{.5cm}
T_1:=T_0+\varphi_1 E\left[h_2-\frac{\partial \Delta_g T_0}{\partial\nu}\right]
\hspace{.5cm}\textrm{and}\hspace{.5cm}
w:=T_1-\fint_M T_1 d\mu_g.$$
By construction, $w$ is smooth and satisfies
$$\int_M wd\mu_g=0,\hspace{.5cm}\frac{\partial w}{\partial\nu} =h_1\hspace{.5cm}\textrm{and}\hspace{.5cm}\frac{\partial\Delta_g w}{\partial\nu}=h_2.$$
Using the compatibility condition \eqref{SchauderCh_EllipticFullCC}, we compute
$$\int_M f-\Delta_g^2 wd\mu_g=\int_M fd\mu_g-\int_{\partial M}\frac{\partial\Delta_g w}{\partial\nu}dS_g=0.$$
So, by Corollary \ref{SchauderCh_EllipticExistenceCor1} we know that there is a function $v\in C^{4,\gamma}(M)$ satisfying
$$\left\{
\begin{aligned}
&\Delta_g^2v=f-\Delta_g^2w\hspace{.55cm}\textrm{in $M$,}\\
&\frac{\partial v}{\partial\nu}=\frac{\partial\Delta_g v}{\partial\nu}=0\hspace{.5cm}\textrm{along $\partial M$,}\\
&\int_M vd\mu_g=0.
\end{aligned}
\right.
$$
We define $u:=v+w$ and compute
$$\Delta_g^2 u=f-\Delta_g ^2w+\Delta_g^2 w=f
\hspace{.5cm}\textrm{and}\hspace{.5cm}
\int_M ud\mu_g=\int_M vd\mu_g+\int_M wd\mu_g=0.$$
Additionally, we have by construction 
$$\frac{\partial u}{\partial\nu}=\frac{\partial v}{\partial\nu}+\frac{\partial w}{\partial\nu}=h_1
\hspace{.5cm}\textrm{and}\hspace{.5cm}
\frac{\partial\Delta_g u}{\partial\nu}=\frac{\partial \Delta_gv}{\partial\nu}+\frac{\partial\Delta_g w}{\partial\nu}=h_2.$$
So, for smooth data, the corollary is proven. Now for the general case. Let $f_k$, $h_{1,k}$ and $h_{2,k}$ be smooth such that $f_k\rightarrow f$ in $C^{0,\gamma}(M)$, $h_{1,k}\rightarrow h_1$ in $C^{3,\gamma}(\partial M)$ and $h_{2,k}\rightarrow h_2$ in $C^{1,\gamma}(\partial M)$ .
In general, these approximations will not satisfy the compatibility condition, so we need to modify them a little. We define 
$$
\tilde f_k:=f_k-\frac1{|M|}\left(\int_M f_k d\mu_g-\int_{\partial M}h_{2,k} dS_g\right).
$$
When $k\rightarrow\infty$, we have $f_k\rightarrow f$ and $h_{2,k}\rightarrow h_2$. Since $f$ and $h_2$ satisfy the compatibility condition \eqref{SchauderCh_EllipticFullCC}, we deduce that $\tilde f_k\rightarrow f$ in $C^{0,\gamma}(M)$ when $k\rightarrow\infty$. Additionally, we have 
$$\int_M \tilde f_k d\mu_g=
\int_M f_k d\mu_g -\left(\int_M f_k d\mu_g-\int_{\partial M}h_{2,k} dS_g\right)
=\int_{\partial M}h_{2,k} dS_g.$$
So, for each $k\in\N$ we get a solution $u_k\in C^{4,\gamma}(M)$ of Problem \eqref{SchauderCh_EllipticFullProblem} with data $\tilde f_k$, $h_{1,k}$ and $h_{2,k}$. Using the Schauder estimate form Corollary \ref{SchauderCh_EllipticSchauderEstimate_Korollar} we deduce that $u_k$ is Cauchy in $C^{4,\gamma}(M)$ and therefore converges to a solution of \eqref{SchauderCh_EllipticFullProblem} with data $f$, $h_1$ and $h_2$.
\end{proof}

\section{A Decay Estimate}\label{SchauderCh_Ch04Sec4}
In general, an estimate of the form derived in Corollary \ref{SchauderCh_SchauderEstimateonmanifolds} cannot be improved in the sense that the Schauder constant has a time dependence. Still, under additional assumptions, more can be proven. In this section we focus on $M=\Sp^2_+:=\set{(\omega_1,\omega_2,\omega_3)\in\Sp^2\ |\ \omega_3\geq 0}$ equipped with the metric $\delta$ induced by the standard metric on $\R^3$. Points on $\Sp^2$, $\Sp^2_+$ and $\Sp^2_-:=\set{(\omega_1,\omega_2,\omega_3)\in\Sp^2\ |\ \omega_3\leq 0}$ are always denoted as $\omega=(\omega_1,\omega_2,\omega_3)$. We denote the Laplace-Beltrami operator by $\Delta$ and the exterior normal along $\partial \Sp^2_+$ by $\nu$. Additionally, we abbreviate $L^2(\Sp^2_+):=L^2_\delta(\Sp^2_+)$.\\

We require the following proposition. For a proof, we refer to \cite{atkinson2012spherical}. (For the density statement, see Chapter 2, Section 8, Subsection 3, Theorem 2.34. For the statement concerning the eigenvalues, see Chapter 3, Section 3, Proposition 3.5. Finally, for the characterization of the spaces $E_0$ and $E_1$, see Chapter 2, Section 1, Subsection 1, Example 2.6 combined with Chapter 2, Section 1, Subsection 3, Definition 2.7) 

\begin{proposition}[Laplace Eigenfunctions on $\Sp^2$]\label{SchauderCh_EigenfunctiononFullSphere}
For $k\in\N_0$, let\\ $E_k:=\set{u\in C^ \infty(\Sp^2)\ |\ -\Delta u=k(k+1)u\textrm{ on }\Sp^2}$. Then the following statements hold:
\begin{itemize}
    \item  $\operatorname{span}_\R\left(\bigcup_{k=0}^ \infty E_k\right)$ is dense in $L^2(\Sp^2)$.
    \item $E_0=\R$ and, putting $f_i:\Sp^2\rightarrow\R,\ f_i(\omega)=\omega_i$,  $E_1=\operatorname{span}_\R\set{f_1,f_2,f_3}$.
\end{itemize}
\end{proposition}
Using Proposition \ref{SchauderCh_EigenfunctiononFullSphere}, we now investigate the spectrum of the Neumann Laplacian on $\Sp^2_+$.
\begin{lemma}\label{SchauderCh_NeumannEVontheHalfSphere}
Using the notation from Theorem \ref{SchauderCh_neumannlaplacianResultsTheorem}: The eigenvalues of the Neumann Laplacian on $(\Sp^2_+,\delta)$ have the following properties:
\begin{itemize}
    \item $\lambda_0=0$ and $\phi_0=\sqrt{2\pi}^{-1}$.
    \item$\lambda_1=\lambda_2=2$ and one can choose $\phi_i(\omega)=\sqrt{\frac{3}{2\pi}}\omega_i$ for $i=1,2$.
    \item $\lambda_k\geq 6$ for $k\geq 3$.
\end{itemize}
\end{lemma}
The proof we present is essentially taken from the proof of Lemma 4 in \cite{AK}. 
 \begin{proof}
Let $\lambda\in\R$ and $\phi\in C^\infty(\Sp^2_+)$ satisfy the problem 
$$\left\{
    \begin{aligned}
        -\Delta\phi&=\lambda \phi\hspace{.5cm}\textrm{in $\Sp^2_+$},\\
        \frac{\partial \phi}{\partial\nu}&=0\hspace{.75cm}\textrm{along $\partial \Sp^2_+$}.
    \end{aligned}
\right.$$
We define the function
$$\bar\phi:\Sp^2\rightarrow\R,\ \bar\phi(\omega_1,\omega_2,\omega_3):=\left\{\begin{array}{ll}
        \displaystyle \phi(\omega_1,\omega_2,\omega_3)   & \displaystyle\textrm{if $\omega_3\geq 0$},\vspace{.15cm} \\
        \displaystyle \phi(\omega_1,\omega_2,-\omega_3)   & \displaystyle\textrm{if $\omega_3< 0$}.
\end{array}\right.$$
Since $\partial_\nu\phi=0$, it follows that $\bar\phi\in C^1(\Sp^2_+)\subset W^ {1,2}(\Sp^2)$. Additionally $\bar\phi\in C^ \infty(\Sp^2_+)$, $\bar\phi\in C^ \infty(\Sp^2_-)$ and $-\Delta\bar\phi=\lambda\bar\phi$ separately on $\Sp^2_+$ and $\Sp^2_-$. We claim that $\bar\phi$ is a weak solution of $-\Delta\bar\phi=\lambda\bar\phi$. To see this, let $\psi\in C^ \infty(\Sp^2)$ and denote the exterior unit normal along $\partial\Sp^2_-$ by $\tilde\nu$. Using $\partial_\nu\bar\phi=0$ along $\partial \Sp^2_+$ and $\partial_{\tilde\nu}\bar\phi=0$ along $\partial\Sp^2_-$, we use integration by parts to deduce
\begin{align*} 
\int_{\Sp^2}\langle\nabla\bar\phi,\nabla\psi\rangle d\mu_{\Sp^2}=&\int_{\Sp^2_+}\langle\nabla\bar\phi,\nabla\psi\rangle d\mu_{\Sp^2}+\int_{\Sp^2_-}\langle\nabla\bar\phi,\nabla\psi\rangle d\mu_{\Sp^2}\\
=&\int_{\partial\Sp^2_+}\frac{\partial\bar\phi}{\partial\nu}\psi dS-\int_{\Sp^2_+}\Delta\bar\phi\psi d\mu_{\Sp^2}+\int_{\partial\Sp^2_-}\frac{\partial\bar\phi}{\partial\tilde \nu}\psi dS-\int_{\Sp^2_-}\Delta\bar\phi \psi d\mu_{\Sp^2}\\
=&\int_{\Sp^2}\lambda\bar\phi\psi d\mu_{\Sp^2}.
\end{align*}
 By Proposition \ref{SchauderCh_RegularityMdfWITHOUTBOundary} $\bar\phi\in C^\infty(\Sp^2)$. Using the notation and statements from Proposition \ref{SchauderCh_EigenfunctiononFullSphere}, there exists $k\in\N_0$ and $f\in E_k$ such that $\langle f, \bar\phi\rangle\neq 0$. Consequently
 $$
\lambda\langle \bar\phi, f\rangle_{L^2(\Sp^2)}
=
-\langle\Delta\bar \phi, f\rangle_{L^2(\Sp^2)}
=
-\langle \bar\phi, \Delta f\rangle_{L^2(\Sp^2)}
=k(k+1) \langle \bar\phi, f\rangle_{L^2(\Sp^2)}
$$
and hence $\lambda=k(k+1)$. The lemma follows from Proposition \ref{SchauderCh_EigenfunctiononFullSphere}. Indeed, either $k\geq 2$ and $\lambda\geq 6$ or $k\in\set{0,1}$ and consequently $\lambda\in\set{0,2}$. If $\lambda=0$ then $\bar\phi\in E_0$ and hence $\phi$ is constant. If $k=1$, then $\bar\phi\in E_1$ and there exist $a_1,a_2,a_3\in\R$ such that $\phi(\omega)=\sum_{i=1}^3a_i\omega_i$. Finally $a_3=0$, since $0\overset!=\frac{\partial\phi}{\partial\nu}=a_3$.
\end{proof}
By Lemma \ref{SchauderCh_NeumannEVontheHalfSphere}, the kernel of $\Delta(\Delta+2)$ is the span of $\phi_0$, $\phi_1$ and $\phi_2$ (when imposing zero boundary values). For a function $\chi\in L^2(\Sp^2_+)$ we define the components $\chi^\parallel$ \emph{parallel} to and $\chi^\perp$ \emph{orthogonal} to the kernel as 
$$\chi^\parallel:=
\frac{\langle \chi,1\rangle_{L^2(\Sp^2_+)}}{\langle 1,1\rangle_{L^2(\Sp^2_+)}}1
+
\sum_{i=1}^2\frac{\langle \chi,\omega_i\rangle_{L^2(\Sp^2_+)}}{\langle \omega_i,\omega_i\rangle_{L^2(\Sp^2_+)}}\omega_i
\hspace{.5cm}\textrm{and}\hspace{.5cm}
\chi^\perp:=\chi-\chi^\parallel.$$
Here, in a typical abuse of notation, we have used the symbol $\omega_i$ also as the map $\omega\mapsto\omega_i$.\\

For $f\in C^{0,0,\gamma}(\Sp^2_+\times[0,T])$, $h_1\in C^{3,0,\gamma}(\partial\Sp^2_+\times[0,T])$, $h_2\in C^{1,0,\gamma}(\partial\Sp^2_+\times[0,T])$ and $u_0\in C^{4,\gamma}(\Sp^2_+)$ satisfying
\begin{equation}\label{SchauderCh_HalfSphereParabolicCompatibility}
\frac{\partial u_0}{\partial\nu}=h_1(\cdot,0)
\hspace{.5cm}\textrm{and}\hspace{.5cm}
\frac{\partial\Delta u_0}{\partial\nu}=h_2(\cdot,0)
\end{equation}
we consider the parabolic problem
\begin{equation}\label{SchauderCh_NullBVProblemonHalfSphere}
\left\{
\begin{array}{ll}
\displaystyle \dot u+\frac12\Delta(\Delta+2)u=f&\displaystyle \textrm{in $\Sp^2_+\times[0,T]$},\vspace{.15cm}\\
\displaystyle \frac{\partial u}{\partial\nu}=h_1,\ \ \frac{\partial\Delta u}{\partial\nu}=h_2&\displaystyle \textrm{along $\partial\Sp^2_+\times[0,T]$},\vspace{.15cm}\\
 \displaystyle u(0)=u_0&\displaystyle \textrm{on $\Sp^2_+$}.
\end{array}
\right.
\end{equation}
Since $u_0$ satisfies Equation \eqref{SchauderCh_HalfSphereParabolicCompatibility}, we may use Theorem \ref{SchauderCh_FinalExistenceThmParabolic} and deduce that Problem \eqref{SchauderCh_NullBVProblemonHalfSphere} has a unique solution $u\in C^{4,1,\gamma}(\Sp^2_+\times[0,T])$.

\begin{theorem}[Decay Estimate]\label{SchauderCh_DecayEstimate}
Let $T\geq 1$ and $u\in C^{4,1,\gamma}(\Sp^2_+\times[0,T])$ be the solution to Problem \eqref{SchauderCh_NullBVProblemonHalfSphere} with $f=0$ and $h_1=h_2=0$. Then for all $t\in[0,T]$ 
$$\|u(t)\|_{C^{4,\gamma}(\Sp^2_+)}\leq C\left(e^{-12 t}\|u_0^\perp\|_{C^{4,\gamma}(\Sp^2_+)}+\|u_0^\parallel\|_{C^{4,\gamma}(\Sp^2_+)}\right).$$
where $C=C_1(\Sp^2_+,\gamma)$ depends on $\Sp^2_+$ and $\gamma$ but importantly not on $T$.
\end{theorem}
\begin{proof}
We put $u^\parallel(t):=(u(t))^\parallel$ and $u^\perp(t):=(u(t))^\perp$. Using the PDE and the boundary conditions from Problem \eqref{SchauderCh_NullBVProblemonHalfSphere} and partial integration, we compute 
\begin{align*}
    \dot u^\parallel+\frac12\Delta(\Delta+2) u ^\parallel
    =&\frac{\langle\dot u(t), 1\rangle_{L^2(\Sp^2_+)}}{\langle 1,1\rangle_{L^2(\Sp^2_+)}}1+\sum_{i=1}^2\frac{\langle \dot u(t),\omega_i\rangle_{L^2(\Sp^2_+)}}{\langle\omega_i,\omega_i\rangle_{L^2(\Sp^2_+)}}\omega_i\\
     =&-\frac12\frac{\langle\Delta(\Delta+2)u(t), 1\rangle_{L^2(\Sp^2_+)}}{\langle 1,1\rangle_{L^2(\Sp^2_+)}}1-
     \frac12\sum_{i=1}^2\frac{\langle\Delta(\Delta+2) u(t),\omega_i\rangle_{L^2(\Sp^2_+)}}{\langle\omega_i,\omega_i\rangle_{L^2(\Sp^2_+)}}\omega_i\\
     =&-\frac12\frac{\langle u(t), \Delta(\Delta+2)1\rangle_{L^2(\Sp^2_+)}}{\langle 1,1\rangle_{L^2(\Sp^2_+)}}1-
     \frac12\sum_{i=1}^2\frac{\langle u(t),\Delta(\Delta+2)\omega_i\rangle_{L^2(\Sp^2_+)}}{\langle\omega_i,\omega_i\rangle_{L^2(\Sp^2_+)}}\omega_i\\
     =&0.
\end{align*}
Since $u^\parallel(t)\in\operatorname{span}(\phi_0,\phi_1,\phi_2)$ we have $\Delta(\Delta+2)u^\parallel(t)=0$ and hence $\dot u^\parallel=0$. So $u^\parallel(t)=u^\parallel(0)=u_0^\parallel$ and thus
\begin{equation}\label{SchauderCh_ParallelDecayEstiamte}
\|u^\parallel(t)\|_{C^{4,\gamma}(\Sp^2_+)}
=\|u_0^\parallel\|_{C^{4,\gamma}(\Sp^2_+)}.
\end{equation}
Next, we consider the evolution of $u^\perp=u-u^\parallel$. Using $\dot u^\parallel=\Delta(\Delta +2) u^\parallel=0$ we obtain
\begin{align*} 
\dot u^\perp+\frac12\Delta(\Delta+2) u^\perp
=\dot u+\frac12\Delta(\Delta+2) u
=0.
\end{align*}
Using the boundary conditions from Problem \eqref{SchauderCh_NullBVProblemonHalfSphere}, it is clear that $\partial_\nu u^\perp=\partial_\nu\Delta u^\perp=0$. Also  $u^\perp(0)=u_0^\perp$. We define $v(t):=e^{12t}u^\perp(t)$. It follows that 
\begin{equation}\label{SchauderCh_NullBVProblemonHalfSphere_PERPEXP}
\left\{
\begin{array}{ll}
\displaystyle\dot v+\frac12\Delta(\Delta+2)v=12v&\displaystyle\textrm{in $\Sp^2_+\times[0,T]$},\vspace{.15cm}\\
\displaystyle\frac{\partial v}{\partial\nu}=\frac{\partial\Delta v}{\partial\nu}=0&\displaystyle\textrm{along $\partial\Sp^2_+\times[0,T]$},\vspace{.15cm}\\
 \displaystyle v(0)=u_0^\perp&\displaystyle\textrm{on $\Sp^2_+$}.
\end{array}
\right.
\end{equation}

For $l=0,1,2$, the functions $\Delta^l v(t)$ are continuous and, by Theorem \ref{SchauderCh_neumannlaplacianResultsTheorem}, they can therefore be expanded into Neumann eigenfunctions. We claim
$$\Delta^l v(\cdot,t)=\sum_{k=0}^\infty c_k(t)(-\lambda_k)^l \phi_k\hspace{.5cm}\textrm{in $L^2(\Sp^2_+)$}.$$
Indeed by the boundary conditions $\partial_\nu v=\partial_\nu\Delta v=0$ from Problem \eqref{SchauderCh_NullBVProblemonHalfSphere_PERPEXP} we have $\langle\phi_k, \Delta^2 v(t)\rangle_{L^2(\Sp^2_+)}=-\lambda_k\langle \phi_k, \Delta v\rangle_{L^2(\Sp^2_+)}=\lambda_k^2 \langle v,\phi_k\rangle_{L^2(\Sp^2_+)}$. 
By definition, $v\perp\phi_0,\phi_1,\phi_2$ in $L^2(\Sp^2_+)$ and hence $c_0=c_1=c_2=0$. Additionally, $\lambda_k(\lambda_k-2)\geq 24$ for $k\geq 3$ since $\lambda_k\geq 6$ for such $k$. Consequently 
$$-\frac12\int_{\Sp^2_+}v(t)\Delta(\Delta +2) v(t)d\mu_{\Sp^2}=\frac12\sum_{k=2}^\infty\lambda_k(2-\lambda_k)|c_k(t)|^2\leq-12\sum_{k=2}^\infty |c_k(t)| ^2=-12\|v(t)\|_{L^2(\Sp^2_+)}^2.$$
Using this estimate, we compute
$$
    \frac12\frac d{dt}\int_{\Sp^2_+} v(t)^2d\mu_{\Sp^2}=\int_{\Sp^2_+} v(t)\dot v(t)d\mu_{\Sp^2}=\int_{\Sp^2_+} -v(t)\frac12\Delta(\Delta+2)v(t)+12v(t)^2d\mu_{\Sp^2}\leq 0.
$$
This implies $\|v(t)\|_{L^2(\Sp^2_+)}\leq \|v_0\|_{L^2(\Sp^2_+)}=\|u_0^\perp\|_{L^2(\Sp^2_+)}$ for all $t\in[0,T]$. Applying the estimate from Theorem \ref{SchauderCh_GoodAPrioriEstimateOnManifold} we deduce that for all $t\in[0,T]$ 
$$\|v(t)\|_{C^{4,\gamma}(\Sp^2_+)}\leq \|v\|_{C^{4,1,\gamma}(\Sp^2_+\times[0,T])}\leq C(\|u_0^\perp\|_{C^{4,\gamma}(\Sp^2_+)}+\sup_{t\in[0,T]}\|v(t)\|_{L^2(\Sp^2_+)})\leq C\|u_0^\perp\|_{C^{4,\gamma}(\Sp^2_+)}.$$
Recalling that $v(t)=e^{12t} u(t)$ we obtain 
\begin{equation}\label{SchauderCh_PerpPartEstimatae}
\|u^\perp(t)\|_{C^{4,\gamma}(\Sp^2_+)}\leq Ce^{-12 t}\|u_0^\perp\|_{C^{4,\gamma}(\Sp^2_+)}.
\end{equation}
Combining Estimates \eqref{SchauderCh_ParallelDecayEstiamte} and \eqref{SchauderCh_PerpPartEstimatae}, we deduce the theorem.
 \end{proof}

 \begin{korollar}\label{SchauderCh_InhomDecayKor}
     Let $T\geq 1$ and $u\in C^{4,1,\gamma}(\Sp^2_+\times[0,T])$ be the solution to Problem \eqref{SchauderCh_NullBVProblemonHalfSphere}. Then for all $t\in[0,T]$ 
\begin{align*} 
\|u(t)\|_{C^{4,\gamma}(\Sp^2_+)}\leq &C_1\left(e^{-12 t}\|u_0^\perp\|_{C^{4,\gamma}(\Sp^2_+)}+\|u_0^\parallel\|_{C^{4,\gamma}(\Sp^2_+)}\right)\\
&+C_2\left(\|f\|_{C^{0,0,\gamma}(\Sp^2_+\times[0,T])}+\|h_1\|_{C^{3,0,\gamma}(\partial\Sp^2_+\times[0,T])}+\|h_2\|_{C^{1,0,\gamma}(\partial\Sp^2_+\times[0,T])}\right)
\end{align*}
where $C_1=C_1(\Sp^2_+,\gamma)$ depends on $\Sp^2_+$ and $\gamma$ but importantly not on $T$ and $C_2=C_2(\Sp^2_+,\gamma,T)$ depends on $\Sp^2_+$ $\gamma$ and $T$.
 \end{korollar}
 \begin{proof}
     We study the elliptic problem
     \begin{equation}\label{SchauderCh_DecayElliptic}
     \left\{
\begin{array}{ll}
\displaystyle\Delta^2w_0=0&\displaystyle\textrm{in $\Sp^2_+$,}\vspace{.2cm}\\
\displaystyle\frac{\partial w_0}{\partial\nu}=h_1(\cdot,0),\ \frac{\partial\Delta w_0}{\partial\nu}=h_2(\cdot,0)&\displaystyle\textrm{along $\partial\Sp^2_+$,}\vspace{.2cm}\\
\displaystyle\int_{\Sp^2_+} w_0d\mu_{\Sp^2}=0.&
\end{array}
\right.
\end{equation}
By Corollary \ref{SchauderCh_EllipticProblemExistenceKorollar}, Problem \eqref{SchauderCh_DecayElliptic} has a unique solution $w_0\in C^{4,\gamma}(\Sp^2_+)$. By Corollary \ref{SchauderCh_EllipticSchauderEstimate_Korollar}, it satisfies the Schauder estimate 
\begin{equation}\label{SchauderCh_DecayEllipticEstimate}
\|w_0\|_{C^{4,\gamma}(\Sp^2_+)}\leq C\left(
\|h_1(\cdot,0)\|_{C^{3,\gamma}(\partial\Sp^2_+)}
+
\|h_2(\cdot,0)\|_{C^{1,\gamma}(\partial\Sp^2_+)}
\right).
\end{equation}
We consider the parabolic problem 
\begin{equation}\label{SchauderCh_NullBVProblemonHalfSphere2}
\left\{
\begin{array}{ll}
\displaystyle \dot {w}+\frac12\Delta(\Delta+2)w=f&\displaystyle \textrm{in $\Sp^2_+\times[0,T]$},\vspace{.15cm}\\
\displaystyle \frac{\partial  w}{\partial\nu}=h_1,\ \ \frac{\partial\Delta w}{\partial\nu}=h_2&\displaystyle \textrm{along $\partial\Sp^2_+\times[0,T]$},\vspace{.15cm}\\
 \displaystyle  w(0)= w_0&\displaystyle \textrm{on $\Sp^2_+$}.
\end{array}
\right.
\end{equation}
By construction of $w_0$, the necessary compatibility conditions are satisfied. So, by Theorem \ref{SchauderCh_FinalExistenceThmParabolic}, Problem \eqref{SchauderCh_NullBVProblemonHalfSphere2} has a unique solution $w\in C^{4,1,\gamma}(\Sp^2_+\times[0,T])$. Combining the Schauder estimate provided by Corollary \ref{SchauderCh_SchauderEstimateonmanifolds} with Estimate \eqref{SchauderCh_DecayEllipticEstimate}, we deduce that $w$ satisfies the Schauder estimate
\begin{equation}\label{SchauderCh_NullBVProblemonHalfSphere2Estiamte}
\begin{aligned}
\|w\|_{C^{4,1,\gamma}(\Sp^2_+\times[0,T])}\leq C_2&\bigg(\|f\|_{C^{0,0,\gamma}(\Sp^2_+\times[0,T])}+\|h_1\|_{C^{3,0,\gamma}(\partial\Sp^2_+\times[0,T])}\\
&\hspace{1cm}+\|h_2\|_{C^{1,0,\gamma}(\partial\Sp^2_+\times[0,T])}\bigg),
\end{aligned}
\end{equation}
where $C_2=C_2(\Sp^2_+,\gamma, T)$. Next, we put $v_0:=u_0-w_0$ and consider the parabolic equation 
\begin{equation}\label{SchauderCh_NullBVProblemonHalfSphere1}
\left\{
\begin{array}{ll}
\displaystyle \dot {v}+\frac12\Delta(\Delta+2)v=0&\displaystyle \textrm{in $\Sp^2_+\times[0,T]$},\vspace{.15cm}\\
\displaystyle \frac{\partial v}{\partial\nu}=0,\ \ \frac{\partial\Delta v}{\partial\nu}=0&\displaystyle \textrm{along $\partial\Sp^2_+\times[0,T]$},\vspace{.15cm}\\
 \displaystyle v(0)=v_0&\displaystyle \textrm{on $\Sp^2_+$}.
\end{array}
\right.
\end{equation}
By construction of $v_0$, the necessary compatibility conditions are satisfied. So, by Theorem \ref{SchauderCh_FinalExistenceThmParabolic}, Problem \eqref{SchauderCh_NullBVProblemonHalfSphere1} has a unique solution $v\in C^{4,1,\gamma}(\Sp^2_+\times[0,T])$, which, by Theorem \ref{SchauderCh_DecayEstimate}, satisfies the decay estimate 
\begin{equation}\label{SchauderCh_NullBVProblemonHalfSphere1Estiamte}
    \|v(t)\|_{C^{4,\gamma}(\Sp^2_+)}\leq C_1\left(e^{-12 t}\|v_0^\perp\|_{C^{4,\gamma}(\Sp^2_+)}+\|v_0^\parallel\|_{C^{4,\gamma}(\Sp^2_+)}\right),
\end{equation}
where $C_1=C_1(\Sp^2_+,\gamma)$ but independent of $T$. Let $*\in\set{\perp,\parallel}$. We use Estimate \eqref{SchauderCh_DecayEllipticEstimate} to estimate 
\begin{align}
\|v_0^*\|_{C^{4,\gamma}(\Sp^2_+)}&\leq \|u_0^*\|_{C^{4,\gamma}(\Sp^2_+)}+\|w_0^*\|_{C^{4,\gamma}(\Sp^2_+)}\nonumber\\
&\leq \|u_0^*\|_{C^{4,\gamma}(\Sp^2_+)}+C\left(\|h_1\|_{C^{3,0,\gamma}(\partial\Sp^2_+\times[0,T])}+\|h_2\|_{C^{1,0,\gamma}(\partial\Sp^2_+\times[0,T])}\right).\label{SchauderCh_ComponentParallelPerpSPlitEstimate}
\end{align}
Since $u=v+w$, the corollary follows by combining Estimates \eqref{SchauderCh_NullBVProblemonHalfSphere2Estiamte}, \eqref{SchauderCh_NullBVProblemonHalfSphere1Estiamte} and \eqref{SchauderCh_ComponentParallelPerpSPlitEstimate}.
 \end{proof}

\bibliographystyle{plain}
\bibliography{quellen}

\begin{thebibliography}{10}

\bibitem{adn1}
Shmuel Agmon, Avron Douglis, and Louis Nirenberg.
\newblock Estimates near the boundary for solutions of elliptic partial
  differential equations satisfying general boundary conditions. i.
\newblock {\em Communications on pure and applied mathematics}, 12(4):623--727,
  1959.

\bibitem{adn2}
Shmuel Agmon, Avron Douglis, and Louis Nirenberg.
\newblock Estimates near the boundary for solutions of elliptic partial
  differential equations satisfying general boundary conditions ii.
\newblock {\em Communications on pure and applied mathematics}, 17(1):35--92,
  1964.

\bibitem{AK}
Roberta Alessandroni and Ernst Kuwert.
\newblock {L}ocal solutions to a free boundary problem for the {W}illmore
  functional.
\newblock {\em Calculus of Variations and Partial Differential Equations},
  55(2):1--29, 2016.

\bibitem{atkinson2012spherical}
Kendall Atkinson and Weimin Han.
\newblock {\em Spherical harmonics and approximations on the unit sphere: an
  introduction}, volume 2044.
\newblock Springer Science \& Business Media, 2012.

\bibitem{eidelman1999parabolic}
Samuil~D Eidelman and Nicolae~V Zhitarashu.
\newblock {\em Parabolic boundary value problems}, volume 101.
\newblock Springer Science \& Business Media, 1999.

\bibitem{eidelman}
{\. S}amuil~D. Ejdel’man.
\newblock Parabolic equations.
\newblock In {\em Partial Differential Equations VI}, pages 203--316. Springer,
  1994.

\bibitem{evans}
Lawrence~C Evans.
\newblock {\em Partial differential equations}, volume~19.
\newblock American Mathematical Society, 2022.

\bibitem{friedman2008partial}
Avner Friedman.
\newblock {\em Partial differential equations of parabolic type}.
\newblock Courier Dover Publications, 2008.

\bibitem{gilbarg1977elliptic}
David Gilbarg, Neil~S Trudinger, David Gilbarg, and NS~Trudinger.
\newblock {\em Elliptic partial differential equations of second order}.
\newblock Number~2 in 224. Springer, 1977.

\bibitem{hirzebruch}
Friedrich Hirzebruch and Winfried Scharlau.
\newblock {\em Einf{\"u}hrung in die Funktionalanalysis}, volume 296.
\newblock BI-Wissenschaftsverlag, 1971.

\bibitem{ladyvzenskaja1988linear}
Olga~A Lady{\v{z}}enskaja, Vsevolod~Alekseevich Solonnikov, and Nina~N
  Ural'ceva.
\newblock {\em Linear and quasi-linear equations of parabolic type}, volume~23.
\newblock American Mathematical Soc., 1988.

\bibitem{LammBiharmonisch}
Tobias Lamm.
\newblock {Biharmonischer Wärmefluß}.
\newblock Master's thesis, {University of Freiburg}, 2001.

\bibitem{LeeRiem}
John~M Lee.
\newblock {\em Introduction to Riemannian manifolds}, volume~2.
\newblock Springer, 2018.

\bibitem{lee2012smooth}
John~M Lee and John~M Lee.
\newblock {\em Introduction to Smooth Manifolds}.
\newblock Springer, 2012.

\bibitem{Metsch}
Jan-Henrik Metsch.
\newblock On the area-preserving {W}illmore flow of small bubbles sliding on a
  domain’s boundary.
\newblock Preprint, 2022.

\bibitem{renardy2006introduction}
Michael Renardy and Robert~C Rogers.
\newblock {\em An introduction to partial differential equations}, volume~13.
\newblock Springer Science \& Business Media, 2006.

\bibitem{schauder01}
Juliusz Schauder.
\newblock {\"U}ber lineare elliptische differentialgleichungen zweiter ordnung.
\newblock {\em Mathematische Zeitschrift}, 38(1):257--282, 1934.

\bibitem{simon}
Leon Simon.
\newblock Schauder estimates by scaling.
\newblock {\em Calculus of Variations and Partial Differential Equations},
  5(5):391--407, 1997.

\end{thebibliography}

\end{document}